\documentclass[11pt,english,bibliography=totoc]{extarticle}
\usepackage[T1]{fontenc}
\usepackage[latin9]{inputenc}
\usepackage{geometry}
\usepackage{color}
\usepackage{babel}
\usepackage{float}
\usepackage{mathrsfs}
\usepackage{mathtools}
\usepackage{enumitem}
\usepackage{amsmath}
\usepackage{amsthm}
\usepackage{amssymb}
\usepackage{graphicx}
\usepackage{wasysym}
\usepackage[authoryear]{natbib}
\PassOptionsToPackage{normalem}{ulem}
\usepackage{ulem}

\usepackage{hyperref}

\hypersetup{
	unicode=true,
	pdfusetitle,
	bookmarks=true,
	bookmarksnumbered=false,
	bookmarksopen=false,
	breaklinks=false,
	pdfborder={0 0 0},
	pdfborderstyle={},
	backref=false,
	colorlinks=true
}
\usepackage{epstopdf}

\makeatletter

\floatstyle{ruled}
\newfloat{algorithm}{tbp}{loa}
\providecommand{\algorithmname}{Algorithm}
\floatname{algorithm}{\protect\algorithmname}

\numberwithin{figure}{section}
\theoremstyle{definition}
\newtheorem{defn}{\protect\definitionname}
\theoremstyle{plain}
\newtheorem{prop}{\protect\propositionname}
\theoremstyle{plain}
\newtheorem{assumption}{\protect\assumptionname}
\theoremstyle{remark}
\newtheorem{rem}{\protect\remarkname}
\theoremstyle{plain}
\newtheorem{thm}{\protect\theoremname}
\theoremstyle{plain}
\newtheorem{cor}{\protect\corollaryname}
\theoremstyle{plain}
\newtheorem{lem}{\protect\lemmaname}

\usepackage{dsfont}
\usepackage{relsize}
\usepackage{subdepth}
\usepackage{xcolor}
\allowdisplaybreaks

\usepackage{url}

\DeclareFontFamily{OT1}{pzc}{}
\DeclareFontShape{OT1}{pzc}{m}{it}{<-> s * [1.200] pzcmi7t}{}
\DeclareMathAlphabet{\mathpzc}{OT1}{pzc}{m}{it}

\usepackage[noadjust]{cite}
\usepackage{filecontents}

\usepackage{stfloats}

\usepackage{babel}

\renewcommand\footnotemark{}

\usepackage{bigints}

\usepackage{scalerel}

\usepackage[framemethod=tikz]{mdframed}

\newmdenv[
  hidealllines=true,
  backgroundcolor=blue!10,
  innerleftmargin=8pt,
  innerrightmargin=8pt,
  innertopmargin=0pt,
  innerbottommargin=6pt,
  leftmargin=-0pt,
  rightmargin=-0pt
]{shadedbox}

\usepackage[titles]{tocloft}
\definecolor{airforceblue}{rgb}{0.36, 0.54, 0.66}
\definecolor{ballblue}{rgb}{0.13, 0.67, 0.8}
\hypersetup{citecolor=ballblue}

\definecolor{alizarin}{rgb}{0.82, 0.1, 0.26}
\definecolor{asparagus}{rgb}{0.53, 0.66, 0.42}
\definecolor{applegreen}{rgb}{0.55, 0.71, 0.0}
\definecolor{armygreen}{rgb}{0.29, 0.33, 0.13}
\definecolor{amber(sae/ece)}{rgb}{1.0, 0.49, 0.0}
\definecolor{coquelicot}{rgb}{1.0, 0.22, 0.0}
\definecolor{ao(english)}{rgb}{0.0, 0.5, 0.0}

\newcommand{\CaseStretch}{1.2}

\renewcommand*\env@cases[1][\CaseStretch]{%
  \let\@ifnextchar\new@ifnextchar
  \left\lbrace
  \def\arraystretch{#1}%
  \array{@{}l@{\quad}l@{}}%
}

\usepackage{anyfontsize}
\usepackage{cleveref}

\makeatother

\providecommand{\assumptionname}{Assumption}
\providecommand{\corollaryname}{Corollary}
\providecommand{\definitionname}{Definition}
\providecommand{\lemmaname}{Lemma}
\providecommand{\propositionname}{Proposition}
\providecommand{\remarkname}{Remark}
\providecommand{\theoremname}{Theorem}

\begin{document}

\title{\textbf{\vspace{-25pt}
}\\
\textbf{Recursive Optimization of Convex Risk Measures:}\\
\textbf{Mean-Semideviation Models}}

\author{Dionysios S. Kalogerias and Warren B. Powell\thanks{The Authors are with the Department of Operations Research \& Financial
Engineering (ORFE), Princeton University, Sherrerd Hall, Charlton
Street, Princeton, NJ 08544, USA. e-mail: \{dkalogerias, powell\}@princeton.edu.}}

\maketitle
\textbf{\vspace{-30pt}
}
\begin{abstract}
We develop recursive, data-driven, stochastic subgradient methods
for optimizing a new, versatile, and application-driven class of convex
risk measures, termed here as \textit{mean-semidevi- ations}, strictly
generalizing the well-known and popular mean-upper-semideviation.
We introduce the $\textit{MESSAGE}^{p}$ \textit{algorithm}, which
is an efficient compositional subgradient procedure for iteratively
solving \textit{convex} mean-semideviation risk-averse problems \textit{to
optimality}, and constitutes a \textit{parallel} variation of the
recently developed, general purpose \textit{$T$-SCGD} \textit{algorithm}
of Yang, Wang \& Fang \citep{Wang2018}. We analyze the asymptotic
behavior of the $\textit{MESSAGE}^{p}$ algorithm under a flexible
and structure-exploiting set of problem assumptions, which reveal
a well-defined trade-off between the expansiveness of the random cost
and the smoothness of the mean-semideviation risk measure under consideration.
In particular:
\begin{itemize}
\item Under appropriate stepsize rules, we establish pathwise convergence
of the $\textit{MESSAGE}^{p}$ algorithm in a strong technical sense,
confirming its asymptotic consistency. 
\item Assuming a \textit{strongly convex} \textit{cost}, we show that, for
fixed semideviation order $p>1$ and for $\epsilon\in\left[0,1\right)$,
the $\textit{MESSAGE}^{p}$ algorithm achieves a squared-${\cal L}_{2}$
solution suboptimality rate of the order of ${\cal O}(n^{-\left(1-\epsilon\right)/2})$
iterations, where, for $\epsilon>0$, pathwise convergence is \textit{simultaneously}
guaranteed. This result establishes a rate of order arbitrarily close
to ${\cal O}(n^{-1/2})$, while ensuring strongly stable pathwise
operation. For $p\equiv1$, the rate order improves to ${\cal O}(n^{-2/3})$,
which also suffices for pathwise convergence, and matches previous
results. 
\item Likewise, in the general case of a \textit{convex cost}, we show that,
for any $\epsilon\in\left[0,1\right)$, the $\textit{MESSAGE}^{p}$
algorithm \textit{with iterate smoothing} achieves an ${\cal L}_{1}$
objective suboptimality rate of the order of ${\cal O}(n^{-\left(1-\epsilon\right)/\left(4\mathds{1}_{\left\{ p>1\right\} }+4\right)})$.
This result provides maximal rates of ${\cal O}(n^{-1/4})$, if $p\equiv1$,
and ${\cal O}(n^{-1/8})$, if $p>1$, matching the state of the art,
as well.
\end{itemize}
Finally, we discuss the superiority of the proposed framework for
convergence, as compared to that employed earlier in \citep{Wang2018},
within the risk-averse context under consideration. By performing
careful analysis and by constructing non-trivial counterexamples,
we explicitly demonstrate that the class of mean-semideviation problems
supported herein is \textit{strictly larger} than the respective class
of problems supported in \citep{Wang2018}. As a result, this work
establishes the applicability of compositional stochastic optimization
for a significantly and strictly wider spectrum of convex mean-semideviation
risk-averse problems, as compared to the state of the art. This fact
justifies the purpose of our work from this perspective, as well.
\end{abstract}
\textbf{\textit{$\quad$}}\textbf{Keywords.} Risk-Averse Optimization,
Risk-Aware Learning, Risk Measures, Mean-Upper-Semi- deviation, Stochastic
Optimization, Stochastic Gradient Descent, Compositional Optimization.

\newpage{}

\hypersetup{linkcolor=ao(english)}\tableofcontents{}

\hypersetup{linkcolor=red}

\newpage{}

\section{\label{sec:Introduction}Introduction}

During the last almost twenty years, many significant advances have
been made in the now relatively mature area of risk-averse modeling
and optimization. These primarily include the fundamental axiomatization
and theoretical characterization of risk functionals, also commonly
known as risk measures \citep{Kijima1993,Rockafellar1997,Artzner1999,Ogryczak1999,Ogryczak2002,Rockafellar2002,Rockafellar2003,Rockafellar2006,Ruszczynski2006a,ShapiroLectures_2ND},
as well as extensive analysis in the context of risk-averse stochastic
programs in both static and sequential decision making problem settings
\citep{Rockafellar1997,Follmer2002,Rockafellar2003,Rockafellar2006,Ruszczynski2006,Collado2012,Cavus2014,Asamov2015,Dentcheva2017,Grechuk2017,Shapiro2017,Fan2018}.
The importance of building a well structured theory of risk is motivated
by its natural and intuitive relevance to problems from a large variety
of applied domains. Arguably the oldest, archetypical application
of risk is in Finance \citep{Kijima1993,Rockafellar1997,Andersson2001,Krokhmal2001,Chen2008,Shang2018},
which has decisively driven pioneering research in risk-averse modeling
and optimization, from its very birth, probably dating back to the
work of Markowitz \citep{Markowitz1952}, to present. Other applications
of risk may be found in both classical and contemporary domains such
as Energy \citep{Moazeni2015,Bruno2016,Jiang2016a}, Wireless Networks
\citep{Ma2018}, Inventory Optimization \citep{Ahmed2007,Chen2007,Xinsheng2015}
and Supply Chain Management \citep{Gan2004,Sawik2016}, to name a
few.

Most recently, the development of effective \textit{computational
methods} for applying risk-averse optimization to actual problems
has also been attracting considerable attention; see, e.g., \citep{Ruszczynski2010,Cavus2014a,Moazeni2017,Tamar2017,Dentcheva2017a,W.Huang2017,Jiang2017,Yu2018}.
This line of work can be divided between sequential settings \citep{Cavus2014a,Moazeni2017,Tamar2017,W.Huang2017,Jiang2017,Yu2018},
and static settings \citep{Tamar2017,Dentcheva2017a}, for a variety
of different problem characteristics. Computational recipes also vary.
\textit{For instance}, \citep{Ruszczynski2010} and \citep{Cavus2014a}
develop and analyze variations of the well known value and policy
iteration algorithms of risk-neutral dynamic programming; \citep{Moazeni2017}
proposes a method for risk-averse nonstationary direct parametric
policy search for finite horizon problems; \citep{Tamar2017}, \citep{Dentcheva2017a}
and \citep{Yu2018} rely on the so-called \textit{Sample Average Approximation
(SAA)} approach \citep{ShapiroLectures_2ND}, where an appropriately
constructed empirical estimate of the original objective is used as
a \textit{surrogate} to that of the original stochastic program, assuming
existence of a sufficiently large sample of the processes introducing
uncertainty into the corresponding risk-averse objective; \citep{W.Huang2017}
and \citep{Jiang2017} consider an \textit{Approximate Dynamic Programming
(ADP)} \citep{Powell2004} approach, where sequential finite state/action
risk-averse stochastic programs are tackled via stochastic approximation
\citep{Kushner2003}.

Following this recent trend, this paper proposes and rigorously analyzes\textit{
recursive stochastic subgradient methods} for an important class of
\textit{static, convex} risk-averse stochastic programs. In a nutshell,
we make the following contributions:
\begin{enumerate}
\item Following the Mean-Risk Model paradigm \citep{ShapiroLectures_2ND},
we introduce a new class of \textit{convex} risk measures, called
\textit{mean-semideviations}. These \textit{strictly generalize} the
well known mean-upper-semideviation risk measure, and are constructed
by replacing the positive part \textit{weighting} function of the
latter by another nonlinear map, termed here as a \textit{risk regularizer},
obeying certain properties. Mean-semideviations share the same core
analytical structure with the mean-upper-semideviation risk measure;
however, they are much more versatile in applications. We study mean-semideviations
in terms of their basic properties, and we present a fundamental \textit{constructive}
characterization result, demonstrating their generality. Specifically,
we show that the class of all mean-semideviation risk measures is
\textit{almost} \textit{in one-to-one correspondence} with the class
of cumulative distribution functions (cdfs) of all integrable random
variables. This result provides an \textit{analytical device} for
constructing mean-semideviations with desirable characteristics, starting
from any cdf of the aforementioned type. The flexibility and effectiveness
of mean-semideviations are explicitly demonstrated on a classical,
chance-constrained newsvendor model, as well.
\item We introduce the $\textit{MESSAGE}^{p}$ \textit{(MEan-Semideviation
Stochastic compositionAl subGradient dEscent of order $p$) algorithm},
an efficient, \textit{data-driven} Stochastic Subgradient Descent
(SSD) -type procedure for iteratively solving \textit{convex} mean-semideviation
risk-averse problems \textit{to optimality}. The $\textit{MESSAGE}^{p}$
algorithm constitutes a \textit{parallel} variation of general purpose\textit{
T-level Stochastic Compositional Gradient Descent} \textit{($T$-SCGD)}
algorithm, recently developed in \citep{Wang2018}, under a generic
theoretical framework. Although risk-averse optimization is listed
in \citep{Wang2018} as a potential application of stochastic compositional
optimization for the mere case of mean-upper-semideviations, this
work is the first to propose a general algorithm, applicable to any
mean-semideviation model of choice.
\item We analyze the asymptotic behavior of the $\textit{MESSAGE}^{p}$
algorithm under a new, flexible and \textit{structure-exploiting}
set of problem assumptions, which reveal a well-defined trade-off
between the expansiveness of the random cost and the smoothness of
the mean-semideviation risk measure under consideration. In particular,
under our proposed structural framework:
\begin{itemize}
\item Under appropriate stepsize rules, we establish pathwise convergence
of the $\textit{MESSAGE}^{p}$ algorithm in a strong technical sense,
confirming its asymptotic consistency. 
\item Assuming a \textit{strongly convex} cost function, the convergence
rate of the $\textit{MESSAGE}^{p}$ algorithm is studied in detail.
More specifically, we show that, for fixed semideviation order $p>1$
and for $\epsilon\in\left[0,1\right)$, the $\textit{MESSAGE}^{p}$
algorithm achieves a squared-${\cal L}_{2}$ solution suboptimality
rate of the order of ${\cal O}(n^{-\left(1-\epsilon\right)/2})$ iterations,
where, for $\epsilon>0$, pathwise convergence is \textit{simultaneously}
guaranteed. Thus, this new result establishes a rate of order arbitrarily
close to ${\cal O}(n^{-1/2})$, also ensuring strongly stable pathwise
operation of the $\textit{MESSAGE}^{p}$ algorithm. In the simpler
case where the semideviation order is chosen as $p\equiv1$, the rate
order of the proposed algorithm improves to ${\cal O}(n^{-2/3})$,
which is sufficient for pathwise convergence as well, and matches
previous results in the related literature \citep{Wang2017}. 
\item For the general case of a \textit{convex cost}, we show that, for
any $\epsilon\in\left[0,1\right)$, the $\textit{MESSAGE}^{p}$ algorithm
\textit{with iterate smoothing} achieves an ${\cal L}_{1}$ objective
suboptimality rate of the order of ${\cal O}(n^{-\left(1-\epsilon\right)/\left(4\mathds{1}_{\left\{ p>1\right\} }+4\right)})$.
As in the strongly convex case, for $\epsilon>0$, pathwise convergence
is also \textit{simultaneously} guaranteed. For $\epsilon\equiv0$,
this result provides maximal rates of ${\cal O}(n^{-1/4})$, if $p\equiv1$,
and ${\cal O}(n^{-1/8})$, if $p>1$, matching the state of the art,
as well.
\end{itemize}
\item We discuss the superiority of the proposed framework for convergence,
as compared to that employed earlier in \citep{Wang2018}, within
the risk-averse context under consideration. By performing careful
analysis and by constructing non-trivial counterexamples, we explicitly
demonstrate that the class of mean-semideviation problems supported
herein is \textit{strictly larger} than the respective class of problems
supported in \citep{Wang2018}. As a result, this paper establishes
the applicability of compositional stochastic optimization for a significantly
and strictly wider spectrum of convex mean-semideviation risk-averse
problems, as compared to the state of the art. This fact justifies
the purpose of our work from this perspective, as well.
\end{enumerate}
Our contributions, briefly outlined above, are now discussed in greater
detail. We also briefly explain how our work relates to and is placed
within the existing literature.

\subsection{Mean-Semideviation Risk Measures}

\textit{Mean-semideviation risk measures}, as proposed and developed
in this work, constitute a new class of risk measures where, given
a random cost, the corresponding dispersion measure (the term penalizing
the ``mean'' part of a mean-risk functional) is defined as the ${\cal L}_{p}$-norm
of a nonlinear, one-dimensional map of the \textit{centered cost},
or, in other words, \textit{its central deviation}. This map is called
a \textit{risk regularizer}, and possesses certain analytical properties:
convexity, nonnegativity, monotonicity and nonexpansiveness. Dispersion
measures with this structure are suggestively called \textit{generalized
semideviations}. 

This terminology originates from the presence of the positive part
function $\left(\cdot\right)_{+}\triangleq\max\left\{ \cdot,0\right\} $,
which is the simplest, prototypical example of a risk regularizer,
in the corresponding dispersion measure of the well known mean-upper-semideviation
risk measure \citep{ShapiroLectures_2ND}, i.e., the upper-(central)-semideviation.
Mean-semideviations are much more versatile, however, since different
choices for the involved risk regularizer correspond to different
rules for ranking the relative effect of both riskier (higher than
the mean) and less risky (lower than the mean) events, corresponding
to specific regions in the range of the (centered) cost. As a result,
the choice of the risk regularizer affects the general quality and
the roughness/stability of an optimal random cost, in a decision making
setting. Consequently, owing to their versatility, mean-semideviations
are practically appealing as well, because they are parametrizable
and they may incorporate domain specific knowledge more easily than
the rigid mean-upper-semideviation.

In this work, after we formulate simple conditions for the existence
of mean-semideviation risk measures, we study their basic geometric
properties, such as convexity and monotonicity. Contrary to the mean-upper-semideviation
alone, mean-semideviations are \textit{not} coherent risk measures,
in general (as a class), because they do not satisfy positive homogeneity
\citep{ShapiroLectures_2ND}. This is due to the potential nonhomogeneity
of the risk regularizer involved. They do satisfy convexity, monotonicity
and translation equivariance, though and, therefore, they belong to
the class of \textit{convex risk measures}, \citep{Follmer2002,ShapiroLectures_2ND},
and that of \textit{convex-monotone risk measures}, as well. 

Further, we present a fundamental \textit{constructive} characterization
result, demonstrating the generality of mean-semideviations. Specifically,
on the one hand, this result shows that the class of all mean-semideviation
risk measures is \textit{almost} \textit{in one-to-one correspondence}
with the class of cdfs of all integrable random variables (on the
line). On the other, it provides an \textit{analytical device} for
constructing such risk measures from any cdf of the aforementioned
type. Although not studied in this paper, this correspondence between
mean-semideviations and cdfs might be of interest in other areas related
to stochastically robust optimization such as stochastic dominance;
see, for instance, the seminal articles \citep{Ogryczak1999,Ogryczak2002}
for some interesting connections.

Our discussion on mean-semideviation risk measures is concluded by
a demonstration of their practical usefulness and flexibility on a
classical, chance-constrained newsvendor model. After we briefly analyze
the structure of the problem under consideration, we put risk regularizers
-each inducing a mean-semideviation risk measure- \textit{in context},
and we explicitly discuss their construction, so that the resulting
mean-semideviation risk measure best reflects problem characteristics,
and the objectives of the decision maker. Additionally, we present
numerical simulations, experimentally confirming the effectiveness
of the proposed risk-averse approach. Our simulations also reveal
some interesting features of the resulting risk-averse solutions,
which we further discuss.

\textbf{\textit{Relation to the Literature:}} We are not the first
to propose convex risk measures featuring nonlinear weighting functions;
see, for instance, \citep{Kijima1993,Chen2011,Fu2017}. In particular,
the recent article \citep{Fu2017} considers risk measures defined
as a nonlinearly weighted, order-$1$ (lower) semideviation \textit{from
a fixed target} (see, for instance, Example 6.25 in \citep{ShapiroLectures_2ND}),
focusing mainly on their applications on a portfolio selection model.
In \citep{Fu2017}, the corresponding weighting function shares the
same properties as a risk regularizer (see above), except for nonexpansiveness.
However, our proposed mean-semideviation risk measures are substantially
different and structurally more complex compared to the risk measures
proposed in \citep{Fu2017}. The main reason is the presence of the
\textit{expected cost}, rather than a fixed target, in the definition
of mean-semideviations; for more details, compare (\citep{Fu2017},
Definition 1) with Section \ref{sec:Mean-Semideviation-Models} herein.

\subsection{Recursive Optimization of Mean-Semideviations}

The main contribution of this work concerns efficient optimization
of mean-semideviations, measuring convexly parameterized random cost
functions, over a closed and convex set. We introduce and rigorously
analyze the $\textit{MESSAGE}^{p}$ \textit{(MEan-Semideviation Stochastic
compositionAl subGradient dEscent of order $p$) algorithm} (Algorithm
\ref{alg:SCGD-3} in Section \ref{subsec:MESSAGE}), which constitutes
an efficient Stochastic Subgradient Descent (SSD) -type procedure
for iteratively solving our base problem \textit{to optimality}. The
$\textit{MESSAGE}^{p}$ algorithm may be seen as a parameterized (relative
to the choice of the risk regularizer), \textit{parallel} variation\textit{
}of the general purpose \textit{T-Level Stochastic Compositional Gradient
Descent} \textit{($T$-SCGD)} \textit{algorithm}, presented and analyzed
very recently in \citep{Wang2018} under generic assumptions. In turn,
the \textit{$T$-SCGD} algorithm is a natural generalization of the
\textit{Basic 2-Level SCGD }algorithm, presented and analyzed earlier
in \citep{Wang2017}. A key feature of the aforementioned \textit{compositional}
stochastic subgradient schemes is the existence of more than one ($T$,
in general), \textit{pairwise coupled} stochastic approximation updates,
or \textit{levels}, \textit{each with a dedicated stepsize}, which
are executed \textit{concurrently }through the operation of the algorithm.
In the case of the $\textit{MESSAGE}^{p}$ algorithm, there exist
three such levels (that is, $T\equiv3$), and this results naturally,
due our specific problem structure. However, contrary to the \textit{$T$-SCGD}
algorithm, all three stochastic approximation levels of the $\textit{MESSAGE}^{p}$
algorithm \textit{are executed completely in parallel within every
iteration}, presenting additional operational efficiency, potentially
important in various applications.

Pathwise convergence and convergence rate analyses of the \textit{$T$-SCGD
}algorithm are presented in \citep{Wang2018}, and \citep{Wang2017}
(where, in the latter, $T\equiv2$). However, the respective structural
framework considered in both \citep{Wang2018} and \citep{Wang2017},
when applied to the problem class considered in this work, imposes
\textit{significant restrictions} in regard to the possible choice
of the risk regularizer, partially related to the expansiveness and
smoothness (or roughness) of the involved random cost function. This
fact significantly limits the type of problems the \textit{$T$-SCGD
}algorithm is provably applicable to, \textit{at least} within the
class of risk-averse problems introduced and studied herein. For example,
when $p\equiv1$, arguably the most popular regularizer $\left(\cdot\right)_{+}$,
leading to the mean-upper-semideviation risk measure, is \textit{not}
supported within the framework of \citep{Wang2017,Wang2018}. This
is because nonsmooth risk regularizers \textit{exhibiting corner points},
such as $\left(\cdot\right)_{+}$, apparently have \textit{discontinuous
subderivatives}, whereas the respective assumptions made in \citep{Wang2017,Wang2018}
essentially require the respective risk regularizer to be \textit{not
only everywhere differentiable}, but to have\textit{ Lipschitz derivatives},
as well. This shortcoming of the theoretical framework of \citep{Wang2017,Wang2018}
naturally carries over to higher values of the semideviation order,
$p$. Naturally, the theoretical narrowness of \citep{Wang2017,Wang2018}
motivates closer study of any compositional subgradient algorithm
whatsoever, one that would exploit the special characteristics of
a mean-semideviation risk measure. The ultimate goal is the development
of a sufficiently general theoretical framework, which will justify
the compositional optimization approach for the whole class of mean-semideviation
risk measures, under as weak structural assumptions as possible.

Following this direction, and focusing on optimizing mean-semideviation
models, we present a new and flexible set of problem assumptions,
\textit{substantially weaker than those employed in} \citep{Wang2017,Wang2018},
under which we analyze the asymptotic behavior of the $\textit{MESSAGE}^{p}$
algorithm, proposed in our work. Our framework carefully exploits
the structure of mean-semideviations, and presents a probably fundamental,
though practically useful, trade-off between the expansiveness of
the random cost function and the smoothness of the chosen risk regularizer,
in a very well-defined sense. As previously outlined, our results
are restated, as follows.

\textit{First}, under appropriate stepsize rules, we establish pathwise
convergence of the $\textit{MESSAGE}^{p}$ algorithm in the same strong
sense as in \citep{Wang2017,Wang2018}, thus confirming its asymptotic
consistency. 

\textit{Second}, assuming a \textit{strongly convex} cost function,
we study the convergence rate of the $\textit{MESSAGE}^{p}$ algorithm,
in detail. More specifically, we show that, for fixed semideviation
order $p>1$ and for any choice of $\epsilon\in\left[0,1\right)$,
the $\textit{MESSAGE}^{p}$ algorithm achieves a squared-${\cal L}_{2}$
solution suboptimality rate of the order of ${\cal O}(n^{-\left(1-\epsilon\right)/2})$
iterations. Here, $\epsilon$ is a user-specified parameter, which
directly affects stepsize selection. If, additionally, $\epsilon$
is chosen to be strictly positive, that is, for $\epsilon>0$, pathwise
convergence is \textit{simultaneously} guaranteed. This completely
novel result establishes a convergence rate of order arbitrarily close
to ${\cal O}(n^{-1/2})$ as $\epsilon\rightarrow0$, while ensuring
strongly stable pathwise operation of the algorithm. In the structurally
simpler case where $p\equiv1$, the rate order improves to ${\cal O}(n^{-2/3})$,
which is sufficient for pathwise convergence as well, and matches
existing results in compositional stochastic optimization, developed
earlier along the lines of \citep{Wang2017}.

\textit{Third}, for the general case of a convex cost function, we
show that, for any $\epsilon\in\left[0,1\right)$, the $\textit{MESSAGE}^{p}$
algorithm \textit{with iterate smoothing} achieves an ${\cal L}_{1}$
objective suboptimality rate of the order of ${\cal O}(n^{-\left(1-\epsilon\right)/\left(4\mathds{1}_{\left\{ p>1\right\} }+4\right)})$.
As in the strongly convex case, for $\epsilon>0$, pathwise convergence
is also \textit{simultaneously} guaranteed. For $\epsilon\equiv0$,
this result provides maximal rates of ${\cal O}(n^{-1/4})$, if $p\equiv1$,
and ${\cal O}(n^{-1/8})$, if $p>1$, matching the state of the art,
as well \citep{Wang2017,Wang2018}. Although those rates may not be
particularly satisfying, they quantitatively demonstrate the remarkable
speedup achieved by assuming and leveraging strong convexity for the
analysis and operation of the $\textit{MESSAGE}^{p}$ algorithm. 

The proposed structural framework adequately mitigates the aforementioned
technical issues of that considered in \citep{Wang2017,Wang2018}.
For example, we show that, when the random cost function has \textit{bounded}
(random) subgradients and its distribution is generally well-behaved,
the choice of the risk regularizer can be\textit{ completely unconstrained},
\textit{regardless of the value of} $p\in\left[1,\infty\right)$.
As a result, under the new framework, the most popular candidate $\left(\cdot\right)_{+}$,
but also every risk regularizer exhibiting corner points, are now
valid choices (under appropriate conditions) for any $p$, contrary
to \citep{Wang2017,Wang2018}.

Finally, in order to show the superiority of our proposed framework
compared to that of \citep{Wang2017,Wang2018}, we present a detailed
analytical comparison, which rigorously demonstrates that the class
of mean-semideviation programs supported within this work \textit{contains}
the respective class of problems supported within \citep{Wang2017,Wang2018};
further,\textit{ the inclusion is strict}. Such comparison is made
possible by performing careful analysis and by constructing non-trivial,
\textit{non-cornercase} counterexamples. As a result, the applicability
of compositional stochastic optimization is established herein for
a significantly and strictly wider spectrum of convex mean-semideviation
risk-averse problems, as compared to the state of the art. This fact
justifies the purpose of our work from this perspective, in addition
to our algorithmic contribution, as well.

\textbf{\textit{Relation to the Literature: }}Apparently, the results
presented in this work are related to those developed in \citep{Wang2017,Wang2018},
for a generic problem setting. Indeed, as already stated, optimization
of mean-upper-semideviation risk measures has been briefly identified
in \citep{Wang2017,Wang2018} as a potential application of the compositional
algorithms proposed therein. However, as mentioned above, the assumptions
on problem structure employed in \citep{Wang2017,Wang2018} are too
restrictive to adequately study the class of mean-semideviation risk
measures introduced herein, which include the mean-upper-semideviation
as a single member of this class. Except for the aforementioned works,
and as also discussed above, there is a significant line of research
considering the SAA approach to risk-averse stochastic optimization,
both from a fundamental, theoretical perspective \citep{Shapiro2013,Guigues2016,Dentcheva2017a}
and from the computational one \citep{Dentcheva2017a,Tamar2017}.
As noted in \citep{Wang2017,Wang2018}, the compositional, SSD-type
optimization algorithms analyzed in this paper present some major
natural advantages over the SAA approach. First, the $\textit{MESSAGE}^{p}$
algorithm solves the \textit{original} risk-averse stochastic program
\textit{asymptotically to optimality}, whereas, in the SAA approach,
the corresponding SAA surrogate to the original program is solved,
producing only an approximate solution; as the number of the sample
increases the solution to the SAA surrogate approaches that of the
original stochastic program, in some well defined sense \citep{Shapiro2013,Dentcheva2017a}.
Second, because of its nature, the SAAs cannot exploit new information
available to the decision maker, so that they can improve their decisions,
based on those made so far; in fact, the SAA surrogate needs to be
redefined using new available information, and then solved afresh.
Of course, the $\textit{MESSAGE}^{p}$ algorithm efficiently exploits
new information, due to its \textit{recursive}, sequential nature.
Third, as a result of the above, SAAs are not suitable for settings
where \textit{information is available sequentially}, and decisions
have to be made adaptively over time. Fourth, SAAs might often require
a very large number of samples for producing accurate approximations
to the optimal decisions corresponding to the original problem, and
this might result in optimization problems whose objective is computationally
difficult to evaluate. For more details on this, see \citep{Wang2017}.
On the contrary, the $\textit{MESSAGE}^{p}$ algorithm is iterative
in nature, and presents \textit{minimal} and \textit{fixed} time and
space complexity per iteration.

\subsection*{\textit{Organization of the Paper}}

The rest of the paper is organized as follows. Section \ref{sec:Problem-Setting}
establishes the stochastic risk-averse convex programming setting
under study, and provides some elementary, albeit necessary preliminaries
on the theory of risk measures. In Section \ref{sec:Mean-Semideviation-Models},
we constructively introduce the class of mean-semideviation risk measures,
we study their existence and their structural properties, we discuss
specific examples, and we develop our above mentioned fundamental
characterization result. Section \ref{sec:Message} is devoted to
the development and analysis of the $\textit{MESSAGE}^{p}$ algorithm,
under our proposed theoretical framework for convergence, and includes
the rigorous comparison of our results with those presented in \citep{Wang2018}.
Finally, Section \ref{sec:Conclusion} concludes the paper.

\textbf{\textit{Note:}} Some longer proofs of the theoretical results
presented in the paper in the form of Theorems, Lemmata and Propositions
are excluded from the main body of the paper for clarity in the exposition,
and are presented in Section \ref{sec:Appendix:-Proofs} (Appendix).

\subsection*{\textit{Notation \& Definitions}}

Matrices and vectors will be denoted by boldface uppercase and boldface
lowercase letters, respectively. Calligraphic letters and formal script
letters will generally denote sets and $\sigma$-algebras, respectively,
except for clearly specified exceptions. The operator $\left(\cdot\right)^{\boldsymbol{T}}$
will denote vector transposition. The $\ell_{p}$-norm of $\boldsymbol{x}\in\mathbb{R}^{n}$
is $\left\Vert \boldsymbol{x}\right\Vert _{p}\triangleq\left(\sum_{i=1}^{n}\left|x\left(i\right)\right|^{p}\right)^{1/p}$,
for all $\mathbb{N}\ni p\ge1$. Similarly, the ${\cal L}_{p}$ norm
of an appropriately measurable function $f\left(\cdot\right)$ will
be $\left\Vert f\right\Vert _{{\cal L}_{p}}\triangleq\left(\int\left|f\left(x\right)\right|^{p}\mathrm{d}\mu\left(x\right)\right)^{1/p}$
for $p\in\left[1,\infty\right)$, and $\left\Vert f\right\Vert _{{\cal L}_{\infty}}\triangleq\mathrm{ess\hspace{1bp}sup}_{x}\left|f\left(x\right)\right|$,
where the reference measure $\mu$ will be clearly specified by the
context. The finite $N$-dimensional identity operator will be denoted
as ${\bf I}_{N}$. Additionally, we define $\mathbb{N}^{+}\triangleq\left\{ 1,2,\ldots\right\} $,
$\mathbb{N}_{n}^{+}\triangleq\left\{ 1,2,\ldots,n\right\} $ and $\mathbb{N}_{n}\triangleq\left\{ 0\right\} \cup\mathbb{N}_{n}^{+}$,
for $n\in\mathbb{N}^{+}$.

If $\Omega$ denotes a \textit{base} sample space and $F:\mathbb{R}^{N}\times\Omega\rightarrow\mathbb{R}$
(referring directly to $\Omega$), then, for the sake of clarity,
we sometimes drop dependence on $\omega\in\Omega$, and write simply
$F\left(\boldsymbol{x},\omega\right)\equiv F\left(\boldsymbol{x}\right)$
(clear by the context).

For every set ${\cal X}\subseteq\mathbb{R}^{N}$, which is nonempty,
closed and convex, the \textit{Euclidean projection onto} ${\cal X}$,
$\Pi_{{\cal X}}:\mathbb{R}^{N}\rightarrow{\cal X}$ is defined, as
usual, as $\Pi_{{\cal X}}\left(\boldsymbol{x}\right)\triangleq\mathrm{arg\hspace{1bp}min}_{\widetilde{\boldsymbol{x}}\in{\cal X}}\left\Vert \widetilde{\boldsymbol{x}}-\boldsymbol{x}\right\Vert _{2}$,
for all $\boldsymbol{x}\in\mathbb{R}^{N}$. Euclidean projections,
as defined above, always exist and are nonexpansive operators.

For every real-valued function $f:\mathbb{R}^{N}\rightarrow\mathbb{R}$,
which is differentiable at a point $\boldsymbol{x}\in\mathbb{R}^{N}$,
the vector $\nabla f\left(\boldsymbol{x}\right)\in\mathbb{R}^{N}$
denotes its \textit{gradient} \textit{at} $\boldsymbol{x}$. If, additionally,
$f$ is differentiable on ${\cal X}\subseteq\mathbb{R}^{N}$, the
function $\nabla f:{\cal X}\rightarrow\mathbb{R}^{N}$ denotes its
\textit{gradient function}, mapping each $\boldsymbol{x}\in{\cal X}$
to $\nabla f\left(\boldsymbol{x}\right)$.

If $f$ is nonsmooth and convex, its \textit{subdifferential} is the
closed-valued multifunction $\partial f:\mathbb{R}^{N}\rightrightarrows\mathbb{R}^{N}$,
defined, for every $\boldsymbol{x}\in\mathbb{R}^{N}$, as the set
of all gradients each corresponding to a linear underestimator of
$f$, or, in other words,
\begin{equation}
\partial f\left(\boldsymbol{x}\right)\triangleq\left\{ \left.\boldsymbol{y}_{\boldsymbol{x}}\in\mathbb{R}^{N}\right|f\left(\boldsymbol{z}\right)\ge f\left(\boldsymbol{x}\right)+\boldsymbol{y}_{\boldsymbol{x}}^{\boldsymbol{T}}\left(\boldsymbol{z}-\boldsymbol{x}\right),\quad\forall\boldsymbol{z}\in\mathbb{R}^{N}\right\} ,\quad\forall\boldsymbol{x}\in\mathbb{R}^{N}.
\end{equation}
\textit{A} \textit{subgradient (function) of $f$}, suggestively denoted
as $\underline{\nabla}f:\mathbb{R}^{N}\rightarrow\mathbb{R}^{N}$,
is defined as \textit{any} \textit{selection} of the subdifferential
multifunction $\partial f$, that is, for every $\boldsymbol{x}\in\mathbb{R}^{N}$,
it is true that $\underline{\nabla}f\left(\boldsymbol{x}\right)\in\partial f\left(\boldsymbol{x}\right)$;
for brevity, we write $\underline{\nabla}f\in\partial f$. For fixed
$\boldsymbol{x}\in\mathbb{R}^{N}$, $\underline{\nabla}f\left(\boldsymbol{x}\right)$
will be called a \textit{subgradient of $f$ at} $\boldsymbol{x}$.

\section{\label{sec:Problem-Setting}Problem Setting \& Preliminaries}

We now formally introduce the problem of interest in this work. Henceforth,
all subsequent probabilistic statements will presume the existence
of a \textit{common} probability space $\left(\Omega,\mathscr{F},{\cal P}\right)$.
We refer to $\left(\Omega,\mathscr{F},{\cal P}\right)$ as the \textit{base
space}. We place no topological restrictions on the sample space $\Omega$.
However, in order for some mild technicalities to be easily resolved,
we conveniently assume that $\left(\Omega,\mathscr{F},{\cal P}\right)$
constitutes a complete measure space.

Let $F:\mathbb{R}^{N}\times\mathbb{R}^{M}\rightarrow\mathbb{R}$ be
a bivariate real-valued mapping, such that, for every $\boldsymbol{x}\in\mathbb{R}^{N}$,
the function $F\left(\boldsymbol{x},\cdot\right)$ is $\mathscr{\mathscr{B}}\left(\mathbb{R}^{M}\right)$-measurable
and, for every $\boldsymbol{w}\in\mathbb{R}^{M}$, the path $F\left(\cdot,\boldsymbol{w}\right)$
is \textit{(real-valued)} \textit{convex (and subdifferentiable)}.
Also, for a given $\mathscr{F}$-measurable (in general) random element
$\boldsymbol{W}:\Omega\rightarrow\mathbb{R}^{M}$, consider the composite
function $\widetilde{F}:\mathbb{R}^{N}\times\Omega\rightarrow\mathbb{R}$,
defined as
\begin{equation}
\widetilde{F}\left(\cdot,\omega\right)\triangleq F\left(\cdot,\boldsymbol{W}\left(\omega\right)\right),\quad\forall\omega\in\Omega.
\end{equation}
It easily follows that, for every $\boldsymbol{x}\in\mathbb{R}^{N}$,
the function $\widetilde{F}\left(\boldsymbol{x},\cdot\right)\equiv F\left(\boldsymbol{x},\boldsymbol{W}\left(\cdot\right)\right)$
is an $\mathscr{F}$-measurable (in general), real-valued random variable.
We additionally assume that, for every $\boldsymbol{x}\in\mathbb{R}^{N}$,
$\widetilde{F}\left(\boldsymbol{x},\cdot\right)$ belongs to the Lebesgue
space ${\cal L}_{q}$ for some fixed choice of $q\in\left[1,\infty\right]$,
relative to the base measure ${\cal P}$, that is, $\widetilde{F}\left(\boldsymbol{x},\cdot\right)\in{\cal L}_{q}\left(\Omega,\mathscr{F},{\cal P};\mathbb{R}\right)\triangleq{\cal Z}_{q}$.
Of course, if, for every $\boldsymbol{x}\in\mathbb{R}^{N}$, $F\left(\boldsymbol{x},\cdot\right)\in{\cal L}_{q}\left(\mathbb{R}^{M},\mathscr{\mathscr{B}}\left(\mathbb{R}^{M}\right),{\cal P}_{\boldsymbol{W}};\mathbb{R}\right)$,
where ${\cal P}_{\boldsymbol{W}}$ is the Borel pushforward of $\boldsymbol{W}$,
then $\widetilde{F}\left(\boldsymbol{x},\cdot\right)\in{\cal Z}_{q}$,
as well. Hereafter, $F\left(\cdot,\boldsymbol{W}\right)$ will be
referred to as a \textit{random cost function}. 

With the term \textit{risk measure}, we refer to some fixed and known
\textit{real-valued} \textit{functional} on the Banach space ${\cal Z}_{q}$
\citep{ShapiroLectures_2ND}. Among all risk measures on ${\cal Z}_{q}$,
we pay special attention to those exhibiting the following basic structural
characteristics.
\begin{defn}
\textbf{(Convex-Monotone Risk Measures)}\label{def:1} A real valued
functional on ${\cal Z}_{q}$, $\rho:{\cal Z}_{q}\rightarrow\mathbb{R}$,
is called a \textit{convex-monotone risk measure}, if and only if
it satisfies the following conditions:
\begin{description}
\item [{$\mathbf{R1}$}] $\,\,\:$\textit{(Convexity)}: For every $Z_{1}\in{\cal Z}_{q}$
and $Z_{2}\in{\cal Z}_{q}$, it is true that
\begin{equation}
\rho\left(\alpha Z_{1}+\left(1-\alpha\right)Z_{2}\right)\le\alpha\rho\left(Z_{1}\right)+\left(1-\alpha\right)\rho\left(Z_{2}\right),
\end{equation}
for all $\alpha\in\left[0,1\right]$.
\item [{$\mathbf{R2}$}] $\,\,\:$\textit{(Monotonicity)}: For every $Z_{1}\in{\cal Z}_{q}$
and $Z_{2}\in{\cal Z}_{q}$, such that $Z_{1}\left(\omega\right)\ge Z_{2}\left(\omega\right)$,
for ${\cal P}$-almost all $\omega\in\Omega$, it is true that $\rho\left(Z_{1}\right)\ge\rho\left(Z_{2}\right)$.
\end{description}
\end{defn}
For a possibly convex-monotone risk measure $\rho:{\cal Z}_{q}\rightarrow\mathbb{R}$
(following Assumption \ref{def:1}), we will be interested in the
``static'' stochastic program
\begin{equation}
\boxed{\begin{array}{rl}
\underset{\boldsymbol{x}}{\mathrm{minimize}} & \rho\left(\widetilde{F}\left(\boldsymbol{x},\cdot\right)\right)\equiv\rho\left(F\left(\boldsymbol{x},\boldsymbol{W}\right)\right)\triangleq\phi^{\widetilde{F}}\left(\boldsymbol{x}\right)\\
\mathrm{subject\,to} & \boldsymbol{x}\in{\cal X}
\end{array},}\label{eq:Prog_1}
\end{equation}
where the set of feasible decisions ${\cal X}\subseteq\mathbb{R}^{N}$
is assumed to be closed and convex.

Under the standard problem setting outlined above, it is straightforward
to formulate the following elementary result, provided here without
proof, and for completeness.
\begin{prop}
\textbf{\textup{(Convexity of Risk-Function Compositions \citep{ShapiroLectures_2ND})\label{prop:Convexity-of-Compositions}}}
Consider a real-valued random function $f:\mathbb{R}^{N}\times\Omega\rightarrow\mathbb{R}$,
as well as a real-valued risk measure $\rho:{\cal Z}_{q}\rightarrow\mathbb{R}$.
Suppose that, for every $\omega\in\Omega$, $f\left(\cdot,\omega\right)$
is convex and that $\rho$ is convex-monotone. Then, the real-valued
composite function $\phi^{f}\left(\cdot\right)\equiv\text{\ensuremath{\rho}}\left(f\left(\cdot,\bullet\right)\right):\mathbb{R}^{N}\rightarrow\mathbb{R}$
is convex.
\end{prop}
Proposition \ref{prop:Convexity-of-Compositions} shows that, under
the respective assumptions, (\ref{eq:Prog_1}) constitutes a convex
mathematical program in standard form. Thus, application of a subgradient
method would require that some selection of the subdifferential multifunction
$\partial\phi^{\widetilde{F}}$ can be evaluated at will, at any $\boldsymbol{x}\in{\cal X}$.
However, for most choices of the random cost function $F\left(\cdot,\boldsymbol{W}\right)$
and of the risk measure $\rho$, even the composition $\phi^{\widetilde{F}}\left(\cdot\right)\equiv\rho\left(F\left(\cdot,\boldsymbol{W}\right)\right)$
is impossible to be evaluated exactly, let alone (a selection of)
$\partial\phi^{\widetilde{F}}$. Instead, we may be given either realizations
of the random exogenous information $\boldsymbol{W}$, or direct evaluations
of $F\left(\cdot,\boldsymbol{W}\right)$ and a subgradient, $\underline{\nabla}F\left(\cdot,\boldsymbol{W}\right)$,
at some \textit{test decision candidate} $\boldsymbol{x}$. It might
also be desirable that decision making is performed \textit{sequentially
over time}, where decisions are updated adaptively as new information
arrives. Such settings motivate the consideration of SSD-type algorithms
for solving (\ref{eq:Prog_1}), which are of main interest in this
paper.

Some basic assumptions follow, fairly standard in the literature of
stochastic approximation \citep{ShapiroLectures_2ND,Wang2017,Wang2018,Kushner2003}.
To this end, let us formally introduce the elementary concept of a
\textit{IID process}. Then, our assumptions follow.
\begin{defn}
\textbf{(IID Process)}\label{def:2} A stochastic sequence $\left\{ \boldsymbol{W}^{n}\right\} _{n\in\mathbb{N}^{+}}$
is called \textit{IID} if and only if it consists of statistically
independent, $\mathbb{R}^{M}$-valued random elements, identically
distributed according to a fixed Borel measure ${\cal P}_{\boldsymbol{W}}$.
\end{defn}
\begin{assumption}
\textbf{\textup{(Availability of Information)}}\label{assu:1} \textbf{Either
one, or more, }mutually independent, IID sequences are available \textbf{sequentially},
all distributed according to ${\cal P}_{\boldsymbol{W}}$.
\end{assumption}
\begin{rem}
Note that in Assumption \ref{assu:1} we\textit{ do not }require that
the process $\boldsymbol{W}^{n}$ is actually \textit{observable}
to the user, but only \textit{available}, either in the form of a
data stream, or by simulation.\hfill{}\ensuremath{\blacksquare}
\end{rem}
\begin{assumption}
\textbf{\textup{(Existence of an ${\cal SO}$)}}\label{assu:2} There
exists a mechanism, called a \textbf{Sampling Oracle} (${\cal SO}$),
which, given $\boldsymbol{x}\in{\cal X}$ and $\boldsymbol{w}\in\mathbb{R}^{M}$,
returns either $F\left(\boldsymbol{x},\boldsymbol{w}\right)$, or
$\underline{\nabla}F\left(\boldsymbol{x},\boldsymbol{w}\right)$,
a subgradient of $F$ relative to $\boldsymbol{x}$, or both. It is
further assumed that the ${\cal SO}$ has \textbf{direct access} to
all available information streams, according to Assumption \ref{assu:1}.
\end{assumption}
In this work, we propose and analyze efficient algorithms for solving
(\ref{eq:Prog_1}) under Assumptions \ref{assu:1} and \ref{assu:2},
and \textit{explicitly assuming no prior knowledge} of either the
random cost function $F\left(\cdot,\boldsymbol{W}\right)$, or its
respective subgradients. We will be restricting our attention to a
new class of convex-monotone risk measures with, however, wide applicability,
and whose general structure follows the so-called \textit{Mean-Risk
Model} (\citep{ShapiroLectures_2ND}, Section 6.2). This special class
of risk measures is introduced and analyzed, in detail, in Section
\ref{sec:Mean-Semideviation-Models}.

\section{\label{sec:Mean-Semideviation-Models}Mean-Semideviation Models}

Under the Mean-Risk Model paradigm \citep{ShapiroLectures_2ND}, a
risk measure $\rho:{\cal Z}_{q}\rightarrow\mathbb{R}$ is defined,
for each random cost $Z\in{\cal Z}_{q}$, as
\begin{equation}
\rho\left(Z\right)\triangleq\mathbb{E}\left\{ Z\right\} +c\mathbb{D}\left\{ Z\right\} ,\label{eq:Mean_Risk}
\end{equation}
where the functional $\mathbb{D}:{\cal Z}_{q}\rightarrow\mathbb{R}$
constitutes a \textit{dispersion measure}, and provided that the respective
quantities are well defined, for the particular choice of $q\in\left[1,\infty\right]$.
The dispersion measure $\mathbb{D}$ may be conveniently thought as
a \textit{penalty}, weighted by the penalty multiplier $c\ge0$, \textit{effectively
quantifying the uncertainty} of the particular cost $Z$. 

In this section, we introduce a special class of dispersion measures,
which constitute natural generalizations of the well-known upper semideviation
of order $p$ \citep{ShapiroLectures_2ND}. This new class of dispersion
measures is termed here as \textit{generalized semideviations}. Reasonably
enough, risk measures of the form of (\ref{eq:Mean_Risk}), where
the respective dispersion measure constitutes a generalized semideviation
will be called either \textit{mean-semideviation risk measures}, or,
interchangeably, \textit{mean-semideviation models}, or, simply, \textit{mean-semideviations.} 

This section is structured as follows. First, the simple notion of
a \textit{risk regularizer} is introduced; risk regularizers constitute
the basic building block of generalized semideviations. The basic
properties of risk regularizers are concisely presented, and a formal
definition of generalized semideviations is also formulated, along
with a brief discussion related to their practical relevance. Mean-semideviation
risk measures are then formally introduced, along with their basic
properties, and specific examples are discussed, highlighting their
versatility. Next, we develop a constructive characterization result,
essentially showing that the class of all mean-semideviation risk
measures is \textit{almost} \textit{in one-to-one correspondence}
with the class of cumulative distribution functions of all integrable
random variables (on the line). This result readily demonstrates an
apparent generality of mean-semideviations, as well. Lastly, the usefulness,
flexibility and effectiveness of mean-semideviation risk measures
are demonstrated on a classical, chance-constrained newsvendor model.
In particular, risk regularizers (each inducing a mean-semideviation
risk measure) are put \textit{in context}, and their construction
is explicitly discussed, reflecting the special characteristics of
the specific newsvendor problem under consideration, and the objectives
of the decision maker.

\subsection{Basic Concepts}

We start by introducing the concept of a \textit{risk regularizer}.
Risk regularizers are simple, real-valued functions of one variable,
which are reasonably structured, so that they, on the one hand, can
be used to quantify risk (see below) and, on the other, can result
in problems which can be solved efficiently and exactly via convex
stochastic optimization.
\begin{defn}
\textbf{(Risk Regularizers)}\label{def:3} A real-valued function
${\cal R}:\mathbb{R}\rightarrow\mathbb{R}$ is called a \textit{risk
regularizer}, if it satisfies the following conditions: 
\begin{description}
\item [{$\mathbf{S1}$}] $\,\,\:$${\cal R}$ is \textit{convex}.
\item [{$\mathbf{S2}$}] $\,\,\:$${\cal R}$ is \textit{nonnegative}.
\item [{$\mathbf{S3}$}] $\,\,\:$${\cal R}$ is \textit{nondecreasing.}
\item [{$\mathbf{S4}$}] $\,\,\:$For every $\alpha\ge0$, it is true that
${\cal R}\left(x+\alpha\right)\le{\cal R}\left(x\right)+\alpha$,
for all $x\in\mathbb{R}$.
\end{description}
\end{defn}
Fig. \ref{fig:Examples-of-RRs} illustrates the shapes of various
risk regularizers, other than the arguably most obvious example of
the positive part function $\left(\cdot\right)_{+}$. Note that a
risk regularizer need not be smooth (a trivial example is $\left(\cdot\right)_{+}$);
several of the examples of Fig. \ref{fig:Examples-of-RRs} are indeed
nonsmooth, with the respective corner points highlighted by black
dots.

Risk regularizers of Definition \ref{def:3} may be further structurally
characterized via the following simple result.
\begin{prop}
\textbf{\textup{(Characterization of ${\cal R}$)\label{prop:CharacterizationR}}}
Consider a real-valued function ${\cal R}:\mathbb{R}\rightarrow\mathbb{R}$,
satisfying condition $\mathbf{S3}$ of Definition \ref{def:4}. Then,
condition $\mathbf{S4}$ holds if and only if ${\cal R}$ is nonexpansive.
\end{prop}
\begin{proof}[Proof of Proposition \ref{prop:CharacterizationR}]
First, assume that condition $\mathbf{S4}$ holds. Then, by the fact
that ${\cal R}$ is nondecreasing ($\mathbf{S3}$), it is true that
\begin{flalign}
\left|{\cal R}\left(x\right)-{\cal R}\left(y\right)\right| & \equiv\left({\cal R}\left(x\right)-{\cal R}\left(y\right)\right)\mathds{1}_{\left\{ x\ge y\right\} }+\left({\cal R}\left(y\right)-{\cal R}\left(x\right)\right)\mathds{1}_{\left\{ x<y\right\} }\nonumber \\
 & \equiv\left({\cal R}\left(y+\left(x-y\right)\right)-{\cal R}\left(y\right)\right)\mathds{1}_{\left\{ x\ge y\right\} }+\left({\cal R}\left(x+\left(y-x\right)\right)-{\cal R}\left(x\right)\right)\mathds{1}_{\left\{ x<y\right\} }\nonumber \\
 & \le\left(x-y\right)\mathds{1}_{\left\{ x\ge y\right\} }+\left(y-x\right)\mathds{1}_{\left\{ x<y\right\} }\equiv\left|x-y\right|,
\end{flalign}
for all $\left(x,y\right)\in\mathbb{R}^{2}$, showing that ${\cal R}$
is a nonexpansive map. \textit{Conversely}, assume that ${\cal R}$
is nonexpansive. Then, for any $\alpha\ge0$, it is true that
\begin{flalign}
0\le{\cal R}\left(x+\alpha\right)-{\cal R}\left(x\right) & \equiv\left|{\cal R}\left(x+\alpha\right)-{\cal R}\left(x\right)\right|\nonumber \\
 & \le\left|x+\alpha-x\right|\equiv\alpha,
\end{flalign}
for all $x\in\mathbb{R}$, verifying condition $\mathbf{S4}$.
\end{proof}
At this point, let us emphasize the elementary fact that, because
of convexity, every (real-valued) risk-regularizer must also be \textit{differentiable
almost everywhere}, relative to the Lebesgue measure on the Borel
space $\left(\mathbb{R},\mathscr{B}\left(\mathbb{R}\right)\right)$.
This also follows either by monotonicity, or due to the fact that
a risk regularizer is nonexpansive and, therefore, Lipschitz continuous
on $\mathbb{R}$. Further, because of convexity, the set of Lebesgue
measure zero of points in $\mathbb{R}$, where a risk regularizer
is \textit{nondifferentiable}, is \textit{at most countable}.

The class all possible risk regularizers\textit{ induces} that of
\textit{generalized semideviations}, which constitute the class of
dispersion measures considered in this paper. The definition of a
generalized semideviation is presented below.
\begin{defn}
\textbf{(Generalized Semideviations)}\label{def:4} Fix $p\in\left[1,\infty\right)$
and choose a risk regularizer ${\cal R}:\mathbb{R}\rightarrow\mathbb{R}$.
A dispersion measure $\mathbb{D}_{p}^{{\cal R}}:{\cal Z}_{q}\rightarrow\mathbb{R}$
is called a \textit{generalized semideviation of order $p$,} if and
only if, for $Z\in{\cal Z}_{q}$,
\begin{equation}
\mathbb{D}_{p}^{{\cal R}}\left\{ Z\right\} \triangleq\left(\mathbb{E}\left\{ \left({\cal R}\left(Z-\mathbb{E}\left\{ Z\right\} \right)\right)^{p}\right\} \right)^{1/p}\equiv\left\Vert {\cal R}\left(Z-\mathbb{E}\left\{ Z\right\} \right)\right\Vert _{{\cal L}_{p}},
\end{equation}
where it is assumed that all involved quantities are well defined
and finite.
\end{defn}
\begin{figure}
\centering\includegraphics[bb=78bp 14bp 812bp 400bp,clip,scale=0.625]{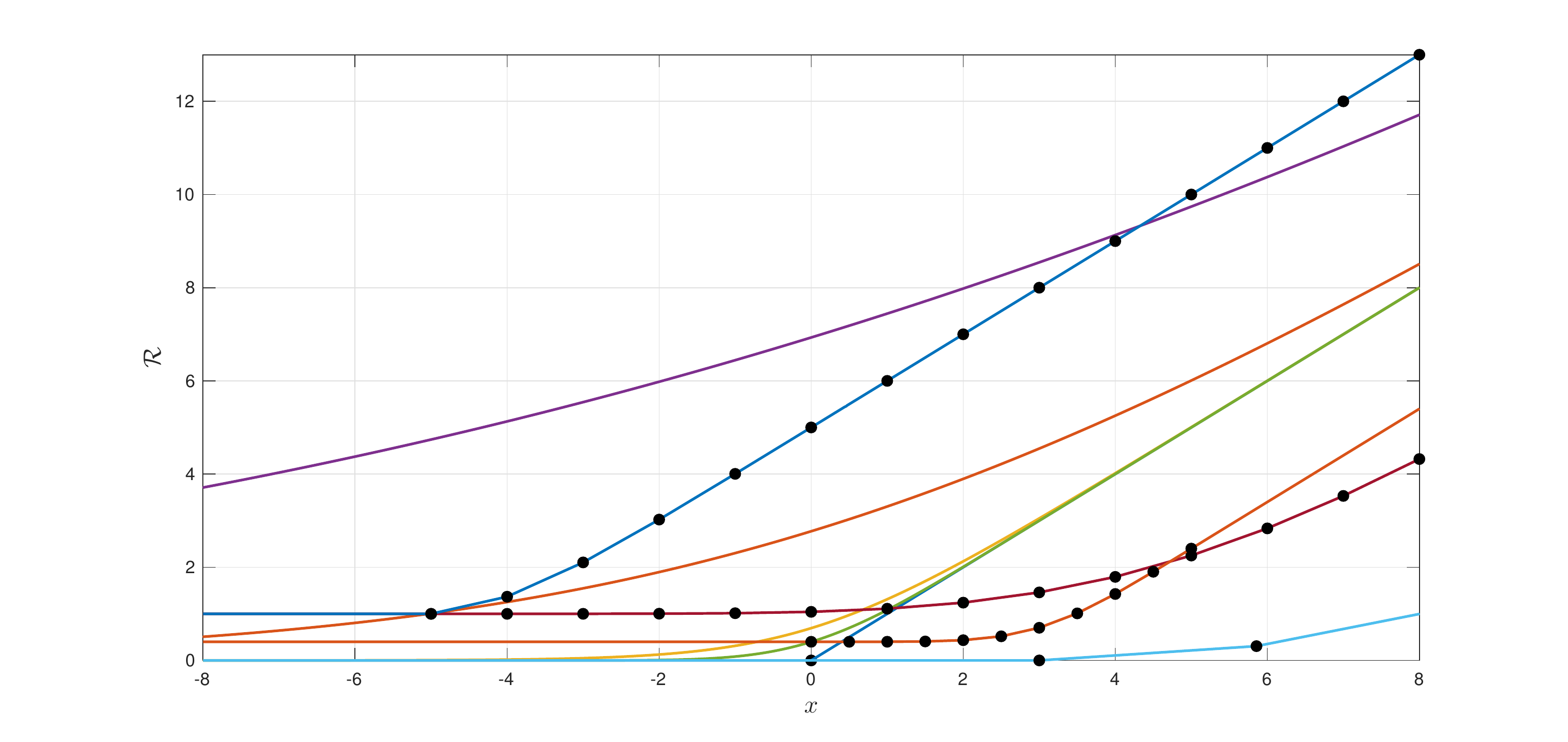}\caption{\label{fig:Examples-of-RRs}Some examples of both smooth and nonsmooth
risk regularizers. Black dots highlight the respective corner points
of nondifferentiability (some imperceptible).}

\vspace{-5bp}
\end{figure}
The power of generalized semideviations is in the fact that they form
a parametric family relative to the choice of the risk regularizer
${\cal R}$; different risk regularizers correspond to different rules
for ranking the relative effect of both riskier (higher than the mean)
and less risky (lower than the mean) events, corresponding to specific
regions in the range of the cost. For more details, see Section \ref{sec:Mean-Semideviation-Models},
where we illustrate the versatility of generalized semideviations
via additional examples, considering various specific choices for
${\cal R}$, with the well known upper-semideviation dispersion measure
\citep{ShapiroLectures_2ND} being the prototypical representative
of this class.

\subsection{Mean-Semideviations: Definition, Existence \& Structure}

Utilizing the concept of generalized semideviations, we may now introduce
the class of risk measures of central interest in this work, as follows.
\begin{defn}
\textbf{(Mean-Semideviation Risk Measures)}\label{def:5} Fix $p\in\left[1,\infty\right)$
and choose a risk regularizer ${\cal R}:\mathbb{R}\rightarrow\mathbb{R}$.
The \textit{mean-semideviation of order $p$, induced by ${\cal R}$,
}or\textit{ $MS_{p}^{{\cal R}}$}, for short, is the real-valued risk
measure defined, for $Z\in{\cal Z}_{q}$, as\footnote{A mean-semideviation risk measure will be denoted either as $\rho_{p}^{{\cal R}}\left(Z;c\right)$,
which is proper, or $\rho\left(Z\right)$, which is simpler, as long
as the choices of $p,{\cal R}$ and $c$ are clearly specified. }
\begin{flalign}
\rho\left(Z\right)\equiv\rho_{p}^{{\cal R}}\left(Z;c\right) & \triangleq\mathbb{E}\left\{ Z\right\} +c\mathbb{D}_{p}^{{\cal R}}\left\{ Z\right\} ,
\end{flalign}
where $c\ge0$ constitutes a fixed penalty multiplier, and provided
that all involved quantities are well defined and finite.
\end{defn}
Next, we state and prove a small number of relatively simple results,
related to the existence of mean-semideviation risk measures, introduced
in Definition \ref{def:5}, as well as their functional structure.
First, as it might be expected, we show that mean-semideviation risk
measures of order $p$ may be naturally associated with costs which
are also in ${\cal L}_{p}$ (i.e., choosing $p\equiv q$). Recall
that, throughout the paper, $p$ is reserved for specifying the order
of the mean-semideviation risk measure under consideration, whereas
$q$ is related to the integrability of the respective cost.
\begin{prop}
\textbf{\textup{(Compatibility of $p$'s and $q$'s)\label{prop:P=000026Q}}}
Fix $p\in\left[1,\infty\right)$, $c\ge0$, and choose any risk regularizer
${\cal R}:\mathbb{R}\rightarrow\mathbb{R}$. Then, as long as $q\ge p$,
the $MS_{p}^{{\cal R}}$ risk measure $\rho_{p}^{{\cal R}}\left(\cdot;c\right)$
is well-defined and finite, for every $Z\in{\cal Z}_{q}$. 
\end{prop}
\begin{proof}[Proof of Proposition \ref{prop:P=000026Q}]
Since $q\ge p\ge1$, it is trivial that $Z\in{\cal Z}_{1}$, simply
due to the inclusion ${\cal Z}_{1}\supset{\cal Z}_{2}\supset\ldots$,
for any choice of $q$. Thus, the expectation of every $Z\in{\cal Z}_{q}$
exists and is finite, and what remains is to prove the result for
the dispersion measure $\mathbb{D}_{p}^{{\cal R}}$.

For simplicity, let $q\equiv p$. Using the fact that $\mathbb{E}\left\{ Z\right\} $
is finite, it is true that, for every $Z\in{\cal Z}_{p}$, the shifted
cost $Z-\mathbb{E}\left\{ Z\right\} $ is in ${\cal Z}_{p}$. It thus
suffices to show that, for every $Z-\mathbb{E}\left\{ Z\right\} \triangleq X\in{\cal Z}_{p}$,
${\cal R}\left(X\right)$ is in ${\cal Z}_{p}$, as well. Because
the risk regularizer ${\cal R}$ is nonnegative (condition $\mathbf{S2}$),
the integral $\mathbb{E}\left\{ \left({\cal R}\left(X\right)\right)^{p}\right\} $
exists. Also, due to condition $\mathbf{S4}$ of Definition \ref{def:3},
it follows that, for every $x\ge0$, ${\cal R}\left(x\right)\le{\cal R}\left(0\right)+x,$
and since ${\cal R}$ is nondecreasing ($\mathbf{S3}$), it is true
that ${\cal R}\left(x\right)\le{\cal R}\left(0\right)+\left|x\right|$,
for all $x\in\mathbb{R}$. Setting $x\equiv X$, this yields
\begin{equation}
0\le{\cal R}\left(X\right)\le{\cal R}\left(0\right)+\left|X\right|,
\end{equation}
and since $X\in{\cal Z}_{p}$, ${\cal R}\left(0\right)+\left|X\right|\in{\cal Z}_{p}$,
as well. Consequently, it is true that
\begin{equation}
\left(\mathbb{E}\left\{ \left({\cal R}\left(X\right)\right)^{p}\right\} \right)^{1/p}\le\left(\mathbb{E}\left\{ \left({\cal R}\left(0\right)+\left|X\right|\right)^{p}\right\} \right)^{1/p}<+\infty,
\end{equation}
showing that $\mathbb{D}_{p}^{{\cal R}}$ and, therefore, $\rho_{p}^{{\cal R}}\left(\cdot;c\right)$,
are both well defined and finite, for every $Z\in{\cal Z}_{p}$.

Now, due to the inclusion ${\cal Z}_{1}\supset{\cal Z}_{2}\supset\ldots$,
we know that, if $Z\in{\cal Z}_{q}$, for some $q\ge p$, then $Z\in{\cal Z}_{p}$,
as well. Enough said.
\end{proof}
\textit{Hereafter}, for the sake of generality, \textit{we will implicitly
assume that $p$ and $q$ are compatible}, so that existence and finiteness
of the resulting risk measures considered is ensured. Of course, in
actual applications, Proposition \ref{prop:P=000026Q} may be directly
invoked on a case-by-case basis, in order to select the order of the
particular dispersion measure of choice, depending on the nature of
the random cost, or a family of those, under study.

After characterizing existence and finiteness of mean-semideviation
risk measures, as introduced in Definition \ref{def:5}, we focus
on their structural properties, from a functional point of view. As
the following result suggests, mean-semideviation risk measures are
indeed convex-monotone under a standardized assumption on the penalty
multiplier $c$. 
\begin{thm}
\textbf{\textup{(When are Mean-Semideviations Convex-Monotone?)\label{thm:ConvexMONO}}}
Fix $p\in\left[1,\infty\right)$ and choose any risk regularizer ${\cal R}:\mathbb{R}\rightarrow\mathbb{R}$.
Then, as long as $c\in\left[0,1\right]$, the $MS_{p}^{{\cal R}}$
risk measure $\rho_{p}^{{\cal R}}\left(\cdot;c\right)$ is convex-monotone;
that is, it satisfies both conditions $\mathbf{R1}$ and $\mathbf{R2}$. 
\end{thm}
\begin{proof}[Proof of Theorem \ref{thm:ConvexMONO}]
Let us start with verifying convexity ($\mathbf{R1}$). Since the
expectation term of $\rho_{p}^{{\cal R}}\left(\cdot;c\right)$ is
a linear functional on ${\cal Z}_{q}$, it will suffice to show that
the generalized semideviation term $\mathbb{D}_{p}^{{\cal R}}$ is
convex. Indeed, for every $Z_{1}\in{\cal Z}_{q}$, $Z_{2}\in{\cal Z}_{q}$
and every $\alpha\in\left[0,1\right]$, we may write
\begin{flalign}
\mathbb{D}_{p}^{{\cal R}}\left\{ \alpha Z_{1}+\left(1-\alpha\right)Z_{2}\right\}  & \equiv\left\Vert {\cal R}\left(\alpha Z_{1}+\left(1-\alpha\right)Z_{2}-\mathbb{E}\left\{ \alpha Z_{1}+\left(1-\alpha\right)Z_{2}\right\} \right)\right\Vert _{{\cal L}_{p}}\nonumber \\
 & \equiv\left\Vert {\cal R}\left(\alpha\left(Z_{1}-\mathbb{E}\left\{ Z_{1}\right\} \right)+\left(1-\alpha\right)\left(Z_{2}-\mathbb{E}\left\{ Z_{2}\right\} \right)\right)\right\Vert _{{\cal L}_{p}}\nonumber \\
 & \le\left\Vert \alpha{\cal R}\left(\left(Z_{1}-\mathbb{E}\left\{ Z_{1}\right\} \right)\right)+\left(1-\alpha\right){\cal R}\left(Z_{2}-\mathbb{E}\left\{ Z_{2}\right\} \right)\right\Vert _{{\cal L}_{p}}\nonumber \\
 & \le\alpha\left\Vert {\cal R}\left(\left(Z_{1}-\mathbb{E}\left\{ Z_{1}\right\} \right)\right)\right\Vert _{{\cal L}_{p}}+\left(1-\alpha\right)\left\Vert {\cal R}\left(Z_{2}-\mathbb{E}\left\{ Z_{2}\right\} \right)\right\Vert _{{\cal L}_{p}}\nonumber \\
 & \equiv\alpha\mathbb{D}_{p}^{{\cal R}}\left\{ Z_{1}\right\} +\left(1-\alpha\right)\mathbb{D}_{p}^{{\cal R}}\left\{ Z_{2}\right\} ,
\end{flalign}
where the first inequality is true due to conditions $\mathbf{S1}$
(convexity) and $\mathbf{S2}$ (nonnegativity), and the second is
due to the triangle (Minkowski) inequality. Thus, $\mathbb{D}_{p}^{{\cal R}}$
is a convex functional, which means that $\rho_{p}^{{\cal R}}\left(\cdot;c\right)$
is also convex on ${\cal Z}_{q}$. Note that the value of $c\ge0$
is not crucial in order to show convexity of $\rho_{p}^{{\cal R}}\left(\cdot;c\right)$.

Let us now study monotonicity ($\mathbf{R2}$) of the risk measure
$\rho_{p}^{{\cal R}}\left(\cdot;c\right)$. For every $Z_{1}\in{\cal Z}_{q}$
and $Z_{2}\in{\cal Z}_{q}$, such that $Z_{1}\left(\omega\right)\ge Z_{2}\left(\omega\right)$,
for ${\cal P}$-almost all $\omega\in\Omega$, we have
\begin{flalign}
\rho_{p}^{{\cal R}}\left(Z_{2};c\right) & \equiv\mathbb{E}\left\{ Z_{2}\right\} +c\left\Vert {\cal R}\left(Z_{2}-\mathbb{E}\left\{ Z_{2}\right\} \right)\right\Vert _{{\cal L}_{p}}\nonumber \\
 & \le\mathbb{E}\left\{ Z_{2}\right\} +c\left\Vert {\cal R}\left(Z_{1}-\mathbb{E}\left\{ Z_{2}\right\} \right)\right\Vert _{{\cal L}_{p}}\nonumber \\
 & \equiv\mathbb{E}\left\{ Z_{2}\right\} +c\left\Vert {\cal R}\left(Z_{1}-\mathbb{E}\left\{ Z_{1}\right\} +\mathbb{E}\left\{ Z_{1}\right\} -\mathbb{E}\left\{ Z_{2}\right\} \right)\right\Vert _{{\cal L}_{p}}\nonumber \\
 & \le\mathbb{E}\left\{ Z_{2}\right\} +c\left\Vert {\cal R}\left(Z_{1}-\mathbb{E}\left\{ Z_{1}\right\} \right)+\mathbb{E}\left\{ Z_{1}\right\} -\mathbb{E}\left\{ Z_{2}\right\} \right\Vert _{{\cal L}_{p}}\nonumber \\
 & \le\mathbb{E}\left\{ Z_{2}\right\} +c\left(\mathbb{E}\left\{ Z_{1}\right\} -\mathbb{E}\left\{ Z_{2}\right\} \right)+c\left\Vert {\cal R}\left(Z_{1}-\mathbb{E}\left\{ Z_{1}\right\} \right)\right\Vert _{{\cal L}_{p}},\label{eq:MONO_1}
\end{flalign}
where the first inequality is due to conditions $\mathbf{S2}$ (nonnegativity)
and $\mathbf{S3}$ (monotonicity), the second is due to conditions
$\mathbf{S2}$ (nonnegativity), $\mathbf{S4}$ (nonexpansiveness),
as well as the fact that $\mathbb{E}\left\{ Z_{1}\right\} \ge\mathbb{E}\left\{ Z_{2}\right\} $,
and the third is again due to the triangle inequality. From (\ref{eq:MONO_1}),
we readily see that, as long as $c\in\left[0,1\right]$, we may further
write
\begin{equation}
\rho_{p}^{{\cal R}}\left(Z_{2};c\right)\le\mathbb{E}\left\{ Z_{1}\right\} +c\left\Vert {\cal R}\left(Z_{1}-\mathbb{E}\left\{ Z_{1}\right\} \right)\right\Vert _{{\cal L}_{p}}\equiv\rho_{p}^{{\cal R}}\left(Z_{1};c\right),
\end{equation}
completing the proof of the theorem.
\end{proof}
We may now invoke Proposition \ref{prop:Convexity-of-Compositions},
presented earlier, to immediately obtain the following key corollary.
The proof is trivial and, thus, omitted.
\begin{cor}
\textbf{\textup{(When is (\ref{eq:Prog_1}) Convex?)\label{cor:Convex_Prog1}}}
Fix $p\in\left[1,\infty\right)$ and choose any risk regularizer ${\cal R}:\mathbb{R}\rightarrow\mathbb{R}$.
Then, as long as $c\in\left[0,1\right]$, the composite function $\phi^{\widetilde{F}}\left(\cdot\right)\equiv\rho_{p}^{{\cal R}}\left(F\left(\cdot,\boldsymbol{W}\right);c\right)\equiv\rho\left(F\left(\cdot,\boldsymbol{W}\right)\right)$
is convex on $\mathbb{R}^{N}$, and (\ref{eq:Prog_1}) constitutes
a convex stochastic program.
\end{cor}
Corollary \ref{cor:Convex_Prog1} is an important result, because
it shows that, for every mean-semideviation risk measure, or equivalently,
for every risk regularizer of choice, problem (\ref{eq:Prog_1}) would
be exactly solvable via, for instance, subgradient methods, if the
function $\phi^{\widetilde{F}}$ was known in advance. This fact reinforces
our hope that it might indeed be possible to solve (\ref{eq:Prog_1})
to optimality, utilizing some carefully designed \textit{stochastic
search}, or, more specifically, and based on the assumed subdifferentiability
of $F\left(\cdot,\boldsymbol{W}\right)$, \textit{stochastic subgradient}
algorithm. Of course, such an algorithm should be designed to work
under Assumptions \ref{assu:1} and \ref{assu:2}, \textit{without}
the need for explicit knowledge of $F\left(\cdot,\boldsymbol{W}\right)$,
or $\underline{\nabla}F\left(\cdot,\boldsymbol{W}\right)$.
\begin{rem}
\textbf{\textit{(Coherence?)}} We should mention that mean-semideviations
\textit{are not coherent} \textit{risk measures} (\citep{ShapiroLectures_2ND},
Section 6.3), since they do not satisfy the axiomatic property of
positive homogeneity. This is simply due to the fact that, in general,
one may find choices for ${\cal R}$ such that
\begin{equation}
\left\Vert {\cal R}\left(tZ-\mathbb{E}\left\{ tZ\right\} \right)\right\Vert _{{\cal L}_{p}}\neq t\left\Vert {\cal R}\left(Z-\mathbb{E}\left\{ Z\right\} \right)\right\Vert _{{\cal L}_{p}},
\end{equation}
for some $t>0$ and $Z\in{\cal Z}_{q}$. Nevertheless, mean-semideviations
may be readily shown to satisfy translation equivariance, although
such property is not explicitly required in this work. As a result,
except for being convex-monotone, mean-semideviations also belong
to the class of \textit{convex risk measures} \citep{Follmer2002,ShapiroLectures_2ND}.\hfill{}\ensuremath{\blacksquare}
\end{rem}

\subsection{\label{subsec:Examples_of_MS}Examples of Mean-Semideviation Models}

Before moving on, it would be instructive to discuss some examples
of mean-semideviations, highlighting the versatility of this particular
class of risk measures. We start from simple, illustrative choices
as far as the involved risk regularizer is concerned, and then we
generalize.

\subsubsection{Mean\textit{-Upper}-Semideviations}

The simplest, \textit{prototypical} example of a mean-semideviation
risk measure is the \textit{mean-upper-semide- viation} of order $p$
(\citep{ShapiroLectures_2ND}, Sections 6.2.2 \& 6.3.2), which is
constructed by choosing as risk regularizer the function
\begin{equation}
{\cal R}\left(x\right)\triangleq\left(x\right)_{+}\triangleq\max\left\{ x,0\right\} ,\quad x\in\mathbb{R},
\end{equation}
yielding the risk measure
\begin{align}
\rho\left(Z\right) & \equiv\mathbb{E}\left\{ Z\right\} +c\left(\mathbb{E}\left\{ \left(\left(Z-\mathbb{E}\left\{ Z\right\} \right)_{+}\right)^{p}\right\} \right)^{1/p}\nonumber \\
 & \equiv\mathbb{E}\left\{ Z\right\} +c\left\Vert \left(Z-\mathbb{E}\left\{ Z\right\} \right)_{+}\right\Vert _{{\cal L}_{p}},
\end{align}
for $Z\in{\cal Z}_{q}$. Of course, in this case, it is trivial to
show that ${\cal R}$ satisfies conditions $\mathbf{S1}$-$\mathbf{S4}$
of Definition \ref{def:3}. Recall that we have assumed that $q$
is appropriately chosen, such that $\rho$ is a well defined, real-valued
functional on ${\cal Z}_{q}$.

\subsubsection{\textit{Entropic} Mean-Semideviations}

Our second example is a generalization of the mean-upper-semideviation
risk measure discussed in the previous example. Here, the risk regularizer
${\cal R}$ is chosen itself from a parametric family, as
\begin{equation}
{\cal R}\left(x;t\right)\triangleq\dfrac{1}{t}\log\left(1+\exp\left(tx\right)\right),\quad\left(x,t\right)\in\mathbb{R}\times\mathbb{R}_{++},
\end{equation}
where $t$ is a parameter, regulating the sharpness of the function
at zero. It is trivial to verify conditions $\mathbf{S2}$ (nonnegativity)
and $\mathbf{S3}$ (monotonicity). Also, for fixed $t$, the first
derivative of ${\cal R}$ relative to $x$ is the logistic function
\begin{equation}
\dfrac{\partial{\cal R}}{\partial x}\left(x;t\right)\equiv\dfrac{\exp\left(tx\right)}{1+\exp\left(tx\right)}\in\left(0,1\right),\quad\forall x\in\mathbb{R},\label{eq:LOGISTIC}
\end{equation}
showing that ${\cal R}$ is a contraction mapping, immediately verifying
condition $\mathbf{S4}$ (nonexpansiveness), via Proposition \ref{prop:CharacterizationR}.
Likewise, the second derivative of ${\cal R}$ is given by
\begin{equation}
\dfrac{\partial^{2}{\cal R}}{\partial x^{2}}\left(x;t\right)\equiv\dfrac{t\exp\left(tx\right)}{\left(1+\exp\left(tx\right)\right)^{2}}>0,\quad\forall x\in\mathbb{R},
\end{equation}
and, thus, $\mathbf{S1}$ (convexity) is readily verified, as well.
Hence, ${\cal R}$ is a valid risk regularizer. Alternatively and
to illustrate the procedure, we may verify condition $\mathbf{S4}$
directly; for fixed $t>0$, for every $\alpha\ge0$ and for every
$x\in\mathbb{R}$, we may write
\begin{flalign}
{\cal R}\left(x+a;t\right) & \equiv\dfrac{1}{t}\log\left(1+\exp\left(t\left(x+\alpha\right)\right)\right)\nonumber \\
 & \equiv\dfrac{1}{t}\log\left(\dfrac{1}{\exp\left(t\alpha\right)}+\exp\left(tx\right)\right)+\alpha\nonumber \\
 & \le\dfrac{1}{t}\log\left(1+\exp\left(tx\right)\right)+\alpha\nonumber \\
 & \equiv{\cal R}\left(x;t\right)+\alpha,
\end{flalign}
where the inequality is due to the fact that $t\alpha\ge0$. It is
also easy to see that, for every $x\in\mathbb{R}$, ${\cal R}\left(x;t\right)\underset{t\rightarrow\infty}{\longrightarrow}\left(x\right)_{+}$,
showing that ${\cal R}\left(\cdot;t\right)$ constitutes a smooth
approximation to the risk regularizer of the mean-upper-semideviation
risk measure discussed previously.

The resulting risk measure is called an \textit{entropic mean-semideviation}
of order $p$, and may be expressed as
\begin{equation}
\rho\left(Z\right)\equiv\mathbb{E}\left\{ Z\right\} +\dfrac{c}{t}\left\Vert \log\left(1+\exp\left(t\left(Z-\mathbb{E}\left\{ Z\right\} \right)\right)\right)\right\Vert _{{\cal L}_{p}},
\end{equation}
for $Z\in{\cal Z}_{q}$. For obvious reasons, this risk measure may
be considered a \textit{soft} version of the mean-upper-semideviation
risk measure.

\subsubsection{\textit{\label{subsec:CDFA_MS}CDF-Antiderivative (CDFA)} Mean-Semideviations}

We now show that, in fact, both previously presented examples are
special cases of a much more general approach, which may be utilized
for the \textit{construction} of risk regularizers. To this end, let
$Y:\Omega\rightarrow\mathbb{R}$ be a random variable in ${\cal Z}_{1}$,
with cumulative distribution function (cdf) $F_{Y}$. Consider the
choice
\begin{equation}
{\cal R}\left(x\right)\triangleq\int_{-\infty}^{x}F_{Y}\left(y\right)\textrm{d}y,\quad x\in\mathbb{R},\label{eq:F_2}
\end{equation}
where, because $F_{Y}$ is a nonnegative Borel measurable function,
the involved integration is \textit{always} well-defined (might be
$+\infty$, though), in the sense of Lebesgue. The particular \textit{antiderivative
of the cdf} $F_{Y}$, as defined in (\ref{eq:F_2}), constitutes a
very important quantity in the theory of stochastic dominance; see,
for instance, related articles \citep{Ogryczak1999} and \citep{Ogryczak2002}
for definition and insights. In particular, via Fubini's Theorem (Theorem
2.6.6 in \citep{Ash2000Probability}), ${\cal R}$ may be easily shown
to admit the alternative integral representation
\begin{equation}
{\cal R}\left(x\right)\equiv\mathbb{E}\left\{ \left(x-Y\right)_{+}\right\} ,\quad\forall x\in\mathbb{R}.\label{eq:F_2_2}
\end{equation}
Exploiting the assumption that $Y\in{\cal Z}_{1}$, it follows that
${\cal R}\left(x\right)<+\infty$, for every $x\in\mathbb{R}$. Also,
from (\ref{eq:F_2_2}), it is trivial to see that, because of the
structure of the function $\left(\cdot\right)_{+}$, ${\cal R}$ is
convex ($\mathbf{S1}$), nonnegative ($\mathbf{S2}$) and nondecreasing
($\mathbf{S3}$) on $\mathbb{R}$. Nonexpansiveness ($\mathbf{S4}$)
may also be readily verified.

Consequently, ${\cal R}$ is a valid risk regularizer, and the resulting
risk measure, called a \textit{CDF-Antiderivative (CDFA) mean-semideviation},
may be expressed in various forms as
\begin{flalign}
\rho\left(Z\right) & \equiv\mathbb{E}\left\{ Z\right\} +c\left\Vert \int_{-\infty}^{Z-\mathbb{E}\left\{ Z\right\} }F_{Y}\left(y\right)\textrm{d}y\right\Vert _{{\cal L}_{p}}\nonumber \\
 & \equiv\mathbb{E}\left\{ Z\right\} +c\left\Vert \left.\mathbb{E}\left\{ \left(x-Y\right)_{+}\right\} \right|_{x\equiv Z-\mathbb{E}\left\{ Z\right\} }\right\Vert _{{\cal L}_{p}}\nonumber \\
 & \equiv\mathbb{E}\left\{ Z\right\} +c\left\Vert \int_{\mathbb{R}}\left(\left[Z-\mathbb{E}\left\{ Z\right\} \right]-y\right)_{+}\textrm{d}{\cal P}_{Y}\left(y\right)\right\Vert _{{\cal L}_{p}}\nonumber \\
 & \equiv\mathbb{E}\left\{ Z\right\} +c\left\Vert \mathbb{E}\left\{ \left.\left(\left[Z-\mathbb{E}\left\{ Z\right\} \right]-Y\right)_{+}\right|Z\right\} \right\Vert _{{\cal L}_{p}},
\end{flalign}
for $Z\in{\cal Z}_{q}$, where $Y$ can be arbitrarily taken to be
\textit{independent of} $Z$ and ${\cal P}_{Y}$ denotes the Borel
pushforward of $Y$. 

We may now verify that both mean-upper-semideviation and entropic
mean-semideviation risk measures discussed above are special cases
of CDFA mean-semideviations. In mean-upper-semidevi- ations, the respective
risk regularizer is an antiderivative (taken piecewise) of the cdf
corresponding to the Dirac measure at zero. In entropic mean-semideviations,
the respective risk regularizer is an antiderivative of (\ref{eq:LOGISTIC})
(by monotone convergence and via a sequential argument), which is
the cdf of a zero-mean element in ${\cal Z}_{1}$. In both cases,
the antiderivatives involved are of the form of (\ref{eq:F_2}).

\paragraph{Special Case:\textit{ Gaussian Antiderivative (GA)} Mean-Semideviations}

An interesting subclass of CDFA mean-semideviations is the one resulting
from taking antiderivatives of the cdf of a standard Gaussian random
variable $Y\sim{\cal N}\left(0,1\right)$. In this case, the simplest
possible risk regularizer may be constructed as
\begin{equation}
{\cal R}\left(x\right)\triangleq\int_{-\infty}^{x}\varPhi\left(y\right)\textrm{d}y\equiv x\varPhi\left(x\right)+\varphi\left(x\right),\quad x\in\mathbb{R},
\end{equation}
where $\varPhi:\mathbb{R}\rightarrow\left[0,1\right]$ and $\varphi:\mathbb{R}\rightarrow\mathbb{R}$
denote the standard Gaussian cdf and density, respectively. This particular
antiderivative of $\varPhi$ appears naturally in standard treatments
of the so-called ranking-\&-selection, or best arm identification
problem and, more specifically, in lookahead selection policies, such
as the Knowledge Gradient and the Expected Improvement \citep{Frazier2008,Ryzhov2016}.

The resulting mean-semideviation risk measure is called a \textit{Gaussian
Antiderivative (GA) mean-semideviation} of order $p$, and may be
expressed as
\begin{equation}
\rho\left(Z\right)\equiv\mathbb{E}\left\{ Z\right\} +c\left\Vert \left(Z-\mathbb{E}\left\{ Z\right\} \right)\varPhi\left(Z-\mathbb{E}\left\{ Z\right\} \right)+\varphi\left(Z-\mathbb{E}\left\{ Z\right\} \right)\right\Vert _{{\cal L}_{p}},\label{eq:GA_MS}
\end{equation}
for $Z\in{\cal Z}_{q}$. Of course, as it happens for all mean semideviations,
the functional $\rho$, as defined in (\ref{eq:GA_MS}), is a convex
risk measure for every $c\ge0$, and a convex-monotone risk measure,
if $c\in\left[0,1\right]$.

\subsection{A Complete Characterization of Mean-Semideviations}

As a result of the discussion in Section \ref{subsec:CDFA_MS} above,
it follows that risk regularizers may be formed by taking antiderivatives
of the cdf of \textit{any integrable} random variable of choice, resulting
in a vast variety of mean-semideviation risk measures, all sharing
a common favorable structure.

Here, we show that if we start from a \textit{given} risk regularizer
${\cal R}$, \textit{the converse statement is also true}. In this
respect, we state and prove the following important result.
\begin{thm}
\textbf{\textup{(CDF-Based Representation of Risk Regularizers)\label{thm:IFF_RR}}}
Let $Y:\Omega\rightarrow\mathbb{R}$ be a random variable, such that,
for every $x\in\mathbb{R}$, $\mathbb{E}\left\{ \left(x-Y\right)_{+}\right\} <+\infty$,
and let $F_{Y}:\mathbb{R}\rightarrow\left[0,1\right]$ denote its
cdf. Then, for any fixed $0\le C_{S}\le1$ and $C_{I}\ge0$, the function
${\cal R}:\mathbb{R}\rightarrow\mathbb{R}$ defined as
\begin{equation}
{\cal R}\left(x\right)\triangleq C_{S}\int_{-\infty}^{x}F_{Y}\left(y\right)\textrm{d}y+C_{I},\quad\forall x\in\mathbb{R},\label{eq:CDF_Rep}
\end{equation}
is a valid risk regularizer, where integration may be interpreted
either in the improper Riemann sense (for computation), or in the
standard sense of Lebesgue (for derivation). 

Conversely, let ${\cal R}:\mathbb{R}\rightarrow\mathbb{R}$ be any
risk regularizer. Then, there exist some random variable $Y:\Omega\rightarrow\mathbb{R}$,
satisfying $\mathbb{E}\left\{ \left(x-Y\right)_{+}\right\} <+\infty$,
for all $x\in\mathbb{R}$, with cdf $F_{Y}:\mathbb{R}\rightarrow\left[0,1\right]$,
and constants $0\le C_{S}\le1$ and $C_{I}\ge0$, such that, for every
$x\in\mathbb{R}$, the representation (\ref{eq:CDF_Rep}) is valid.
In particular, if ${\cal R}'_{+}:\mathbb{R}\rightarrow\mathbb{R}$
denotes the right derivative of ${\cal R}$, it is always true that
$C_{S}\equiv\sup_{x\in\mathbb{R}}{\cal R}'_{+}\left(x\right)$, $C_{I}\equiv\inf_{x\in\mathbb{R}}{\cal R}\left(x\right)$,
and, as long as ${\cal R}$ is nonconstant, it holds that $C_{S}\neq0$,
and $F_{Y}$ is given by $F_{Y}\left(x\right)\equiv C_{S}^{-1}{\cal R}'_{+}\left(x\right),$
for all $x\in\mathbb{R}$.
\end{thm}
\begin{proof}[Proof of Theorem \ref{thm:IFF_RR}]
See Section \ref{subsec:Proof-of-CHAR} (Appendix).
\end{proof}
Theorem \ref{thm:IFF_RR} is important for two main reasons, the first
being related to the forward statement, and the second to the converse.
On the one hand, Theorem \ref{thm:IFF_RR} provides us with the clean,
very versatile and analytically friendly integral formula (\ref{eq:CDF_Rep})
for constructing risk regularizers of various shapes and types. On
the other hand, it informs us that, \textit{necessarily}, any risk
regularizer can be expressed in the form of (\ref{eq:CDF_Rep}) and,
as a result, \textit{all possible} risk regularizers may be constructed
utilizing (\ref{eq:CDF_Rep}), each time for some suitably chosen
cdf. Therefore, risk regularizers are completely characterized by
the cdf-based representation of Theorem \ref{thm:IFF_RR}.

Of course, every risk regularizer induces a unique mean-semideviation
risk measure. But also notice that, trivially, every mean-semideviation
risk measure corresponds to a uniquely specified risk regularizer
(as a functional, or when \textit{all} costs in the corresponding
${\cal L}_{p}$-space ${\cal Z}_{p}$ -the \textit{largest} such space,
for the \textit{smallest possible} $q$- are considered). Therefore,
Theorem \ref{thm:IFF_RR} provides a complete characterization of
the whole class of mean-semideviation risk measures. In particular,
Theorem \ref{thm:IFF_RR} implies that the class of all mean-semideviation
risk measures is \textbf{\textit{almost}} \textit{in one-to-one correspondence}
with the class of cdfs of all integrable $\mathbb{R}$-valued random
elements. The ``almost'' in the preceding statement is due to the
presence of constants $C_{S}$ and $C_{I}$ in Theorem \ref{thm:IFF_RR},
and that actually slightly less is required than (absolute) integrability
of the involved random variable $Y$.

\subsection{Practical Illustration of Mean-Semideviation Models}

We conclude this section by briefly outlining the relevance of mean-semideviation
models in applications, also putting our proposed risk regularizers
in context. More specifically, we consider a chance-constrained version
of the prototypical, single-product newsvendor problem (see, for instance,
Chapter 1 in \citep{ShapiroLectures_2ND}), upon which we are based
in order to formulate a \textit{doubly risk-averse }newsvendor problem,
which jointly controls \textit{both} unmet demand \textit{and} holding
costs. We also explicitly demonstrate how the respective risk regularizer
may be potentially designed, based on the characteristics of the particular
problem under consideration.

Although the single-product newsvendor problem (and its variations)
indeed constitutes a one-dimensional, toy example, it provides insights
and highlights some important features of the mean-semideviation risk
measures advocated herein. Additionally, the simplicity of such a
problem facilitates numerical solution, and enables us to present
some numerical results, verifying the effectiveness of the proposed
mean-semideviation risk measures experimentally, as well.

\subsubsection{A Chance-Constrained Single-Product NewsVendor}

Suppose that a newsvendor is interested in optimally producing newspapers
for an uncertain market, so that they minimize the cost incurred by
actual production \textit{and} by \textit{not meeting} market demand,
while respecting their holding capacity, or a predefined holding cost
target. Let $K^{P}>0$, $K^{U}>0$ and $K^{H}>0$ be known constants,
standing for the production, unmet demand and holding costs \textit{per
production unit}. Also let $W:\Omega\rightarrow\mathbb{R}_{+}$ be
the random market demand, a random variable with cdf $F_{W}:\mathbb{R}\rightarrow\left[0,1\right]$,
for simplicity assumed to be absolutely continuous relative to the
Lebesgue measure on $\left(\mathbb{R},\mathscr{B}\left(\mathbb{R}\right)\right)$.
Since the market is uncertain, the newsvendor resorts to stochastically
deciding their production plan by solving the \textit{chance-constrained
program }\renewcommand{\arraystretch}{1.5}
\begin{equation}
\begin{array}{rl}
\underset{x}{\mathrm{minimize}} & K^{P}x+\mathbb{E}\left\{ K^{U}\left(W-x\right)_{+}\right\} \\
\mathrm{subject\,to} & {\cal P}\left(K^{H}\left(x-W\right)_{+}>h\right)\le\alpha\\
 & x\ge0
\end{array},\label{eq:Joes_Prog}
\end{equation}
\renewcommand{\arraystretch}{1}where, also for simplicity, we assume
that the decision variable is real-valued, and where $0\le\alpha\le1$
constitutes the newsvendor's tolerance in the event that their holding
cost $K^{H}\left(x-W\right)_{+}$ will exceed a prescribed threshold
$h\ge0$. Both $\alpha$ and $h$ are fixed design parameters decided
by the newsvendor beforehand. Chance-constrained newsvendor problems
similar to (\ref{eq:Joes_Prog}) have been previously considered in
the literature; see, for instance, the related article \citep{Zhang2009}.
Here, an important detail is that, despite the probabilistic constraint,
problem (\ref{eq:Joes_Prog}) is risk-neutral as far as treatment
of unmet demand is concerned. This is because only the expectation
of the cost of not meeting the demand, corresponding to $K^{U}\left(W-x\right)_{+}$,
is considered in the objective. 

Problem (\ref{eq:Joes_Prog}) exhibits some interesting features and
may be significantly simplified, as follows. First, we may observe
that, for every fixed choice of $h\ge0$,
\begin{flalign}
{\cal P}\left(K^{H}\left(x-W\right)_{+}>h\right) & ={\cal P}\left(K^{H}\left(x-W\right)>h\right)\nonumber \\
 & \equiv{\cal P}\left(W<x-\dfrac{h}{K^{H}}\right)\nonumber \\
 & \equiv F_{W}\left(x-\dfrac{h}{K^{H}}\right),\quad\forall x\in\mathbb{R}.
\end{flalign}
Consequently, it is true that
\begin{flalign}
{\cal P}\left(K^{H}\left(x-W\right)_{+}>h\right)\le\alpha & \iff F_{W}\left(x-\dfrac{h}{K^{H}}\right)\le\alpha\nonumber \\
 & \iff x\le F_{W}^{-1}\left(\text{\ensuremath{\alpha}}\right)+\dfrac{h}{K^{H}}\label{eq:Joes_Constraint_1}\\
 & \iff{\cal P}\left(W<x\right)\le F_{W}\left(F_{W}^{-1}\left(\text{\ensuremath{\alpha}}\right)+\dfrac{h}{K^{H}}\right),\label{eq:Joes_Constraint_2}
\end{flalign}
where, due to $F_{W}$ being continuous, the pseudo-inverse or quantile
function $F_{W}^{-1}:\left[0,1\right]\rightarrow\mathbb{R}_{+}$ is
defined as
\begin{equation}
F_{W}^{-1}\left(\alpha\right)\triangleq\inf\left\{ x\in\mathbb{R}\left|F_{W}\left(x\right)\ge\alpha\right.\right\} \equiv\sup\left\{ x\in\mathbb{R}\left|F_{W}\left(x\right)\le\alpha\right.\right\} .
\end{equation}
Thus, problem (\ref{eq:Joes_Prog}) is convex and may be reformulated
as\renewcommand{\arraystretch}{1.3}
\begin{equation}
\begin{array}{rl}
\underset{x}{\mathrm{minimize}} & K^{P}x+\mathbb{E}\left\{ K^{U}\left(W-x\right)_{+}\right\} \\
\mathrm{subject\,to} & x\in\left[0,F_{W}^{-1}\left(\text{\ensuremath{\alpha}}\right)+\dfrac{h}{K^{H}}\right]
\end{array}.\label{eq:Joes_Prog_REF}
\end{equation}
\renewcommand{\arraystretch}{1}Hereafter, without loss of generality,
we may assume that $\alpha$ and $h$ are chosen such that $F_{W}^{-1}\left(\text{\ensuremath{\alpha}}\right)+h/K^{H}>0$.
Otherwise, the problem is trivially solved at $x^{*}\equiv0$. To
be fully compatible with the generic notation utilized in this paper,
we may also define $F\left(\cdot,\bullet\right)\triangleq K^{P}\left(\cdot\right)+K^{U}\left(\left(\bullet\right)-\left(\cdot\right)\right)_{+},$
and ${\cal X}\triangleq\left[0,F_{W}^{-1}\left(\text{\ensuremath{\alpha}}\right)+h/K^{H}\right]$.

Next, let us consider the derivative of the objective of (\ref{eq:Joes_Prog_REF}),
relative to $x$. We have
\begin{equation}
\nabla\mathbb{E}\left\{ F\left(x,W\right)\right\} =K^{P}-K^{U}{\cal P}\left(W\ge x\right),\quad\forall x\in{\cal X}.
\end{equation}
Hence, unless $K^{U}>K^{P}$, it readily follows that, for every $x\in{\cal X}$,
$\nabla\mathbb{E}\left\{ F\left(x,W\right)\right\} \ge0$, again implying
that the choice $x^{*}\equiv0$ constitutes a solution of (\ref{eq:Joes_Prog_REF});
in other words, producing nothing is always optimal whenever $K^{U}\le K^{P}$.
On the other hand, it is apparently true that $\nabla\mathbb{E}\left\{ F\left(x,W\right)\right\} <0$,
for all $x\in{\cal X}$, if and only if
\begin{equation}
K^{P}-K^{U}\left(1-{\cal P}\left(W<x\right)\right)<0,\quad\forall x\in{\cal X},
\end{equation}
implying that the condition
\begin{equation}
K^{P}-K^{U}\left(1-F_{W}\left(F_{W}^{-1}\left(\text{\ensuremath{\alpha}}\right)+\dfrac{h}{K^{H}}\right)\hspace{-2pt}\right)<0
\end{equation}
is sufficient to ensure negativity of $\nabla\mathbb{E}\left\{ F\left(x,W\right)\right\} $
everywhere within the feasible set ${\cal X}$ (where (\ref{eq:Joes_Constraint_2})
is always satisfied), in which case the choice $x^{*}\equiv F_{W}^{-1}\left(\text{\ensuremath{\alpha}}\right)+h/K^{H}$
constitutes the optimal production level. Putting it altogether, whenever
\begin{equation}
F_{W}\left(F_{W}^{-1}\left(\text{\ensuremath{\alpha}}\right)+\dfrac{h}{K^{H}}\right)>0,
\end{equation}
the condition
\begin{equation}
K^{U}\left(1-F_{W}\left(F_{W}^{-1}\left(\text{\ensuremath{\alpha}}\right)+\dfrac{h}{K^{H}}\right)\hspace{-2pt}\right)\le K^{P}<K^{U}
\end{equation}
ensures that problem (\ref{eq:Joes_Prog}) admits nontrivial solutions,
thus being of technical interest.

Problem (\ref{eq:Joes_Prog}) may be solved in closed form. Indeed,
either by considering the Karush-Kuhn-Tucker (KKT) conditions for
problem (\ref{eq:Joes_Prog}) (for a constraint qualification, we
may observe that Slater's condition is satisfied trivially whenever
$F_{W}^{-1}\left(\text{\ensuremath{\alpha}}\right)+h/K^{H}>0$), or
by looking at its geometric structure directly, it may be easily shown
that its optimal solution may be expressed analytically as
\begin{equation}
x^{*}=\begin{cases}
0, & \text{if }K^{U}\le K^{P}\\
\min\left\{ F_{W}^{-1}\left(\dfrac{K^{U}-K^{P}}{K^{U}}\right)\hspace{-1pt},F_{W}^{-1}\left(\text{\ensuremath{\alpha}}\right)+\dfrac{h}{K^{H}}\right\} \hspace{-1pt}, & \text{if }K^{U}>K^{P}
\end{cases},
\end{equation}
representing the newsvendor's optimal decision in regard to the quantity
of newspapers they would have to plan for, before the random market
demand $W$ is revealed.
\begin{rem}
Problems of the type of (\ref{eq:Joes_Prog}) are meaningful in various
settings; specifically, they are most suitable when holding is operationally
more important than unmet demand. For instance, it might be the case
that the event where holding exceeds some threshold might have severe
economic consequences, while not meeting the demand might be tolerable,
although undesirable.

Additionally and perhaps more importantly, we should mention that
a chance-constrained approach such as that adopted in (\ref{eq:Joes_Prog})
allows to efficiently \textit{blend economic with physical quantities}
in a single stochastic program. This is simply due to the fact that
by defining a quantity $\widetilde{h}\triangleq h/K^{H}\ge0$, the
probabilistic constraint of (\ref{eq:Joes_Prog}) may be written as
\begin{equation}
{\cal P}\left(\left(x-W\right)_{+}>\widetilde{h}\right)\le\alpha,
\end{equation}
implying that, if we want to, we may directly choose $\widetilde{h}$
as a probabilistic upper bound directly on the excess production $\left(x-W\right)_{+}$.

The modification above can be very useful if we are willing to consider
the problem\renewcommand{\arraystretch}{1.7}
\begin{equation}
\begin{array}{rl}
\underset{x}{\mathrm{minimize}} & K^{P}x+\mathbb{E}\left\{ K^{H}\left(x-W\right)_{+}\right\} \\
\mathrm{subject\,to} & {\cal P}\left(K^{U}\left(W-x\right)_{+}>u\right)\le\alpha\\
 & x\ge0
\end{array},\label{eq:Joes_Prog_D}
\end{equation}
\renewcommand{\arraystretch}{1}which constitutes a dual version of
the initial newsvendor problem (\ref{eq:Joes_Prog}) resulting by
interchanging the two respective stochastic costs and where, similarly
to (\ref{eq:Joes_Prog}), $u\ge0$ is a prescribed threshold. In this
case, unmet demand is operationally more important than holding, by
choice. Of course, problem (\ref{eq:Joes_Prog_D}) is structurally
very similar to (\ref{eq:Joes_Prog}), and can be analyzed via almost
the same procedure as above. By defining $\widetilde{u}\triangleq u/K^{U}\ge0$,
problem (\ref{eq:Joes_Prog_D}) may be reformulated as\renewcommand{\arraystretch}{1.7}
\begin{equation}
\begin{array}{rl}
\underset{x}{\mathrm{minimize}} & K^{P}x+\mathbb{E}\left\{ K^{H}\left(x-W\right)_{+}\right\} \\
\mathrm{subject\,to} & {\cal P}\left(\left(W-x\right)_{+}>\widetilde{u}\right)\le\alpha\\
 & x\ge0
\end{array},\label{eq:Joes_Prog_D-1}
\end{equation}
\renewcommand{\arraystretch}{1}where $\widetilde{u}$ can now be
preselected directly. We may readily observe that the objective of
(\ref{eq:Joes_Prog_D-1}) constitutes an economic quantity (a cost),
whereas the probabilistic constraint is placed on the unmet demand
itself, which, of course, is a physical quantity. This modification
can be extremely useful in a more realistic scenario, since in many
practical cases the unit cost of unmet demand, $K^{U}$, is either
completely unknown, or extremely difficult to estimate based on experience.\hfill{}\ensuremath{\blacksquare}
\end{rem}

\subsubsection{A Doubly Risk-Averse Single-Product NewsVendor}

Suppose now that, due to high variability of the market demand, the
newsvendor realizes that minimizing their unmet demand cost in expectation
does not constitute a very meaningful objective. Thus, the newsvendor
would like to decide on their newspaper production size by explicitly
accounting for market variability in their model and, because they
are reasonable, they are willing to settle with a potentially slightly
higher \textit{expected} monetary penalty for not meeting market demand.
In effect, the newsvendor is interested in making their decision by
additionally considering the risk incurred due to stochastic variability
in the resulting unmet demand, the latter realized when market demand
is revealed. In other words, the newsvendor would like to come up
with a meaningful \textit{doubly risk-averse }version\textit{ }of
the original, chance-constrained problem (\ref{eq:Joes_Prog}).

The newsvendor may think as follows. For every \textit{fixed and feasible}
production decision $x\in{\cal X}$, if the \textit{noisy} unmet demand
$\left(W-x\right)_{+}$ is \textit{smaller than} $\mathbb{E}\left\{ \left(W-x\right)_{+}\right\} $,
which is the newsvendor's expectation, then there is no risk incurred,
since the newsvendor has been prepared for and has agreed to settle
with a cost of unmet demand equal to $\mathbb{E}\left\{ K^{U}\left(W-x\right)_{+}\right\} $.
In an actual production scenario, $\mathbb{E}\left\{ \left(W-x\right)_{+}\right\} $
might correspond to a small quantity of newspapers which are \textit{not
actually produced}, but for which resources have been allocated beforehand,
to compensate for the case $W$ is greater than $x$, but not too
much. In other words, we might think about the quantity $\mathbb{E}\left\{ \left(W-x\right)_{+}\right\} $
as a \textit{risk-free, first-level safety stock}.

Positive risk is incurred whenever $\left(W-x\right)_{+}>\mathbb{E}\left\{ \left(W-x\right)_{+}\right\} $.
However, the newsvendor realizes that \textit{not} all values of the
\textit{central deviation}
\begin{equation}
CD\left(x,W\right)\triangleq\left(W-x\right)_{+}-\mathbb{E}\left\{ \left(W-x\right)_{+}\right\} 
\end{equation}
are of equal importance, or equal severity. In other words, the newsvendor's
risk is variable relative to the value of $CD\left(x,W\right)$. Under
the reasonable assumption that positive risk should be increasing
as a function of the deviation $CD\left(x,W\right)$, the newsvendor's
realization translates naturally into a variable and increasing rate
of change of the risk, relative to the values of the deviation. In
particular, whenever $CD\left(x,W\right)>0$, the newsvendor identifies
the following risk-incurring regions of increasing severity:
\begin{enumerate}
\item $CD\left(x,W\right)\in\left(0,t_{1}>0\right]$. In this case, unmet
demand is higher than what the newsvendor expects, but its deviation
from their expectation is no higher than a fixed threshold $t_{1}$.
The value $\mathbb{E}\left\{ \left(W-x\right)_{+}\right\} +t_{1}$
corresponds to the maximum \textit{partially unplanned or unexpected}
production quantity that the newsvendor may be able to produce \textit{today},
potentially using \textit{presently unallocated} resources. We might
think about the threshold $t_{1}$ as a \textit{risk-incurring,} \textit{second-level
safety stock}.
\item $CD\left(x,W\right)\in\left(t_{1},t_{2}>t_{1}\right]$. Here, the
deviation of the unmet demand from the newsvendor's expectation is
exceeds $t_{1}$, but is no higher than another fixed threshold $t_{2}$.
The value $t_{2}-t_{1}$ corresponds to the maximum quantity of newspapers
that the newsvendor cannot produce in-house today, but may ask a nearby
vendor to produce for them. Of course, such events should incur higher
and more severely increasing risk, since the newsvendor essentially
\textit{borrows resources} from the nearby vendor. We might call $t_{2}$
as the \textit{borrowing threshold}.
\item $CD\left(x,W\right)\in\left(t_{2},\infty\right)$. This constitutes
an event of ``total disaster,'' in which it is impossible for the
newsvendor to compensate for unmet market demand. When $CD\left(x,W\right)>t_{2}$,
unmet demand is so high that it cannot be met even if the newsvendor
borrows the maximum amount of resources from some nearby newsvendor.
This might have severe consequences for the newsvendor, since they
either might be in debt, or even lose their professional credibility,
or both.
\end{enumerate}
Although potentially simplified, a narrative such as the above is
reasonable and quite realistic. Of course, what is important for us
in the context of this paper, is the fact that the characteristics
of the relatively complex risk dynamics discussed above can be succinctly
captured by an appropriately shaped risk regularizer, as proposed
and analyzed herein. As a simplest example, we may define a \textit{piecewise
linear} risk regularizer ${\cal R}^{nv}:\mathbb{R}\rightarrow\mathbb{R}$
as
\begin{equation}
{\cal R}^{nv}\left(x\right)\triangleq\begin{cases}
0, & \text{if }x\le0\\
\psi_{1}x, & \text{if }0<x\le K^{U}t_{1}\\
\psi_{2}x+\left(\psi_{1}-\psi_{2}\right)K^{U}t_{1}, & \text{if }K^{U}t_{1}<x\le K^{U}t_{2}\\
x+\left(\psi_{2}-1\right)K^{U}t_{2}+\left(\psi_{1}-\psi_{2}\right)K^{U}t_{1}, & \text{if }x>K^{U}t_{2}
\end{cases},\label{eq:NV_Regularizer}
\end{equation}
where the \textit{risk slopes} $\psi_{1}\ge0$ and $\psi_{2}\ge0$
are chosen such that $\psi_{1}\le\psi_{2}\le1$. Of course, ${\cal R}$
may be rewritten as
\begin{flalign}
{\cal R}^{nv}\left(x\right) & \equiv\psi_{1}x\mathds{1}_{\left[0,K^{U}t_{1}\right)}\left(x\right)+\left(\psi_{2}x+\left(\psi_{1}-\psi_{2}\right)K^{U}t_{1}\right)\mathds{1}_{\left[K^{U}t_{1},K^{U}t_{2}\right)}\left(x\right)\nonumber \\
 & \quad\quad+\left(x+\left(\psi_{2}-1\right)K^{U}t_{2}+\left(\psi_{1}-\psi_{2}\right)K^{U}t_{1}\right)\mathds{1}_{\left[K^{U}t_{2},\infty\right)}\left(x\right),
\end{flalign}
for all $x\in\mathbb{R}$, and may be conveniently thought as a generalization
of the positive part function of the upper-semideviation dispersion
measure. \textit{Equivalently}, the risk regularizer ${\cal R}^{nv}$
may be defined as an antiderivative of the cdf $F_{Y}^{nv}:\mathbb{R}\rightarrow\left[0,1\right]$
corresponding to some random variable $Y:\Omega\rightarrow\mathbb{R}$
in ${\cal Z}_{\infty}$, and defined as
\begin{equation}
F_{Y}^{nv}\left(x\right)\triangleq\psi_{1}\mathds{1}_{\left[0,K^{U}t_{1}\right)}\left(x\right)+\psi_{2}\mathds{1}_{\left[K^{U}t_{1},K^{U}t_{2}\right)}\left(x\right)+\mathds{1}_{\left[K^{U}t_{2},\infty\right)}\left(x\right),\quad x\in\mathbb{R},
\end{equation}
as suggested by Theorem \ref{thm:IFF_RR}. The cdf $F_{Y}^{nv}$ expresses
precisely the \textit{rate of increase of the risk} incurred at each
$x\in\mathbb{R}$, where $x$ may be thought of as the central deviation
of the quantity whose risk is assessed by the risk regularizer ${\cal R}^{nv}$;
in the newsvendor's case, this quantity should be the noisy economic
consequence due to unmet demand, i.e., $K^{U}\left(W-x\right)_{+}$,
also justifying the multiplication of thresholds $t_{1}$ and $t_{2}$
with the unit cost $K^{U}$ in (\ref{eq:NV_Regularizer}). Essentially,
$F_{Y}^{nv}$ admits an intuitive interpretation, and can be utilized
in order to \textit{actually design} ${\cal R}^{nv}$, as well; also
see Theorem \ref{thm:IFF_RR}. 

If the newsvendor chooses the risk slopes $\psi_{1}$ and $\psi_{2}$
such that $\psi_{1}<\psi_{2}<1$ (only \textit{they} know how set
specific appropriate values), the quantity ${\cal R}^{nv}\left(K^{U}CD\left(x,W\right)\right)$
(for a fixed and feasible $x\in{\cal X}$) captures the dynamic behavior
of the risk incurred by the central deviation $CD\left(x,W\right)$
when taking different values in $\mathbb{R}_{+}$, as described above.
Essentially, ${\cal R}^{nv}$ may be regarded as a nonlinear weighting
function acting on $K^{U}CD\left(x,W\right)$, whose shape has been
carefully designed in order to reflect the newsvendor's particular
context.

Now, the newsvendor is interested in utilizing ${\cal R}^{nv}\left(K^{U}CD\left(\cdot,W\right)\right)$
for decision making purposes. Since, for each $x\in{\cal X}$, ${\cal R}^{nv}\left(K^{U}CD\left(\cdot,W\right)\right)$
depends pointwise on the random market demand $W$, which is unobservable
when the newsvendor decides their production plan, it is most reasonable
to consider the ${\cal L}_{p}$-norm of ${\cal R}^{nv}\left(K^{U}CD\left(\cdot,W\right)\right)$,
for some prespecified $p\in\left[1,\infty\right)$, \textit{as a}
\textit{measure of magnitude}. Of course, lower values for such deterministic
term are preferred. In order to effectively manage their risk during
decision making, the newsvendor proceeds by adding the ${\cal L}_{p}$-norm
of ${\cal R}^{nv}\left(K^{U}CD\left(\cdot,W\right)\right)$ as a penalty
term to their original, risk-neutral objective, leading to the risk-averse
stochastic program
\begin{figure}
\centering\includegraphics[bb=80bp 145bp 1000bp 1114bp,clip,scale=0.56]{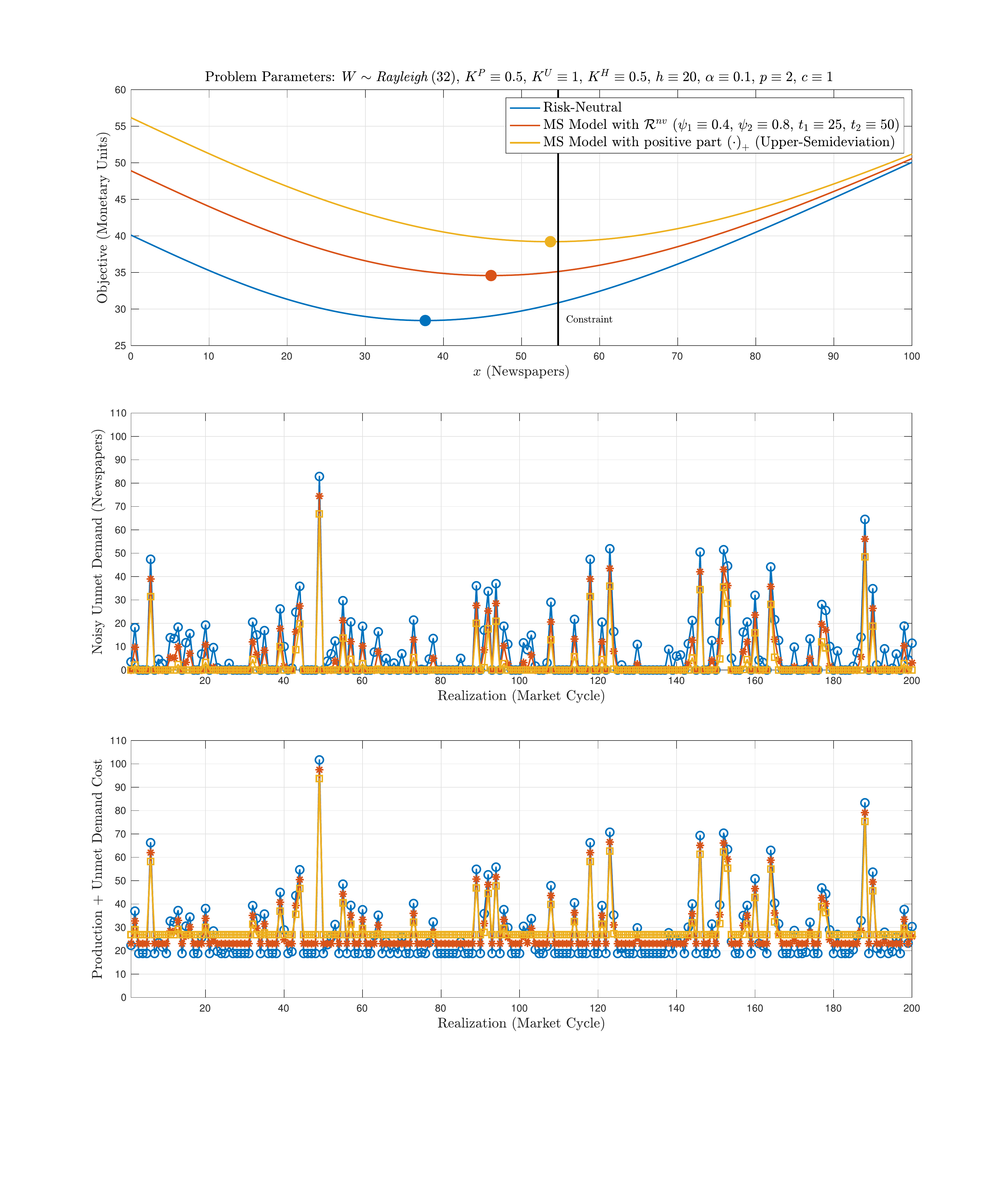}\caption{\label{fig:RA_NewsVendor}Top: Objective as a function of the production
decision, for each of the following three cases: Risk neutral (problem
(\ref{eq:Joes_Prog})), risk-averse of the form (\ref{eq:Joes_Prog_RA}),
and risk-averse of the form (\ref{eq:Joes_Prog_RA}), but with ${\cal R}^{nv}$
being replaced by $\left(\cdot\right)_{+}$. Middle: Unmet demand
realizations as a function of the market cycle, for each of the three
cases. Bottom: Combined monetary cost of production plus unmet demand
realizations as a function of the market cycle, for each of the three
cases.}

\vspace{-5bp}
\end{figure}
\begin{equation}
\begin{array}{rl}
\underset{x}{\mathrm{minimize}} & K^{P}x+\mathbb{E}\left\{ K^{U}\left(W-x\right)_{+}\right\} +c\left\Vert {\cal R}^{nv}\left(K^{U}\left(W-x\right)_{+}\hspace{-2pt}-\mathbb{E}\left\{ K^{U}\left(W-x\right)_{+}\right\} \right)\right\Vert _{{\cal L}_{p}}\\
\mathrm{subject\,to} & {\cal P}\left(K^{H}\left(x-W\right)_{+}>h\right)\le\alpha\\
 & x\ge0
\end{array}\hspace{-2pt},\label{eq:Joes_Prog_RA_PRE}
\end{equation}
where, in general, $c\ge0$ denotes the corresponding penalty tradeoff
multiplier. It is then easy to see that the objective of problem (\ref{eq:Joes_Prog_RA_PRE})
constitutes a mean-semideviation model. Indeed, by equivalently rewriting
${\cal R}^{nv}\left(K^{U}CD\left(\cdot,W\right)\right)$ as
\begin{flalign}
{\cal R}^{nv}\left(K^{U}CD\left(x,W\right)\right) & \equiv{\cal R}^{nv}\left(K^{P}x+K^{U}\left(W-x\right)_{+}-\mathbb{E}\left\{ K^{P}x+K^{U}\left(W-x\right)_{+}\right\} \right)\nonumber \\
 & \equiv{\cal R}^{nv}\left(F\left(x,W\right)-\mathbb{E}\left\{ F\left(x,W\right)\right\} \right),\quad\forall x\in{\cal X},
\end{flalign}
problem (\ref{eq:Joes_Prog_RA_PRE}) may be equivalently restated
as
\begin{equation}
\begin{array}{rl}
\underset{x}{\mathrm{minimize}} & \mathbb{E}\left\{ F\left(x,W\right)\right\} +c\left\Vert {\cal R}^{nv}\left(F\left(x,W\right)-\mathbb{E}\left\{ F\left(x,W\right)\right\} \right)\right\Vert _{{\cal L}_{p}}\\
\mathrm{subject\,to} & x\in{\cal X}
\end{array}.\label{eq:Joes_Prog_RA}
\end{equation}
Apparently, the objective of problem (\ref{eq:Joes_Prog_RA}) is a
mean-semideviation risk measure. Of course, whenever $c\in\left[0,1\right]$,
(\ref{eq:Joes_Prog_RA}) (and thus, (\ref{eq:Joes_Prog_RA_PRE}),
as well) constitutes a convex, risk-averse stochastic program, precisely
of the form considered in this paper.

To empirically demonstrate the effectiveness of the risk-averse newsvendor
problem (\ref{eq:Joes_Prog_RA}), we have conducted numerical simulations
concerning the following three cases: Initial problem (\ref{eq:Joes_Prog})
(risk neutral), risk-averse problem (\ref{eq:Joes_Prog_RA}), and
risk-averse of the form (\ref{eq:Joes_Prog_RA}), but with ${\cal R}^{nv}$
being replaced by $\left(\cdot\right)_{+}$, resulting to the mean-upper-semideviation
risk measure. In all simulations, the random market demand $W$ follows
a Rayleigh distribution, and all necessary expectations present in
each objective have been approximated utilizing $5\cdot10^{6}$ demand
realizations, sampled independently. The precise values for the scale
of the aforementioned distribution and for all the rest of parameters
involved in problem (\ref{eq:Joes_Prog_RA}) are shown in the title
of Fig. \ref{fig:RA_NewsVendor} (top), respectively.

From Fig. \ref{fig:RA_NewsVendor} (top), we observe that the optimal
production decision obtained by solving (\ref{eq:Joes_Prog_RA}) is
distinctly different from the respective solutions obtained by solving
both the risk neutral problem (\ref{eq:Joes_Prog}), and the risk-averse
problem employing the mean-upper-semideviation risk measure (all optimal
solutions are represented by appropriately colored dots in Fig. \ref{fig:RA_NewsVendor}
(top)). In particular, the solution of (\ref{eq:Joes_Prog_RA}) lies
somewhere near the midpoint of the respective solutions of the remaining
two aforementioned problems. Therefore, we may conclude that the solution
of (\ref{eq:Joes_Prog_RA}) constitutes a \textit{less conservative}
risk-averse production decision, compared to the case of the mean-upper-semideviation
risk measure, which essentially presumes that all risk-incurring events
are of equal operational severity for the newsvendor. Equivalently,
the mean-semideviation model utilized in (\ref{eq:Joes_Prog_RA})
constitutes a less conservative risk-averse objective compared to
that involving the mean-upper-semideviation risk measure (which, of
course, is itself a mean-semideviation model induced by the trivial
risk regularizer $\left(\cdot\right)_{+}$). As it can be readily
observed in Fig. \ref{fig:RA_NewsVendor} (middle \& bottom), the
less conservative character of problem (\ref{eq:Joes_Prog_RA}) translates
directly to the statistical behavior of the realized unmet demand,
and that of the combined cost due to production and unmet demand.
This is obviously expected in this example, and is due to the simple
structure to the original newsvendor problem we started with.

\section{\label{sec:Message}The \textit{$\textit{MESSAGE}^{p}$} Algorithm}

This section is devoted to the introduction and detailed analysis
of the $\textit{MESSAGE}^{p}$ algorithm. As also stated in Section
\ref{sec:Introduction}, the $\textit{MESSAGE}^{p}$ algorithm is
a parameterized (relative to the choice of ${\cal R}$) parallel version\textit{
}of the general purpose \textit{$T$-SCGD} algorithm \citep{Wang2018}.
Both the algorithm and analysis presented in this work are new; as
compared to \citep{Wang2018}, we propose a significantly milder set
of problem assumptions, which, nonetheless, result in asymptotic guarantees
of \textit{at least} the same quality, \textit{and more}.

Before proceeding, let us restate the stochastic program under study.
Formally, for fixed $p\in\left[1,\infty\right)$, we are interested
in the convex optimization problem
\begin{equation}
\begin{array}{rl}
\underset{\boldsymbol{x}}{\mathrm{minimize}} & \mathbb{E}\left\{ F\left(\boldsymbol{x},\boldsymbol{W}\right)\right\} +c\left\Vert {\cal R}\left(F\left(\boldsymbol{x},\boldsymbol{W}\right)-\mathbb{E}\left\{ F\left(\boldsymbol{x},\boldsymbol{W}\right)\right\} \right)\right\Vert _{{\cal L}_{p}}\\
\mathrm{subject\,to} & \boldsymbol{x}\in{\cal X}
\end{array},\label{eq:MAIN_PROB}
\end{equation}
where, for every $\boldsymbol{x}\in{\cal X}$, the convex real-valued
random cost $F\left(\boldsymbol{x},\boldsymbol{W}\left(\cdot\right)\right)\equiv\widetilde{F}\left(\boldsymbol{x},\cdot\right)$
is in ${\cal Z}_{q}$, the set of feasible decisions ${\cal X}\subseteq\mathbb{R}^{N}$
is closed and convex, the risk related penalty multiplier is denoted
by $c\ge0$, and where ${\cal R}:\mathbb{R}\rightarrow\mathbb{R}$
constitutes any risk regularizer of choice. Recall that we implicitly
assume that $q$ and $p$ are compatible according to Proposition
\ref{prop:P=000026Q}. Also, based on our definitions, the objective
is identified as either of the functions $\rho\left(F\left(\cdot,\boldsymbol{W}\right)\right)\equiv\rho_{p}^{{\cal R}}\left(F\left(\cdot,\boldsymbol{W}\right);c\right)$
and $\phi^{\widetilde{F}}$, where the choices of $\rho$-related
quantities $p$, $c$, ${\cal R}$ are assumed to be \textit{fixed}
and made in advance; as such, they will not be explicitly referred
to in our notation. Additionally, in the following, we assume that
$c\in\left[0,1\right]$, so that, by Corollary \ref{cor:Convex_Prog1},
(\ref{eq:MAIN_PROB}) constitutes a \textit{convex} problem.

In the following, we first discuss the reformulation of the objective
of (\ref{eq:MAIN_PROB}) in a convenient compositional form, key to
the development of any compositional algorithm whatsoever. Second,
we discuss the technical reasons that motivate the consideration of
a compositional SSD-type algorithm for solving (\ref{eq:MAIN_PROB}),
as well as differentiability of its objective. Then, we present the
$\textit{MESSAGE}^{p}$ algorithm, along with some of its key characteristics.

The section proceeds with the asymptotic analysis of the $\textit{MESSAGE}^{p}$
algorithm. First, our structural assumptions are presented and their
main implications are discussed. Second, pathwise convergence of the
$\textit{MESSAGE}^{p}$ algorithm is established, in a strong technical
sense. Our proof follows the somewhat standard\textit{ ``almost-supermartingale
approach''}, also adopted in \citep{Wang2017,Wang2018}. Third, we
present a detailed convergence rate analysis of the $\textit{MESSAGE}^{p}$
algorithm, where we systematically develop all the results advertised
in Section \ref{sec:Introduction}, along with relevant discussions. 

Finally, the generality of our structural framework against that utilized
in \citep{Wang2018} is also clearly demonstrated, rigorously showing
that the class of mean-semideviation programs supported herein is
\textit{strictly larger} than the respective class of problems supported
within \citep{Wang2017,Wang2018}. This result concludes our discussion
related to the consistency of the $\textit{MESSAGE}^{p}$ algorithm,
and justifies our effort.

\subsection{Mean-Semideviations in Compositional Form}

Because we will be interested in determining the structure of the
subdifferential of the objective of (\ref{eq:MAIN_PROB}), $\phi^{\widetilde{F}}$,
it is convenient to express $\phi^{\widetilde{F}}$ in \textit{compositional
form}, similar to the general approach adopted in \citep{Wang2017,Wang2018}.
To do this, let us define the \textit{expectation functions} $\varrho:\mathbb{R}_{+}\rightarrow\mathbb{R}$,
$g^{\widetilde{F}}:\mathbb{R}^{N}\times\mathbb{R}\rightarrow\mathbb{R}_{+}$
and $\boldsymbol{h}^{\widetilde{F}}:\mathbb{R}^{N}\rightarrow\mathbb{R}^{N}\times\mathbb{R}$
as
\begin{flalign}
\varrho\left(x\right) & \triangleq x^{1/p},\quad x>0,\\
g^{\widetilde{F}}\left(\boldsymbol{x},y\right) & \triangleq\mathbb{E}\left\{ \left({\cal R}\left(F\left(\boldsymbol{x},\boldsymbol{W}\right)-y\right)\right)^{p}\right\} \quad\text{and}\\
\boldsymbol{h}^{\widetilde{F}}\left(\boldsymbol{x}\right) & \triangleq\left[\boldsymbol{x}\:\mathbb{E}\left\{ F\left(\boldsymbol{x},\boldsymbol{W}\right)\right\} \right],
\end{flalign}
for every admissible choice of $F$ and ${\cal P}_{\boldsymbol{W}}$.
Then, $\phi^{\widetilde{F}}$ may be alternatively expressed as
\begin{equation}
\phi^{\widetilde{F}}\left(\boldsymbol{x}\right)\equiv\mathbb{E}\left\{ F\left(\boldsymbol{x},\boldsymbol{W}\right)\right\} +c\varrho\left(g^{\widetilde{F}}\left(\boldsymbol{h}^{\widetilde{F}}\left(\boldsymbol{x}\right)\right)\right),\quad\boldsymbol{x}\in{\cal X}.\label{eq:MAIN_2}
\end{equation}
We observe that the functional composition term on the RHS of (\ref{eq:MAIN_2})
coincides with the dispersion measure in the objective of (\ref{eq:MAIN_PROB}),
simply rewritten as a composition of real-valued and, of course, deterministic,
functions. In the special case where $p\equiv1$, (\ref{eq:MAIN_2})
becomes
\begin{equation}
\phi^{\widetilde{F}}\left(\boldsymbol{x}\right)\equiv\mathbb{E}\left\{ F\left(\boldsymbol{x},\boldsymbol{W}\right)\right\} +cg^{\widetilde{F}}\left(\boldsymbol{h}^{\widetilde{F}}\left(\boldsymbol{x}\right)\right),\quad\boldsymbol{x}\in{\cal X}.\label{eq:MAIN_3}
\end{equation}
Also, it is trivial to see that, if $p\equiv1$, then, by defining
another function $\mathring{g}^{\widetilde{F}}:\mathbb{R}^{N}\times\mathbb{R}\rightarrow\mathbb{R}$
as
\begin{equation}
\mathring{g}^{\widetilde{F}}\left(\boldsymbol{x},y\right)\triangleq y+c\mathbb{E}\left\{ {\cal R}\left(F\left(\boldsymbol{x},\boldsymbol{W}\right)-y\right)\right\} ,
\end{equation}
we may as well write
\begin{equation}
\phi^{\widetilde{F}}\left(\boldsymbol{x}\right)\equiv\mathring{g}^{\widetilde{F}}\left(\boldsymbol{h}^{\widetilde{F}}\left(\boldsymbol{x}\right)\right),\quad\boldsymbol{x}\in{\cal X},\label{eq:MAIN_4}
\end{equation}
which is exactly the type of problem considered in \citep{Wang2017},
except for the fact that, here, it is formulated under weaker assumptions.
See Section \ref{subsec:MESSAGE} for details. 

The difference between (\ref{eq:MAIN_3}) and (\ref{eq:MAIN_2}) is
subtle. As we will see later on, based on our assumptions, the structure
of any SSD-type optimization algorithm suitable for handling objectives
of the form of (\ref{eq:MAIN_2}) (and, thus, (\ref{eq:MAIN_PROB}))
for $p\in\left(1,\infty\right)$ is inherently more complicated, compared
to the case of the slightly simpler objective resulting by setting
$p\equiv1$.
\begin{rem}
Note that, alternatively, we could reexpress $\phi^{\widetilde{F}}$
in the compositional form outlined in (\citep{Wang2017}, Supplemental
Materials, Section H.4, or \citep{Wang2018}, Section 4). In particular,
if we define the expectation functions $\widehat{\varrho}:\mathbb{R}\times\mathbb{R}_{+}\rightarrow\mathbb{R}$,
$\widehat{g}^{\widetilde{F}}:\mathbb{R}^{N}\times\mathbb{R}\rightarrow\mathbb{R}\times\mathbb{R}_{+}$
and $\widehat{\boldsymbol{h}}^{\widetilde{F}}:\mathbb{R}^{N}\rightarrow\mathbb{R}^{N}\times\mathbb{R}$
as
\begin{flalign}
\widehat{\varrho}\left(x,y\right) & \triangleq x+cy^{1/p},\quad y>0,\\
\widehat{\boldsymbol{g}}^{\widetilde{F}}\left(\boldsymbol{x},y\right) & \triangleq\left[y\:\mathbb{E}\left\{ \left({\cal R}\left(F\left(\boldsymbol{x},\boldsymbol{W}\right)-y\right)\right)^{p}\right\} \right]\quad\text{and}\\
\widehat{\boldsymbol{h}}^{\widetilde{F}}\left(\boldsymbol{x}\right) & \triangleq\left[\boldsymbol{x}\:\mathbb{E}\left\{ F\left(\boldsymbol{x},\boldsymbol{W}\right)\right\} \right],
\end{flalign}
then $\phi^{\widetilde{F}}$ may be written as
\begin{equation}
\phi^{\widetilde{F}}\left(\boldsymbol{x}\right)=\widehat{\varrho}\left(\widehat{\boldsymbol{g}}^{\widetilde{F}}\left(\widehat{\boldsymbol{h}}^{\widetilde{F}}\left(\boldsymbol{x}\right)\right)\right),\quad\forall\boldsymbol{x}\in{\cal X}.\label{eq:MAIN_Mengdi}
\end{equation}
Of course, the compositional representation (\ref{eq:MAIN_Mengdi})
is equivalent to (\ref{eq:MAIN_2}). For our purposes, though, (\ref{eq:MAIN_2})
is perfectly sufficient and, additionally, it is cleaner and somewhat
more compact.\hfill{}\ensuremath{\blacksquare}
\end{rem}

\subsection{Algorithm Motivation and Differentiability of $\phi^{\widetilde{F}}$}

As a result of the discussion above, the original problem (\ref{eq:Prog_1})
can be equivalently written as
\begin{equation}
\begin{array}{rl}
\underset{\boldsymbol{x}}{\mathrm{minimize}} & \mathbb{E}\left\{ F\left(\boldsymbol{x},\boldsymbol{W}\right)\right\} +c\varrho\left(g^{\widetilde{F}}\left(\boldsymbol{h}^{\widetilde{F}}\left(\boldsymbol{x}\right)\right)\right)\\
\mathrm{subject\,to} & \boldsymbol{x}\in{\cal X}
\end{array}.\label{eq:MAIN_PROB_2}
\end{equation}
Exploiting the assumed convexity of $\phi^{\widetilde{F}}$ on ${\cal X}$,
and given that we are interested in solving (\ref{eq:MAIN_PROB_2}),
one would most reasonably hope for a SSD-type algorithm, whose gradient
evaluation policy follows a path of the stochastic differential equation
\begin{equation}
\boldsymbol{x}^{n+1}=\Pi_{{\cal X}}\left\{ \boldsymbol{x}^{n}-\gamma_{n}\left[\underline{\widetilde{\nabla}}^{n+1}\phi^{\widetilde{F}}\left(\boldsymbol{x}^{n}\right)\right]\right\} ,\quad n\in\mathbb{N},\label{eq:SGD_MUS1_1}
\end{equation}
where $\boldsymbol{x}^{0}\in{\cal X}$ is arbitrarily chosen, $\left\{ \gamma_{n}>0\right\} _{n\in\mathbb{N}}$
is an appropriately chosen stepsize sequence, and $\left\{ \underline{\widetilde{\nabla}}^{n}\phi^{\widetilde{F}}\right\} _{n\in\mathbb{N}^{+}}$
denotes a sequence of \textit{stochastic subgradients}, that is, a
sequence of $\mathbb{R}^{N}$-valued, \textit{appropriately measurable}
\textit{random functions on} $\mathbb{R}^{N}\times\Omega$, such that
\begin{equation}
\mathbb{E}\left\{ \underline{\widetilde{\nabla}}^{n}\phi^{\widetilde{F}}\left(\boldsymbol{x}\right)\right\} \equiv\mathbb{E}\left\{ \underline{\widetilde{\nabla}}^{n}\phi^{\widetilde{F}}\left(\boldsymbol{x},\cdot\right)\right\} \in\partial\phi^{\widetilde{F}}\left(\boldsymbol{x}\right),\quad\forall\left(n,\boldsymbol{x}\right)\in\mathbb{N}^{+}\times{\cal X},
\end{equation}
where we recall that the compact-valued multifunction $\partial\phi^{\widetilde{F}}:\mathbb{R}^{N}\rightrightarrows\mathbb{R}^{N}$
constitutes the subdifferential associated with the convex function
$\phi^{\widetilde{F}}$.

We are now interested in the structural characterization of $\partial\phi^{\widetilde{F}}$.
However, even though $\phi^{\widetilde{F}}$ is convex, a \textit{computationally
useful} characterization of its subdifferential is highly nontrivial;
this is mainly due to the fact that, although any mean-semideviation
risk measure is convex-monotone (as long as $c\in\left[0,1\right]$),
the corresponding dispersion measure can be only guaranteed to be
convex. This implies that composition of the latter with a convex
random function on $\mathbb{R}^{N}$ (such as $F$) does \textit{not}
yield a convex function on $\mathbb{R}^{N}$. Unfortunately, common
rules from subdifferential calculus, such as addition and composition,
which are essential in order to tractably determine the structure
of the multifunction $\partial\phi^{\widetilde{F}}$, are rather complicated
for nonconvex functions (see, e.g., Chapter 10 in \citep{Rockafellar2009VarAn}),
not only conceptually, but most importantly, from a computational
point of view, as well. Fortunately, the problem simplifies considerably
if we impose some mild regularity requirements on the structure of
the random cost function $F$, thus avoiding unnecessary technical
complications.
\begin{assumption}
\label{assu:F_AS_1}The random function $F$ possesses the following
properties:
\begin{description}
\item [{$\mathbf{P1}$}] $\,\,\:$For every $\boldsymbol{x}\in{\cal X}$,
there exists a measurable set $\mathsf{D}_{\boldsymbol{x}}\subseteq\Omega$,
with ${\cal P}\left(\mathsf{D}_{\boldsymbol{x}}\right)\equiv1$, such
that, for all $\omega\in\mathsf{D}_{\boldsymbol{x}}$, the random
function $F\left(\cdot,\boldsymbol{W}\left(\omega\right)\right)$
is differentiable at $\boldsymbol{x}$. In other words, $F$ is differentiable
at each $\boldsymbol{x}\in{\cal X}$, almost everywhere relative to
the base measure ${\cal P}$.
\item [{$\mathbf{P2}$}] $\,\,\:$Let ${\cal A}$ be the countable Borel
nullset of points where the risk regularizer ${\cal R}$ is nondifferentiable.
For every $\boldsymbol{x}\in{\cal X}$, there exists an event $\mathsf{N}_{\boldsymbol{x}}\subseteq\Omega$,
with ${\cal P}\left(\mathsf{N}_{\boldsymbol{x}}\right)\equiv1$, such
that, for all $\omega\in\mathsf{N}_{\boldsymbol{x}}$, $F\left(\boldsymbol{x},\boldsymbol{W}\left(\omega\right)\right)-\mathbb{E}\left\{ F\left(\boldsymbol{x},\boldsymbol{W}\right)\right\} \notin{\cal A}$.
In other words, for every $\boldsymbol{x}\in{\cal X}$, it is true
that
\begin{equation}
{\cal P}\left(F\left(\boldsymbol{x},\boldsymbol{W}\right)-\mathbb{E}\left\{ F\left(\boldsymbol{x},\boldsymbol{W}\right)\right\} \notin{\cal A}\right)\equiv1.
\end{equation}
\end{description}
\end{assumption}
In addition to Properties $\mathbf{P1}$ and $\mathbf{P2}$, it is
technically necessary to make the following basic assumption, concerning
the random subdifferential multifunction of $F$, relative to $\boldsymbol{x}$.
\begin{assumption}
\label{assu:F_AS_1_SUB}There exists a jointly $\mathscr{B}\left(\mathbb{R}^{N}\right)\otimes\mathscr{B}\left(\mathbb{R}^{M}\right)$-measurable
selection of the closed-valued multifunction $\partial F\left(\cdot,\bullet\right):\mathbb{\mathbb{R}}^{N}\times\mathbb{R}^{M}\rightrightarrows\mathbb{R}^{N}$,
say $\underline{\nabla}F\left(\cdot,\bullet\right):\mathbb{\mathbb{R}}^{N}\times\mathbb{R}^{M}\rightarrow\mathbb{R}^{N}$;
this is provided by the ${\cal SO}$, at each $n\in\mathbb{N}$, given
current iterate $\boldsymbol{x}^{n}\in\mathbb{R}^{N}$ and IID process
realization $\boldsymbol{w}\in\mathbb{R}^{M}$.
\end{assumption}
Assumption \ref{assu:F_AS_1_SUB} is important, because it allows
us to integrate $\underline{\nabla}F$ on $\mathbb{R}^{N}\times\mathbb{R}^{M}$,
relative to any qualifying Borel measure, provided such an integral
is well defined. This is extremely useful, in case \textit{both arguments}
$\boldsymbol{x}$ and $\boldsymbol{W}$ are random elements. Hereafter,
Assumption \ref{assu:F_AS_1_SUB} will be considered \textit{implicitly
true}; although it has to be verified case-by-case, it is almost always
true in practice. Utilizing both properties $\mathbf{P1}$ and $\mathbf{P2}$,
the following result may be formulated; it will then be utilized in
the design of SSD-type algorithms, specialized for the convex problem
(\ref{eq:MAIN_PROB}).
\begin{lem}
\textbf{\textup{(Differentiability of $\phi^{\widetilde{F}}$)}}\label{lem:Sub_Grad}
Consider the convex function $\phi^{\widetilde{F}}$. Let Assumption
\ref{assu:F_AS_1} be in effect, and suppose that ${\cal R}$ is not
identically zero everywhere on $\mathbb{R}$. Also, \uline{if \mbox{$p\in\left(1,\infty\right)$}},
and with
\begin{equation}
\kappa_{{\cal R}}\triangleq\sup\left\{ x\in\mathbb{R}\left|{\cal R}\left(x\right)\equiv0\right.\right\} \in\left[-\infty,\infty\right),
\end{equation}
suppose that
\begin{equation}
{\cal P}\left(F\left(\boldsymbol{x},\boldsymbol{W}\right)-\mathbb{E}\left\{ F\left(\boldsymbol{x},\boldsymbol{W}\right)\right\} \le\kappa_{{\cal R}}\right)<1,\quad\forall\boldsymbol{x}\in{\cal X}.\label{eq:Condition_Diff}
\end{equation}
Then $\phi^{\widetilde{F}}$ is differentiable everywhere on ${\cal X}$,
and its gradient $\nabla\phi^{\widetilde{F}}:\mathbb{R}^{N}\rightarrow\mathbb{R}^{N}$
may be expressed as
\begin{flalign}
\nabla\phi^{\widetilde{F}}\left(\boldsymbol{x}\right) & \equiv\mathbb{E}\left\{ \underline{\nabla}F\left(\boldsymbol{x},\boldsymbol{W}\right)\right\} +c\nabla\boldsymbol{h}^{\widetilde{F}}\left(\boldsymbol{x}\right)\nabla g^{\widetilde{F}}\left(\boldsymbol{h}^{\widetilde{F}}\left(\boldsymbol{x}\right)\right)\nabla\varrho\left(g^{\widetilde{F}}\left(\boldsymbol{h}^{\widetilde{F}}\left(\boldsymbol{x}\right)\right)\right),\quad\forall\boldsymbol{x}\in{\cal X},\label{eq:MUS1_Rep1}
\end{flalign}
where the derivative $\nabla\varrho:\mathbb{R}_{+}\rightarrow\mathbb{R}$,
Jacobian $\nabla\boldsymbol{h}^{\widetilde{F}}:\mathbb{R}^{N}\rightarrow\mathbb{R}^{N\times\left(N+1\right)}$
and gradient $\nabla g^{\widetilde{F}}:\mathbb{R}^{N+1}\rightarrow\mathbb{R}^{N+1}$
are given by the expectation functions
\begin{align}
\nabla\varrho\left(x\right) & \equiv\begin{cases}
1,\quad x\ge0, & \text{if }p\equiv1\\
\dfrac{1}{p}x^{\left(1-p\right)/p},\quad x>0, & \text{if }p\in\left(1,\infty\right)
\end{cases},\label{eq:rho_MUS1}\\
\nabla\boldsymbol{h}^{\widetilde{F}}\left(\boldsymbol{x}\right) & \equiv\mathbb{E}\left\{ \hspace{-2pt}\left[\hspace{-2pt}\hspace{-2pt}\begin{array}{c|c}
\\
\boldsymbol{I}_{N}\hspace{-2pt} & \underline{\nabla}F\left(\boldsymbol{x},\boldsymbol{W}\right)\\
\\
\end{array}\hspace{-2pt}\hspace{-2pt}\right]\hspace{-2pt}\right\} \quad\text{and}\label{eq:JAC_MUS1}\\
\nabla g^{\widetilde{F}}\left(\boldsymbol{x},y\right) & \equiv\mathbb{E}\left\{ p\left({\cal R}\left(F\left(\boldsymbol{x},\boldsymbol{W}\right)-y\right)\right)^{p-1}\underline{\nabla}{\cal R}\left(F\left(\boldsymbol{x},\boldsymbol{W}\right)-y\right)\left[\begin{array}{c}
\vphantom{{\displaystyle \int}}\underline{\nabla}F\left(\boldsymbol{x},\boldsymbol{W}\right)\\
\hline \vphantom{{\displaystyle \int}}-1
\end{array}\right]\hspace{-2pt}\right\} ,\label{eq:GRAD_MUS1}
\end{align}
respectively, for every $\left(\boldsymbol{x},y\right)\in\left\{ \left.\left(\boldsymbol{x},y\right)\in{\cal X}\times\mathbb{R}\right|y\equiv\mathbb{E}\left\{ F\left(\boldsymbol{x},\boldsymbol{W}\right)\right\} \right\} \triangleq\mathrm{Graph}_{{\cal X}}\left(\mathbb{E}\left\{ F\left(\cdot,\boldsymbol{W}\right)\right\} \right)$.
\end{lem}
\begin{proof}[Proof of Lemma \ref{lem:Sub_Grad}]
See Section \ref{subsec:MUS-1-Direct-Derivation} (Appendix).
\end{proof}
\begin{rem}
Note that, in Lemma \ref{lem:Sub_Grad}, we have explicitly assumed
that the risk regularizer ${\cal R}$ is \textit{not} identically
zero everywhere on $\mathbb{R}$. If it is, (\ref{eq:Prog_1}) reduces
to the standard, risk neutral stochastic program of minimizing the
expectation of a convex random function over a closed convex set,
well-studied in the literature of stochastic approximation.\hfill{}\ensuremath{\blacksquare}
\end{rem}
\begin{rem}
We would also like to comment on the potential restrictiveness of
condition (\ref{eq:Condition_Diff}). Suppose that $\kappa_{{\cal R}}\le0$.
This essentially means that the risk regularizer ${\cal R}$ always
\textit{positively} penalizes events for which $F\left(\boldsymbol{x},\boldsymbol{W}\right)>\mathbb{E}\left\{ F\left(\boldsymbol{x},\boldsymbol{W}\right)\right\} $.
In such a case, (\ref{eq:Condition_Diff}) will be true if, for every
$\boldsymbol{x}\in{\cal X}$,
\begin{equation}
{\cal P}\left(F\left(\boldsymbol{x},\boldsymbol{W}\right)-\mathbb{E}\left\{ F\left(\boldsymbol{x},\boldsymbol{W}\right)\right\} \le0\right)<1.\label{eq:Condition_0}
\end{equation}
It is then a standard exercise to show that, for every $\boldsymbol{x}\in{\cal X}$,
(\ref{eq:Condition_0}) is true if and only if
\begin{equation}
{\cal P}\left(F\left(\boldsymbol{x},\boldsymbol{W}\right)\equiv\mathbb{E}\left\{ F\left(\boldsymbol{x},\boldsymbol{W}\right)\right\} \right)<1.
\end{equation}
This means that, if $\kappa_{{\cal R}}\le0$, (\ref{eq:Condition_Diff})
translates to a truly mild condition on the structure of $F\left(\cdot,\boldsymbol{W}\right)$,
namely, that, for every feasible decision $\boldsymbol{x}\in{\cal X}$,
the cost $F\left(\boldsymbol{x},\boldsymbol{W}\right)$ cannot be
equal to a constant, almost everywhere relative to ${\cal P}$. In
other words, for every $\boldsymbol{x}\in{\cal X}$, $F\left(\boldsymbol{x},\boldsymbol{W}\right)$
has to be a nontrivial random variable. Of course, it is not hard
to satisfy such a condition in practice.\hfill{}\ensuremath{\blacksquare}
\end{rem}
Lemma \ref{lem:Sub_Grad} establishes that, under Assumption \ref{assu:F_AS_1}
(properties $\mathbf{P1}$ and $\mathbf{P2}$) and, \textit{if} $p\in\left(1,\infty\right)$,
under condition (\ref{eq:Condition_Diff}), the subdifferential of
$\phi^{\widetilde{F}}$ is a singleton, and provides an explicit representation
of the gradient vector $\nabla\phi^{\widetilde{F}}$.

Of course, in the original version of the SSD algorithm, one would
require the availability of a stochastic subgradient sequence in order
to perform the usual SSD update step. However, Lemma \ref{lem:Sub_Grad}
reveals that, under our base assumptions, such a stochastic subgradient
process is far from obvious to obtain, mainly due to the functional
form of $\nabla\phi^{\widetilde{F}}$. In particular, by inspection
of (\ref{eq:MUS1_Rep1}) in Lemma \ref{lem:Sub_Grad}, it is easy
to see that $\nabla\phi^{\widetilde{F}}$ exhibits itself \textit{compositional
structure}, consisting of products of nested expectation functions.
This fact implies that it is generally \textit{not} possible to generate
a stochastic gradient in a \textit{single} sampling step, thus leading
naturally to the idea of developing a \textit{compositional stochastic
subgradient algorithm} (see, for instance, \citep{Wang2017}). In
such an algorithm, the respective stochastic gradient step would be
implemented in a hierarchical fashion \textit{at every iteration},
starting from the ``deepest'' \textit{Stochastic Approximation (SA)
level}, to the ``shallowest'' (in most cases, a \textit{biased}
procedure); see Section \ref{subsec:MESSAGE} for details.

Exploiting Assumptions \ref{assu:1}, \ref{assu:2} \textit{and} \ref{assu:F_AS_1},
and to efficiently exploit the special compositional structure of
$\nabla\phi^{\widetilde{F}}$, we will be particularly interested
in sampled \textit{approximations} of $\nabla\phi^{\widetilde{F}}$,
which are constructed using \textit{sampled realizations} of the random
cost $F\left(\cdot,\boldsymbol{W}\right)$, as well as some corresponding
subgradient, $\underline{\nabla}F\left(\cdot,\boldsymbol{W}\right)$.
From now on, we will implicitly assume that the risk regularizer ${\cal R}$
\textit{is not identically zero everywhere on} $\mathbb{R}$. Otherwise,
the problem reduces to its risk-neutral counterpart.

\subsection{\label{subsec:MESSAGE}The \textit{$\textit{MESSAGE}^{p}$} Algorithm}

For any value of $p\in\left[1,\infty\right)$, the proposed \textit{$\textit{MESSAGE}^{p}$}
algorithm consists of \textit{three} SA levels, as naturally suggested
by the structure of $\nabla\phi^{\widetilde{F}}$, and assumes the
existence of \textit{two} mutually independent, IID information streams,
$\boldsymbol{W}_{1}^{n}$, $\boldsymbol{W}_{2}^{n}$, accessible by
the ${\cal SO}$ (see\textit{ }Assumption \ref{assu:1}), as follows.
In the first (shallowest) SA level, at iteration $n\in\mathbb{N}$,
and given current, random iterates $\boldsymbol{x}^{n}\equiv\boldsymbol{x}^{n}\left(\omega\right)\in{\cal X}$
and $y^{n}\equiv y^{n}\left(\omega\right)\in\mathbb{R}$, the ${\cal SO}$
provides the samples $F\hspace{-2pt}\left(\boldsymbol{x}^{n},\boldsymbol{W}_{1}^{n+1}\right)$
and $\underline{\nabla}F\hspace{-2pt}\left(\boldsymbol{x}^{n},\boldsymbol{W}_{1}^{n+1}\right)$
and the smoothing update
\begin{algorithm}[t]
\begin{shadedbox}
\vspace{3bp}
\textbf{Input}: Initial points $\boldsymbol{x}^{0}\in{\cal X}$, $y^{0}\in\mathbb{R}$,
$z^{0}\in\mathbb{R}$, stepsize sequences $\left\{ \alpha_{n}\right\} _{n\in\mathbb{N}}$,
$\left\{ \beta_{n}\right\} _{n\in\mathbb{N}}$, $\left\{ \gamma_{n}\right\} _{n\in\mathbb{N}}$,
IID sequences $\left\{ \boldsymbol{W}_{1}^{n}\right\} _{n\in\mathbb{N}}$,
$\left\{ \boldsymbol{W}_{2}^{n}\right\} _{n\in\mathbb{N}}$ and penalty
coefficient $c\in\left[0,1\right]$.

\textbf{Output}: Sequence $\left\{ \boldsymbol{x}^{n}\right\} _{n\in\mathbb{N}}$.

1:$\;\;$\textbf{for} $n=0,1,2,\ldots$ \textbf{do}

2:$\;\;$$\quad$Obtain $F\hspace{-2pt}\left(\boldsymbol{x}^{n},\boldsymbol{W}_{1}^{n+1}\right)$
and $\underline{\nabla}F\hspace{-2pt}\left(\boldsymbol{x}^{n},\boldsymbol{W}_{1}^{n+1}\right)$
from the ${\cal SO}$.

3:$\;\;$$\quad$Update (First SA Level):
\begin{flalign*}
y^{n+1} & =\left(1-\beta_{n}\right)y^{n}+\beta_{n}F\hspace{-2pt}\left(\boldsymbol{x}^{n},\boldsymbol{W}_{1}^{n+1}\right)
\end{flalign*}

4:$\;\;$$\quad$Obtain $F\hspace{-2pt}\left(\boldsymbol{x}^{n},\boldsymbol{W}_{2}^{n+1}\right)$
and $\underline{\nabla}F\hspace{-2pt}\left(\boldsymbol{x}^{n},\boldsymbol{W}_{2}^{n+1}\right)$
from the ${\cal SO}$.

5:$\;\;$$\quad$Update (Second SA Level):
\[
z^{n+1}=\begin{cases}
1, & \text{if }p=1\\
\left(1-\gamma_{n}\right)z^{n}+\gamma_{n}\left({\cal R}\left(F\hspace{-2pt}\left(\boldsymbol{x}^{n},\boldsymbol{W}_{2}^{n+1}\right)\hspace{-2pt}-\hspace{-2pt}y^{n}\right)\right)^{p}, & \text{if }p>1
\end{cases}
\]

6:$\;\;$$\quad$Define auxiliary variables:
\begin{flalign*}
\delta & =F\left(\boldsymbol{x}^{n},\boldsymbol{W}_{2}^{n+1}\right)-y^{n}\\
\boldsymbol{\delta}^{\underline{\nabla}} & =\underline{\nabla}F\hspace{-2pt}\left(\boldsymbol{x}^{n},\boldsymbol{W}_{2}^{n+1}\right)-\underline{\nabla}F\hspace{-2pt}\left(\boldsymbol{x}^{n},\boldsymbol{W}_{1}^{n+1}\right)\\
\Delta & =\boldsymbol{\delta}^{\underline{\nabla}}\underline{\nabla}{\cal R}\left(\delta\right)\left({\cal R}\left(\delta\right)\right)^{p-1}\left(z^{n}\right)^{\left(1-p\right)/p}
\end{flalign*}
7:$\;\;$$\quad$Update (Third SA Level):
\[
\boldsymbol{x}^{n+1}=\Pi_{{\cal X}}\left\{ \boldsymbol{x}^{n}-\alpha_{n}\left(\underline{\nabla}F\hspace{-2pt}\left(\boldsymbol{x}^{n},\boldsymbol{W}_{2}^{n+1}\right)\hspace{-2pt}+c\Delta\right)\hspace{-2pt}\right\} 
\]

8:$\;\;$\textbf{end for}
\end{shadedbox}
\vspace{-9bp}

\caption{\label{alg:SCGD-3}\textit{$\;\textit{MESSAGE}^{p}$}}
\end{algorithm}
\begin{equation}
y^{n+1}=\left(1-\beta_{n}\right)y^{n}+\beta_{n}F\hspace{-2pt}\left(\boldsymbol{x}^{n},\boldsymbol{W}_{1}^{n+1}\right)\label{eq:HIER_1-1}
\end{equation}
is performed, where $\left\{ \beta_{n}>0\right\} _{n\in\mathbb{N}}$
is an appropriately chosen stepsize sequence. In the second SA level,
the ${\cal SO}$ provides the samples $F\hspace{-2pt}\left(\boldsymbol{x}^{n},\boldsymbol{W}_{2}^{n+1}\right)$
and $\underline{\nabla}F\hspace{-2pt}\left(\boldsymbol{x}^{n},\boldsymbol{W}_{2}^{n+1}\right)$,
and another smoothing update
\begin{equation}
z^{n+1}=\begin{cases}
1, & \text{if }p=1\\
\left(1-\gamma_{n}\right)z^{n}+\gamma_{n}\left({\cal R}\left(F\hspace{-2pt}\left(\boldsymbol{x}^{n},\boldsymbol{W}_{2}^{n+1}\right)\hspace{-2pt}-\hspace{-2pt}y^{n}\right)\right)^{p}, & \text{if }p>1
\end{cases}
\end{equation}
is performed, with $\left\{ \gamma_{n}>0\right\} _{n\in\mathbb{N}}$
being another appropriately chosen stepsize sequence. Of course, in
the simpler case where $p\equiv1$, no actual update is performed.
In the third (deepest) SA level, \textit{with no additional information
by the ${\cal SO}$}, and by defining the stochastic gradient \textit{approximation}
$\widehat{\nabla}^{n+1}\phi^{\widetilde{F}}:\mathbb{R}^{N+2}\times\Omega\rightarrow\mathbb{R}$
as
\begin{flalign}
 & \hspace{-2pt}\hspace{-2pt}\hspace{-2pt}\hspace{-2pt}\widehat{\nabla}^{n+1}\phi^{\widetilde{F}}\left(\boldsymbol{x}^{n},y^{n},z^{n}\right)\equiv\widehat{\nabla}^{n+1}\phi^{\widetilde{F}}\left(\boldsymbol{x}^{n},y^{n},z^{n},\cdot\right)\nonumber \\
 & \triangleq\underline{\nabla}F\hspace{-2pt}\left(\boldsymbol{x}^{n},\hspace{-1pt}\boldsymbol{W}_{2}^{n+1}\right)\hspace{-2pt}+\hspace{-2pt}c\left(z^{n}\right)^{\left(1-p\right)/p}\left[\hspace{-2pt}\hspace{-2pt}\hspace{-2pt}\begin{array}{c|c}
\\
\boldsymbol{I}_{N}\hspace{-2pt}\hspace{-1pt} & \hspace{-2pt}\underline{\nabla}F\hspace{-2pt}\left(\boldsymbol{x}^{n},\hspace{-1pt}\boldsymbol{W}_{1}^{n+1}\right)\\
\\
\end{array}\hspace{-2pt}\hspace{-2pt}\hspace{-2pt}\right]\hspace{-2pt}\hspace{-2pt}\hspace{-2pt}\nonumber \\
 & \quad\quad\times\left[\begin{array}{c}
\vphantom{{\displaystyle \int}}\underline{\nabla}F\hspace{-2pt}\left(\boldsymbol{x}^{n},\hspace{-1pt}\boldsymbol{W}_{2}^{n+1}\right)\\
\hline \vphantom{{\displaystyle \int}}-1
\end{array}\right]\underline{\nabla}{\cal R}\hspace{-2pt}\left(F\hspace{-2pt}\left(\boldsymbol{x}^{n},\hspace{-1pt}\boldsymbol{W}_{2}^{n+1}\right)\hspace{-2pt}-y^{n}\right)\hspace{-2pt}\hspace{-2pt}\left({\cal R}\hspace{-2pt}\left(F\hspace{-2pt}\left(\boldsymbol{x}^{n},\hspace{-1pt}\boldsymbol{W}_{2}^{n+1}\right)\hspace{-2pt}-y^{n}\right)\right)^{p-1}\nonumber \\
 & \equiv\underline{\nabla}F\hspace{-2pt}\left(\boldsymbol{x}^{n},\hspace{-1pt}\boldsymbol{W}_{2}^{n+1}\right)\hspace{-2pt}+\hspace{-2pt}c\left(z^{n}\right)^{\left(1-p\right)/p}\left(\underline{\nabla}F\hspace{-2pt}\left(\boldsymbol{x}^{n},\hspace{-1pt}\boldsymbol{W}_{2}^{n+1}\right)-\underline{\nabla}F\hspace{-2pt}\left(\boldsymbol{x}^{n},\hspace{-1pt}\boldsymbol{W}_{1}^{n+1}\right)\right)\nonumber \\
 & \quad\quad\quad\quad\quad\quad\quad\quad\quad\quad\quad\times\underline{\nabla}{\cal R}\hspace{-2pt}\left(F\hspace{-2pt}\left(\boldsymbol{x}^{n},\hspace{-1pt}\boldsymbol{W}_{2}^{n+1}\right)\hspace{-2pt}-y^{n}\right)\hspace{-2pt}\hspace{-2pt}\left({\cal R}\hspace{-2pt}\left(F\hspace{-2pt}\left(\boldsymbol{x}^{n},\hspace{-1pt}\boldsymbol{W}_{2}^{n+1}\right)\hspace{-2pt}-y^{n}\right)\right)^{p-1}\nonumber \\
 & \triangleq\underline{\nabla}F\hspace{-2pt}\left(\boldsymbol{x}^{n},\hspace{-1pt}\boldsymbol{W}_{2}^{n+1}\right)\hspace{-2pt}+c\Delta^{n+1}\left(\boldsymbol{x}^{n},y^{n},z^{n}\right),
\end{flalign}
we update the current estimate $\boldsymbol{x}^{n}$ as
\begin{flalign}
\boldsymbol{x}^{n+1}\hspace{-2pt} & \equiv\hspace{-2pt}\Pi_{{\cal X}}\hspace{-2pt}\left\{ \boldsymbol{x}^{n}\hspace{-2pt}-\hspace{-2pt}\alpha_{n}\widehat{\nabla}^{n+1}\phi^{\widetilde{F}}\left(\boldsymbol{x}^{n},y^{n},z^{n}\right)\right\} \nonumber \\
\hspace{-2pt} & \equiv\hspace{-2pt}\Pi_{{\cal X}}\hspace{-2pt}\left\{ \boldsymbol{x}^{n}\hspace{-2pt}-\hspace{-2pt}\alpha_{n}\hspace{-2pt}\left(\underline{\nabla}F\hspace{-2pt}\left(\boldsymbol{x}^{n},\hspace{-1pt}\boldsymbol{W}_{2}^{n+1}\right)\hspace{-2pt}+\hspace{-2pt}c\Delta^{n+1}\left(\boldsymbol{x}^{n},y^{n},z^{n}\right)\right)\hspace{-2pt}\right\} ,\label{eq:HIER_2-1}
\end{flalign}
where $\left\{ \alpha_{n}\ge0\right\} _{n\in\mathbb{N}}$ is another
appropriately chosen stepsize sequence. In the above, $\Delta^{n}:\mathbb{R}^{N+2}\times\Omega\rightarrow\mathbb{R}$,
$n\in\mathbb{N}^{+}$ (a random function of the involved quantities)
may be viewed as a \textit{risk-averse correction sequence}, weighted
by the penalty multiplier $c\in\left[0,1\right]$. Again, if $p\equiv1$,
the correction $\Delta^{n}$ is simplified accordingly as
\begin{flalign}
\Delta^{n+1}\left(\boldsymbol{x}^{n},y^{n},z^{n}\right) & \equiv\Delta^{n+1}\left(\boldsymbol{x}^{n},y^{n},1\right)\nonumber \\
 & \equiv\underline{\nabla}{\cal R}\hspace{-2pt}\left(F\hspace{-2pt}\left(\boldsymbol{x}^{n},\hspace{-1pt}\boldsymbol{W}_{2}^{n+1}\right)\hspace{-2pt}-y^{n}\right)\left(\underline{\nabla}F\hspace{-2pt}\left(\boldsymbol{x}^{n},\hspace{-1pt}\boldsymbol{W}_{2}^{n+1}\right)-\underline{\nabla}F\hspace{-2pt}\left(\boldsymbol{x}^{n},\hspace{-1pt}\boldsymbol{W}_{1}^{n+1}\right)\right),
\end{flalign}
for all $n\in\mathbb{N}$. As we will shortly see, whether $p\equiv1$
or $p>1$ has nontrivial consequences in regard to the asymptotic
performance of the \textit{$\textit{MESSAGE}^{p}$ }algorithm. The
iterative optimization procedure outlined above (the \textit{$\textit{MESSAGE}^{p}$}
algorithm) is summarized in Algorithm \ref{alg:SCGD-3}.

Exploiting Assumption \ref{assu:F_AS_1_SUB}, it is then easy to verify
that, for every $\left(n,\boldsymbol{x}\right)\in\mathbb{N}^{+}\times{\cal X}$,
\begin{align}
\mathbb{E}\left\{ \widehat{\nabla}^{n}\phi^{\widetilde{F}}\left(\boldsymbol{x},\mathbb{E}\left\{ F\hspace{-2pt}\left(\boldsymbol{x},\boldsymbol{W}\right)\right\} ,\mathbb{E}\left\{ \left({\cal R}\left(F\hspace{-2pt}\left(\boldsymbol{x},\boldsymbol{W}\right)\hspace{-2pt}-\hspace{-2pt}\mathbb{E}\left\{ F\hspace{-2pt}\left(\boldsymbol{x},\boldsymbol{W}\right)\right\} \right)\right)^{p}\right\} \right)\right\}  & \equiv\nabla\phi^{\widetilde{F}}\left(\boldsymbol{x}\right).
\end{align}
This implies that while the random function $\widehat{\nabla}^{n}\phi^{\widetilde{F}}\left(\cdot,y^{n},z^{n}\right)$
\textit{does not} necessarily yield a stochastic gradient of $\phi^{\widetilde{F}}$,
for $n\in\mathbb{N}$, $\widehat{\nabla}^{n}\phi^{\widetilde{F}}\left(\cdot,\mathbb{E}\left\{ F\hspace{-2pt}\left(\cdot,\boldsymbol{W}\right)\right\} ,\mathbb{E}\left\{ \left({\cal R}\left(F\hspace{-2pt}\left(\cdot,\boldsymbol{W}\right)\hspace{-2pt}-\hspace{-2pt}\mathbb{E}\left\{ F\hspace{-2pt}\left(\cdot,\boldsymbol{W}\right)\right\} \right)\right)^{p}\right\} \right)$
\textit{does}; this is a key fact in the analysis of general purpose
compositional stochastic subgradient algorithms \citep{Wang2018}.
\begin{rem}
In relation to the brief discussion above, it might be helpful to
observe that, by the substitution rule for conditional expectations
(again due to Assumption \ref{assu:F_AS_1_SUB}), it is also true
that
\begin{equation}
\hspace{-2pt}\hspace{-2pt}\mathbb{E}\hspace{-2pt}\left\{ \hspace{-1pt}\left.\widehat{\nabla}^{n+1}\phi^{\widetilde{F}}\hspace{-2pt}\left(\boldsymbol{x}^{n},\left.\mathbb{E}\hspace{-2pt}\left\{ F\hspace{-2pt}\left(\boldsymbol{x},\hspace{-1pt}\boldsymbol{W}\right)\right\} \right|_{\boldsymbol{x}\equiv\boldsymbol{x}^{n}}\hspace{-1pt},\left.\mathbb{E}\hspace{-2pt}\left\{ \left({\cal R}\left(F\hspace{-2pt}\left(\boldsymbol{x},\hspace{-1pt}\boldsymbol{W}\right)\hspace{-2pt}-\hspace{-2pt}\mathbb{E}\hspace{-2pt}\left\{ F\hspace{-2pt}\left(\boldsymbol{x},\hspace{-1pt}\boldsymbol{W}\right)\right\} \right)\right)^{p}\right\} \right|_{\boldsymbol{x}\equiv\boldsymbol{x}^{n}}\right)\right|\boldsymbol{x}^{n}\right\} \hspace{-2pt}\equiv\hspace{-2pt}\nabla\phi^{\widetilde{F}}\hspace{-2pt}\left(\boldsymbol{x}^{n}\right)\hspace{-1pt},\hspace{-2pt}
\end{equation}
almost everywhere relative to ${\cal P}$ and, apparently, this is
\textit{not} the case for the conditional expectation of $\widehat{\nabla}^{n+1}\phi^{\widetilde{F}}\left(\boldsymbol{x}^{n},y^{n},z^{n}\right)$
relative to $\sigma\left\{ \boldsymbol{x}^{n}\right\} $, or even
$\sigma\left\{ \boldsymbol{x}^{0},\boldsymbol{x}^{1},\ldots,\boldsymbol{x}^{n}\right\} $.
In this sense, we might say that the hierarchical approximate gradient
sampling scheme described by (\ref{eq:HIER_1-1}) and (\ref{eq:HIER_2-1})
is \textit{conditionally biased}.\hfill{}\ensuremath{\blacksquare}
\end{rem}
\vspace{-5bp}

\subsection{\label{subsec:Convergence-Analysis}Convergence Analysis}

Next, we present and discuss the proposed structural framework, explicitly
demonstrating its generality and flexibility. Subsequently, we proceed
with a detailed presentation of our main technical results, concerning
the asymptotic behavior of the \textit{$\textit{MESSAGE}^{p}$} algorithm;
we study pathwise convergence first, and rate of convergence second.
\vspace{-5bp}

\subsubsection{\label{par:Structural-Assumptions}Structural Assumptions}

Hereafter, for some measurable set $\Omega_{E}\subseteq\Omega$, such
that ${\cal P}\left(\Omega_{E}\right)\equiv1$, let us define the
quantities
\begin{equation}
m_{l}\triangleq\inf_{\boldsymbol{x}\in{\cal X}}\inf_{\omega\in\Omega_{E}}F\left(\boldsymbol{x},\boldsymbol{W}\left(\omega\right)\right)\quad\text{and}\quad m_{h}\triangleq\sup_{\boldsymbol{x}\in{\cal X}}\sup_{\omega\in\Omega_{E}}F\left(\boldsymbol{x},\boldsymbol{W}\left(\omega\right)\right),
\end{equation}
and let $\mathfrak{R}^{\widetilde{F}}\triangleq\textrm{cl}\left\{ \left(m_{l}-m_{h},m_{h}-m_{l}\right)\right\} $,
where the closure is taken relative to the usual Euclidean topology
on $\mathbb{R}$. The structural problem assumptions considered in
this paper follow. 
\begin{shadedbox}
\begin{assumption}
\label{assu:F_AS_Main}For $P\in\left[2,\infty\right]$ and $Q\in\left[P/\hspace{-2pt}\left(P\hspace{-2pt}-\hspace{-2pt}1\right)\hspace{-2pt},\infty\right]$\footnote{As usual, the case $P\equiv\infty$ is understood as a limit.},
$F\left(\cdot,\hspace{-1pt}\boldsymbol{W}\right)$ and ${\cal R}$
satisfy the conditions:
\begin{description}
\item [{$\mathbf{C1}$}] $\,\,\:$For chosen random subgradient $\underline{\nabla}F\left(\cdot,\hspace{-1pt}\boldsymbol{W}\right)$,
there exists a number $G<\infty$, such that
\[
\sup_{\boldsymbol{x}\in{\cal X}}\left[\mathbb{E}\left\{ \left\Vert \underline{\nabla}F\left(\boldsymbol{x},\hspace{-1pt}\boldsymbol{W}\right)\right\Vert _{2}^{P}\right\} \right]^{1/P}\triangleq\sup_{\boldsymbol{x}\in{\cal X}}\left\Vert \vphantom{\varint}\hspace{-2pt}\left\Vert \underline{\nabla}F\left(\boldsymbol{x},\hspace{-1pt}\boldsymbol{W}\right)\right\Vert _{2}\right\Vert _{{\cal L}_{P}}\le G.
\]
In other words, the $\ell_{2}$-norm of $\underline{\nabla}F\left(\cdot,\hspace{-1pt}\boldsymbol{W}\right)$
has bounded ${\cal L}_{P}$-norm, uniformly over ${\cal X}$.
\item [{$\mathbf{C2}$}] $\,\,\:$There exists a number $V<\infty$, such
that
\[
\sup_{\boldsymbol{x}\in{\cal X}}\mathbb{V}\left\{ F\left(\boldsymbol{x},\hspace{-1pt}\boldsymbol{W}\right)\right\} \triangleq\sup_{\boldsymbol{x}\in{\cal X}}\left[\mathbb{E}\left\{ \left(F\left(\boldsymbol{x},\hspace{-1pt}\boldsymbol{W}\right)\right)^{2}\right\} -\left(\mathbb{E}\left\{ F\left(\boldsymbol{x},\hspace{-1pt}\boldsymbol{W}\right)\right\} \right)^{2}\right]\le V,
\]
that is, uniformly on ${\cal X}$, $F\left(\cdot,\hspace{-1pt}\boldsymbol{W}\right)$
is of bounded variance.
\item [{$\mathbf{C3}$}] $\,\,\:$For chosen subderivative $\underline{\nabla}{\cal R}$,
there exists another number $D<\infty$, such that\footnote{Note that, by convexity, it always is true that $\underline{\nabla}\left({\cal R}\left(z\right)\right)^{p}\equiv p\left({\cal R}\left(z\right)\right)^{p-1}\underline{\nabla}{\cal R}\left(z\right)$.}
\[
\sup_{\boldsymbol{x}\in{\cal X}}\left\Vert \vphantom{\varint}\hspace{-2pt}\left|\left.\underline{\nabla}\left({\cal R}\left(z\right)\right)^{p}\right|_{z\equiv F\left(\boldsymbol{x},\hspace{-1pt}\boldsymbol{W}\right)-y_{1}}-\left.\underline{\nabla}\left({\cal R}\left(z\right)\right)^{p}\right|_{z\equiv F\left(\boldsymbol{x},\hspace{-1pt}\boldsymbol{W}\right)-y_{2}}\right|\right\Vert _{{\cal L}_{Q}}\le D\left|y_{1}-y_{2}\right|,
\]
for all $\left(y_{1},y_{2}\right)\in\left[\mathrm{cl}\left\{ \left(m_{l},m_{h}\right)\right\} \right]^{2}$.
This is a Lipschitz-in-Expectation type of condition.
\item [{$\mathbf{C4}$}] $\,\,\:$\uline{Whenever \mbox{$p>1$}}, it
is true that $-\infty<m_{l}\le m_{h}<\infty$, and
\[
0<\varepsilon\triangleq{\cal R}\left(m_{l}-m_{h}\right)\le{\cal R}\left(m_{h}-m_{l}\right)\triangleq{\cal E}<\infty.
\]
In other words, the risk regularizer ${\cal R}$ is strictly positively
uniformly bounded within $\mathfrak{R}^{\widetilde{F}}$.
\end{description}
\end{assumption}
\end{shadedbox}

Let us briefly comment on the various conditions of Assumption \ref{assu:F_AS_Main}.
First, conditions $\mathbf{C1}$ and $\mathbf{C3}$ reveal a probably
fundamental trade-off between the size of the $\ell_{2}$-norm of
the random subgradient $\underline{\nabla}F\left(\boldsymbol{x},\hspace{-1pt}\boldsymbol{W}\right)$,
which may be thought of as a measure of the expansiveness of the random
cost function $F\left(\cdot,\boldsymbol{W}\right)$, and the size
of the \textit{slope} of the random subgradient $p\left({\cal R}\left(F\left(\cdot,\hspace{-1pt}\boldsymbol{W}\right)-\bullet\right)\right)^{p-1}$
$\times\underline{\nabla}{\cal R}\left(F\left(\cdot,\hspace{-1pt}\boldsymbol{W}\right)-\bullet\right)$,
which is directly related to the smoothness of the risk regularizer
${\cal R}$, as well as the smoothness of the distribution of $F\left(\cdot,\boldsymbol{W}\right)$.
This trade-off is explicitly demonstrated through the following simple
result, which presents at least three ways of increasing generality,
for ensuring validity of condition $\mathbf{C3}$, for certain values
of the exponent pair $\left(P,Q\right)$.
\begin{prop}
\textbf{\textup{(Ensuring Validity of }}$\mathbf{C3}$\textbf{\textup{)\label{prop:C3_Valid}}}
Assume that, whenever $p>1$, condition ${\bf C4}$ is satisfied.
Then, the following statements are true:
\begin{enumerate}
\item Suppose that the $p$-th power of ${\cal R}$ is differentiable on
$\mathfrak{R}^{\widetilde{F}}$, and that there is $D_{{\cal R},p}<\infty$,
such that
\begin{equation}
\left|\nabla\left({\cal R}\left(y_{1}\right)\right)^{p}-\nabla\left({\cal R}\left(y_{1}\right)\right)^{p}\right|\le D_{{\cal R},p}\left|y_{1}-y_{2}\right|,\quad\forall\left(y_{1},y_{2}\right)\in\left[\mathfrak{R}^{\widetilde{F}}\right]^{2}.\label{eq:Valid_1}
\end{equation}
Then, condition $\mathbf{C3}$ is satisfied for every choice of $Q\in\left[P/\hspace{-2pt}\left(P\hspace{-2pt}-\hspace{-2pt}1\right)\hspace{-2pt},\infty\right]$,
for every choice of $P\in\left[2,\infty\right]$.
\item Choose $\underline{\nabla}{\cal R}\equiv{\cal R}'_{+}$, and if $F_{\boldsymbol{W}}^{\left(\cdot\right)}:\mathbb{R}\rightarrow\left[0,1\right]$
denotes the cdf of $F\left(\cdot,\boldsymbol{W}\right)$, suppose
that there exists $D_{\widetilde{F}}<\infty$, such that
\begin{equation}
\sup_{\boldsymbol{x}\in{\cal X}}\left|F_{\boldsymbol{W}}^{\boldsymbol{x}}\left(y_{1}\right)-F_{\boldsymbol{W}}^{\boldsymbol{x}}\left(y_{2}\right)\right|\le D_{\widetilde{F}}\left|y_{1}-y_{2}\right|,\quad\forall\left(y_{1},y_{2}\right)\in\left[\mathrm{cl}\left\{ \left(m_{l},m_{h}\right)\right\} \right]^{2}.\label{eq:Valid_2}
\end{equation}
Then, condition $\mathbf{C3}$ is satisfied for $Q\equiv1$ (implying
that $P\equiv\infty$), for every value of $p\in\left[1,\infty\right)$.
\item More generally, whenever $p\equiv1$ and for any choice of ${\cal R}$,
take $\underline{\nabla}{\cal R}\equiv{\cal R}'_{+}$, let $Y:\Omega\rightarrow\mathbb{R}$
be the random variable associated with ${\cal R}$, as in Theorem
\ref{thm:IFF_RR}, and suppose that $F_{\boldsymbol{W}}^{\left(\cdot\right)}$
is continuous everywhere on $\mathbb{R}$ (not necessarily Lipschitz).
Then, condition $\mathbf{C3}$ is satisfied for $Q\equiv1$ (implying
that $P\equiv\infty$) if and only if there exists some $D_{{\cal R}}^{\widetilde{F}}<\infty$,
such that the Lipschitz-in-Expectation condition
\begin{equation}
\sup_{\boldsymbol{x}\in{\cal X}}\int\left|F_{\boldsymbol{W}}^{\boldsymbol{x}}\left(y+y_{1}\right)-F_{\boldsymbol{W}}^{\boldsymbol{x}}\left(y+y_{2}\right)\right|\mathrm{d}{\cal P}_{Y}\left(y\right)\le D_{{\cal R}}^{\widetilde{F}}\left|y_{1}-y_{2}\right|,\label{eq:Valid_3}
\end{equation}
is satisfied, for all $\left(y_{1},y_{2}\right)\in\left[\mathrm{cl}\left\{ \left(m_{l},m_{h}\right)\right\} \right]^{2}$.
If $p>1$, (\ref{eq:Valid_3}) is only sufficient for condition ${\bf C3}$.
\end{enumerate}
\end{prop}
\begin{proof}[Proof of Proposition \ref{prop:C3_Valid}]
See Section \ref{subsec:Proof-of-Proposition_5} (Appendix). 
\end{proof}
Proposition \ref{prop:C3_Valid} demonstrates the versatility of condition
$\mathbf{C3}$, mainly relative to the choice of $Q$. First, observe
that if $P\equiv2$, then $Q$ can be \textit{anything} in $\left[2,\infty\right]$.
If, for instance, we choose $Q\equiv\infty$ (this is most easiest
to verify from a technical viewpoint; also see Remark \ref{rem:We-would-like}
below), the \textit{almost Lipschitz} assumption imposed by condition
$\mathbf{C3}$ on the $p$-th power of the chosen risk regularizer
${\cal R}$ might be \textit{severely restrictive}, depending on the
value $p$. More specifically, a model satisfying (or required to
satisfy) Assumption \ref{assu:F_AS_Main} for $Q\equiv\infty$ (in
which case $\mathbf{C3}$ is \textit{almost equivalent} to the respective
condition in the first part of Proposition \ref{prop:C3_Valid}) \textit{might}
\textit{not} allow for risk regularizers exhibiting \textit{corner
points}. Let us illustrate this by means of an example. Let $p\equiv1$,
choose ${\cal R}$ to be the upper-semideviation regularizer, that
is, ${\cal R}\equiv\left(\cdot\right)_{+}\equiv\max\left\{ \cdot,0\right\} $,
and consider the \textit{linear} objective
\begin{equation}
F\left(\boldsymbol{x},\hspace{-1pt}\boldsymbol{W}\right)\triangleq\boldsymbol{\boldsymbol{W}}_{1}^{\boldsymbol{T}}\boldsymbol{x}+W\in\mathbb{R},
\end{equation}
where $\boldsymbol{W}\triangleq\left[\boldsymbol{W}_{1}^{\boldsymbol{T}}\,W\right]^{\boldsymbol{T}}$,
where $\boldsymbol{\boldsymbol{W}}_{1}:\Omega\rightarrow\mathbb{R}^{N}$
constitutes an absolutely continuous random element almost everywhere
in $\left[0,1\right]^{N}$, $\mathbb{E}\left\{ \boldsymbol{\boldsymbol{W}}_{1}\right\} \equiv\boldsymbol{\mu}$,
and $W\sim{\cal N}\left(0,1\right)$. Then, we are interested in the
\textit{nonlinear} \textit{(convex)}, risk-averse stochastic program
\begin{equation}
\underset{\boldsymbol{x}\in{\cal X}}{\inf}\:\boldsymbol{\mu}^{\boldsymbol{T}}\boldsymbol{x}+c\mathbb{E}\left\{ \left(\boldsymbol{\boldsymbol{W}}_{1}^{\boldsymbol{T}}\boldsymbol{x}-\boldsymbol{\mu}^{\boldsymbol{T}}\boldsymbol{x}+W\right)_{+}\right\} ,
\end{equation}
for some closed, convex set ${\cal X}$. Note that, for every choice
of ${\cal X}$ (compact or not), it is true that $m_{l}\equiv-\infty$
and $m_{h}\equiv+\infty$, since the random element $W$ is unbounded.
Thus, $\mathrm{cl}\left\{ \left(m_{l},m_{h}\right)\right\} \equiv\mathbb{R}$.
Consequently, for this problem, condition ${\bf C3}$ (for $Q\equiv\infty$)
demands the existence of a number $D<\infty$, such that\footnote{In this case, we simply take $\underline{\nabla}{\cal R}\left(\cdot\right)\equiv\underline{\nabla}\left(\cdot\right)_{+}\equiv\mathds{1}_{\left\{ \left(\cdot\right)\ge0\right\} }$}
\begin{equation}
\sup_{\boldsymbol{x}\in{\cal X}}\underset{\omega\in\Omega}{\mathrm{ess\hspace{1bp}sup}}\left|\mathds{1}_{\left\{ F\left(\boldsymbol{x},\boldsymbol{W}\left(\omega\right)\right)\ge y_{1}\right\} }-\mathds{1}_{\left\{ F\left(\boldsymbol{x},\boldsymbol{W}\left(\omega\right)\right)\ge y_{2}\right\} }\right|\le D\left|y_{1}-y_{2}\right|,
\end{equation}
for all $\left(y_{1},y_{2}\right)\in\mathbb{R}^{2}$, or by definition
of the essential supremum (\citep{Bogachev2007}, p. 250), 
\begin{equation}
\sup_{\boldsymbol{x}\in{\cal X}}\hspace{1bp}\inf_{\mathscr{F}\ni\Omega'\subseteq\Omega:{\cal P}\left(\Omega'\right)\equiv1}\hspace{1bp}\sup_{\omega\in\Omega'}\left|\mathds{1}_{\left\{ F\left(\boldsymbol{x},\boldsymbol{W}\left(\omega\right)\right)\ge y_{1}\right\} }-\mathds{1}_{\left\{ F\left(\boldsymbol{x},\boldsymbol{W}\left(\omega\right)\right)\ge y_{2}\right\} }\right|\le D\left|y_{1}-y_{2}\right|,\label{eq:ESSSUP_OPEN}
\end{equation}
for all $\left(y_{1},y_{2}\right)\in\mathbb{R}^{2}$. It is relatively
easy to show that it is actually impossible for (\ref{eq:ESSSUP_OPEN})
to hold \textit{uniformly in} $\left(y_{1},y_{2}\right)\in\mathbb{R}^{2}$,
for any possible choice of $D>0$. Indeed, for simplicity, consider
the symmetric ``antidiagonal'' case where $y_{1}\equiv-y_{2}\triangleq z>0$.
Then, for each fixed $\boldsymbol{x}\in{\cal X}$, it is true that
\begin{flalign}
\left|\mathds{1}_{\left\{ F\left(\boldsymbol{x},\boldsymbol{W}\left(\omega\right)\right)\ge y_{1}\right\} }-\mathds{1}_{\left\{ F\left(\boldsymbol{x},\boldsymbol{W}\left(\omega\right)\right)\ge y_{2}\right\} }\right| & \equiv\left|\mathds{1}_{\left\{ F\left(\boldsymbol{x},\boldsymbol{W}\left(\omega\right)\right)\ge z\right\} }-\mathds{1}_{\left\{ F\left(\boldsymbol{x},\boldsymbol{W}\left(\omega\right)\right)\ge-z\right\} }\right|\nonumber \\
 & =\mathds{1}_{\left\{ F\left(\boldsymbol{x},\boldsymbol{W}\left(\omega\right)\right)\in\left[-z,z\right)\right\} }\nonumber \\
 & \equiv\mathds{1}_{\Omega_{\boldsymbol{x}}^{z}}\left(\omega\right),\quad\forall\omega\in\Omega,
\end{flalign}
where the event $\Omega_{\boldsymbol{x}}^{z}\in\mathscr{F}$ is defined
as
\begin{equation}
\Omega_{\boldsymbol{x}}^{z}\triangleq\left\{ \omega\in\Omega\left|F\left(\boldsymbol{x},\boldsymbol{W}\left(\omega\right)\right)\in\left[-z,z\right)\right.\right\} ,\quad\forall\left(\boldsymbol{x},z\right)\in{\cal X}\times\mathbb{R}_{++},
\end{equation}
and where we emphasize that, due to our assumptions, it holds that
${\cal P}\left(\Omega_{\boldsymbol{x}}^{z}\right)>0$, for every choice
of $\left(\boldsymbol{x},z\right)\in{\cal X}\times\mathbb{R}_{++}$.
Consequently, for an arbitrary event $\Omega'\subseteq\Omega$ such
that ${\cal P}\left(\Omega'\right)\equiv1$, we have 
\begin{flalign}
 & \hspace{-2pt}\hspace{-2pt}\hspace{-2pt}\hspace{-2pt}\hspace{-2pt}\hspace{-2pt}\hspace{-2pt}\hspace{-2pt}\hspace{-2pt}\sup_{\omega\in\Omega'}\left|\mathds{1}_{\left\{ F\left(\boldsymbol{x},\boldsymbol{W}\left(\omega\right)\right)\ge y_{1}\right\} }-\mathds{1}_{\left\{ F\left(\boldsymbol{x},\boldsymbol{W}\left(\omega\right)\right)\ge y_{2}\right\} }\right|\equiv\sup_{\omega\in\Omega'}\mathds{1}_{\Omega_{\boldsymbol{x}}^{z}}\left(\omega\right)\nonumber \\
 & \equiv\max\left\{ \sup_{\omega\in\Omega'\cap\Omega_{\boldsymbol{x}}^{z}}\mathds{1}_{\Omega_{\boldsymbol{x}}^{z}}\left(\omega\right),\sup_{\omega\in\Omega'\cap\left(\Omega_{\boldsymbol{x}}^{z}\right)^{c}}\mathds{1}_{\Omega_{\boldsymbol{x}}^{z}}\left(\omega\right)\right\} =\max\left\{ 1,0\right\} \equiv1,\quad\forall\left(\boldsymbol{x},z\right)\in{\cal X}\times\mathbb{R}_{++},\hspace{-2pt}\hspace{-2pt}\hspace{-2pt}\hspace{-2pt}\hspace{-2pt}
\end{flalign}
which implies in particular that, unless $z\ge\left(2D\right)^{-1}$,
\begin{equation}
1\equiv\sup_{\boldsymbol{x}\in{\cal X}}\underset{\omega\in\Omega}{\mathrm{ess\hspace{1bp}sup}}\left|\mathds{1}_{\left\{ F\left(\boldsymbol{x},\boldsymbol{W}\left(\omega\right)\right)\ge y_{1}\right\} }-\mathds{1}_{\left\{ F\left(\boldsymbol{x},\boldsymbol{W}\left(\omega\right)\right)\ge y_{2}\right\} }\right|>D\left|y_{1}-y_{2}\right|\equiv2Dz,
\end{equation}
for any fixed and finite choice of $D>0$, thus immediately disproving
uniform validity of (\ref{eq:ESSSUP_OPEN}). Of course, the apparent
impossibility of (\ref{eq:ESSSUP_OPEN}) may be seen as a consequence
of the fact that, for \textit{any} \textit{fixed} $z\in\mathbb{R}$,
the function $\mathds{1}_{\left\{ z\ge y\right\} }$ \textit{is} \textit{discontinuous}
\textit{when} $y\in\mathbb{R}$.

For this example, it is also possible to show that condition ${\bf C3}$
is impossible to hold for $Q\equiv2$, as well, which corresponds
to the smallest choice of $Q$, when $P\equiv2$. Indeed, in this
case, condition ${\bf C3}$ demands the existence of a number $D<\infty$,
such that
\begin{equation}
\sup_{\boldsymbol{x}\in{\cal X}}\left\Vert \vphantom{\varint}\hspace{-2pt}\left|\mathds{1}_{\left\{ F\left(\boldsymbol{x},\boldsymbol{W}\right)\ge y_{1}\right\} }-\mathds{1}_{\left\{ F\left(\boldsymbol{x},\boldsymbol{W}\right)\ge y_{2}\right\} }\right|\right\Vert _{{\cal L}_{2}}\le D\left|y_{1}-y_{2}\right|,\label{eq:No_L2}
\end{equation}
for all $\left(y_{1},y_{2}\right)\in\mathbb{R}^{2}$. For every $\boldsymbol{x}\in{\cal X}$
and for every pair $\left(y_{1},y_{2}\right)\in\mathbb{R}^{2}$, we
may write
\begin{flalign}
 & \hspace{-2pt}\hspace{-2pt}\hspace{-2pt}\hspace{-2pt}\hspace{-2pt}\hspace{-2pt}\mathbb{E}\left\{ \left(\mathds{1}_{\left\{ F\left(\boldsymbol{x},\boldsymbol{W}\right)\ge y_{1}\right\} }-\mathds{1}_{\left\{ F\left(\boldsymbol{x},\boldsymbol{W}\right)\ge y_{2}\right\} }\right)^{2}\right\} \nonumber \\
 & \equiv\mathbb{E}\left\{ \mathds{1}_{\left\{ F\left(\boldsymbol{x},\boldsymbol{W}\right)\ge y_{1}\right\} }+\mathds{1}_{\left\{ F\left(\boldsymbol{x},\boldsymbol{W}\right)\ge y_{2}\right\} }-2\mathds{1}_{\left\{ F\left(\boldsymbol{x},\boldsymbol{W}\right)\ge y_{1}\right\} }\mathds{1}_{\left\{ F\left(\boldsymbol{x},\boldsymbol{W}\right)\ge y_{2}\right\} }\right\} \nonumber \\
 & ={\cal P}\left(F\left(\boldsymbol{x},\hspace{-1pt}\boldsymbol{W}\right)\ge y_{1}\right)+{\cal P}\left(F\left(\boldsymbol{x},\hspace{-1pt}\boldsymbol{W}\right)\ge y_{2}\right)-2{\cal P}\left(F\left(\boldsymbol{x},\hspace{-1pt}\boldsymbol{W}\right)\ge\max\left\{ y_{1},y_{2}\right\} \right)\nonumber \\
 & ={\cal P}\left(F\left(\boldsymbol{x},\hspace{-1pt}\boldsymbol{W}\right)\ge\min\left\{ y_{1},y_{2}\right\} \right)-{\cal P}\left(F\left(\boldsymbol{x},\hspace{-1pt}\boldsymbol{W}\right)\ge\max\left\{ y_{1},y_{2}\right\} \right)\nonumber \\
 & \equiv\left|{\cal P}\left(F\left(\boldsymbol{x},\hspace{-1pt}\boldsymbol{W}\right)\ge y_{1}\right)-{\cal P}\left(F\left(\boldsymbol{x},\hspace{-1pt}\boldsymbol{W}\right)\ge y_{2}\right)\right|\nonumber \\
 & \equiv\left|F_{\boldsymbol{W}}^{\boldsymbol{x}}\left(y_{1}\right)-F_{\boldsymbol{W}}^{\boldsymbol{x}}\left(y_{2}\right)\right|.
\end{flalign}
Therefore, for (\ref{eq:No_L2}) to hold, it must be true that, for
every $\boldsymbol{x}\in{\cal X}$ and for every $\left(y_{1},y_{2}\right)\in\mathbb{R}^{2}$,
\begin{equation}
\sqrt{\left|F_{\boldsymbol{W}}^{\boldsymbol{x}}\left(y_{1}\right)-F_{\boldsymbol{W}}^{\boldsymbol{x}}\left(y_{2}\right)\right|}\le D\left|y_{1}-y_{2}\right|\iff\left|F_{\boldsymbol{W}}^{\boldsymbol{x}}\left(y_{1}\right)-F_{\boldsymbol{W}}^{\boldsymbol{x}}\left(y_{2}\right)\right|\le D^{2}\left|y_{1}-y_{2}\right|^{2},
\end{equation}
implying that $F_{\boldsymbol{W}}^{\left(\cdot\right)}$ \textit{must
be constant} \textit{on} $\mathbb{R}$. This is absurd, however, since
$F_{\boldsymbol{W}}^{\left(\cdot\right)}$ is a proper cdf.

On the other hand, the second and third parts of Proposition \ref{prop:C3_Valid}
show that, if the random variable $\left\Vert \underline{\nabla}F\left(\cdot,\hspace{-1pt}\boldsymbol{W}\right)\right\Vert _{2}$
can be afforded to be uniformly in ${\cal Z}_{\infty}$ (for $Q\equiv1$),
the choice of the risk regularizer ${\cal R}$ may be completely unconstrained,
as long as the Borel pushforward of $F\left(\cdot,\boldsymbol{W}\right)$
is uniformly well behaved, in the sense of either (\ref{eq:Valid_2}),
or, more generally, (\ref{eq:Valid_3}). Of course, this constitutes
a major improvement compared to the case where $Q\equiv\infty$, discussed
above, at least in regard to the shape of ${\cal R}$. For example,
in our previous example, it is also true that
\begin{equation}
\underline{\nabla}F\left(\boldsymbol{x},\hspace{-1pt}\boldsymbol{W}\right)=\nabla F\left(\boldsymbol{x},\hspace{-1pt}\boldsymbol{W}\right)=\boldsymbol{\boldsymbol{W}}_{1}\in\left[0,1\right]^{N},\quad{\cal P}-a.e.,
\end{equation}
and, hence,
\begin{equation}
\sup_{\boldsymbol{x}\in{\cal X}}\left\Vert \vphantom{\varint}\hspace{-2pt}\left\Vert \nabla F\left(\boldsymbol{x},\hspace{-1pt}\boldsymbol{W}\right)\right\Vert _{2}\right\Vert _{{\cal L}_{\infty}}\equiv\sup_{\boldsymbol{x}\in{\cal X}}\underset{\omega\in\Omega}{\mathrm{ess\hspace{1bp}sup}}\left\Vert \boldsymbol{\boldsymbol{W}}_{1}\right\Vert _{2}\le\sqrt{N}.\label{eq:ESSSUP}
\end{equation}
In this case, condition ${\bf C3}$ is loosened to
\begin{equation}
\sup_{\boldsymbol{x}\in{\cal X}}\left\Vert \vphantom{\varint}\hspace{-2pt}\left|\mathds{1}_{\left\{ F\left(\boldsymbol{x},\boldsymbol{W}\right)\ge y_{1}\right\} }-\mathds{1}_{\left\{ F\left(\boldsymbol{x},\boldsymbol{W}\right)\ge y_{2}\right\} }\right|\right\Vert _{{\cal L}_{1}}\le D\left|y_{1}-y_{2}\right|,
\end{equation}
for all $\left(y_{1},y_{2}\right)\in\mathbb{R}^{2}$, whose validity
may now be verified for appropriate choices of the distribution ${\cal P}_{\boldsymbol{W}}$,
as Proposition \ref{prop:C3_Valid} suggests.

We have seen that the price to be paid for choosing a lower value
for the exponent $Q$ is a potentially stronger requirement on the
size of the random subgradient $\underline{\nabla}F\left(\cdot,\boldsymbol{W}\right)$
(condition $\mathbf{C1}$). Still, such a requirement is relatively
easy to satisfy for many interesting models, other than our particular
example discussed above. For example, in the extreme case where $P\equiv\infty$,
${\bf C1}$ will indeed be satisfied in cases involving a compact
feasible set ${\cal X}$ and a Borel measure ${\cal P}_{\boldsymbol{W}}$
with bounded essential support (recall that, by assumption, the domain
of $F\left(\cdot,\boldsymbol{W}\right)$ is the whole Euclidean space
$\mathbb{R}^{N}$).
\begin{rem}
\label{rem:We-would-like}We would like to emphasize that, for $P\equiv\infty$,
a sometimes more easily verifiable sufficient condition for $\mathbf{C1}$
is the existence of an event $\Omega_{\widetilde{F}}\subseteq\Omega$,
with ${\cal P}\left(\Omega_{\widetilde{F}}\right)\equiv1$, as well
as a number $G<\infty$ such that
\begin{equation}
\sup_{\boldsymbol{x}\in{\cal X}}\sup_{\omega\in\Omega_{\widetilde{F}}}\left\Vert \underline{\nabla}F\left(\boldsymbol{x},\hspace{-1pt}\boldsymbol{W}\hspace{-2pt}\left(\omega\right)\right)\right\Vert _{2}\equiv\sup_{\omega\in\Omega_{\widetilde{F}}}\sup_{\boldsymbol{x}\in{\cal X}}\left\Vert \underline{\nabla}F\left(\boldsymbol{x},\hspace{-1pt}\boldsymbol{W}\hspace{-2pt}\left(\omega\right)\right)\right\Vert _{2}\le G.
\end{equation}
This follows by definition of the essential supremum (\citep{Bogachev2007},
p. 250); indeed, we may write
\begin{flalign}
\sup_{\boldsymbol{x}\in{\cal X}}\left\Vert \vphantom{\varint}\left\Vert \underline{\nabla}F\left(\boldsymbol{x},\hspace{-1pt}\boldsymbol{W}\right)\right\Vert _{2}\right\Vert _{{\cal L}_{\infty}} & \triangleq\sup_{\boldsymbol{x}\in{\cal X}}\underset{\omega\in\Omega}{\mathrm{ess\hspace{1bp}sup}}\left\Vert \underline{\nabla}F\left(\boldsymbol{x},\hspace{-1pt}\boldsymbol{W}\hspace{-2pt}\left(\omega\right)\right)\right\Vert _{2}\nonumber \\
 & \triangleq\sup_{\boldsymbol{x}\in{\cal X}}\hspace{1bp}\inf_{\mathscr{F}\ni\Omega'\subseteq\Omega:{\cal P}\left(\Omega'\right)\equiv1}\hspace{1bp}\sup_{\omega\in\Omega'}\left\Vert \underline{\nabla}F\left(\boldsymbol{x},\hspace{-1pt}\boldsymbol{W}\hspace{-2pt}\left(\omega\right)\right)\right\Vert _{2}\nonumber \\
 & \le\sup_{\boldsymbol{x}\in{\cal X}}\sup_{\omega\in\widetilde{\Omega}}\left\Vert \underline{\nabla}F\left(\boldsymbol{x},\hspace{-1pt}\boldsymbol{W}\hspace{-2pt}\left(\omega\right)\right)\right\Vert _{2},
\end{flalign}
where $\widetilde{\Omega}\subseteq\Omega$ is \textit{any} measurable
set in $\mathscr{F}$, such that ${\cal P}\left(\widetilde{\Omega}\right)\equiv1$.
This technical fact was previously utilized in (\ref{eq:ESSSUP}).\hfill{}\ensuremath{\blacksquare}
\end{rem}
Let us also comment on the hard boundedness condition ${\bf C4}$,
which is assumed \textit{only when $p>1$}. Although condition ${\bf C4}$
may not impose significant restrictions on the choice of ${\cal R}$,
it \textit{does} require that $F\left(\cdot,\boldsymbol{W}\right)$
is \textit{uniformly bounded} on ${\cal X}$, almost everywhere on
$\Omega$ relative to ${\cal P}$. This assumption is explicitly made
in condition ${\bf C4}$ mainly for analytical tractability, and is
due to the slightly more complicated form of the gradient $\nabla\phi^{\widetilde{F}}$
(see Lemma \ref{lem:Sub_Grad}). Without ${\bf C4}$, asymptotic analysis
of the \textit{$\textit{MESSAGE}^{p}$} algorithm becomes unnecessarily
and uninsightfully complicated, when $p$ is chosen greater than one.
Still, uniform boundedness of $F\left(\cdot,\boldsymbol{W}\right)$
may be verified in many common optimization settings, such as when
${\cal X}$ is compact and ${\cal P}_{\boldsymbol{W}}$ has bounded
essential support (see also the discussion above), or in case $F$
is itself uniformly bounded on ${\cal X}\times\mathbb{R}^{M}$. 

Another important structural reason for imposing condition ${\bf C4}$
is that for every choice of $p>1$, we have implicitly assumed that
$F\left(\cdot,\boldsymbol{W}\right)\in{\cal Z}_{q}$, for some $q\ge p$,
so that problem (\ref{eq:MAIN_PROB}) is well defined. Thus, choosing
larger values for $p$ implies that $F\left(\cdot,\boldsymbol{W}\right)$
behaves more or less like a bounded function (\textit{pointwise on
${\cal X}$}). Therefore, condition ${\bf C4}$ may be regarded as
an easy way of exploiting this approximate boundedness, compared to
the imposition of integral ${\cal L}_{q}$-norm constraints, which
are more complicated and harder to handle.

Nevertheless, there are important cases where the choice of ${\cal R}$
might make it very difficult to guarantee that $\varepsilon\equiv{\cal R}\left(m_{l}-m_{h}\right)>0$,
even if $m_{l}$ and $m_{h}$ are finite. For example, simply take
${\cal R}\left(\cdot\right)\equiv\left(\cdot\right)_{+}$ (note that
$m_{l}<m_{h}$, by definition); in this case, $\varepsilon\equiv0$.
Fortunately, there is a simple, cheap-trick remedy to this technical
issue. For fixed slack $\eta>0$, consider a function ${\cal R}_{\eta}:\mathbb{R}\rightarrow\mathbb{R}$,
defined as
\begin{equation}
{\cal R}_{\eta}\left(x\right)\triangleq{\cal R}\left(x\right)+\eta,\quad\forall x\in\mathbb{R}.
\end{equation}
It can be readily verified that ${\cal R}_{\eta}$ is a valid risk
regularizer and may be seen as a variable, \textit{lower-biased} version
of ${\cal R}$. Then, problem (\ref{eq:MAIN_PROB}) is replaced by
its \textit{slack-adjusted} version
\begin{equation}
\begin{array}{rl}
\underset{\boldsymbol{x}}{\mathrm{minimize}} & \mathbb{E}\left\{ F\left(\boldsymbol{x},\boldsymbol{W}\right)\right\} +c\left\Vert {\cal R}_{\eta}\left(F\left(\boldsymbol{x},\boldsymbol{W}\right)-\mathbb{E}\left\{ F\left(\boldsymbol{x},\boldsymbol{W}\right)\right\} \right)\right\Vert _{{\cal L}_{p}}\\
\mathrm{subject\,to} & \boldsymbol{x}\in{\cal X}
\end{array},\label{eq:MAIN_Adjusted}
\end{equation}
and that it is always true that $\varepsilon\ge\eta$, satisfying
the respective requirement of condition ${\bf C4}$. One then hopes
that, as $\eta\rightarrow0$, (\ref{eq:MAIN_Adjusted}) becomes increasingly
equivalent to (\ref{eq:MAIN_PROB}). The size of $\eta$ should be
chosen so that a certain trade-off between algorithmic stability and
closeness to the original problem (\ref{eq:MAIN_PROB}) is satisfied. 

Lastly, condition ${\bf C2}$ constitutes a common restriction of
SSD-type algorithms (either compositional or not) \citep{Kushner2003,ShapiroLectures_2ND,Wang2017,Wang2018},
and will also be made here, without any further comment.

Next, we will exploit Assumption \ref{assu:F_AS_Main}, in order to
show asymptotic consistency for Algorithm \ref{alg:SCGD-3}, in a
strong, pathwise sense.

\subsubsection{\label{subsec:Pathwise-Convergence}Pathwise Convergence of\textit{
}the\textit{ $\textit{MESSAGE}^{p}$} Algorithm}

Proving convergence of Algorithm \ref{alg:SCGD-3} will be based on
the so-called \textit{$T$-Level Almost-Supermartingale Convergence
Lemma }by Yang, Wang \& Fang \citep{Wang2018}, presented below. This
is an inductive generalization of the \textit{Coupled} \textit{Almost-Supermartingale
Convergence Lemma }by Wang \& Bertsekas \citep{Wang2016,Wang2017},
which in turn generalizes the well-known \textit{Almost-Supermartingale
Convergence Lemma} by Robbins and Siegmund \citep{Robbins1971}.
\begin{lem}
\textbf{\textup{($T$-Level Almost-Supermartingale Convergence Lemma
\citep{Wang2018})}}\label{lem:SUPER_M} Let $\left\{ \xi^{n}\right\} _{n\in\mathbb{N}}$,
$\left\{ \eta^{n}\right\} _{n\in\mathbb{N}}$, $\{\zeta^{n,j}\}_{n\in\mathbb{N}}$,
$\{\theta^{n,j}\}_{n\in\mathbb{N}}$, for $j\in\mathbb{N}_{T-1}^{+}$,
$\{u^{n,j}\}_{n\in\mathbb{N}}$ and $\{\mu^{n,j}\}_{n\in\mathbb{N}}$,
for $j\in\mathbb{N}_{T}^{+}$, be nonnegative random sequences on
$\left(\Omega,\mathscr{F}\right)$, and consider the global filtration
$\left\{ \mathscr{G}^{n}\subseteq\mathscr{F}\right\} _{n\in\mathbb{N}}$,
where
\begin{equation}
\mathscr{G}^{n}\triangleq\sigma\left\{ \xi^{i},\eta^{i},\zeta^{i,j},\theta^{i,j},u^{i,j},\mu^{i,j},u^{i,T},\mu^{i,T},\;\forall j\in\mathbb{N}_{T-1}^{+}\;\text{and}\;\forall i\in\mathbb{N}_{n}\right\} .
\end{equation}
Let $c_{j}>0,j\in\mathbb{N}_{T}^{+}$ and suppose that
\begin{flalign}
\mathbb{E}\left\{ \left.\xi^{n+1}\right|\mathscr{G}^{n}\right\}  & \le\left(1+\eta^{n}\right)\xi^{n}-u^{n,T}+\sum_{j\in\mathbb{N}_{T-1}^{+}}c_{j}\theta^{n,j}\zeta^{n,j}+\mu^{n,T}\quad\text{and}\label{eq:Almost_1}\\
\mathbb{E}\left\{ \left.\zeta^{n+1,j}\right|\mathscr{G}^{n}\right\}  & \le\left(1-\theta^{n,j}\right)\zeta^{n,j}-u^{n,j}+\mu^{n,j},\quad\forall j\in\mathbb{N}_{T-1}^{+},
\end{flalign}
for all $n\in\mathbb{N},$ and that $\sum_{n\in\mathbb{N}}\eta^{n}<\infty$,
$\sum_{n\in\mathbb{N}}\mu^{n,j}<\infty$, for all $j\in\mathbb{N}_{T}^{+}$,
all almost everywhere relative to ${\cal P}$. Then, there exist random
variables $\xi_{*}$ and $\zeta_{*}^{j}$, $j\in\mathbb{N}_{T-1}^{+}$
such that $\xi^{n}\underset{n\rightarrow\infty}{\longrightarrow}\xi_{*}$
and $\zeta^{n,j}\underset{n\rightarrow\infty}{\longrightarrow}\zeta_{*}^{j}$,
for all $j\in\mathbb{N}_{T-1}^{+}$ and $\sum_{n\in\mathbb{N}}u^{n,j}<\infty$,
for all $j\in\mathbb{N}_{T}^{+}$, $\sum_{n\in\mathbb{N}}\theta^{n,j}\zeta^{n,j}<\infty$,
for all $j\in\mathbb{N}_{T-1}^{+}$, all almost everywhere relative
to ${\cal P}$.
\end{lem}
Our proof roadmap is similar to that presented in, say, \citep{Wang2017,Wang2018},
and is somewhat standard in the literature of stochastic approximation,
in general. Before proceeding, let us define the filtration $\left\{ \mathscr{D}^{n}\subseteq\mathscr{F}\right\} _{n\in\mathbb{N}}$,
generated from \textit{all data observed so far, by both the user
and the} ${\cal SO}$, with each sub $\sigma$-algebra $\mathscr{D}^{n}$
given by
\begin{equation}
\mathscr{D}^{n}\triangleq\sigma\left\{ \boldsymbol{x}^{0},\ldots,\boldsymbol{x}^{n},y^{0},\ldots,y^{n},z^{0},\ldots,z^{n},\boldsymbol{W}_{1}^{0},\ldots,\boldsymbol{W}_{1}^{n},\boldsymbol{W}_{2}^{0},\ldots,\boldsymbol{W}_{2}^{n}\right\} ,\quad\forall n\in\mathbb{N}.
\end{equation}
Also, for the sake of clarity, if $\mathscr{C}$ is some sub $\sigma$-algebra
of $\mathscr{F}$, we will employ the more compact notation $\mathbb{E}\left\{ \cdot\left|\mathscr{C}\right.\right\} \equiv\mathbb{E}_{\mathscr{C}}\left\{ \cdot\right\} $,
especially for larger expressions involving conditional expectations.

Our first basic result follows, characterizing the rate of decay of
the squared ${\cal L}_{2}$-norm of the inter-iteration error $\boldsymbol{x}^{n+1}-\boldsymbol{x}^{n}$,
relative to $\mathscr{D}^{n}$.
\begin{lem}
\textbf{\textup{(Deep SA Level: Iterate Increment Growth)}}\label{lem:INTER_1}
Let Assumption \ref{assu:F_AS_Main} be in effect and define a constant
\begin{equation}
\mathsf{R}_{p}\equiv\begin{cases}
1, & \text{if }p\equiv1\\
\left(\dfrac{{\cal E}}{\varepsilon}\right)^{p-1} & \text{if }p>1
\end{cases}.
\end{equation}
Then, for every $p\in\left[1,\infty\right)$, the process $\left\{ \boldsymbol{x}^{n}\right\} _{n\in\mathbb{N}}$
generated by the $\textit{MESSAGE}^{p}$ algorithm satisfies
\begin{equation}
{\cal O}\left(\alpha_{n}^{2}\right)\equiv\mathbb{E}_{\mathscr{D}^{n}}\left\{ \left\Vert \boldsymbol{x}^{n+1}-\boldsymbol{x}^{n}\right\Vert _{2}^{2}\right\} \le\alpha_{n}^{2}\left(2c\mathsf{R}_{p}+1\right)^{2}G^{2},\quad\forall n\in\mathbb{N},
\end{equation}
almost everywhere relative to ${\cal P}$.
\end{lem}
\begin{proof}[Proof of Lemma \ref{lem:INTER_1}]
See Section \ref{subsec:RM_1} (Appendix). 
\end{proof}
Exploiting Lemma \ref{lem:INTER_1}, one may also prove the following
result, which will be useful later in our analysis. The proof is omitted,
since it is essentially provided in (\citep{Wang2017}, Supplementary
Material, Section G.1).
\begin{lem}
\textbf{\textup{(Iterate Increment Summability)}}\label{lem:INTER_1-1-2}
Let Assumption \ref{assu:F_AS_Main} be in effect. Also, consider
a sequence $\left\{ \delta_{n}>0\right\} _{n\in\mathbb{N}}$, such
that $\sum_{n\in\mathbb{N}}\alpha_{n}^{2}\delta_{n}^{-1}<\infty$.
Then, the iterate process $\left\{ \boldsymbol{x}^{n}\right\} _{n\in\mathbb{N}}$
generated by the $\textit{MESSAGE}^{p}$ algorithm satisfies
\begin{equation}
\sum_{n\in\mathbb{N}}\delta_{n}^{-1}\mathbb{E}_{\mathscr{D}^{n}}\left\{ \left\Vert \boldsymbol{x}^{n+1}\hspace{-1pt}\hspace{-1pt}-\hspace{-1pt}\boldsymbol{x}^{n}\right\Vert _{2}^{2}\right\} <\infty,
\end{equation}
almost everywhere relative to ${\cal P}$. 
\end{lem}
Let us now consider Borel measurable functions ${\cal S}^{\widetilde{F}}:{\cal X}\rightarrow\mathbb{R}$,
${\cal D}^{\widetilde{F}}:{\cal X}\times\mathbb{R}\rightarrow\mathbb{R}_{+}$
and ${\cal D}^{\widetilde{F}}:\mathbb{R}\rightarrow\mathbb{R}_{+}$,
defined as\footnote{Without any risk of confusion, we use the same name ${\cal D}^{\widetilde{F}}$
to refer to the two very similar functions (\ref{eq:SIM_1}) and (\ref{eq:SIM_2}).
The two functions will be distinguished by their different number
of arguments. }
\begin{flalign}
{\cal S}^{\widetilde{F}}\left(\boldsymbol{x}\right) & \triangleq\mathbb{E}\left\{ F\left(\boldsymbol{x},\boldsymbol{W}'\right)\right\} ,\\
{\cal D}^{\widetilde{F}}\left(\boldsymbol{x},y\right) & \triangleq\mathbb{E}\hspace{-2pt}\left\{ \left({\cal R}\left(F\hspace{-2pt}\left(\boldsymbol{x},\hspace{-1pt}\boldsymbol{W}'\right)\hspace{-2pt}-\hspace{-2pt}y\right)\right)^{p}\right\} \quad\text{and}\label{eq:SIM_1}\\
{\cal D}^{\widetilde{F}}\left(\boldsymbol{x}\right) & \triangleq\mathbb{E}\hspace{-2pt}\left\{ \left({\cal R}\left(F\hspace{-2pt}\left(\boldsymbol{x},\hspace{-1pt}\boldsymbol{W}'\right)\hspace{-2pt}-\hspace{-2pt}{\cal S}^{\widetilde{F}}\left(\boldsymbol{x}\right)\right)\right)^{p}\right\} .\label{eq:SIM_2}
\end{flalign}
where the random element $\boldsymbol{W}':\Omega\rightarrow\mathbb{R}^{M}$
is distributed according to the Borel measure ${\cal P}_{\boldsymbol{W}}$,
and is arbitrarily taken as independent of the \textit{whole} filtration
$\left\{ \mathscr{D}^{n}\right\} _{n\in\mathbb{N}}$. For instance,
for each $n\in\mathbb{N}$, $\boldsymbol{W}'$ may be substituted
by the information stream $\boldsymbol{W}_{2}^{n+1}$, which is by
assumption statistically independent of the sub $\sigma$-algebra
$\sigma\left\{ \mathscr{D}^{n}\right\} $. The main purpose of the
auxiliary expectation functions ${\cal S}^{\widetilde{F}}$ and ${\cal D}^{\widetilde{F}}$
is convenience.

Utilizing ${\cal S}^{\widetilde{F}}$, the behavior of the running
approximation error $y^{n+1}-{\cal S}^{\widetilde{F}}\left(\boldsymbol{x}^{n+1}\right)$
may be analyzed in a similar manner as the respective quantity of
Lemma \ref{lem:INTER_1}. The relevant result follows.
\begin{lem}
\textbf{\textup{(First SA Level: Error Growth)}}\label{lem:INTER_1-1}
Let Assumption \ref{assu:F_AS_Main} be in effect. Also, let $\beta_{n}\in\left(0,1\right]$,
for all $n\in\mathbb{N}$. Then, the composite process $\left\{ \left(\boldsymbol{x}^{n},y^{n}\right)\right\} _{n\in\mathbb{N}}$
generated by the $\textit{MESSAGE}^{p}$ algorithm satisfies
\begin{align}
\hspace{-2pt}\mathbb{E}_{\mathscr{D}^{n}}\hspace{-2pt}\left\{ \left|y^{n+1}\hspace{-1pt}\hspace{-1pt}-\hspace{-1pt}{\cal S}^{\widetilde{F}}\hspace{-1pt}\hspace{-1pt}\left(\boldsymbol{x}^{n+1}\right)\right|^{2}\right\}  & \hspace{-2pt}\le\hspace{-1pt}\left(1\hspace{-1pt}-\hspace{-1pt}\beta_{n}\right)\hspace{-1pt}\left|y^{n}\hspace{-1pt}\hspace{-1pt}-\hspace{-1pt}{\cal S}^{\widetilde{F}}\hspace{-1pt}\hspace{-1pt}\left(\boldsymbol{x}^{n}\right)\right|^{2}\hspace{-2pt}+\hspace{-1pt}\beta_{n}^{-1}2G^{2}\hspace{-1pt}\mathbb{E}_{\mathscr{D}^{n}}\hspace{-2pt}\left\{ \left\Vert \boldsymbol{x}^{n+1}\hspace{-1pt}\hspace{-1pt}-\hspace{-1pt}\boldsymbol{x}^{n}\right\Vert _{2}^{2}\right\} \hspace{-1pt}+\hspace{-1pt}\beta_{n}^{2}2V,\label{eq:Deep_SA}
\end{align}
for all $n\in\mathbb{N}$, almost everywhere relative to ${\cal P}$.
\end{lem}
\begin{proof}[Proof of Theorem \ref{lem:INTER_1-1}]
See Section \ref{subsec:RM_0} (Appendix).
\end{proof}
At this point, let us make the following additional assumption, related
with the initial conditions of the first and second SA levels of the
\textit{$\textit{MESSAGE}^{p}$} algorithm.
\begin{assumption}
\textbf{\textup{(Initial Values)}}\label{assu:Initial-Values} \uline{Whenever
\mbox{$p>1$}}, $y^{0}$,
$\beta_{0}$ and $z^{0}$, $\gamma_{0}$ are chosen such that
\begin{equation}
\left\{ \hspace{-1pt}\hspace{-1pt}\hspace{-1pt}\hspace{-1pt}\begin{array}{l}
\text{either}\quad y^{0}\in\left[m_{l},m_{h}\right],\quad\text{or}\quad\beta_{0}\equiv1\\
\text{either}\quad z^{0}\in\left[\varepsilon^{p},{\cal E}^{p}\right],\quad\text{or}\quad\gamma_{0}\equiv1
\end{array}\right..
\end{equation}
\end{assumption}
It is trivial to see that Assumption \ref{assu:Initial-Values} \textit{can
always be satisfied}, one way or another. However, choosing \textbf{$\beta_{0}\triangleq1$
}and $\gamma_{0}\triangleq1$ might be advantageous in practice, especially
for cases where specific values of the constants $m_{l},m_{h},\varepsilon,{\cal E}$
are unknown. In our analysis, Assumption \ref{assu:Initial-Values}
will help us guarantee uniform boundedness of the iterates $\left\{ y^{n}\right\} _{n\in\mathbb{N}}$
and $\left\{ z^{n}\right\} _{n\in\mathbb{N}}$ of the first and second
SA levels of the \textit{$\textit{MESSAGE}^{p}$} algorithm, respectively,
whenever the semideviation order is chosen greater than unity, that
is, $p>1$ (see Lemma \ref{lem:Iterate_Boundedness} in Section \ref{subsec:Some-Auxiliary-Results}
(Appendix)). 

Now, similarly to Lemma \ref{lem:INTER_1-1}, the growth of the running
approximation error $z^{n}-{\cal D}^{\widetilde{F}}\left(\boldsymbol{x}^{n},y^{n}\right)$
may be characterized as follows.
\begin{lem}
\textbf{\textup{(Second SA Level: Error Growth)}}\label{lem:INTER_Middle}
Let Assumptions \ref{assu:F_AS_Main} and \ref{assu:Initial-Values}
be in effect. Also, choose $p>1$, and let $\beta_{n}\in\left(0,1\right]$,
$\gamma_{n}\in\left(0,1\right]$, for all $n\in\mathbb{N}$. Then,
the composite process $\left\{ \left(\boldsymbol{x}^{n},y^{n},z^{n}\right)\right\} _{n\in\mathbb{N}}$
generated by the $\textit{MESSAGE}^{p}$ algorithm satisfies
\begin{flalign}
\mathbb{E}_{\mathscr{D}^{n}}\left\{ \left|z^{n+1}-{\cal D}^{\widetilde{F}}\left(\boldsymbol{x}^{n+1},y^{n+1}\right)\right|^{2}\right\}  & \le\left(1-\gamma_{n}\right)\left|z^{n}-{\cal D}^{\widetilde{F}}\left(\boldsymbol{x}^{n},y^{n}\right)\right|^{2}\nonumber \\
 & \hspace{-2pt}\hspace{-2pt}\hspace{-2pt}\hspace{-2pt}\hspace{-2pt}\hspace{-2pt}\hspace{-2pt}\hspace{-2pt}\hspace{-2pt}\hspace{-2pt}\hspace{-2pt}\hspace{-2pt}\hspace{-2pt}\hspace{-2pt}\hspace{-2pt}\hspace{-2pt}\hspace{-2pt}\hspace{-2pt}\hspace{-2pt}\hspace{-2pt}\hspace{-2pt}\hspace{-2pt}\hspace{-2pt}\hspace{-2pt}\hspace{-2pt}\hspace{-2pt}\hspace{-2pt}\hspace{-2pt}\hspace{-2pt}\hspace{-2pt}\hspace{-2pt}\hspace{-2pt}\hspace{-2pt}\hspace{-2pt}\hspace{-2pt}\hspace{-2pt}\hspace{-2pt}\hspace{-2pt}\hspace{-2pt}\hspace{-2pt}\hspace{-2pt}\hspace{-2pt}\hspace{-2pt}\hspace{-2pt}\hspace{-2pt}\hspace{-2pt}\hspace{-2pt}\hspace{-2pt}\hspace{-2pt}\hspace{-2pt}\hspace{-2pt}\hspace{-2pt}\hspace{-2pt}\hspace{-2pt}\hspace{-2pt}\hspace{-2pt}\hspace{-2pt}\hspace{-2pt}\hspace{-2pt}\hspace{-2pt}\hspace{-2pt}\hspace{-2pt}\hspace{-2pt}+\gamma_{n}^{-1}16G^{2}{\cal E}^{2p-2}p^{2}\mathbb{E}_{\mathscr{D}^{n}}\left\{ \left\Vert \boldsymbol{x}^{n+1}-\boldsymbol{x}^{n}\right\Vert _{2}^{2}\right\} +\beta_{n}^{2}\gamma_{n}^{-1}4{\cal E}^{2p-2}p^{2}\left(m_{h}-m_{l}\right)^{2}+\gamma_{n}^{2}2{\cal E}^{2p},\label{eq:Middle_SA}
\end{flalign}
for all $n\in\mathbb{N}$, almost everywhere relative to ${\cal P}$.
\end{lem}
\begin{proof}[Proof of Theorem \ref{lem:INTER_Middle}]
See Section \ref{subsec:RM_2} (Appendix).
\end{proof}
Let $\boldsymbol{x}^{*}\in{\cal X}$ be an optimal solution of problem
(\ref{eq:MAIN_PROB}), assuming such solution exists. We now characterize
the evolution of optimality error $\boldsymbol{x}^{n+1}\hspace{-1pt}-\hspace{-1pt}\boldsymbol{x}^{*}$,
showing that the quantity $\left\Vert \boldsymbol{x}^{n+1}\hspace{-1pt}-\hspace{-1pt}\boldsymbol{x}^{*}\right\Vert _{2}^{2}$
is an almost-supermartingale nonnegative sequence, of the form (\ref{eq:Almost_1})
in Lemma \ref{lem:SUPER_M}.
\begin{lem}
\textbf{\textup{(Third SA Level: Optimality Error Growth)}}\label{lem:INTER_1-1-1}
Let Assumptions \ref{assu:F_AS_Main} and \ref{assu:Initial-Values}
be in effect, let $\beta_{n}\in\left(0,1\right]$, $\gamma_{n}\in\left(0,1\right]$,
for all $n\in\mathbb{N}$, and define the constant
\begin{equation}
\mathsf{B}_{p}\triangleq\begin{cases}
D, & \text{if }p\equiv1\\
\left(1+{\cal E}^{p-1}p\right)\max\left\{ \hspace{-2pt}\left(\dfrac{1}{\varepsilon}\right)^{\hspace{-1pt}p-1}\hspace{-2pt}D,\left(p-1\right)\dfrac{{\cal E}^{p-1}}{\varepsilon^{2p-1}}\right\} , & \text{if }p>1
\end{cases}<\infty.
\end{equation}
Also, suppose that ${\cal X}^{*}\triangleq\mathrm{arg\,min}_{\boldsymbol{x}\in{\cal X}}\phi^{\widetilde{F}}\left(\boldsymbol{x}\right)\neq\emptyset$
and consider any $\boldsymbol{x}^{*}\in{\cal X}^{*}$. Then, the composite
process $\left\{ \left(\boldsymbol{x}^{n},y^{n},z^{n}\right)\right\} _{n\in\mathbb{N}}$
generated by the $\textit{MESSAGE}^{p}$ algorithm satisfies
\begin{flalign}
 & \hspace{-2pt}\hspace{-2pt}\hspace{-2pt}\hspace{-2pt}\hspace{-2pt}\hspace{-2pt}\mathbb{E}_{\mathscr{D}^{n}}\left\{ \left\Vert \boldsymbol{x}^{n+1}\hspace{-1pt}-\hspace{-1pt}\boldsymbol{x}^{*}\right\Vert _{2}^{2}\right\} \nonumber \\
 & \le\hspace{-2pt}\hspace{-2pt}\left(\hspace{-1pt}1\hspace{-1pt}+\hspace{-1pt}4\mathsf{B}_{p}^{2}G^{2}c^{2}\left(\dfrac{\alpha_{n}^{2}}{\beta_{n}}+\dfrac{\alpha_{n}^{2}}{\gamma_{n}}\mathds{1}_{\left\{ p>1\right\} }\right)\hspace{-2pt}\hspace{-1pt}\right)\hspace{-2pt}\left\Vert \boldsymbol{x}^{n}\hspace{-1pt}-\hspace{-1pt}\boldsymbol{x}^{*}\right\Vert _{2}^{2}\hspace{-1pt}+\hspace{-1pt}\alpha_{n}^{2}\left(2c\mathsf{R}_{p}\hspace{-2pt}+\hspace{-1pt}1\right)^{2}G^{2}\hspace{-1pt}-\hspace{-1pt}2\alpha_{n}\hspace{-1pt}\left(\phi^{\widetilde{F}}\hspace{-1pt}\left(\boldsymbol{x}^{n}\right)\hspace{-1pt}-\hspace{-1pt}\phi_{*}^{\widetilde{F}}\right)\nonumber \\
 & \quad\quad+\beta_{n}\hspace{-1pt}\left|y^{n}\hspace{-1pt}\hspace{-1pt}-\hspace{-1pt}{\cal S}^{\widetilde{F}}\hspace{-1pt}\hspace{-1pt}\left(\boldsymbol{x}^{n}\right)\right|^{2}\hspace{-2pt}+\hspace{-1pt}\gamma_{n}\left|z^{n}-{\cal D}^{\widetilde{F}}\left(\boldsymbol{x}^{n},y^{n}\right)\right|^{2}\mathds{1}_{\left\{ p>1\right\} },\label{eq:BIG_P1_3-1}
\end{flalign}
for all $n\in\mathbb{N}$, almost everywhere relative to ${\cal P}$,
where $\phi_{*}^{\widetilde{F}}\in\mathbb{R}$ is the optimal value
of problem (\ref{eq:MAIN_PROB}).
\end{lem}
\begin{proof}[Proof of Lemma \ref{lem:INTER_1-1-1}]
See Section \ref{subsec:RM_1-1-1} (Appendix). 
\end{proof}
Under the proposed Assumption \ref{assu:F_AS_Main}, pathwise convergence
of the \textit{$\textit{MESSAGE}^{p}$} algorithm is established next,
in a rather strong sense. Here, we directly invoke Lemma \ref{lem:SUPER_M}.
\begin{thm}
\textbf{\textup{(Pathwise Convergence of the }}\textbf{$\textit{MESSAGE}^{p}$
}\textbf{\textup{Algorithm)}}\label{thm:Alg_1_Conv} Let Assumptions
\ref{assu:F_AS_Main} and \ref{assu:Initial-Values} be in effect,
and let $\beta_{n}\in\left(0,1\right]$, $\gamma_{n}\in\left(0,1\right]$,
for all $n\in\mathbb{N}$. Whenever $p\equiv1$, suppose that
\begin{gather}
\sum_{n\in\mathbb{N}}\alpha_{n}\equiv\infty\quad\text{and}\quad\sum_{n\in\mathbb{N}}\alpha_{n}^{2}+\beta_{n}^{2}+\dfrac{\alpha_{n}^{2}}{\beta_{n}}<\infty,
\end{gather}
whereas, whenever $p>1$, suppose additionally that
\begin{gather}
\sum_{n\in\mathbb{N}}\gamma_{n}^{2}+\dfrac{\alpha_{n}^{2}}{\gamma_{n}}+\dfrac{\beta_{n}^{2}}{\gamma_{n}}<\infty.
\end{gather}
Then, as long as ${\cal X}^{*}\equiv\mathrm{arg\,min}_{\boldsymbol{x}\in{\cal X}}\phi^{\widetilde{F}}\left(\boldsymbol{x}\right)\neq\emptyset$,
the process $\left\{ \boldsymbol{x}^{n}\right\} _{n\in\mathbb{N}}$
generated by the $\textit{MESSAGE}^{p}$ algorithm satisfies\footnote{Note that (\ref{eq:BIG_PROB}) is meaningful only if the involved
outcome set is an event, that is, $\mathscr{F}$-measurable. In our
case, such measurability follows by completeness of the base space
$\left(\Omega,\mathscr{F},{\cal P}\right)$. }
\begin{equation}
{\cal P}\left(\left\{ \omega\in\Omega\left|\exists\,\boldsymbol{x}^{*}\left(\omega\right)\in{\cal X}^{*}\;\text{such that}\;\,\boldsymbol{x}^{n}\left(\omega\right)\underset{n\rightarrow\infty}{\longrightarrow}\boldsymbol{x}^{*}\left(\omega\right)\right.\right\} \right)\equiv1.\label{eq:BIG_PROB}
\end{equation}
In other words, almost everywhere relative to ${\cal P}$, the process
$\left\{ \boldsymbol{x}^{n}\right\} _{n\in\mathbb{N}}$ converges
to a random point in the set of optimal solutions of (\ref{eq:MAIN_PROB_2}).
\end{thm}
\begin{proof}[Proof of Theorem \ref{thm:Alg_1_Conv}]
We present the proof assuming that $p>1$. If $p\equiv1$, the proof
is almost the same, albeit simpler. Under the assumptions of the theorem,
Lemmata \ref{lem:INTER_1-1}, \ref{lem:INTER_1-1-2} and \ref{lem:INTER_Middle}
imply that
\begin{equation}
\sum_{n\in\mathbb{N}}\beta_{n}^{-1}2G^{2}\hspace{-1pt}\mathbb{E}_{\mathscr{D}^{n}}\left\{ \left\Vert \boldsymbol{x}^{n+1}\hspace{-1pt}\hspace{-1pt}-\hspace{-1pt}\boldsymbol{x}^{n}\right\Vert _{2}^{2}\right\} <\infty,
\end{equation}
and
\begin{equation}
\sum_{n\in\mathbb{N}}\gamma_{n}^{-1}16G^{2}{\cal E}^{2p-2}p^{2}\mathbb{E}_{\mathscr{D}^{n}}\left\{ \left\Vert \boldsymbol{x}^{n+1}\hspace{-1pt}\hspace{-1pt}-\hspace{-1pt}\boldsymbol{x}^{n}\right\Vert _{2}^{2}\right\} <\infty,
\end{equation}
whereas it is also true that
\begin{gather}
\sum_{n\in\mathbb{N}}\alpha_{n}^{2}\left(2c\mathsf{R}_{p}+1\right)^{2}G^{2}<\infty,\quad\sum_{n\in\mathbb{N}}\beta_{n}^{2}\gamma_{n}^{-1}4{\cal E}^{2p-2}p^{2}\left(m_{h}-m_{l}\right)^{2}<\infty,\\
\sum_{n\in\mathbb{N}}\beta_{n}^{2}V<\infty\quad\text{and}\quad\sum_{n\in\mathbb{N}}\gamma_{n}^{2}2{\cal E}^{2p}<\infty,
\end{gather}
as well. Therefore, we may apply Lemma \ref{lem:SUPER_M} with the
identifications
\begin{gather*}
\begin{array}{ccc}
\xi^{n}\equiv\left\Vert \boldsymbol{x}^{n}\hspace{-1pt}-\hspace{-1pt}\boldsymbol{x}^{*}\right\Vert _{2}^{2}, & \eta^{n}\equiv4\mathsf{B}_{p}^{2}G^{2}c^{2}\left(\dfrac{\alpha_{n}^{2}}{\beta_{n}}\hspace{-1pt}+\hspace{-1pt}\dfrac{\alpha_{n}^{2}}{\gamma_{n}}\right), & u^{n,3}\equiv2\alpha_{n}\left(\phi^{\widetilde{F}}\hspace{-1pt}\left(\boldsymbol{x}^{n}\right)\hspace{-1pt}-\hspace{-1pt}\phi_{*}^{\widetilde{F}}\right)\\
\zeta^{n,1}\equiv\left|y^{n}\hspace{-1pt}\hspace{-1pt}-\hspace{-1pt}{\cal S}^{\widetilde{F}}\hspace{-1pt}\hspace{-1pt}\left(\boldsymbol{x}^{n}\right)\right|^{2}, & \theta^{n,1}\equiv\beta_{n}, & u^{n,1}\equiv0,\\
\zeta^{n,2}\equiv\left|z^{n}-{\cal D}^{\widetilde{F}}\left(\boldsymbol{x}^{n},y^{n}\right)\right|^{2}, & \theta^{n,2}\equiv\gamma_{n}, & u^{n,2}\equiv0,
\end{array}\\
\mu^{n,3}\equiv\alpha_{n}^{2}\left(2c\mathsf{R}_{p}+1\right)^{2}G^{2}\\
\mu^{n,1}\equiv\beta_{n}^{-1}2G^{2}\hspace{-1pt}\mathbb{E}_{\mathscr{D}^{n}}\left\{ \left\Vert \boldsymbol{x}^{n+1}-\boldsymbol{x}^{n}\right\Vert _{2}^{2}\right\} \hspace{-1pt}+\hspace{-1pt}\beta_{n}^{2}V,\\
\mu^{n,2}\equiv\gamma_{n}^{-1}16G^{2}{\cal E}^{2p-2}p^{2}\mathbb{E}_{\mathscr{D}^{n}}\left\{ \left\Vert \boldsymbol{x}^{n+1}-\boldsymbol{x}^{n}\right\Vert _{2}^{2}\right\} +\beta_{n}^{2}\gamma_{n}^{-1}4{\cal E}^{2p-2}p^{2}\left(m_{h}-m_{l}\right)^{2}+\gamma_{n}^{2}2{\cal E}^{2p},\\
\mathscr{G}^{n}\equiv\mathscr{D}^{n},
\end{gather*}
and with $c_{1}\equiv c_{2}\equiv1$. The rest of the proof is identical
to (\citep{Wang2017}, Proof of Theorem 1 (a)), or (\citep{Wang2018},
Proof of Theorem 2.1 (a) \& Proof of Lemma 2.5).
\end{proof}
\begin{rem}
\label{rem:We-would-like-1}Note that, in both (\citep{Wang2017},
Theorem 1) and (\citep{Wang2018}, Theorem 2.1), in addition to the
stepsize requirements of Theorem \ref{thm:Alg_1_Conv}, it is assumed
that
\begin{equation}
\sum_{n\in\mathbb{N}^{+}}\beta_{n}\equiv\infty\quad\text{and}\quad\sum_{n\in\mathbb{N}^{+}}\gamma_{n}\equiv\infty,
\end{equation}
as well. To be best of our knowledge, however, although they do not
hurt, \textit{none} of the aforementioned (non)summability conditions
are necessary in order to guarantee pathwise convergence of the \textit{$\textit{MESSAGE}^{p}$}
algorithm, and the same is true for the general purpose \textit{SCGD
}algorithm of \citep{Wang2017} (see statement and proof of Theorem
1 in \citep{Wang2017}) and \textit{$T$-SCGD }algorithm of \citep{Wang2018}
(see statement and proof of Theorem 2.1 in \citep{Wang2018}).\hfill{}\ensuremath{\blacksquare}
\end{rem}
Besides the classical Robbins-Monro (RM) conditions \citep{Robbins1951}
on the stepsize sequence $\left\{ \alpha_{n}\right\} _{n\in\mathbb{N}^{+}}$
and the square summability conditions on $\left\{ \beta_{n}\right\} _{n\in\mathbb{N}^{+}}$and
$\left\{ \gamma_{n}\right\} _{n\in\mathbb{N}^{+}}$, and in agreement
with (\citep{Wang2018}, Theorem 2.1), Theorem \ref{thm:Alg_1_Conv}
demands that 
\begin{equation}
\sum_{n\in\mathbb{N}^{+}}\dfrac{\alpha_{n}^{2}}{\beta_{n}}<\infty,\quad\sum_{n\in\mathbb{N}^{+}}\dfrac{\alpha_{n}^{2}}{\gamma_{n}}<\infty\quad\text{and}\quad\sum_{n\in\mathbb{N}^{+}}\dfrac{\beta_{n}^{2}}{\gamma_{n}}<\infty.\label{eq:STEP_D}
\end{equation}
These stepsize requirements imposed by Theorem \ref{thm:Alg_1_Conv}
might seem quite complicated. Nevertheless, there are lots of viable
choices for the sequences $\left\{ \alpha_{n}\right\} _{n\in\mathbb{N}^{+}}$,
$\left\{ \beta_{n}\right\} _{n\in\mathbb{N}^{+}}$and $\left\{ \gamma_{n}\right\} _{n\in\mathbb{N}^{+}}$,
satisfying the conditions in (\ref{eq:STEP_D}).

Let us present a simple, but instructive example. Take, for every
$n\in\mathbb{N}^{+}$, 
\begin{equation}
\alpha_{n}\equiv\dfrac{1}{n^{\tau_{1}}},\quad\beta_{n}\equiv\dfrac{1}{n^{\tau_{2}}}\quad\text{and}\quad\gamma_{n}\equiv\dfrac{1}{n^{\tau_{3}}},
\end{equation}
for some $\tau_{j}\in\left(0.5,1\right]$, for $j\in\left\{ 1,2,3\right\} $.
In such a case, the RM conditions are automatically satisfied for
all three stepsizes since
\begin{equation}
\sum_{n\in\mathbb{N}^{+}}\dfrac{1}{n^{\tau_{j}}}\equiv\infty\quad\text{and}\quad\sum_{n\in\mathbb{N}^{+}}\dfrac{1}{n^{2\tau_{j}}}<\infty,\quad j\in\left\{ 1,2,3\right\} .
\end{equation}
We would like to see how we may choose $\tau_{j}$, $j\in\left\{ 1,2,3\right\} $,
such that the summability conditions in (\ref{eq:STEP_D}) are satisfied.
First, we demand that
\begin{equation}
\sum_{n\in\mathbb{N}^{+}}\dfrac{\alpha_{n}^{2}}{\beta_{n}}\equiv\sum_{n\in\mathbb{N}^{+}}\dfrac{1}{n^{2\tau_{1}-\tau_{2}}}<\infty,
\end{equation}
which equivalently yields
\begin{equation}
1<2\tau_{1}-\tau_{2}\quad\iff\quad\tau_{2}<2\tau_{1}-1.
\end{equation}
Also, for the preceding inequality to yield a feasible lower bound
for $\tau_{2}$, it must be true that
\begin{equation}
2\tau_{1}-1>\dfrac{1}{2}\quad\iff\quad\tau_{1}>\dfrac{3}{4}.
\end{equation}
Consequently, we obtain the conditions
\begin{equation}
\dfrac{3}{4}<\tau_{1}\le1\quad\text{and}\quad\dfrac{1}{2}<\tau_{2}<2\tau_{1}-1.\label{eq:STEP1}
\end{equation}
Similarly, for the second condition of (\ref{eq:STEP_D})
\begin{equation}
\sum_{n\in\mathbb{N}^{+}}\dfrac{\alpha_{n}^{2}}{\gamma_{n}}\equiv\sum_{n\in\mathbb{N}^{+}}\dfrac{1}{n^{2\tau_{1}-\tau_{3}}}<\infty,
\end{equation}
we obtain the constraints
\begin{equation}
\dfrac{3}{4}<\tau_{1}\le1\quad\text{and}\quad\dfrac{1}{2}<\tau_{3}<2\tau_{1}-1.\label{eq:STEP2}
\end{equation}
Now, for the third condition of (\ref{eq:STEP_D}), we demand that
\begin{figure}
\centering\hspace{-14bp}\includegraphics[bb=16bp 245bp 580bp 533bp,clip,scale=0.63]{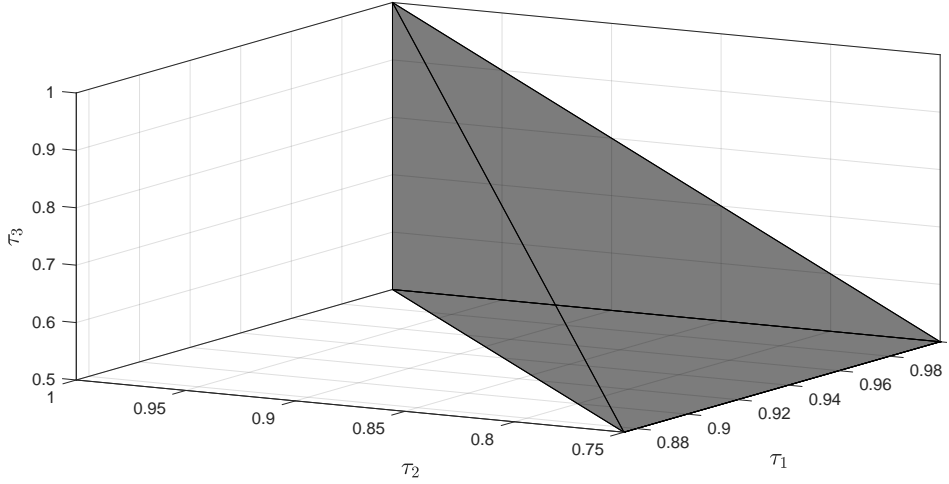}\caption{\label{fig:Comparison-of-stepsize} A graphical representation of
the stepsize constraint set (\ref{eq:STEP_1_W})-(\ref{eq:STEP_2_W}).}
\end{figure}
\begin{equation}
\sum_{n\in\mathbb{N}^{+}}\dfrac{\beta_{n}^{2}}{\gamma_{n}}\equiv\sum_{n\in\mathbb{N}^{+}}\dfrac{1}{n^{2\tau_{2}-\tau_{3}}}<\infty,
\end{equation}
yielding
\begin{equation}
\dfrac{3}{4}<\tau_{2}\le1\quad\text{and}\quad\dfrac{1}{2}<\tau_{3}<2\tau_{2}-1.\label{eq:STEP3}
\end{equation}
Of course, the linear constraints (\ref{eq:STEP1}), (\ref{eq:STEP2})
and (\ref{eq:STEP3}) need to be satisfied simultaneously, yielding
the feasible set
\begin{flalign}
\dfrac{7}{8}<\, & \tau_{1}\le1,\label{eq:STEP_1_W}\\
\dfrac{3}{4}<\, & \tau_{2}<2\tau_{1}-1\quad\text{and}\\
\dfrac{1}{2}<\, & \tau_{3}<2\tau_{2}-1.\label{eq:STEP_2_W}
\end{flalign}

We observe that there are lots of feasible choices for the exponents
$\tau_{1}$, $\tau_{2}$ and $\tau_{3}$. For example, one may take
$\tau_{1}\equiv1$, $\tau_{2}\equiv0.9$ and $\tau_{3}\equiv0.7$.
A graphical representation of the constraint set (\ref{eq:STEP_1_W})-(\ref{eq:STEP_2_W})
is shown in Fig. \ref{fig:Comparison-of-stepsize}.

\subsubsection{Convergence Rates of the\textit{ $\textit{MESSAGE}^{p}$} Algorithm}

We study two standard settings considered in the literature, namely,
that involving a convex risk-averse objective, matching all problem
assumptions we have made so far, and that involving a \textit{strongly
convex objective}, which, as we will shortly see, results naturally
by imposing strong convexity \textit{directly} on the random cost
function under consideration.

For the convex case, we employ \textit{iterate smoothing}, and we
provide detailed bounds on the ${\cal L}_{1}$ \textit{objective suboptimality
rate} of the \textit{$\textit{MESSAGE}^{p}$} algorithm. The proof
of our result follows directly, by appealing to the respective results
developed recently in \citep{Wang2018}.

For the strongly convex case, we develop completely new, detailed
and much stronger results on the \textit{squared-${\cal L}_{2}$ solution
suboptimality rate} of the \textit{$\textit{MESSAGE}^{p}$} algorithm,
which provide substantial improvement over the convex case, and are
much more comparable to rates achievable in risk-neutral stochastic
optimization.

The next basic technical result will be useful in our analysis.
\begin{lem}
\textbf{\textup{(Approximation Error Boundedness)}}\label{lem:Iter_Error_Bound}
Let Assumptions \ref{assu:F_AS_Main}, \ref{assu:Initial-Values}
be in effect, and let $\beta_{n}\in\left(0,1\right]$, $\gamma_{n}\in\left(0,1\right]$,
for all $n\in\mathbb{N}$. Also, whenever $p\equiv1$, suppose that
$\sup_{n\in\mathbb{N}}\alpha_{n}^{2}/\beta_{n}^{2}<\infty.$ Then,
it is true that
\begin{equation}
\sup_{n\in\mathbb{N}}\mathbb{E}\left\{ \left|y^{n}-{\cal S}^{\widetilde{F}}\hspace{-1pt}\hspace{-1pt}\left(\boldsymbol{x}^{n}\right)\right|^{2}\right\} <\infty,
\end{equation}
for every choice of $p\in\left[1,\infty\right)$, and
\begin{equation}
\sup_{n\in\mathbb{N}}\mathbb{E}\left\{ \left|z^{n}-{\cal D}^{\widetilde{F}}\hspace{-1pt}\hspace{-1pt}\left(\boldsymbol{x}^{n},y^{n}\right)\right|^{2}\right\} <\infty,
\end{equation}
for every choice of $p\in\left(1,\infty\right)$.
\end{lem}
\begin{proof}[Proof of Lemma \ref{lem:Iter_Error_Bound}]
If $p>1$, the conclusion of the lemma is trivial, due to Assumptions
\ref{assu:F_AS_Main} and \ref{assu:Initial-Values}, which imply
that the involved quantities $y^{n}$, ${\cal S}^{\widetilde{F}}\hspace{-1pt}\hspace{-1pt}\left(\cdot\right)$,
$z^{n}$ and ${\cal D}^{\widetilde{F}}\hspace{-2pt}\left(\cdot,\cdot\right)$
are uniformly bounded almost everywhere relative to ${\cal P}$ (see
Lemma \ref{lem:Iterate_Boundedness} in the Appendix (Section \ref{sec:Appendix:-Proofs})).

For the remaining case where $p\equiv1$, if $\sup_{n\in\mathbb{N}}\alpha_{n}^{2}\beta_{n}^{-2}<\infty$,
we may use simple induction exactly as in (\citep{Wang2018}, Appendix,
Proof of Lemma 2.3 (c)), exploiting the recursion (\ref{eq:Deep_SA})
in Lemma (\ref{lem:INTER_1-1}), respectively.
\end{proof}

\paragraph{\label{par:Convex-Random-Cost}Convex Random Cost with Iterate Smoothing}

For the convex case, we consider \textit{iterate smoothing} on top
of the \textit{$\textit{MESSAGE}^{p}$} algorithm, by defining averages
\begin{equation}
\widehat{\boldsymbol{x}}^{n}\triangleq\dfrac{1}{\left\lceil n/2\right\rceil }\sum_{i\in\mathbb{N}_{n}^{n-\left\lceil n/2\right\rceil }}\widehat{\boldsymbol{x}}^{i},\quad n\in\mathbb{N}^{+},
\end{equation}
exactly as in \citep{Wang2017,Wang2018}. Under this setting, the
next result characterizes the ${\cal L}_{1}$ objective suboptimality
rate of the \textit{$\textit{MESSAGE}^{p}$} algorithm, when iterate
smoothing is employed, for any choice of the semideviation order $p$.
\begin{thm}
\textbf{\textup{(Rate | Convex Case | Subharmonic Stepsizes)}}\label{lem:Rate_Convex}
Let Assumptions \ref{assu:F_AS_Main}, \ref{assu:Initial-Values}
be in effect, and let the stepsize sequences $\left\{ \alpha_{n}\right\} _{n\in\mathbb{N}}$,
$\left\{ \beta_{n}\right\} _{n\in\mathbb{N}}$ and $\left\{ \gamma_{n}\right\} _{n\in\mathbb{N}}$
follow the subharmonic rules\renewcommand{\CaseStretch}{2}
\begin{equation}
\begin{cases}
\alpha_{n}\triangleq\dfrac{1}{n^{\tau_{1}}},\quad\beta_{n}\triangleq\dfrac{1}{n^{\tau_{2}}}, & \text{if }p\equiv1\text{ with }1/2\le\tau_{2}<\tau_{1}<1\\
\alpha_{n}\triangleq\dfrac{1}{n^{\tau_{1}}},\quad\beta_{n}\triangleq\dfrac{1}{n^{\tau_{2}}},\quad\gamma_{n}\triangleq\dfrac{1}{n^{\tau_{3}}}, & \text{if }p>1\text{ with }1/2\le\tau_{3}<\tau_{2}<\tau_{1}<1
\end{cases},\quad\forall n\in\mathbb{N}^{+},
\end{equation}
\renewcommand{\CaseStretch}{1.2}with initial values $\alpha_{0}\equiv\beta_{0}\equiv\gamma_{0}\equiv1$.
Additionally, suppose that $\sup_{n\in\mathbb{N}}\mathbb{E}\left\{ \left\Vert \boldsymbol{x}^{n}\hspace{-1pt}-\hspace{-1pt}\boldsymbol{x}^{*}\right\Vert _{2}^{2}\right\} <\infty$,
where $\boldsymbol{x}^{*}\in{\cal X}^{*}$. Then, for every $n\in\mathbb{N}^{+}$,
it is true that
\begin{equation}
\left\Vert \phi^{\widetilde{F}}\hspace{-1pt}\left(\widehat{\boldsymbol{x}}^{n}\right)\hspace{-1pt}-\hspace{-1pt}\phi_{*}^{\widetilde{F}}\right\Vert _{{\cal L}_{1}}\equiv\mathbb{E}\left\{ \phi^{\widetilde{F}}\hspace{-1pt}\left(\widehat{\boldsymbol{x}}^{n}\right)\hspace{-1pt}-\hspace{-1pt}\phi_{*}^{\widetilde{F}}\right\} \le\begin{cases}
{\cal K}_{1}n^{-\min\left\{ 1-\tau_{1},\tau_{1}-\tau_{2},2\tau_{2}-\tau_{1}\right\} }, & \text{if }p\equiv1\\
{\cal K}_{p}n^{-\min\left\{ 1-\tau_{1},\tau_{1}-\tau_{2},2\tau_{3}-\tau_{1},2\tau_{2}-\tau_{1}-\tau_{3}\right\} }, & \text{if }p>1
\end{cases},
\end{equation}
where $0<{\cal K}_{p}<\infty,p\in\left[1,\infty\right)$ is a problem
dependent constant. In particular, if, for some $\epsilon\in\left[0,1\right)$,
$\delta\in\left(0,1\right)$ and $\zeta\in\left(0,1\right)$ such
that $\delta\ge\zeta$, \renewcommand{\CaseStretch}{2}
\begin{equation}
\begin{cases}
\tau_{1}\equiv\dfrac{3+\epsilon}{4}\quad\text{and}\quad\tau_{2}\equiv\dfrac{1+\delta\epsilon}{2}, & \text{if }p\equiv1\\
\tau_{1}\equiv\dfrac{7+\epsilon}{8},\quad\tau_{2}\equiv\dfrac{3+\delta\epsilon}{4}\quad\text{and}\quad\tau_{3}\equiv\dfrac{1+\zeta\epsilon}{2}, & \text{if }p>1
\end{cases},
\end{equation}
\renewcommand{\CaseStretch}{1.2}then the $\textit{MESSAGE}^{p}$
algorithm satisfies
\begin{equation}
\mathbb{E}\left\{ \phi^{\widetilde{F}}\hspace{-1pt}\left(\widehat{\boldsymbol{x}}^{n}\right)\hspace{-1pt}-\hspace{-1pt}\phi_{*}^{\widetilde{F}}\right\} \le{\cal K}_{p}n^{-\left(1-\epsilon\right)/\left(4\mathds{1}_{\left\{ p>1\right\} }+4\right)},
\end{equation}
for every $n\in\mathbb{N}^{+}$, for each fixed $\epsilon$.
\end{thm}
\begin{proof}[Proof of Theorem \ref{lem:Rate_Convex}]
Although the \textit{$\textit{MESSAGE}^{p}$} algorithm is different
from the \textit{$T$-SCGD }algorithm of \citep{Wang2018}, the proof
of Theorem \ref{lem:Rate_Convex} shares essentially the same structure
with (\citep{Wang2018}, Proof of Theorem 2.2). In a nutshell, except
for its native assumptions, the proof exploits Lemma \ref{lem:Iter_Error_Bound}
discussed above, the bound of Lemma \ref{lem:INTER_1} and the recursions
of Lemmata \ref{lem:INTER_1-1}, \ref{lem:INTER_Middle}, and \ref{lem:INTER_1-1-1},
developed in Section \ref{subsec:Pathwise-Convergence}, the convexity
of $\phi^{\widetilde{F}}$, as well as the stepsize exponent constraint
set (\ref{eq:STEP_1_W})-(\ref{eq:STEP_2_W}). The details of the
proof are omitted, and the reader is referred to \citep{Wang2018},
instead.
\end{proof}
It should be mentioned that, for $\epsilon>0$, the exponents of the
subharmonic stepsizes of Theorem \ref{lem:Rate_Convex} \textit{simultaneously}
satisfy the constraints (\ref{eq:STEP_1_W})-(\ref{eq:STEP_2_W}),
as discussed in Section \ref{subsec:Pathwise-Convergence}, which
are sufficient for guaranteeing convergence of the \textit{$\textit{MESSAGE}^{p}$}
algorithm in the pathwise sense. Therefore, for $\epsilon>0$, the
\textit{$\textit{MESSAGE}^{p}$} algorithm attains a ${\cal L}_{1}$
objective suboptimality rate of order \textit{arbitrarily close to}
${\cal O}(n^{-1/\left(4\mathds{1}_{\left\{ p>1\right\} }+4\right)})$,
while provably exhibiting pathwise stability, as well. On the other
hand, if $\epsilon\equiv0$, then the \textit{$\textit{MESSAGE}^{p}$}
algorithm attains a rate of order precisely ${\cal O}(n^{-1/\left(4\mathds{1}_{\left\{ p>1\right\} }+4\right)})$,
that is, ${\cal O}(n^{-1/4})$, if $p\equiv1$, and ${\cal O}(n^{-1/4})$,
if $p>1$, but pathwise convergence is not guaranteed, at least based
on the results presented in Section \ref{subsec:Pathwise-Convergence}.

Indeed, the conclusions of Theorem \ref{lem:Rate_Convex}, which match
the respective rate results previously developed for the general purpose
\textit{$T$-SCGD }algorithm in \citep{Wang2018}, are somewhat disappointing,
especially when $p>1$. However, we should note that this result assumes
nothing but mere convexity on the random cost $F\left(\cdot,\boldsymbol{W}\right)$
and, therefore, on $\phi^{\widetilde{F}}$, as well. This means that
Theorem \ref{lem:Rate_Convex} is valid for any problematic or pathological
choice of the potentially nonsmooth cost $F\left(\cdot,\boldsymbol{W}\right)$,
as long as it is convex (and of course satisfying any additional regularity
assumptions made throughout this work). 

Nonetheless, the situation changes dramatically if we strengthen our
assumptions on the convexity of $\phi^{\widetilde{F}}$, which, as
we will see shortly, can be guaranteed very naturally by in turn strengthening
the convexity of $F\left(\cdot,\boldsymbol{W}\right)$, as in classical
risk-neutral stochastic optimization. This is the subject of the next
paragraph.

\paragraph{Strongly Convex Random Cost}

Here, we assume that the risk-averse objective under consideration,
$\phi^{\widetilde{F}}$, is $\sigma$\textit{-strongly convex}, \textit{in
the sense that} there exists $\sigma>0$, such that
\begin{equation}
\phi^{\widetilde{F}}\left(\boldsymbol{x}\right)-\phi_{*}^{\widetilde{F}}\ge\sigma\left\Vert \boldsymbol{x}-\boldsymbol{x}^{*}\right\Vert _{2}^{2},\quad\forall\boldsymbol{x}\in{\cal X},\label{eq:Strong_Convex_O}
\end{equation}
where $\boldsymbol{x}^{*}\in{\cal X}^{*}$, and ${\cal X}^{*}$ is
singleton. Although condition (\ref{eq:Strong_Convex_O}) will turn
out to be very central in our analysis, is not very useful per se,
unless we can show that it can be satisfied under reasonable choices
of the random cost $F\left(\cdot,\boldsymbol{W}\right)$, and the
risk measure $\rho$, such that $\rho\left(F\left(\cdot,\boldsymbol{W}\right)\right)\equiv\phi^{\widetilde{F}}\left(\cdot\right)$.
In other words, it is important to be able to satisfy condition (\ref{eq:Strong_Convex_O})
\textit{constructively} within our problem setting, starting from
appropriate assumptions on its basic components (bottom-up).

In fact, it turns out that imposing $\sigma$-strong convexity on
$F\left(\cdot,\boldsymbol{W}\right)$ \textit{in the usual sense that}
there exists $\sigma>0$, such that
\begin{equation}
F\left(\cdot,\boldsymbol{w}\right)-\sigma\left\Vert \cdot\right\Vert _{2}^{2}\text{ is convex, for all }\boldsymbol{w}\in\mathbb{R}^{M},
\end{equation}
is \textit{all that is needed} in order to guarantee condition (\ref{eq:Strong_Convex_O})
for the objective of our base problem, $\phi^{\widetilde{F}}$. This
is a very simple, but important consequence of the fact that mean-semideviations
are convex risk measures (that is, convex, monotone and translation
equivariant real-valued functionals on ${\cal Z}_{q}$). The relevant
results follow.
\begin{prop}
\textbf{\textup{(Strong Convexity of Risk-Function Compositions)\label{prop:StrongConvexity-of-Compositions}}}
Consider a real-valued random function $f:\mathbb{R}^{N}\times\Omega\rightarrow\mathbb{R}$,
as well as a real-valued risk measure $\rho:{\cal Z}_{q}\rightarrow\mathbb{R}$.
Suppose that, for every $\omega\in\Omega$, $f\left(\cdot,\omega\right)$
is $\sigma$-strongly convex, and that $\rho$ is convex. Then, the
real-valued composite function $\phi^{f}\left(\cdot\right)\equiv\text{\ensuremath{\rho}}\left(f\left(\cdot,\bullet\right)\right):\mathbb{R}^{N}\rightarrow\mathbb{R}$
is $\sigma$-strongly convex, as well.
\end{prop}
\begin{proof}[Proof of Proposition \ref{prop:StrongConvexity-of-Compositions}]
By $\sigma$-strong convexity of $f\left(\cdot,\omega\right)$ for
all $\omega\in\Omega$, it is true that $f\left(\cdot,\omega\right)-\sigma\left\Vert \cdot\right\Vert _{2}^{2}$
is convex, for all $\omega\in\Omega$. But $\rho$ is a convex-monotone
risk measure and, thus, $\text{\ensuremath{\rho}}\left(f\left(\cdot,\bullet\right)-\sigma\left\Vert \cdot\right\Vert _{2}^{2}\right)$
is also convex. Since, additionally, $\rho$ is translation equivariant,
it is true that, for every $\boldsymbol{x}\in\mathbb{R}^{N}$,
\begin{equation}
\text{\ensuremath{\rho}}\left(f\left(\boldsymbol{x},\bullet\right)-\sigma\left\Vert \boldsymbol{x}\right\Vert _{2}^{2}\right)\equiv\text{\ensuremath{\rho}}\left(f\left(\boldsymbol{x},\bullet\right)\right)-\sigma\left\Vert \boldsymbol{x}\right\Vert _{2}^{2}\equiv\phi^{f}\left(\boldsymbol{x}\right)-\sigma\left\Vert \boldsymbol{x}\right\Vert _{2}^{2},
\end{equation}
which, of course, implies that the function $\phi^{f}\left(\cdot\right)-\sigma\left\Vert \cdot\right\Vert _{2}^{2}$
is convex. Enough said.
\end{proof}
For the special case of mean-semideviation models, Proposition \ref{prop:StrongConvexity-of-Compositions}
may be specialized accordingly, as follows. The proof is trivial,
and therefore is omitted.
\begin{prop}
\textbf{\textup{(Strong Convexity of $\phi^{\widetilde{F}}$)\label{prop:StrongConvexity-of-MS}}}
Fix $p\in\left[1,\infty\right)$ and choose any risk regularizer ${\cal R}:\mathbb{R}\rightarrow\mathbb{R}$.
Suppose that, for every $\boldsymbol{w}\in\mathbb{R}^{M}$, $F\left(\cdot,\boldsymbol{w}\right)$
is $\sigma$-strongly convex on ${\cal X}$. Then, as long as $c\in\left[0,1\right]$,
the composite function $\phi^{\widetilde{F}}\left(\cdot\right)\equiv\rho\left(F\left(\cdot,\boldsymbol{W}\right)\right)$
is $\sigma$-strongly convex on ${\cal X}$, as well, and satisfies
condition (\ref{eq:Strong_Convex_O}).
\end{prop}
Consequently, we see that strong convexity of $F\left(\cdot,\boldsymbol{W}\right)$
suffices for $\phi^{\widetilde{F}}$ being strongly convex, as well.
This fact is very important from an operational/practical point of
view, because it implies that guaranteeing strong convexity for a
risk-averse problem is in principle no harder than guaranteeing strong
convexity for the respective risk-neutral problem when mean-semideviations,
or, more generally, convex risk measures, are involved. In particular,
Proposition \ref{prop:StrongConvexity-of-Compositions} holds true
for all coherent risk measures, being also convex. 

What remains now is to quantitatively characterize the effect of strong
convexity on the convergence rates achieved by the \textit{$\textit{MESSAGE}^{p}$}
algorithm, and how those compare to the more general convex case,
briefly analyzed in Section \ref{par:Convex-Random-Cost}.

Our first result is a parametric \textit{``rate generator''}, which
provides general stepsize conditions, under which the optimality and
approximation errors of all three levels of the \textit{$\textit{MESSAGE}^{p}$}
algorithm may be jointly combined into a recursive inequality, resembling
the respective recursions arising in the rate analysis of standard,
risk-neutral SSD algorithms. This result is new, and builds upon some
basic ideas found in the earlier work of \citep{Wang2017}, concerning
the rate of the general purpose ($2$-level) \textit{SCGD} algorithm
developed therein.
\begin{lem}
\textbf{\textup{(Rate Generator | Strongly Convex Case)}}\label{lem:Rate_Generator}
Let Assumptions \ref{assu:F_AS_Main}, \ref{assu:Initial-Values}
be in effect, and let $\beta_{n}\in\left(0,1\right]$, $\gamma_{n}\in\left(0,1\right]$,
for all $n\in\mathbb{N}$. Also, suppose that $\phi^{\widetilde{F}}$
is $\sigma$-strongly convex, and that there exists $n_{o}\in\mathbb{N}^{+}$,
such that, for all $n\in\mathbb{N}^{n_{o}}$, the following conditions
hold simultaneously:
\begin{description}
\item [{$\mathbf{G1}$}] $\sigma\alpha_{n}\le\dfrac{K-1}{K}\min\left\{ \beta_{n-1},\gamma_{n-1}\right\} $,
for some bounded constant $K>1$.
\item [{$\mathbf{G2}$}] $\alpha_{n+1}\beta_{n-1}\le\alpha_{n}\beta_{n}$
and, likewise, $\alpha_{n+1}\gamma_{n-1}\le\alpha_{n}\gamma_{n}$.
\end{description}
For nonnegative sequences $\left\{ \boldsymbol{\Delta}_{B}^{n}\right\} _{n\in\mathbb{N}}$
and $\left\{ \boldsymbol{\Delta}_{C}^{n}\right\} _{n\in\mathbb{N}}$,
consider the process
\begin{flalign}
J^{n} & \triangleq\mathbb{E}\left\{ \left\Vert \boldsymbol{x}^{n}-\boldsymbol{x}^{*}\right\Vert _{2}^{2}\right\} +\boldsymbol{\Delta}_{B}^{n-1}\mathbb{E}\left\{ \left|y^{n-1}-{\cal S}^{\widetilde{F}}\hspace{-1pt}\hspace{-1pt}\left(\boldsymbol{x}^{n-1}\right)\right|^{2}\right\} \hspace{-1pt}\nonumber \\
 & \quad\quad\quad\quad\quad\quad\quad+\boldsymbol{\Delta}_{C}^{n-1}\mathbb{E}\left\{ \left|z^{n-1}-{\cal D}^{\widetilde{F}}\hspace{-2pt}\left(\boldsymbol{x}^{n-1},y^{n-1}\right)\right|^{2}\right\} \mathds{1}_{\left\{ p>1\right\} },\quad n\in\mathbb{N}^{+}.
\end{flalign}
Then, $\left\{ \boldsymbol{\Delta}_{B}^{n}\right\} _{n\in\mathbb{N}}$
and $\left\{ \boldsymbol{\Delta}_{C}^{n}\right\} _{n\in\mathbb{N}}$
may be chosen such that
\begin{flalign}
J^{n+1} & \le\left(1-\sigma\alpha_{n}\right)J^{n}+\widetilde{\Sigma}\left(\sigma^{2}\alpha_{n}^{2}+\dfrac{\sigma^{3}\alpha_{n}\alpha_{n-1}^{2}}{\beta_{n-1}^{2}}+\sigma\alpha_{n}\beta_{n-1}\right)\nonumber \\
 & \quad\quad\quad\quad\quad\quad+\widetilde{\Sigma}\left(\dfrac{\sigma^{3}\alpha_{n}\alpha_{n-1}^{2}}{\gamma_{n-1}^{2}}+\dfrac{\sigma\alpha_{n}\beta_{n-1}^{2}}{\gamma_{n-1}^{2}}+\sigma\alpha_{n}\gamma_{n-1}\right)\mathds{1}_{\left\{ p>1\right\} },\quad\forall n\in\mathbb{N}^{n_{o}},
\end{flalign}
for some constant $0<\widetilde{\Sigma}<\infty$. Additionally, under
the assumptions of Lemma \ref{lem:Iter_Error_Bound}, and for the
same choices of $\left\{ \boldsymbol{\Delta}_{B}^{n}\right\} _{n\in\mathbb{N}}$
and $\left\{ \boldsymbol{\Delta}_{C}^{n}\right\} _{n\in\mathbb{N}}$,
it is true that $\sup_{n\in\mathbb{N}^{+}}J^{n}<\infty.$
\end{lem}
\begin{proof}[Proof of Lemma \ref{lem:Rate_Generator}]
See Section \ref{subsec:Rate_Gen} (Appendix).
\end{proof}
Leveraging Lemma \ref{lem:Rate_Generator}, along with a simple generalization
of Chung's Lemma (see Section \ref{subsec:Chung} (Appendix)), we
may characterize the convergence rates of the \textit{$\textit{MESSAGE}^{p}$}
algorithm in the strongly convex case, in full and transparent technical
detail. We start with the case where $p>1$.
\begin{thm}
\textbf{\textup{(Rate | Strongly Convex Case | Subharmonic Stepsizes
| $p>1$)}}\label{lem:Rate_SConvex} Let Assumptions \ref{assu:F_AS_Main}
and \ref{assu:Initial-Values} be in effect. Suppose that $\phi^{\widetilde{F}}$
is $\sigma$-strongly convex, and that the stepsize sequences $\left\{ \alpha_{n}\right\} _{n\in\mathbb{N}}$,
$\left\{ \beta_{n}\right\} _{n\in\mathbb{N}}$ and $\left\{ \gamma_{n}\right\} _{n\in\mathbb{N}}$
satisfy the subharmonic rules
\begin{equation}
\alpha_{n}\triangleq\dfrac{1}{\sigma n},\quad\beta_{n}\triangleq\dfrac{1}{n^{\tau_{2}}}\quad\text{and}\quad\gamma_{n}\triangleq\dfrac{1}{n^{\tau_{3}}},\quad\forall n\in\mathbb{N}^{+},
\end{equation}
where $1/2\le\tau_{3}<\tau_{2}<1$, and with initial values $\alpha_{0}\equiv\beta_{0}\equiv\gamma_{0}\equiv1$.
Also, define the quantities
\begin{equation}
\text{\ensuremath{n_{o}}}\hspace{-1pt}\left(\tau_{2}\right)\equiv\left\lceil \dfrac{1}{1-\tau_{2}^{1/\left(\tau_{2}+1\right)}}\right\rceil \in\mathbb{N}^{3}\quad\text{and}\quad\mathsf{R}\left(\tau_{2},\tau_{3}\right)\triangleq\dfrac{1}{1-\max\left\{ 2-2\tau_{2},2\tau_{2}-2\tau_{3},\tau_{3}\right\} }>1.
\end{equation}
Then, for every $n\in\mathbb{N}^{n_{o}\left(\tau_{2}\right)}$, it
is true that
\begin{equation}
\mathbb{E}\left\{ \left\Vert \boldsymbol{x}^{n+1}-\boldsymbol{x}^{*}\right\Vert _{2}^{2}\right\} \le\widehat{\Sigma}\dfrac{\text{\ensuremath{n_{o}}}\hspace{-1pt}\left(\tau_{2}\right)}{n}+\widehat{\Sigma}\dfrac{\mathsf{R}\left(\tau_{2},\tau_{3}\right)}{n^{2\min\left\{ 1-\tau_{2},\tau_{2}-\tau_{3}\right\} }},
\end{equation}
for some constant $0<\widehat{\Sigma}<\infty$. In particular, if,
for some $\epsilon\in\left[0,1\right)$ and $\delta\in\left(0,1\right)$,
\begin{equation}
\tau_{2}\equiv\dfrac{3+\epsilon}{4}\quad\text{and}\quad\tau_{3}\equiv\dfrac{1+\delta\epsilon}{2},
\end{equation}
then the $\textit{MESSAGE}^{p}$ algorithm satisfies
\begin{equation}
{\cal O}\left(n^{-\left(1-\epsilon\right)/2}\right)\equiv\mathbb{E}\left\{ \left\Vert \boldsymbol{x}^{n+1}-\boldsymbol{x}^{*}\right\Vert _{2}^{2}\right\} \le\dfrac{\widehat{\Sigma}\left(n_{o}\left(\epsilon\right)+\dfrac{2}{1-\epsilon}\right)}{n^{\left(1-\epsilon\right)/2}},
\end{equation}
for every $n\in\mathbb{N}^{n_{o}\left(\epsilon\right)}$, for each
fixed $\epsilon$. 
\end{thm}
\begin{proof}[Proof of Theorem \ref{lem:Rate_SConvex}]
See Section \ref{subsec:Rate_Sconvex} (Appendix).
\end{proof}
The main conclusion of Theorem \ref{lem:Rate_SConvex} is that, for
fixed semideviation order $p>1$ and for any choice of the user-specified
parameter $\epsilon\in\left[0,1\right)$, the $\textit{MESSAGE}^{p}$
algorithm achieves a squared-${\cal L}_{2}$ solution suboptimality
rate of the order of ${\cal O}(n^{-\left(1-\epsilon\right)/2})$ iterations.
If, additionally, $\epsilon$ is chosen to be strictly positive, that
is, for $\epsilon>0$, pathwise convergence is \textit{simultaneously}
guaranteed, since the constraints (\ref{eq:STEP_1_W})-(\ref{eq:STEP_2_W})
of Section \ref{subsec:Pathwise-Convergence} are also satisfied.
Similarly to the convex case, this completely novel result establishes
a convergence rate of order arbitrarily close to ${\cal O}(n^{-1/2})$
as $\epsilon\rightarrow0$, while ensuring stable pathwise operation
of the algorithm. Of course, when $\epsilon\equiv0$, the rate of
${\cal O}(n^{-1/2})$ iterations is attained, but pathwise convergence
of the algorithm is not guaranteed.

Setting aside the fact that rate quantification is different for the
convex and strongly convex cases, and by looking at the respective
rate exponents, we observe that Theorem \ref{lem:Rate_SConvex} provides
a rate \textit{strictly four ($4$) times faster} than that provided
by Theorem \ref{lem:Rate_Convex}. Of course, this substantial improvement
on the rate of convergence of the $\textit{MESSAGE}^{p}$ algorithm
is made possible due to imposition of strong convexity on the risk-averse
objective $\phi^{\widetilde{F}}$.

When $p\equiv1$, we also have the following simpler result. As the
proof is very similar to that of Theorem \ref{lem:Rate_SConvex},
it is omitted.
\begin{thm}
\textbf{\textup{(Rate | Strongly Convex Case | Subharmonic Stepsizes
| $p\equiv1$)}}\label{lem:Rate_SConvex-1} Let Assumptions \ref{assu:F_AS_Main}
and \ref{assu:Initial-Values} be in effect. Suppose that $\phi^{\widetilde{F}}$
is $\sigma$-strongly convex, and that the stepsize sequences $\left\{ \alpha_{n}\right\} _{n\in\mathbb{N}}$,
$\left\{ \beta_{n}\right\} _{n\in\mathbb{N}}$ follow the subharmonic
rules
\begin{equation}
\alpha_{n}\triangleq\dfrac{1}{\sigma n},\quad\text{and}\quad\beta_{n}\triangleq\dfrac{1}{n^{\tau_{2}}}\quad\forall n\in\mathbb{N}^{+},
\end{equation}
where $1/2<\tau_{2}<1$, and with initial values $\alpha_{0}\equiv\beta_{0}\equiv1$.
Choose $\text{\ensuremath{n_{o}}}\hspace{-1pt}\left(\tau_{2}\right)$
as in Theorem \ref{lem:Rate_SConvex}, and define
\begin{equation}
\mathsf{R}\left(\tau_{2}\right)\triangleq\dfrac{1}{1-\max\left\{ 2-2\tau_{2},\tau_{2}\right\} }>1
\end{equation}
 Then, for every $n\in\mathbb{N}^{n_{o}\left(\tau_{2}\right)}$, it
is true that
\begin{equation}
\mathbb{E}\left\{ \left\Vert \boldsymbol{x}^{n+1}-\boldsymbol{x}^{*}\right\Vert _{2}^{2}\right\} \le\dfrac{\overline{\Sigma}\left(n_{o}\left(\tau_{2}\right)+\mathsf{R}\left(\tau_{2}\right)\right)}{n^{\min\left\{ 2-2\tau_{2},\tau_{2}\right\} }},
\end{equation}
for some constant $0<\widehat{\Sigma}<\infty$. In particular, the
exponent in the denominator is maximized at $\tau_{2}^{*}\equiv2/3$,
yielding a rate of the order of ${\cal O}(n^{-2/3})$.
\end{thm}
In the structurally simpler case where $p\equiv1$, the rate order
improves to ${\cal O}(n^{-2/3})$, which is sufficient for pathwise
convergence as well, and matches existing results in compositional
stochastic optimization, developed earlier along the lines of \citep{Wang2017}.
Compared to the convex case (Theorem \ref{lem:Rate_Convex}), Theorem
\ref{lem:Rate_SConvex-1} provides a rate which is \textit{at most
$8/3\approx2.67$ times faster}. Note, however, that whereas in the
strongly convex case pathwise convergence is always guaranteed for
the particular selection of stepsizes, this does not happen in the
convex case, which also involves the choice of $\epsilon$. This seems
to be a unique feature of mean-semideviation problems of order $p\equiv1$
(two SA levels), since if $p>1$ (three SA levels), the trade-off
between achieving pathwise convergence and a fast rate of convergence
exists in both the convex and strongly convex cases.

\subsection{\label{subsec:The-Choice-of-R}The Choice of ${\cal R}$: Comparison
with Assumption 2.1 of \citep{Wang2018}}

We now present a detailed comparison between Assumption \ref{assu:F_AS_Main},
which is proposed in this paper, and \textit{Assumption 2.1 of \citep{Wang2018}},
which is utilized for analyzing and proving convergence of the general
purpose \textit{$T$-SCGD }algorithm, formulated therein. In the following,
we rigorously show that, as far as problem (\ref{eq:MAIN_PROB}) is
concerned, Assumption \ref{assu:F_AS_Main} imposes substantially
weaker restrictions on problem structure, compared with (\citep{Wang2018},
Assumption 2.1), therefore providing a much broader structural framework
for recursive, compositional SSD-type optimization of mean-semideviation
risk measures.

Recall from (\ref{eq:MAIN_4}) that the objective of our base problem
(\ref{eq:MAIN_PROB}) may be equivalently represented in the form
considered in \citep{Wang2017,Wang2018} as
\begin{flalign}
\phi^{\widetilde{F}}\left(\boldsymbol{x}\right) & \equiv\widehat{\varrho}\left(\widehat{\boldsymbol{g}}^{\widetilde{F}}\left(\widehat{\boldsymbol{h}}^{\widetilde{F}}\left(\boldsymbol{x}\right)\right)\hspace{-2pt}\right),\quad\forall\boldsymbol{x}\in{\cal X},\quad\text{with}\\
\widehat{\varrho}\left(x,y\right) & \equiv\mathbb{E}\left\{ x+cy^{1/p}\right\} \equiv x+cy^{1/p},\\
\widehat{\boldsymbol{g}}^{\widetilde{F}}\left(\boldsymbol{x},y\right) & \equiv\mathbb{E}\left\{ \left[y\:\left({\cal R}\left(F\left(\boldsymbol{x},\boldsymbol{W}\right)-y\right)\right)^{p}\right]\right\} \nonumber \\
 & \triangleq\mathbb{E}\left\{ \widehat{\boldsymbol{g}}_{\boldsymbol{W}}^{\widetilde{F}}\left(\boldsymbol{x},y\right)\right\} \quad\text{and}\\
\widehat{\boldsymbol{h}}^{\widetilde{F}}\left(\boldsymbol{x}\right) & \equiv\mathbb{E}\left\{ \left[\boldsymbol{x}\:F\left(\boldsymbol{x},\boldsymbol{W}\right)\right]\right\} \nonumber \\
 & \triangleq\mathbb{E}\left\{ \widehat{\boldsymbol{h}}_{\boldsymbol{W}}^{\widetilde{F}}\left(\boldsymbol{x}\right)\right\} ,\quad\boldsymbol{x}\in{\cal X}.\label{eq:Mengdi_1}
\end{flalign}

Via careful, but relatively straightforward comparison, it follows
that, relative to problem (\ref{eq:MAIN_PROB}), (\citep{Wang2018},
Assumption 2.1) translates into the following structural requirements,
where the quantities $\varepsilon$ and ${\cal E}$ are defined precisely
as in condition ${\bf C4}$ of Assumption \ref{assu:F_AS_Main}. Recall
that $\varepsilon$ and ${\cal E}$ characterize the essential range
of ${\cal R}\left(F\left(\cdot,\boldsymbol{W}\right)-\bullet\right)$
and the iterate process $\left\{ z^{n}\right\} _{n\in\mathbb{N}^{+}}$,
if $z^{0}\in\left[\varepsilon^{p},{\cal E}^{p}\right]$ (see Lemma
\ref{lem:Iterate_Boundedness}). Here, though, we allow the possibility
of $\varepsilon$ and ${\cal E}$ attaining the values zero and infinity,
respectively, and we explicitly adopt the generalized definitions
\begin{flalign}
\varepsilon & \triangleq\lim_{x\rightarrow m_{l}-m_{h}}{\cal R}\left(x\right)\quad\text{and}\\
{\cal E} & \triangleq\lim_{x\rightarrow m_{h}-m_{l}}{\cal R}\left(x\right),
\end{flalign}
where $m_{l}\in\left[-\infty,\infty\right]$ and $m_{h}\in\left[-\infty,\infty\right]$,
respecting the constraint $m_{l}\le m_{h}$ (note that, by our assumptions,
$m_{l}$ and $m_{h}$ cannot be equal and infinite at the same time).
Also, ${\cal E}$ is finite if and only if \textit{both} $m_{l}$
and $m_{h}$ are finite.
\begin{description}
\item [{$\mathbf{W1}$}] $\,\,\hspace{0.9bp}$The random functions $\widehat{\boldsymbol{g}}_{\boldsymbol{W}}^{\widetilde{F}}$
and $\widehat{\boldsymbol{h}}_{\boldsymbol{W}}^{\widetilde{F}}$ are
of uniformly bounded variance.
\item [{$\mathbf{W2}$}] $\,\,\hspace{0.9bp}$Almost everywhere on $\Omega$,
$\widehat{g}_{\boldsymbol{W}}^{\widetilde{F}}$ is \textit{differentiable
everywhere} on ${\cal X}\times\textrm{cl}\left\{ \left(m_{l},m_{h}\right)\right\} $.
In other words, it is true that
\begin{equation}
{\cal P}\left(\left\{ \omega\in\Omega\left|\nabla\widehat{g}_{\boldsymbol{W}\left(\omega\right)}^{\widetilde{F}}\text{ exists for all }{\cal X}\times\textrm{cl}\left\{ \left(m_{l},m_{h}\right)\right\} \right.\right\} \right)\equiv1.
\end{equation}
\item [{$\mathbf{W3}$}] $\,\,\hspace{0.9bp}$The \textit{squared} induced
operator $\ell_{2}$-norms of the random subgradient $\underline{\nabla}\widehat{\boldsymbol{h}}_{\boldsymbol{W}}^{\widetilde{F}}$
and the almost everywhere existent random Jacobian function $\nabla\widehat{\boldsymbol{g}}_{\boldsymbol{W}}^{\widetilde{F}}$
are uniformly bounded in expectation.
\item [{$\mathbf{W4}$}] $\,\,\hspace{0.9bp}$The random Jacobian of $\widehat{\boldsymbol{g}}_{\boldsymbol{W}}^{\widetilde{F}}$
is \textit{uniformly Lipschitz} on ${\cal X}\times\textrm{cl}\left\{ \left(m_{l},m_{h}\right)\right\} $,
that is, there exists some constant, say $L<\infty$, such that
\begin{equation}
\left\Vert \nabla\widehat{\boldsymbol{g}}_{\boldsymbol{W}}^{\widetilde{F}}\left(\boldsymbol{x}_{1},y_{1}\right)-\nabla\widehat{\boldsymbol{g}}_{\boldsymbol{W}}^{\widetilde{F}}\left(\boldsymbol{x}_{2},y_{2}\right)\right\Vert _{2}\le L\sqrt{\left\Vert \boldsymbol{x}_{1}-\boldsymbol{x}_{2}\right\Vert _{2}^{2}+\left|y_{1}-y_{2}\right|^{2}},
\end{equation}
for all $\left(\left[\boldsymbol{x}_{1}\,y_{1}\right],\left[\boldsymbol{x}_{2}\,y_{2}\right]\right)\in\left[{\cal X}\times\textrm{cl}\left\{ \left(m_{l},m_{h}\right)\right\} \right]^{2}$,
\textit{almost everywhere on} $\Omega$.
\item [{$\mathbf{W5}$}] $\,\,\hspace{0.9bp}$The expectation function
$\widehat{\boldsymbol{h}}^{\widetilde{F}}$ is Lipschitz on ${\cal X}$.
\item [{$\mathbf{W6}$}] $\,\,\hspace{0.9bp}$The gradient $\nabla\widehat{\varrho}$
of the outer function $\widehat{\varrho}$ is both uniformly bounded
(relative to any $\ell_{p}$-norm) and Lipschitz on $\textrm{cl}\left\{ \left(m_{l},m_{h}\right)\right\} \times\left[\varepsilon^{p},{\cal E}^{p}\right]$.
\end{description}
Conditions ${\bf W1}-{\bf W6}$ match precisely (\citep{Wang2018},
Assumption 2.1), when the latter is applied to the class of risk-averse
problems considered in this paper. In our analysis, we will also impose
the following condition \textit{in addition to} ${\bf W1}-{\bf W6}$,
closely resembling condition ${\bf C4}$ of Assumption \ref{assu:F_AS_Main}.
\begin{description}
\item [{$\mathbf{\widetilde{W}7}$}] $\,\,\hspace{0.9bp}$Whenever $p>1$,
it is true that ${\cal E}<\infty$.
\end{description}
Although condition $\mathbf{\widetilde{W}7}$ is not explicitly considered
in \citep{Wang2018}, it is made here in order to simplify and free
the comparison with our proposed Assumption \ref{assu:F_AS_Main}
from unnecessary technical complications. In effect, considering condition
$\mathbf{\widetilde{W}7}$ together with conditions ${\bf W1}-{\bf W6}$
slightly restricts the class of problems supported by the latter.
Nonetheless, such restriction is by no means that severe. On the other
hand, imposing condition $\mathbf{\widetilde{W}7}$ provides great
analytical flexibility; without it, verification of conditions $\mathbf{W1}$
and $\mathbf{W3}$ within the framework of \citep{Wang2018}, referring
in particular to (\citep{Wang2018}, Assumption 2.1 (iii) \& (iv)),
becomes rather problematic and uninsightful, for reasons very similar
to those justifying condition ${\bf C4}$ as part of Assumption \ref{assu:F_AS_Main}.
The usefulness of condition $\mathbf{\widetilde{W}7}$ in addition
to conditions ${\bf W1}-{\bf W6}$ in the framework of \citep{Wang2018}
is clearly demonstrated in our discussion below. 

Of course, condition $\mathbf{\widetilde{W}7}$ is trivially equivalent
with \textit{almost half of} condition ${\bf C4}$ of Assumption \ref{assu:F_AS_Main};
thus, no further comment is necessary. Amongst all remaining conditions
${\bf W1}-{\bf W6}$, conditions $\mathbf{W5}$ and $\mathbf{W6}$
are the easiest to discuss and may be almost trivially shown to be
automatically satisfied by all problems considered in this paper.
Between the latter, let us strategically consider condition $\mathbf{W6}$
first, which requires that the gradient function
\begin{equation}
\nabla\widehat{\varrho}\left(x,y\right)\equiv\begin{bmatrix}1\\
c\dfrac{1}{p}y^{\frac{1-p}{p}}
\end{bmatrix}
\end{equation}
is uniformly bounded and Lipschitz on $\textrm{cl}\left\{ \left(m_{l},m_{h}\right)\right\} \times\left[\varepsilon^{p},{\cal E}^{p}\right]$.
If $p>1$ (if not, the situation is trivial), in order for $\nabla\widehat{\varrho}$
to be uniformly bounded, we of course need to verify that (any $\ell_{p}$-norm
is fine)
\begin{flalign}
\sup_{\left(x,y\right)\in\textrm{cl}\left\{ \left(m_{l},m_{h}\right)\right\} \times\left[\varepsilon^{p},{\cal E}^{p}\right]}\left\Vert \nabla\widehat{\varrho}\left(x,y\right)\right\Vert _{2} & \equiv\sup_{y\in\left[\varepsilon^{p},{\cal E}^{p}\right]}\sqrt{1+c^{2}\dfrac{1}{p^{2}}y^{\frac{2\left(1-p\right)}{p}}}\nonumber \\
 & \equiv\sqrt{1+c^{2}\dfrac{1}{p^{2}}\sup_{y\in\left[\varepsilon^{p},{\cal E}^{p}\right]}y^{\frac{2\left(1-p\right)}{p}}}<\infty,
\end{flalign}
which, due to the fact that $\left(\cdot\right)^{\frac{2\left(1-p\right)}{p}}$
is a hyperbola, \textit{is} \textit{only possible if} $\varepsilon>0$,
yielding
\begin{equation}
\sup_{y\in\left[\varepsilon^{p},{\cal E}^{p}\right]}y^{\frac{2\left(1-p\right)}{p}}=\dfrac{1}{\varepsilon^{2\left(p-1\right)}}.
\end{equation}
By taking the Jacobian of $\nabla\widehat{\varrho}$, it can be easily
shown that strict positivity of $\varepsilon$ ensures that $\nabla\widehat{\varrho}$
is Lipschitz on $\textrm{cl}\left\{ \left(m_{l},m_{h}\right)\right\} \times\left[\varepsilon^{p},{\cal E}^{p}\right]$,
as well. Consequently, we see that condition $\mathbf{W6}$ is implied
by condition ${\bf C4}$ of Assumption \ref{assu:F_AS_Main}. It is
also relatively easy to show that condition $\mathbf{W6}$ \textit{together
with} $\mathbf{\widetilde{W}7}$ are in fact equivalent to ${\bf C4}$.
As far as condition $\mathbf{W5}$ is concerned, this can be directly
verified exploiting Lipschitz continuity of the function $\mathbb{E}\left\{ F\left(\cdot,\boldsymbol{W}\right)\right\} $
on ${\cal X}$ (see Lemma \ref{prop:EF_LIP}). In the following, we
examine the less obvious, remaining conditions ${\bf W1}-{\bf W4}$,
in greater detail. 

We start with condition ${\bf W1}$. In order for $\widehat{\boldsymbol{g}}_{\boldsymbol{W}}^{\widetilde{F}}$
and $\widehat{\boldsymbol{h}}_{\boldsymbol{W}}^{\widetilde{F}}$ to
be of uniformly bounded variance, it must be true that
\begin{flalign}
\sup_{\boldsymbol{x}\in{\cal X}}\sup_{y\in\textrm{cl}\left\{ \left(m_{l},m_{h}\right)\right\} }\mathbb{E}\left\{ \left\Vert \widehat{\boldsymbol{g}}_{\boldsymbol{W}}^{\widetilde{F}}\left(\boldsymbol{x},y\right)-\widehat{\boldsymbol{g}}^{\widetilde{F}}\left(\boldsymbol{x},y\right)\right\Vert _{2}^{2}\right\}  & <\infty\quad\text{and}\\
\sup_{\boldsymbol{x}\in{\cal X}}\mathbb{E}\left\{ \left\Vert \widehat{\boldsymbol{h}}_{\boldsymbol{W}}^{\widetilde{F}}\left(\boldsymbol{x}\right)-\widehat{\boldsymbol{h}}^{\widetilde{F}}\left(\boldsymbol{x}\right)\right\Vert _{2}^{2}\right\}  & <\infty,
\end{flalign}
respectively. Let $\boldsymbol{W}':\Omega\rightarrow\mathbb{R}^{M}$
be an independent copy of the information variable $\boldsymbol{W}$.
Then, regarding $\widehat{\boldsymbol{g}}_{\boldsymbol{W}}^{\widetilde{F}}$,
we have, for every $\boldsymbol{x}\in{\cal X}$ and for every $y\in\textrm{cl}\left\{ \left(m_{l},m_{h}\right)\right\} $,
\begin{flalign}
 & \hspace{-2pt}\hspace{-2pt}\hspace{-2pt}\hspace{-2pt}\hspace{-2pt}\hspace{-2pt}\hspace{-2pt}\hspace{-2pt}\hspace{-2pt}\mathbb{E}\left\{ \left\Vert \widehat{\boldsymbol{g}}_{\boldsymbol{W}}^{\widetilde{F}}\left(\boldsymbol{x},y\right)-\widehat{\boldsymbol{g}}^{\widetilde{F}}\left(\boldsymbol{x},y\right)\right\Vert _{2}^{2}\right\} \nonumber \\
 & \equiv\mathbb{E}\left\{ \left\Vert \left[0\:\left({\cal R}\left(F\left(\boldsymbol{x},\boldsymbol{W}\right)-y\right)\right)^{p}-\mathbb{E}\left\{ \left({\cal R}\left(F\left(\boldsymbol{x},\boldsymbol{W}\right)-y\right)\right)^{p}\right\} \right]\right\Vert _{2}^{2}\right\} \nonumber \\
 & \equiv\mathbb{E}\left\{ \left(\left({\cal R}\left(F\left(\boldsymbol{x},\boldsymbol{W}\right)-y\right)\right)^{p}-\mathbb{E}\left\{ \left({\cal R}\left(F\left(\boldsymbol{x},\boldsymbol{W}\right)-y\right)\right)^{p}\right\} \right)^{2}\right\} \nonumber \\
 & \equiv\mathbb{E}\left\{ \left(\left({\cal R}\left(F\left(\boldsymbol{x},\boldsymbol{W}\right)-y\right)\right)^{p}-\mathbb{E}\left\{ \left.\left({\cal R}\left(F\left(\boldsymbol{x},\boldsymbol{W}'\right)-y\right)\right)^{p}\right|\boldsymbol{W}\right\} \right)^{2}\right\} \nonumber \\
 & \equiv\mathbb{E}\left\{ \left(\mathbb{E}\left\{ \left({\cal R}\left(F\left(\boldsymbol{x},\boldsymbol{W}\right)-y\right)\right)^{p}-\left.\left({\cal R}\left(F\left(\boldsymbol{x},\boldsymbol{W}'\right)-y\right)\right)^{p}\right|\boldsymbol{W}\right\} \right)^{2}\right\} ,
\end{flalign}
which, by Jensen, yields
\begin{equation}
\mathbb{E}\left\{ \left\Vert \widehat{\boldsymbol{g}}_{\boldsymbol{W}}^{\widetilde{F}}\left(\boldsymbol{x},y\right)-\widehat{\boldsymbol{g}}^{\widetilde{F}}\left(\boldsymbol{x},y\right)\right\Vert _{2}^{2}\right\} \le\mathbb{E}\left\{ \left|\left({\cal R}\left(F\left(\boldsymbol{x},\boldsymbol{W}\right)-y\right)\right)^{p}-\left({\cal R}\left(F\left(\boldsymbol{x},\boldsymbol{W}'\right)-y\right)\right)^{p}\right|^{2}\right\} .
\end{equation}
If $p\equiv1$, nonexpansiveness of ${\cal R}$ (condition ${\bf S4}$)
further implies that
\begin{flalign}
\mathbb{E}\left\{ \left\Vert \widehat{\boldsymbol{g}}_{\boldsymbol{W}}^{\widetilde{F}}\left(\boldsymbol{x},y\right)-\widehat{\boldsymbol{g}}^{\widetilde{F}}\left(\boldsymbol{x},y\right)\right\Vert _{2}^{2}\right\}  & \le\mathbb{E}\left\{ \left|F\left(\boldsymbol{x},\boldsymbol{W}\right)-F\left(\boldsymbol{x},\boldsymbol{W}'\right)\right|^{2}\right\} \nonumber \\
 & \equiv\mathbb{E}\left\{ \left(F\left(\boldsymbol{x},\boldsymbol{W}\right)\right)^{2}+\left(F\left(\boldsymbol{x},\boldsymbol{W}'\right)\right)^{2}-2F\left(\boldsymbol{x},\boldsymbol{W}\right)F\left(\boldsymbol{x},\boldsymbol{W}'\right)\right\} \nonumber \\
 & \equiv2\mathbb{V}\left\{ F\left(\boldsymbol{x},\boldsymbol{W}\right)\right\} ,
\end{flalign}
for all $\boldsymbol{x}\in{\cal X}$ and for all $y\in\textrm{cl}\left\{ \left(m_{l},m_{h}\right)\right\} $.
For $p>1$, Lemma \ref{lem:P_is_LIP} similarly implies that (recall
that we have assumed that condition $\mathbf{\widetilde{W}7}$ is
true)
\begin{equation}
\mathbb{E}\left\{ \left\Vert \widehat{\boldsymbol{g}}_{\boldsymbol{W}}^{\widetilde{F}}\left(\boldsymbol{x},y\right)-\widehat{\boldsymbol{g}}^{\widetilde{F}}\left(\boldsymbol{x},y\right)\right\Vert _{2}^{2}\right\} \le2{\cal E}^{p-1}p\mathbb{V}\left\{ F\left(\boldsymbol{x},\boldsymbol{W}\right)\right\} ,
\end{equation}
for all $\boldsymbol{x}\in{\cal X}$ and for all $y\in\textrm{cl}\left\{ \left(m_{l},m_{h}\right)\right\} $.
Consequently, $F\left(\cdot,\boldsymbol{W}\right)$ being uniformly
in ${\cal Z}_{2}$ is \textit{sufficient}, so that $\widehat{\boldsymbol{g}}_{\boldsymbol{W}}^{\widetilde{F}}$
is also uniformly in ${\cal Z}_{2}$, as required. Now, note that,
for every $\boldsymbol{x}\in{\cal X}$,
\begin{equation}
\mathbb{E}\left\{ \left\Vert \widehat{\boldsymbol{h}}_{\boldsymbol{W}}^{\widetilde{F}}\left(\boldsymbol{x}\right)-\widehat{\boldsymbol{h}}^{\widetilde{F}}\left(\boldsymbol{x}\right)\right\Vert _{2}^{2}\right\} \equiv\mathbb{E}\left\{ \left(F\left(\boldsymbol{x},\boldsymbol{W}\right)-\mathbb{E}\left\{ F\left(\boldsymbol{x},\boldsymbol{W}\right)\right\} \right)^{2}\right\} \equiv\mathbb{V}\left\{ F\left(\boldsymbol{x},\boldsymbol{W}\right)\right\} ,
\end{equation}
and thus $F\left(\cdot,\boldsymbol{W}\right)$ being uniformly in
${\cal Z}_{2}$ is \textit{equivalent} to $\widehat{\boldsymbol{h}}_{\boldsymbol{W}}^{\widetilde{F}}$
being uniformly in ${\cal Z}_{2}$. Apparently, condition ${\bf W1}$
\textit{implies} condition ${\bf C2}$ of Assumption \ref{assu:F_AS_Main},
which directly requires that $F\left(\cdot,\boldsymbol{W}\right)$
is uniformly in ${\cal Z}_{2}$, in turn implying condition ${\bf W1}$.
Therefore, conditions ${\bf C2}$ and ${\bf W1}$ are \textit{equivalent}.

Second, we examine the consequences of assuming everywhere differentiability
of $\widehat{\boldsymbol{g}}_{\boldsymbol{W}}^{\widetilde{F}}$ primarily
on the smoothness on the risk regularizer ${\cal R}$, but also that
of the random cost function $F\left(\cdot,\boldsymbol{W}\right)$.
Suppose that there exists a measurable set $\widehat{\Omega}\subseteq\Omega$,
with ${\cal P}\left(\widehat{\Omega}\right)\equiv1$, such that, for
all $\omega\in\widehat{\Omega}$, $\widehat{\boldsymbol{g}}_{\boldsymbol{W}\left(\omega\right)}^{\widetilde{F}}$
is differentiable everywhere on ${\cal X}\times\textrm{cl}\left\{ \left(m_{l},m_{h}\right)\right\} $,
as in condition ${\bf W2}$. Without loss of generality, we can take
$\widehat{\Omega}\equiv\Omega_{E}$. Then, for every $\omega\in\widehat{\Omega}$
and for every $\left(\boldsymbol{x},y\right)\in{\cal X}\times\textrm{cl}\left\{ \left(m_{l},m_{h}\right)\right\} $,
and due to convexity (see also Proof of Lemma \ref{lem:Sub_Grad}
in Section \ref{subsec:MUS-1-Direct-Derivation}), the (random) Jacobian
of $\widehat{\boldsymbol{g}}_{\boldsymbol{W}}^{\widetilde{F}}$ may
be expressed as\renewcommand{\arraystretch}{1.2}
\begin{flalign}
\nabla\widehat{\boldsymbol{g}}_{\boldsymbol{W}}^{\widetilde{F}}\left(\boldsymbol{x},y\right) & =\begin{bmatrix}{\bf 0}_{N} & \nabla_{\boldsymbol{x}}\left[\left({\cal R}\left(F\left(\boldsymbol{x},\boldsymbol{W}\right)-y\right)\right)^{p}\right]\\
1 & \nabla_{y}\left[\left({\cal R}\left(F\left(\boldsymbol{x},\boldsymbol{W}\right)-y\right)\right)^{p}\right]
\end{bmatrix}\nonumber \\
 & =\begin{bmatrix}{\bf 0}_{N} & p\left({\cal R}\left(F\left(\boldsymbol{x},\boldsymbol{W}\right)-y\right)\right)^{p-1}\underline{\nabla}{\cal R}\left(F\left(\boldsymbol{x},\hspace{-1pt}\boldsymbol{W}\right)-y\right)\underline{\nabla}F\left(\boldsymbol{x},\boldsymbol{W}\right)\\
1 & -p\left({\cal R}\left(F\left(\boldsymbol{x},\boldsymbol{W}\right)-y\right)\right)^{p-1}\underline{\nabla}{\cal R}\left(F\left(\boldsymbol{x},\hspace{-1pt}\boldsymbol{W}\right)-y\right)
\end{bmatrix}\nonumber \\
 & \equiv\begin{bmatrix}{\bf 0}_{N} & \left.\underline{\nabla}\left({\cal R}\left(z\right)\right)^{p}\right|_{z\equiv F\left(\boldsymbol{x},\boldsymbol{W}\right)-y}\times\underline{\nabla}F\left(\boldsymbol{x},\boldsymbol{W}\right)\\
1 & -\left.\underline{\nabla}\left({\cal R}\left(z\right)\right)^{p}\right|_{z\equiv F\left(\boldsymbol{x},\boldsymbol{W}\right)-y}
\end{bmatrix},
\end{flalign}
\renewcommand{\arraystretch}{1}where we have assumed that, although
$\nabla\widehat{\boldsymbol{g}}_{\boldsymbol{W}}^{\widetilde{F}}$
exists, the convex functions $F\left(\cdot,\boldsymbol{W}\right)$
and ${\cal R}$ may \textit{not} be differentiable \textit{everywhere}
on ${\cal X}$ and $\mathfrak{R}^{\widetilde{F}}$ (the effective
domain of ${\cal R}\left(F\left(\cdot,\boldsymbol{W}\right)-\left(\bullet\right)\right)$),
respectively. Then, the following proposition is true.
\begin{prop}
\textbf{\textup{(Masks of Nondifferentiability)\label{prop:NonDiff}}}
Assume that, for some fixed value of $p\in\left[1,\infty\right)$,
condition ${\bf W2}$ is satisfied. Then, the following statements
are necessarily true:
\begin{enumerate}
\item The $p$-th power of ${\cal R}$ is differentiable everywhere on $\mathfrak{R}^{\widetilde{F}}$.
\item Either:
\begin{enumerate}
\item \uline{Everywhere on \mbox{$\widehat{\Omega}$}}, the random cost
function $F\left(\cdot,\boldsymbol{W}\right)$ is differentiable everywhere
on ${\cal X}$,

\hspace{-23.7bp}Or:
\item If, \uline{for at least one \mbox{$\omega\in\widehat{\Omega}$}},
$F\left(\cdot,\boldsymbol{W}\right)$ is nondifferentiable at some
$\boldsymbol{x}\in{\cal X}$, it must be true that
\begin{equation}
{\cal R}\left(z\right)\equiv C_{ND},\quad\forall z\in\bigcup_{\omega\in\widehat{\Omega}_{ND}}\mathrm{cl}\left\{ \left(-\infty,F_{ND}^{*}\left(\omega\right)-m_{l}\right)\right\} ,\label{eq:COMP_2-1}
\end{equation}
where $0\le C_{ND}<\infty$ is some constant, the function $F_{ND}^{*}:\widehat{\Omega}\rightarrow\left[-\infty,\infty\right]$
is defined as
\begin{equation}
F_{ND}^{*}\left(\omega\right)\triangleq\sup\hspace{-1pt}\left\{ \hspace{-1pt}F\left(\boldsymbol{x},\boldsymbol{W}\left(\omega\right)\right)\in\mathbb{R}\left|F\left(\cdot,\boldsymbol{W}\left(\omega\right)\right)\text{ is nondifferentiable at }\boldsymbol{x}\right.\hspace{-1pt}\right\} \hspace{-2pt},\;\omega\in\widehat{\Omega},
\end{equation}
and the set of elementary events $\widehat{\Omega}_{ND}\subseteq\widehat{\Omega}$
is defined as
\begin{equation}
\widehat{\Omega}_{ND}\triangleq\left\{ \hspace{-1pt}\left.\omega\in\widehat{\Omega}\right|F\left(\cdot,\boldsymbol{W}\left(\omega\right)\right)\text{ is nondifferentiable at some }\boldsymbol{x}\right\} .\label{eq:Omega_ND}
\end{equation}
 In other words, ${\cal R}$ must be partially constant, as prescribed
by (\ref{eq:COMP_2-1}).
\end{enumerate}
\end{enumerate}
\end{prop}
\begin{proof}[Proof of Proposition \ref{prop:NonDiff}]
First, \textit{existence} of the gradient $\nabla_{y}\left[\left({\cal R}\left(F\left(\boldsymbol{x},\boldsymbol{W}\right)-y\right)\right)^{p}\right]$
\textit{for all} $\left(\boldsymbol{x},y\right)\in{\cal X}\times\textrm{cl}\left\{ \left(m_{l},m_{h}\right)\right\} $,
which is true by our hypothesis, necessarily implies that the function
$\left({\cal R}\left(\cdot\right)\right)^{p}$ is \textit{differentiable
everywhere on} $\mathfrak{R}^{\widetilde{F}}$, for the particular
choice of $p\in\left[1,\infty\right)$. The case of the gradient $\nabla_{\boldsymbol{x}}\left[\left({\cal R}\left(F\left(\boldsymbol{x},\boldsymbol{W}\right)-y\right)\right)^{p}\right]$
is slightly more complicated. In order for $\nabla_{\boldsymbol{x}}\left[\left({\cal R}\left(F\left(\boldsymbol{x},\boldsymbol{W}\right)-y\right)\right)^{p}\right]$
to exist, it must be the case that, for each fixed $\boldsymbol{x}_{0}\in{\cal X}$,
either $F\left(\cdot,\boldsymbol{W}\right)$ is differentiable at
$\boldsymbol{x}_{0}$, or whenever $\boldsymbol{x}_{0}$ is in the
restriction of the set of nondifferentiability points of $F\left(\cdot,\boldsymbol{W}\right)$
to ${\cal X}$, say ${\cal C}_{\boldsymbol{W}\left(\cdot\right)}^{F}\left.\hspace{-2pt}\hspace{-2pt}\vphantom{{\cal C}}\right|_{{\cal X}}:\widehat{\Omega}\rightrightarrows\mathbb{R}^{N}$,
defined as
\begin{equation}
{\cal C}_{\boldsymbol{W}\left(\omega\right)}^{F}\left.\hspace{-2pt}\hspace{-2pt}\vphantom{{\cal C}}\right|_{{\cal X}}\triangleq\left\{ \hspace{-1pt}\left.\boldsymbol{x}\in{\cal X}\right|F\left(\cdot,\boldsymbol{W}\left(\omega\right)\right)\text{ is nondifferentiable at }\boldsymbol{x}\right\} ,\quad\omega\in\widehat{\Omega},
\end{equation}
the stationary point condition 
\begin{equation}
\nabla_{y}\left[\left({\cal R}\left(F\left(\boldsymbol{x}_{0},\boldsymbol{W}\right)-y\right)\right)^{p}\right]\equiv0
\end{equation}
is satisfied, \textit{for all} $y\in\textrm{cl}\left\{ \left(m_{l},m_{h}\right)\right\} $.
Equivalently, we demand that
\begin{equation}
\nabla\left[\left({\cal R}\left(z\right)\right)^{p}\right]\equiv0,
\end{equation}
\textit{for all} $z\in\textrm{cl}\left\{ \left(F\left(\boldsymbol{x}_{0},\boldsymbol{W}\right)-m_{h},F\left(\boldsymbol{x}_{0},\boldsymbol{W}\right)-m_{l}\right)\right\} $.
Due to convexity, nonnegativity, and monotonicity of ${\cal R}$ (conditions
${\bf S1}$, ${\bf S2}$ and ${\bf S3}$), it is not hard to see that,
for every \textit{qualifying} $\omega\in\widehat{\Omega}$ and for
every $\boldsymbol{x}_{0}\in{\cal C}_{\boldsymbol{W}\left(\omega\right)}^{F}\left.\hspace{-2pt}\hspace{-2pt}\vphantom{{\cal C}}\right|_{{\cal X}}$,
${\cal R}$ must partially be of the form 
\begin{equation}
{\cal R}\left(z\right)\equiv C_{ND},\quad\forall z\in\textrm{cl}\left\{ \left(-\infty,F\left(\boldsymbol{x}_{0},\boldsymbol{W}\left(\omega\right)\right)-m_{l}\right)\right\} ,
\end{equation}
where $0\le C_{ND}<\infty$ is some constant. Working in the same
fashion, utilizing the fact that the multifunction ${\cal C}_{\boldsymbol{W}\left(\cdot\right)}^{F}\left.\hspace{-2pt}\hspace{-2pt}\vphantom{{\cal C}}\right|_{{\cal X}}$
is countable-valued and by defining the function $F_{ND}^{*}:\widehat{\Omega}\rightarrow\left[-\infty,\infty\right]$
as
\begin{equation}
F_{ND}^{*}\left(\omega\right)\triangleq\sup\left\{ \hspace{-1pt}F\left(\boldsymbol{x},\boldsymbol{W}\left(\omega\right)\right)\in\mathbb{R}\left|\boldsymbol{x}\in{\cal C}_{\boldsymbol{W}\left(\omega\right)}^{F}\left.\hspace{-2pt}\hspace{-2pt}\vphantom{{\cal C}}\right|_{{\cal X}}\right.\right\} ,\quad\omega\in\widehat{\Omega},
\end{equation}
we may also obtain the uniform requirement
\begin{equation}
{\cal R}\left(z\right)\equiv C_{ND},\quad\forall z\in\bigcup_{\omega\in\widehat{\Omega}_{ND}}\mathrm{cl}\left\{ \left(-\infty,F_{ND}^{*}\left(\omega\right)-m_{l}\right)\right\} ,\label{eq:COMP_2}
\end{equation}
where the set $\widehat{\Omega}_{ND}\subseteq\widehat{\Omega}$ is
defined as in (\ref{eq:Omega_ND}). Therefore, if, for at least one
$\omega\in\widehat{\Omega}$, $F\left(\cdot,\boldsymbol{W}\right)$
is nondifferentiable at some $\boldsymbol{x}\in{\cal X}$, ${\cal R}$
must be \textit{partially constant}, as prescribed by (\ref{eq:COMP_2}).
\end{proof}
As implied by Proposition \ref{prop:NonDiff}, condition ${\bf W2}$
\textit{always} requires differentiability of the $p$-th power of
${\cal R}$, everywhere on $\mathfrak{R}^{\widetilde{F}}$. Additionally,
any potential nonsmoothness of $F\left(\cdot,\boldsymbol{W}\right)$
\textit{always} imposes further requirements on the structure of ${\cal R}$,
significantly restricting the allowable choices in regard to the latter.
On the contrary, this is not the case as far as Assumption \ref{assu:F_AS_Main}
is concerned, regarding the choice of ${\cal R}$. Specifically, there
are a lot of cases where, not only ${\cal R}$ and/or its powers are
allowed to exhibit corner points, but also $F\left(\cdot,\boldsymbol{W}\right)$
may be nonsmooth, as well. To show that indeed this is the case, let
us consider the following, simple example.

Let $p\equiv1$ (for simplicity), let ${\cal R}$ be \textit{any risk
regularizer}, and consider the objective function
\begin{equation}
F\left(\boldsymbol{x},\boldsymbol{W}\right)\equiv F\left(x,W\right)\triangleq\left|x-W\right|,
\end{equation}
where $W\sim{\cal N}\left(0,1\right)$. Although nonsmooth, the random
cost function $F\left(\cdot,W\right)$ is differentiable almost everywhere
relative to ${\cal P}$, at each fixed $x\in\mathbb{R}$, thus satisfying
condition ${\bf P1}$. It may also easily argued that condition ${\bf P2}$
is also satisfied, as well. Then, we are interested in the scalar-decision,
risk-averse stochastic program
\begin{equation}
\begin{array}{rl}
\underset{x}{\mathrm{minimize}} & \mathbb{E}\hspace{-1pt}\left\{ \left|x-W\right|\right\} +c\mathbb{E}\hspace{-1pt}\left\{ \vphantom{\int^{\int}}{\cal R}\left(\left|x-W\right|-\mathbb{E}\left\{ \left|x-W\right|\right\} \right)\right\} \\
\mathrm{subject\,to} & x\in{\cal X}
\end{array},\label{eq:Example_2}
\end{equation}
for some nonempty, non-singleton, convex and compact set ${\cal X}$.
Note that, for every choice of ${\cal X}$, it is true that
\begin{equation}
m_{l}\equiv0\quad\text{and}\quad m_{h}\equiv+\infty,
\end{equation}
since $W$ is unbounded. Thus, $\mathrm{cl}\left\{ \left(m_{l},m_{h}\right)\right\} \equiv\left[0,\infty\right)$.
The random subdifferential multifunction of $F\left(\cdot,W\right)$
may be expressed as
\begin{equation}
\partial F\left(x,W\right)=\begin{cases}
\left\{ 1\right\} , & \text{if }x>W\\
\left[-1,1\right], & \text{if }x\equiv W\\
\left\{ -1\right\} , & \text{if }x<W
\end{cases},
\end{equation}
and, thus, every subgradient of $F\left(\cdot,W\right)$ has the form
\begin{equation}
\underline{\nabla}F\left(x,W\right)=\mathds{1}_{\left\{ x>W\right\} }-\mathds{1}_{\left\{ x<W\right\} }+\delta\mathds{1}_{\left\{ x\equiv W\right\} },
\end{equation}
where $\delta\in\left[-1,1\right]$. Consequently, it is true that
$\left|\underline{\nabla}F\left(\cdot,W\right)\right|\le1$ \textit{uniformly
on} ${\cal X}\times\Omega$, and condition ${\bf C1}$ of Assumption
\ref{assu:F_AS_Main} is satisfied with $P\equiv\infty$. Also, due
to $W$ being integrable and ${\cal X}$ being compact, it is easy
to see that condition ${\bf C2}$ of Assumption \ref{assu:F_AS_Main}
is satisfied, as well. Let us now study condition ${\bf C3}$, related
to the choice of ${\cal R}$. For each $x\in{\cal X}$, the cost $F\left(x,W\right)$
follows a \textit{folded normal distribution} with scale $x$ and
location $1$, since $x-W\sim{\cal N}\left(x,1\right)$. On $\left[0,\infty\right)$
and for fixed $x\in{\cal X}$, the cdf of $F\left(x,W\right)$ is
given by
\begin{flalign}
F_{W}^{x}\left(y\right) & =\Phi\left(y+x\right)+\Phi\left(y-x\right)-1.
\end{flalign}
Hence, $F_{W}^{x}$ is (uniformly) Lipschitz on $\left[0,\infty\right)$,
since the Gaussian cdf $\Phi$ is Lipschitz on $\mathbb{R}$. Consequently,
by choosing $\underline{\nabla}{\cal R}\equiv{\cal R}'_{+}$, case
${\bf \left(2\right)}$ of Proposition \ref{prop:C3_Valid} implies
that the choice of ${\cal R}$ can be \textit{completely unconstrained}.

Now, let us see if problem (\ref{eq:Example_2}) is supported within
the framework set by the necessary conditions of Proposition \ref{prop:NonDiff}.
First, case ${\bf \left(1\right)}$ of Proposition \ref{prop:NonDiff}
directly implies that ${\cal R}$ \textit{must be differentiable}
\textit{on} $\mathfrak{R}^{\widetilde{F}}$. Second, case ${\bf \left(2\right)}$
of Proposition \ref{prop:NonDiff} implies that, since almost everywhere
on $\Omega,$ $F\left(\cdot,W\right)$ is not differentiable everywhere
on ${\cal X}$, \textit{${\cal R}$ must be partially constant, in
addition to the differentiability requirement}. In particular, for
every choice of a \textit{certain} event $\widehat{\Omega}$, there
exists a pair $\left(x_{0},\omega_{0}\right)\in{\cal X}\times\widehat{\Omega}$,
such that $W\left(\omega_{0}\right)\equiv x_{0}$, implying the existence
of \textit{at least} \textit{one }point of nondifferentiability of
$F\left(\cdot,W\right)$ on ${\cal X}$ (however happening with ${\cal P}$-measure
zero, since $W$ is Gaussian). This fact may be shown by the following
simple argument. Let $\widehat{\Omega}\subseteq\Omega$ be \textit{any}
event such that ${\cal P}\left(\widehat{\Omega}\right)\equiv1$, and
consider the preimage
\begin{equation}
W^{-1}\left({\cal X}\right)\triangleq\left\{ \omega\in\Omega\left|W\left(\omega\right)\in{\cal X}\right.\right\} \in\mathscr{F}.
\end{equation}
Of course, we have ${\cal P}\left(W^{-1}\left({\cal X}\right)\right)>0$.
Suppose that $\widehat{\Omega}\bigcap W^{-1}\left({\cal X}\right)$
is empty, implying that\linebreak{}
${\cal P}\left(\widehat{\Omega}\bigcap W^{-1}\left({\cal X}\right)\right)\equiv0$.
But then it would be true that
\begin{flalign}
{\cal P}\left(\widehat{\Omega}\bigcup W^{-1}\left({\cal X}\right)\right) & \equiv{\cal P}\left(\widehat{\Omega}\right)+{\cal P}\left(W^{-1}\left({\cal X}\right)\right)-{\cal P}\left(\widehat{\Omega}\bigcap W^{-1}\left({\cal X}\right)\right)\nonumber \\
 & \equiv1+{\cal P}\left(W^{-1}\left({\cal X}\right)\right)>1,
\end{flalign}
which is, of course, absurd. Therefore, the events $\widehat{\Omega}$
and $W^{-1}\left({\cal X}\right)$ must necessarily have at least
one element in common. Call this element $\omega_{0}$. Since $\omega_{0}\in W^{-1}\left({\cal X}\right)$,
there must exist some $x_{0}$ in ${\cal X}$, such that $W\left(\omega_{0}\right)\equiv x_{0}$.
Now, for every possible choice of $\widehat{\Omega}$, it is trivially
true that, if $\omega\in\widehat{\Omega}$ and $F\left(\cdot,W\left(\omega\right)\right)$
is nondifferentiable at some $x\in{\cal X}$, then $W\left(\omega\right)\equiv x$
and $F\left(x,W\left(\omega\right)\right)\equiv0$. Consequently,
with the notation of Proposition \ref{prop:NonDiff}, it follows that,
\textit{regardless of} the choice of the nonempty feasible set ${\cal X}$,
\begin{flalign}
F_{ND}^{*}\left(\omega\right) & \equiv0,\quad\forall\omega\in\widehat{\Omega},
\end{flalign}
implying that ${\cal R}$ \textit{must be constant on} $\left(-\infty,0\right]$.
This yields a major limitation of condition ${\bf W2}$. We should
also mention that, for this very simple example, even the choice ${\cal R}\left(\cdot\right)\equiv\left(\cdot\right)_{+}$,
which gives the mean-upper-semideviation risk measure, is excluded
if condition ${\bf W2}$ is imposed.

Next, let us consider condition ${\bf W3}$. In this case, the situation
is very similar to condition ${\bf W1}$. Consider the \textit{Frobenius
norm} of $\nabla\widehat{\boldsymbol{g}}_{\boldsymbol{W}}^{\widetilde{F}}$
and $\underline{\nabla}\widehat{\boldsymbol{h}}_{\boldsymbol{W}}^{\widetilde{F}}$,
respectively. Regarding the Jacobian $\nabla\widehat{\boldsymbol{g}}_{\boldsymbol{W}}^{\widetilde{F}}$,
it is true that
\begin{flalign}
\left\Vert \nabla\widehat{\boldsymbol{g}}_{\boldsymbol{W}}^{\widetilde{F}}\left(\boldsymbol{x},y\right)\right\Vert _{F}^{2} & \equiv\left\Vert \hspace{-2pt}\begin{bmatrix}{\bf 0}_{N} & \left.\nabla\left({\cal R}\left(z\right)\right)^{p}\right|_{z\equiv F\left(\boldsymbol{x},\boldsymbol{W}\right)-y}\underline{\nabla}F\left(\boldsymbol{x},\boldsymbol{W}\right)\\
1 & \left.\nabla\left({\cal R}\left(z\right)\right)^{p}\right|_{z\equiv F\left(\boldsymbol{x},\boldsymbol{W}\right)-y}
\end{bmatrix}\hspace{-2pt}\right\Vert _{F}^{2}\nonumber \\
 & \equiv1+\left(\left.\nabla\left({\cal R}\left(z\right)\right)^{p}\right|_{z\equiv F\left(\boldsymbol{x},\boldsymbol{W}\right)-y}\right)^{2}+\left(\left.\nabla\left({\cal R}\left(z\right)\right)^{p}\right|_{z\equiv F\left(\boldsymbol{x},\boldsymbol{W}\right)-y}\right)^{2}\left\Vert \underline{\nabla}F\left(\boldsymbol{x},\boldsymbol{W}\right)\right\Vert _{2}^{2}\nonumber \\
 & \le\begin{cases}
2+\left\Vert \underline{\nabla}F\left(\boldsymbol{x},\boldsymbol{W}\right)\right\Vert _{2}^{2}, & \text{if }p\equiv1\\
1+p^{2}{\cal E}^{2\left(p-1\right)}\left(1+\left\Vert \underline{\nabla}F\left(\boldsymbol{x},\boldsymbol{W}\right)\right\Vert _{2}^{2}\right), & \text{if }p>1
\end{cases},
\end{flalign}
for all $\boldsymbol{x}\in{\cal X}$ and for all $y\in\textrm{cl}\left\{ \left(m_{l},m_{h}\right)\right\} $,
as a result of Lemma \ref{lem:Iterate_Boundedness}, implying that,
as long as $\left\Vert \underline{\nabla}F\left(\boldsymbol{x},\boldsymbol{W}\right)\right\Vert _{2}$
is uniformly in ${\cal Z}_{2}$, $\left\Vert \nabla\widehat{\boldsymbol{g}}_{\boldsymbol{W}}^{\widetilde{F}}\left(\boldsymbol{x},y\right)\right\Vert _{F}$
must be uniformly in ${\cal Z}_{2}$. Since $\nabla\widehat{\boldsymbol{g}}_{\boldsymbol{W}}^{\widetilde{F}}$
is of rank at most two, we also have
\begin{equation}
\left\Vert \nabla\widehat{\boldsymbol{g}}_{\boldsymbol{W}}^{\widetilde{F}}\left(\boldsymbol{x},y\right)\right\Vert _{2}^{2}\le\left\Vert \nabla\widehat{\boldsymbol{g}}_{\boldsymbol{W}}^{\widetilde{F}}\left(\boldsymbol{x},y\right)\right\Vert _{F}^{2}\le2\left\Vert \nabla\widehat{\boldsymbol{g}}_{\boldsymbol{W}}^{\widetilde{F}}\left(\boldsymbol{x},y\right)\right\Vert _{2}^{2},
\end{equation}
which means that, if $\left\Vert \nabla\widehat{\boldsymbol{g}}_{\boldsymbol{W}}^{\widetilde{F}}\left(\boldsymbol{x},y\right)\right\Vert _{F}$
is uniformly in ${\cal Z}_{2}$, so is the spectral norm $\left\Vert \nabla\widehat{\boldsymbol{g}}_{\boldsymbol{W}}^{\widetilde{F}}\left(\boldsymbol{x},y\right)\right\Vert _{2}$
(and conversely). Similarly, the Frobenius norm of $\underline{\nabla}\widehat{\boldsymbol{h}}_{\boldsymbol{W}}^{\widetilde{F}}$
may be explicitly expressed as
\begin{flalign}
\left\Vert \underline{\nabla}\widehat{\boldsymbol{h}}_{\boldsymbol{W}}^{\widetilde{F}}\left(\boldsymbol{x}\right)\right\Vert _{F}^{2} & \equiv\left\Vert \hspace{-2pt}\left[\hspace{-2pt}\hspace{-2pt}\begin{array}{c|c}
\\
\boldsymbol{I}_{N}\hspace{-2pt} & \underline{\nabla}F\left(\boldsymbol{x},\boldsymbol{W}\right)\\
\\
\end{array}\hspace{-2pt}\hspace{-2pt}\right]\hspace{-2pt}\right\Vert _{F}^{2}=N+\left\Vert \underline{\nabla}F\left(\boldsymbol{x},\boldsymbol{W}\right)\right\Vert _{2}^{2},
\end{flalign}
and because $\underline{\nabla}\widehat{\boldsymbol{h}}_{\boldsymbol{W}}^{\widetilde{F}}$
is of rank at most $N$, it is true that

\begin{equation}
\left\Vert \underline{\nabla}\widehat{\boldsymbol{h}}_{\boldsymbol{W}}^{\widetilde{F}}\left(\boldsymbol{x}\right)\right\Vert _{2}^{2}\le\left\Vert \underline{\nabla}\widehat{\boldsymbol{h}}_{\boldsymbol{W}}^{\widetilde{F}}\left(\boldsymbol{x}\right)\right\Vert _{F}^{2}\equiv N+\left\Vert \underline{\nabla}F\left(\boldsymbol{x},\boldsymbol{W}\right)\right\Vert _{2}^{2}\le N\left\Vert \underline{\nabla}\widehat{\boldsymbol{h}}_{\boldsymbol{W}}^{\widetilde{F}}\left(\boldsymbol{x}\right)\right\Vert _{2}^{2},
\end{equation}
for all $\boldsymbol{x}\in{\cal X}$. Apparently, we get that the
spectral norm of $\underline{\nabla}\widehat{\boldsymbol{h}}_{\boldsymbol{W}}^{\widetilde{F}}$
is uniformly in ${\cal Z}_{2}$ if and only if $\left\Vert \underline{\nabla}F\left(\boldsymbol{x},\boldsymbol{W}\right)\right\Vert _{2}$
is uniformly in ${\cal Z}_{2}$, as well. This simply means that condition
${\bf C1}$ of Assumption \ref{assu:F_AS_Main} is\textit{ equivalent
to} condition ${\bf W3}$, as with the case of condition ${\bf W1}$
and condition ${\bf C2}$ of Assumption \ref{assu:F_AS_Main}, discussed
above.

We now continue with condition ${\bf W4}$. Utilizing the fact that,
for almost all $\omega\in\Omega$, $\widehat{\boldsymbol{g}}_{\boldsymbol{W}}^{\widetilde{F}}$
is differentiable everywhere on ${\cal X}\times\textrm{cl}\left\{ \left(m_{l},m_{h}\right)\right\} $,
condition ${\bf W4}$ demands that, for almost all $\omega\in\Omega$,
it is true that\renewcommand{\arraystretch}{1.2}
\begin{flalign}
 & \hspace{-2pt}\hspace{-2pt}\hspace{-2pt}\left\Vert \nabla\widehat{\boldsymbol{g}}_{\boldsymbol{W}}^{\widetilde{F}}\left(\boldsymbol{x}_{1},y_{1}\right)-\nabla\widehat{\boldsymbol{g}}_{\boldsymbol{W}}^{\widetilde{F}}\left(\boldsymbol{x}_{2},y_{2}\right)\right\Vert _{2}\nonumber \\
 & \equiv\left\Vert \hspace{-2pt}\begin{bmatrix}{\bf 0}_{N} & \nabla_{\boldsymbol{x}}\left[\left({\cal R}\left(F\left(\boldsymbol{x}_{1},\boldsymbol{W}\right)-y_{1}\right)\right)^{p}\right]\\
1 & \nabla_{y}\left[\left({\cal R}\left(F\left(\boldsymbol{x}_{1},\boldsymbol{W}\right)-y_{1}\right)\right)^{p}\right]
\end{bmatrix}\hspace{-2pt}-\hspace{-2pt}\begin{bmatrix}{\bf 0}_{N} & \nabla_{\boldsymbol{x}}\left[\left({\cal R}\left(F\left(\boldsymbol{x}_{2},\boldsymbol{W}\right)-y_{2}\right)\right)^{p}\right]\\
1 & \nabla_{y}\left[\left({\cal R}\left(F\left(\boldsymbol{x}_{2},\boldsymbol{W}\right)-y_{2}\right)\right)^{p}\right]
\end{bmatrix}\hspace{-2pt}\right\Vert _{2}\nonumber \\
 & \equiv\left\Vert \hspace{-2pt}\begin{bmatrix}{\bf 0}_{N} & \nabla_{\boldsymbol{x}}\left[\left({\cal R}\left(F\left(\boldsymbol{x}_{1},\boldsymbol{W}\right)-y_{1}\right)\right)^{p}\right]-\nabla_{\boldsymbol{x}}\left[\left({\cal R}\left(F\left(\boldsymbol{x}_{2},\boldsymbol{W}\right)-y_{2}\right)\right)^{p}\right]\\
0 & \nabla_{y}\left[\left({\cal R}\left(F\left(\boldsymbol{x}_{1},\boldsymbol{W}\right)-y_{1}\right)\right)^{p}\right]-\nabla_{y}\left[\left({\cal R}\left(F\left(\boldsymbol{x}_{2},\boldsymbol{W}\right)-y_{2}\right)\right)^{p}\right]
\end{bmatrix}\hspace{-2pt}\right\Vert _{2}\nonumber \\
 & =\left\Vert \hspace{-2pt}\begin{bmatrix}\left.\nabla\left({\cal R}\left(z\right)\right)^{p}\right|_{z\equiv F\left(\boldsymbol{x}_{1},\boldsymbol{W}\right)-y_{1}}\hspace{-2pt}\times\hspace{-2pt}\underline{\nabla}F\left(\boldsymbol{x}_{1},\boldsymbol{W}\right)-\left.\nabla\left({\cal R}\left(z\right)\right)^{p}\right|_{z\equiv F\left(\boldsymbol{x}_{2},\boldsymbol{W}\right)-y_{2}}\hspace{-2pt}\times\hspace{-2pt}\underline{\nabla}F\left(\boldsymbol{x}_{2},\boldsymbol{W}\right)\\
-\left.\nabla\left({\cal R}\left(z\right)\right)^{p}\right|_{z\equiv F\left(\boldsymbol{x}_{1},\boldsymbol{W}\right)-y_{1}}+\left.\nabla\left({\cal R}\left(z\right)\right)^{p}\right|_{z\equiv F\left(\boldsymbol{x}_{2},\boldsymbol{W}\right)-y_{2}}
\end{bmatrix}\hspace{-2pt}\right\Vert _{2}\nonumber \\
 & \le L\sqrt{\left\Vert \boldsymbol{x}_{1}-\boldsymbol{x}_{2}\right\Vert _{2}^{2}+\left|y_{1}-y_{2}\right|^{2}},\label{eq:Mengdi_2}
\end{flalign}
\renewcommand{\arraystretch}{1}for all $\left(\left[\boldsymbol{x}_{1}\,y_{1}\right],\left[\boldsymbol{x}_{2}\,y_{2}\right]\right)\in\left[{\cal X}\times\textrm{cl}\left\{ \left(m_{l},m_{h}\right)\right\} \right]^{2}$,
where $L<\infty$. Without loss of generality, let us call $\widehat{\Omega}$
the certain subset of $\Omega$, such that (\ref{eq:Mengdi_2}) is
true. As with condition ${\bf W2}$, without loss of generality, we
can take $\widehat{\Omega}\equiv\Omega_{E}$.

We compare (\ref{eq:Mengdi_2}) with the \textit{strongest variation}
of condition $\mathbf{C3}$, that is, case ${\bf \left(1\right)}$
of Proposition \ref{prop:C3_Valid}. We will see that, \textit{under
no additional assumptions}, case ${\bf \left(1\right)}$ of Proposition
\ref{prop:C3_Valid} cannot imply (\ref{eq:Mengdi_2}), by construction.
Indeed, \textit{even when} $\boldsymbol{x}_{1}\equiv\boldsymbol{x}_{2}\triangleq\boldsymbol{x}\in{\cal X}$,
we may write\renewcommand{\arraystretch}{1.2}
\begin{flalign}
 & \hspace{-2pt}\hspace{-2pt}\hspace{-2pt}\left\Vert \nabla\widehat{\boldsymbol{g}}_{\boldsymbol{W}}^{\widetilde{F}}\left(\boldsymbol{x},y_{1}\right)-\nabla\widehat{\boldsymbol{g}}_{\boldsymbol{W}}^{\widetilde{F}}\left(\boldsymbol{x},y_{2}\right)\right\Vert _{2}\nonumber \\
 & \equiv\left\Vert \hspace{-2pt}\begin{bmatrix}\left.\nabla\left({\cal R}\left(z\right)\right)^{p}\right|_{z\equiv F\left(\boldsymbol{x}_{1},\boldsymbol{W}\right)-y_{1}}\hspace{-2pt}\times\hspace{-2pt}\underline{\nabla}F\left(\boldsymbol{x},\boldsymbol{W}\right)-\left.\nabla\left({\cal R}\left(z\right)\right)^{p}\right|_{z\equiv F\left(\boldsymbol{x}_{2},\boldsymbol{W}\right)-y_{2}}\hspace{-2pt}\times\hspace{-2pt}\underline{\nabla}F\left(\boldsymbol{x},\boldsymbol{W}\right)\\
-\left.\nabla\left({\cal R}\left(z\right)\right)^{p}\right|_{z\equiv F\left(\boldsymbol{x},\boldsymbol{W}\right)-y_{1}}+\left.\nabla\left({\cal R}\left(z\right)\right)^{p}\right|_{z\equiv F\left(\boldsymbol{x},\boldsymbol{W}\right)-y_{2}}
\end{bmatrix}\hspace{-2pt}\right\Vert _{2}\nonumber \\
 & \equiv\left\Vert \hspace{-2pt}\left[\begin{array}{c}
\vphantom{{\displaystyle \int}}\underline{\nabla}F\hspace{-2pt}\left(\boldsymbol{x},\hspace{-1pt}\boldsymbol{W}\right)\\
\hline \vphantom{{\displaystyle \int}}-1
\end{array}\right]\hspace{-2pt}\hspace{-2pt}\left(\left.\nabla\left({\cal R}\left(z\right)\right)^{p}\right|_{z\equiv F\left(\boldsymbol{x},\boldsymbol{W}\right)-y_{1}}-\left.\nabla\left({\cal R}\left(z\right)\right)^{p}\right|_{z\equiv F\left(\boldsymbol{x},\boldsymbol{W}\right)-y_{2}}\right)\right\Vert _{2},\label{eq:Detail_1}
\end{flalign}
\renewcommand{\arraystretch}{1}for all $\left(y_{1},y_{2}\right)\in\left[\textrm{cl}\left\{ \left(m_{l},m_{h}\right)\right\} \right]^{2}$.
It is clear that, in order for (\ref{eq:Detail_1}) to yield a Lipschitz
inequality for the involved function, \textit{assuming that case ${\bf \left(1\right)}$
of Proposition \ref{prop:C3_Valid} is true}, it would be necessary
to impose assumptions on the size of $\nabla F\left(\cdot,\hspace{-1pt}\boldsymbol{W}\right)$.
Specifically, it is true that
\begin{flalign}
 & \hspace{-2pt}\hspace{-2pt}\hspace{-2pt}\hspace{-2pt}\hspace{-2pt}\hspace{-2pt}\left\Vert \hspace{-2pt}\left[\begin{array}{c}
\vphantom{{\displaystyle \int}}\underline{\nabla}F\hspace{-2pt}\left(\boldsymbol{x},\hspace{-1pt}\boldsymbol{W}\right)\\
\hline \vphantom{{\displaystyle \int}}-1
\end{array}\right]\hspace{-2pt}\hspace{-2pt}\left(\left.\nabla\left({\cal R}\left(z\right)\right)^{p}\right|_{z\equiv F\left(\boldsymbol{x},\boldsymbol{W}\right)-y_{1}}-\left.\nabla\left({\cal R}\left(z\right)\right)^{p}\right|_{z\equiv F\left(\boldsymbol{x},\boldsymbol{W}\right)-y_{2}}\right)\right\Vert _{2}\nonumber \\
 & \equiv\sqrt{\left\Vert \underline{\nabla}F\hspace{-2pt}\left(\boldsymbol{x},\hspace{-1pt}\boldsymbol{W}\right)\right\Vert _{2}^{2}+1}\left|\left.\nabla\left({\cal R}\left(z\right)\right)^{p}\right|_{z\equiv F\left(\boldsymbol{x},\boldsymbol{W}\right)-y_{1}}-\left.\nabla\left({\cal R}\left(z\right)\right)^{p}\right|_{z\equiv F\left(\boldsymbol{x},\boldsymbol{W}\right)-y_{2}}\right|\nonumber \\
 & \le\left(\left\Vert \underline{\nabla}F\hspace{-2pt}\left(\boldsymbol{x},\hspace{-1pt}\boldsymbol{W}\right)\right\Vert _{2}+1\right)\left|y_{1}-y_{2}\right|,
\end{flalign}
for all $\left(y_{1},y_{2}\right)\in\left[\textrm{cl}\left\{ \left(m_{l},m_{h}\right)\right\} \right]^{2}$,
demonstrating need of a bound on $\left\Vert \underline{\nabla}F\hspace{-2pt}\left(\boldsymbol{x},\hspace{-1pt}\boldsymbol{W}\right)\right\Vert _{2}$,
uniform on ${\cal X}\times\Omega'$, where $\Omega'\subseteq\Omega$
is a certain event, so that (\ref{eq:Mengdi_2}) can be verified.
Of course, such uniform boundedness assumption is not made directly
neither in Assumption \ref{assu:F_AS_Main}, nor in (\citep{Wang2018},
Assumption 2.1) (it is made\textit{ in expectation}, though). The
closest relative to our framework would be to assume that $\left\Vert \underline{\nabla}F\hspace{-2pt}\left(\boldsymbol{x},\hspace{-1pt}\boldsymbol{W}\right)\right\Vert _{2}$
is in ${\cal Z}_{\infty}$ (that is, with bounded essential supremum),
uniformly on ${\cal X}$, but this condition is too restrictive if
it is imposed \textit{together} \textit{with} assuming differentiability
of the $p$-th power of ${\cal R}$ (case ${\bf \left(1\right)}$
of Proposition \ref{prop:C3_Valid}).

On the other hand, suppose that (\ref{eq:Mengdi_2}) is true. Then,
for every $\boldsymbol{x}_{1}\equiv\boldsymbol{x}_{2}\equiv\boldsymbol{x}\in{\cal X}$
and everywhere on $\widehat{\Omega}\equiv\Omega_{E}$, it is true
that
\begin{flalign}
L\left|y_{1}-y_{2}\right| & \ge\sqrt{\left\Vert \underline{\nabla}F\hspace{-2pt}\left(\boldsymbol{x},\hspace{-1pt}\boldsymbol{W}\right)\right\Vert _{2}^{2}+1}\left|\left.\nabla\left({\cal R}\left(z\right)\right)^{p}\right|_{z\equiv F\left(\boldsymbol{x},\boldsymbol{W}\right)-y_{1}}-\left.\nabla\left({\cal R}\left(z\right)\right)^{p}\right|_{z\equiv F\left(\boldsymbol{x},\boldsymbol{W}\right)-y_{2}}\right|\nonumber \\
 & \ge\left|\left.\nabla\left({\cal R}\left(z\right)\right)^{p}\right|_{z\equiv F\left(\boldsymbol{x},\boldsymbol{W}\right)-y_{1}}-\left.\nabla\left({\cal R}\left(z\right)\right)^{p}\right|_{z\equiv F\left(\boldsymbol{x},\boldsymbol{W}\right)-y_{2}}\right|,
\end{flalign}
implying that
\begin{equation}
\hspace{-2pt}\left|\left.\nabla\left({\cal R}\left(z\right)\right)^{p}\right|_{z\equiv F\left(\boldsymbol{x},\boldsymbol{W}\right)-y_{1}}-\left.\nabla\left({\cal R}\left(z\right)\right)^{p}\right|_{z\equiv F\left(\boldsymbol{x},\boldsymbol{W}\right)-y_{2}}\right|\hspace{-2pt}\le L\left|F\left(\boldsymbol{x},\boldsymbol{W}\right)-y_{2}-\left(F\left(\boldsymbol{x},\boldsymbol{W}\right)-y_{1}\right)\right|,
\end{equation}
for all $\left(F\left(\boldsymbol{x},\boldsymbol{W}\right)-y_{1},F\left(\boldsymbol{x},\boldsymbol{W}\right)-y_{2}\right)\in\left[\mathfrak{R}^{\widetilde{F}}\right]^{2}$.
Therefore, it follows that case ${\bf \left(1\right)}$ of Proposition
\ref{prop:C3_Valid} is satisfied with $D_{{\cal R},p}\equiv L$.
This shows that \textit{condition} \textit{${\bf W4}$ is in general
stronger than the strongest assumption on the smoothness of $\left({\cal R}\left(z\right)\right)^{p}$}
\textit{considered in this paper} whatsoever.

Driven by the detailed discussion above, let us now formulate the
following proposition, which constitutes a precise statement of the
fact that the structural framework considered in this work is more
general than the one considered in \citep{Wang2017,Wang2018}. The
proof is based on the above and is omitted.
\begin{prop}
\textbf{\textup{(Structural Comparisons)\label{prop:Structural}}}
The class of mean-semideviation programs supported under Assumptions
\ref{assu:F_AS_1} and \ref{assu:F_AS_Main} contains the respective
class supported under conditions ${\bf W1}-{\bf W6}$ plus $\mathbf{\widetilde{W}7}$
(i.e., Assumption 2.1 of \citep{Wang2018} $+$ $\mathbf{\widetilde{W}7}$).
Further, the inclusion is strict.
\end{prop}

\section{\label{sec:Conclusion}Conclusion}

We have introduced the $\textit{MESSAGE}^{p}$ \textit{algorithm},
which is an efficient, data-driven compositional stochastic subgradient
procedure for iteratively solving \textit{convex} mean-semideviation
risk-averse problems \textit{to optimality}, and constitutes a \textit{parallel}
variation of the recently developed, general purpose \textit{$T$-SCGD}
\textit{algorithm} of Yang, Wang \& Fang \citep{Wang2018}. We have
proposed a flexible and structure-exploiting set of problem assumptions,
under which we have rigorously analyzed the asymptotic behavior of
the $\textit{MESSAGE}^{p}$ algorithm. Specifically:
\begin{itemize}
\item We have established pathwise convergence of the $\textit{MESSAGE}^{p}$
algorithm in a strong technical sense, confirming its asymptotic consistency. 
\item In the case of a \textit{strongly convex} \textit{cost}, we have shown
that, for fixed semideviation order $p>1$, the $\textit{MESSAGE}^{p}$
algorithm achieves a squared-${\cal L}_{2}$ solution suboptimality
rate of the order of ${\cal O}(n^{-\left(1-\epsilon\right)/2})$ iterations,
where $\epsilon\in\left[0,1\right)$ is a user-specified constant,
related to the stepsize selection. In particular, for $\epsilon>0$,
pathwise convergence of the $\textit{MESSAGE}^{p}$ algorithm is \textit{simultaneously}
guaranteed, establishing a rate of order arbitrarily close to ${\cal O}(n^{-1/2})$,
while ensuring stable pathwise operation. For $p\equiv1$, the rate
order improves to ${\cal O}(n^{-2/3})$, which also suffices for pathwise
convergence, and matches previous results. 
\item Likewise, in the general case of a \textit{convex cost}, we have shown
that, for any $\epsilon\in\left[0,1\right)$, the $\textit{MESSAGE}^{p}$
algorithm \textit{with iterate smoothing} achieves an ${\cal L}_{1}$
objective suboptimality rate of the order of ${\cal O}(n^{-\left(1-\epsilon\right)/\left(4\mathds{1}_{\left\{ p>1\right\} }+4\right)})$.
This result provides maximal rates ${\cal O}(n^{-1/4})$, if $p\equiv1$,
and ${\cal O}(n^{-1/8})$, if $p>1$, matching the state of the art,
as well.
\end{itemize}
We have also discussed the superiority of the proposed framework for
convergence, as compared to that employed earlier in \citep{Wang2018},
within the risk-averse context under consideration. First, contrary
to \citep{Wang2018}, a unique feature of our framework is that it
clearly reveals a well-defined trade-off between the expansiveness
of the random cost and the smoothness of the particular mean-semideviation
risk measure. This provides great analytical flexibility, which is
very important for practical considerations. Additionally, we have
rigorously demonstrated that the class of mean-semideviation problems
supported herein is \textit{strictly larger} than the respective class
of problems supported in \citep{Wang2018}. As a result, this work
establishes the applicability of compositional stochastic optimization
for a significantly and strictly wider spectrum of convex mean-semideviation
risk-averse problems, as compared to the state of the art. Consequently,
the purpose of our work is justified from this perspective, as well.

\section{Acknowledgements}

We would like to thank our colleague Weidong Han for his useful suggestions.

This material is based upon work supported by the U.S. Navy / SPAWAR
Systems Center Pacific under Contract No. N66001-18-C-4031. Any opinions,
findings and conclusions or recommendations expressed in this material
are those of the author(s) and do not necessarily reflect the views
of the U.S. Navy / SPAWAR Systems Center Pacific.

\section{\label{sec:Appendix:-Proofs}Appendix: Proofs}

\subsection{\label{subsec:Proof-of-CHAR}Proof of Theorem \ref{thm:IFF_RR}}

The first part of the theorem has essentially already been proved
in earlier in Section \ref{subsec:CDFA_MS} (in particular, (\ref{eq:CDF_Rep})
with $C_{S}\equiv1$ and $C_{I}\equiv0$, and for some $Y\in{\cal Z}_{1}$,
which implies that $\mathbb{E}\left\{ \left(x-Y\right)_{+}\right\} <+\infty$,
for all $x\in\mathbb{R}$), except for explicitly showing equivalence
of interpreting the involved integral in the Lebesgue and improper
Riemann senses. Therefore, in addition to this detail, it suffices
to prove the second part of the theorem (the converse). The proof,
presented below, is technical, but clean and simple.

Consider any \textit{nonconstant} risk regularizer ${\cal R}:\mathbb{R}\rightarrow\mathbb{R}$.
By definition, ${\cal R}$ is convex on $\mathbb{R}$ ($\mathbf{S1}$)
and, thus, it admits both left and right (directional) derivatives,
which are \textit{nondecreasing}, \textit{everywhere} on $\mathbb{R}$.
Let ${\cal R}'_{+}:\mathbb{R}\rightarrow\mathbb{R}$ be the right
derivative of ${\cal R}$. Because ${\cal R}'_{+}$ is nondecreasing
on $\mathbb{R}$, it exhibits an at most countable number of discontinuities,
and of the jump type. By convexity, it follows that ${\cal R}'_{+}$
is right continuous at every such point of discontinuity, as well. 

By definition of ${\cal R}'_{+}$, it is true that, for every $x\in\mathbb{R}$,
${\cal R}'_{+}\left(x\right)\in\partial{\cal R}\left(x\right)$, where
the compact-valued multifunction $\partial{\cal R}:\mathbb{R}\rightrightarrows\mathbb{R}$
denotes the subdifferential of ${\cal R}$. Therefore, for every $x\in\mathbb{R}$,
the subderivative ${\cal R}'_{+}\left(x\right)$ satisfies the defining
inequality
\begin{equation}
{\cal R}\left(y\right)-{\cal R}\left(x\right)\ge{\cal R}'_{+}\left(x\right)\left(y-x\right),\label{eq:SubDerivative}
\end{equation}
for every $y\in\mathbb{R}$. Exploiting (\ref{eq:SubDerivative}),
monotonicity of ${\cal R}$ ($\mathbf{S3}$) readily implies that
${\cal R}'_{+}\left(x\right)\ge0$, for all $x\in\mathbb{R}$, whereas,
from nonexpansiveness of ${\cal R}$ $(\mathbf{S4})$, it easily follows
that ${\cal R}'_{+}\left(x\right)\le1$, for all $x\in\mathbb{R}$.
Additionally, from nonnegativity of ${\cal R}$ ($\mathbf{S2}$),
it is true that, for every $x\in\mathbb{R}$,
\begin{equation}
{\cal R}'_{+}\left(x\right)\le\dfrac{{\cal R}\left(y\right)-{\cal R}\left(x\right)}{y-x}\le\dfrac{{\cal R}\left(y\right)}{y-x},\quad\forall y\in\left(x,+\infty\right),
\end{equation}
and, using the fact that ${\cal R}'_{+}\left(x\right)\ge0$, for all
$x\in\mathbb{R}$, we may pass to the limit as $x\rightarrow-\infty$,
yielding
\begin{equation}
0\le\limsup_{x\rightarrow-\infty}{\cal R}'_{+}\left(x\right)\le\limsup_{x\rightarrow-\infty}\dfrac{{\cal R}\left(y\right)}{y-x}\equiv0,
\end{equation}
implying that ${\cal R}'_{+}\left(x\right)\underset{x\rightarrow-\infty}{\longrightarrow}0$,
as well. On the other hand, since ${\cal R}'_{+}\left(x\right)\le1$,
for all $x\in\mathbb{R}$, it is trivial to see that $0<\sup_{x\in\mathbb{R}}{\cal R}'_{+}\left(x\right)\le1$
(for nonconstant ${\cal R}$). Consequently, the function $F_{Y}:\mathbb{R}\rightarrow\left[0,1\right]$
defined as
\begin{equation}
F_{Y}\left(x\right)\triangleq\dfrac{{\cal R}'_{+}\left(x\right)}{\sup_{x\in\mathbb{R}}{\cal R}'_{+}\left(x\right)},\quad\forall x\in\mathbb{R},
\end{equation}
qualifies as the cdf of some random variable $Y:\Omega\rightarrow\mathbb{R}$,
and we may obviously write
\begin{equation}
{\cal R}'_{+}\left(x\right)\equiv F_{Y}\left(x\right)\sup_{x\in\mathbb{R}}{\cal R}'_{+}\left(x\right),\quad\forall x\in\mathbb{R}.\label{eq:CDF_Proof}
\end{equation}

Now, we know that ${\cal R}$ is convex on $\mathbb{R}$ and, if ${\cal A}$
denotes the countable set of points where ${\cal R}$ is \textit{nondifferentiable},
its derivative exists on $\mathbb{R}\setminus{\cal A}$. Let ${\cal R}':\mathbb{R}\rightarrow\mathbb{R}$
denote this derivative, defined on the set it exists. Then, by definition,
it is true that
\begin{equation}
{\cal R}'\left(x\right)\equiv{\cal R}'_{+}\left(x\right),\quad\forall x\in\mathbb{R}\setminus{\cal A},
\end{equation}
where, of course, ${\cal A}$ is of Lebesgue measure zero. Consequently,
for every $\left(\alpha,x\right)\in\mathbb{R}^{2}$, such that $\alpha\le x$,
it follows that ${\cal R}'\equiv{\cal R}'_{+}$, almost everywhere
relative to the Lebesgue measure on $\left[\alpha,x\right]$. Also
due to convexity on $\mathbb{R}$ (say), ${\cal R}$ is absolutely
continuous on $\left[\alpha,x\right]$, for every qualifying choice
of $\alpha$ and $x$. Therefore, Lebesgue's Fundamental Theorem of
Integral Calculus (Theorems 2.3.4 \& 2.3.10 in \citep{Ash2000Probability})
implies that
\begin{equation}
{\cal R}\left(x\right)-{\cal R}\left(\alpha\right)\equiv\int_{\alpha}^{x}{\cal R}'\left(x\right)\textrm{d}y\equiv\int_{\alpha}^{x}{\cal R}'_{+}\left(y\right)\textrm{d}y,\label{eq:Riemman}
\end{equation}
where integration is interpreted in the sense of Lebesgue, relative
to the Lebesgue measure on the Borel space $\left(\mathbb{R},\mathscr{B}\left(\mathbb{R}\right)\right)$.
By monotone convergence, we may deduce that, since ${\cal R}$ is
nondecreasing and uniformly bounded from below, its limit at $-\infty$
is finite and, in particular,
\begin{equation}
{\cal R}\left(x\right)\underset{x\rightarrow-\infty}{\longrightarrow}\inf_{x\in\mathbb{R}}{\cal R}\left(x\right)\ge0.
\end{equation}
Also, for every $x\in\mathbb{R}$, (\ref{eq:Riemman}) is true for
every $\mathbb{R}\ni\alpha\le x$. Therefore, we may pass to the limit
in (\ref{eq:Riemman}) as $\alpha\rightarrow-\infty$, to obtain
\begin{equation}
{\cal R}\left(x\right)-\inf_{x\in\mathbb{R}}{\cal R}\left(x\right)\equiv\lim_{\alpha\rightarrow-\infty}\int_{\alpha}^{x}{\cal R}'_{+}\left(y\right)\textrm{d}y,\quad\forall x\in\mathbb{R}.\label{eq:Improp}
\end{equation}
Invoking Lebesgue's Monotone Convergence Theorem and via a standard
sequential argument, it follows that
\begin{flalign}
\lim_{\alpha\rightarrow-\infty}\int_{\alpha}^{x}{\cal R}'_{+}\left(y\right)\textrm{d}y & \equiv\lim_{\alpha\rightarrow-\infty}\int{\cal R}'_{+}\left(y\right)\mathds{1}_{\left[\alpha,x\right]}\left(y\right)\textrm{d}y\nonumber \\
 & =\int\lim_{\alpha\rightarrow-\infty}{\cal R}'_{+}\left(y\right)\mathds{1}_{\left[\alpha,x\right]}\left(y\right)\textrm{d}y\nonumber \\
 & \equiv\int{\cal R}'_{+}\left(y\right)\mathds{1}_{\left[-\infty,x\right]}\left(y\right)\textrm{d}y\equiv\int_{-\infty}^{x}{\cal R}'_{+}\left(y\right)\textrm{d}y,\quad\forall x\in\mathbb{R},
\end{flalign}
which, together with (\ref{eq:CDF_Proof}), further implies that
\begin{equation}
{\cal R}\left(x\right)\equiv\left(\sup_{x\in\mathbb{R}}{\cal R}'_{+}\left(x\right)\right)\int_{-\infty}^{x}F_{Y}\left(y\right)\textrm{d}y+\inf_{x\in\mathbb{R}}{\cal R}\left(x\right),\quad\forall x\in\mathbb{R}.\label{eq:Riemman_2-1}
\end{equation}
In addition to the above, Fubini's Theorem (Theorem 2.6.6 in \citep{Ash2000Probability})
implies that
\begin{equation}
+\infty>\int_{-\infty}^{x}F_{Y}\left(y\right)\textrm{d}y\equiv\mathbb{E}\left\{ \left(x-Y\right)_{+}\right\} ,
\end{equation}
for all $x\in\mathbb{R}$ and for every random variable $Y:\Omega\rightarrow\mathbb{R}$
having $F_{Y}$ as its cdf.

To show that the integral involved in (\ref{eq:Riemman_2-1}) is well
defined in the improper Riemann sense, note first that the nondecreasing
function is Riemann integrable. Therefore, the Lebesgue integral in
(\ref{eq:Riemman}) is necessarily equal to the respective Riemann
integral. Equivalently, integration in (\ref{eq:Riemman}) may be
interpreted in the Riemann sense, as well. Then, (\ref{eq:Improp})
remains true, and the limit on the RHS may be interpreted as an improper
Riemann integral, by definition. The validity of (\ref{eq:Riemman_2-1}),
where integration is in the improper Riemann sense, follows. Note
that, as far as the direct statement of Theorem \ref{thm:IFF_RR}
is concerned, equivalence of the aforementioned Lebesgue and improper
Riemann integrals may be shown in exactly the same fashion as above.

Finally, let ${\cal R}$ be constant on $\mathbb{R}$. Then, it is
trivial to see that $\sup_{x\in\mathbb{R}}{\cal R}'_{+}\left(x\right)\equiv0$,
and $\inf_{x\in\mathbb{R}}{\cal R}\left(x\right)\equiv{\cal R}\left(x\right)$,
for all $x\in\mathbb{R}$. Then, for \textit{any} random variable
$Y:\Omega\rightarrow\mathbb{R}$, such that $\mathbb{E}\left\{ \left(x-Y\right)_{+}\right\} <+\infty$,
for all $x\in\mathbb{R}$, (\ref{eq:Riemman_2-1}) is trivially true,
and, apparently, there is at least one such random variable. The result
now follows.\hfill{}\ensuremath{\blacksquare}

\subsection{\label{subsec:MUS-1-Direct-Derivation}Proof of Lemma \ref{lem:Sub_Grad}}

Certainly, because $\phi^{\widetilde{F}}$ admits the compositional
representation
\begin{equation}
\phi^{\widetilde{F}}\left(\boldsymbol{x}\right)\equiv\mathbb{E}\left\{ F\left(\boldsymbol{x},\boldsymbol{W}\right)\right\} +c\varrho\left(g^{\widetilde{F}}\left(\boldsymbol{h}^{\widetilde{F}}\left(\boldsymbol{x}\right)\right)\right),\quad\forall\boldsymbol{x}\in{\cal X},\label{eq:comp_rep_proof}
\end{equation}
it follows that $\phi^{\widetilde{F}}$ will be differentiable as
long as the functions $\mathbb{E}\left\{ F\left(\cdot,\boldsymbol{W}\right)\right\} $,
$\varrho$, $g^{\widetilde{F}}$ and $\boldsymbol{h}^{\widetilde{F}}$
are in the respective effective domains, in which case it must be
true that
\begin{equation}
\nabla\phi^{\widetilde{F}}\left(\boldsymbol{x}\right)\equiv\nabla\mathbb{E}\left\{ F\left(\boldsymbol{x},\boldsymbol{W}\right)\right\} +c\nabla\boldsymbol{h}^{\widetilde{F}}\left(\boldsymbol{x}\right)\left.\nabla g^{\widetilde{F}}\left(\boldsymbol{y}\right)\right|_{\boldsymbol{y}\equiv\boldsymbol{h}^{\widetilde{F}}\left(\boldsymbol{x}\right)}\left.\nabla\varrho\left(z\right)\right|_{z\equiv g^{\widetilde{F}}\left(\boldsymbol{h}^{\widetilde{F}}\left(\boldsymbol{x}\right)\right)},\quad\forall\boldsymbol{x}\in{\cal X},\label{eq:CHAIN_1}
\end{equation}
where $\nabla\boldsymbol{h}^{\widetilde{F}}:\mathbb{R}^{N}\rightarrow\mathbb{R}^{N\times\left(N+1\right)}$
denotes the Jacobian of $\boldsymbol{h}^{\widetilde{F}}$, $\nabla g^{\widetilde{F}}:\mathbb{R}^{N+1}\rightarrow\mathbb{R}^{N+1}$
denotes the gradient of $g^{\widetilde{F}}$, assumed to exist at
least for all $\boldsymbol{y}\in\mathrm{Graph}_{{\cal X}}\left(\mathbb{E}\left\{ F\left(\cdot,\boldsymbol{W}\right)\right\} \right)$,
and $\nabla\varrho:\mathbb{R}\rightarrow\mathbb{R}$ denotes the derivative
of $\varrho$, also assumed to be well defined at least for every
$z$ in the range of $g^{\widetilde{F}}$.

Following a bottom-up approach, we first exploit our basic assumption
that $F\left(\cdot,\boldsymbol{W}\right)$ is convex on ${\cal X}$
(at least), for every realization $\boldsymbol{W}\equiv\boldsymbol{W}\left(\omega\right),\omega\in\Omega$,
as well as property $\mathbf{P1}$. Under this setting, we may invoke
(\citep{ShapiroLectures_2ND}, Theorem 7.51) for each $\boldsymbol{x}\in{\cal X}$,
from where it follows that the function $\mathbb{E}\left\{ F\left(\cdot,\boldsymbol{W}\right)\right\} $
is differentiable \textit{everywhere} on ${\cal X}$ and that, further,
we may interchange differentiation with integration, implying that
\begin{equation}
\nabla\mathbb{E}\left\{ F\left(\boldsymbol{x},\boldsymbol{W}\right)\right\} \equiv\mathbb{E}\left\{ \underline{\nabla}F\left(\boldsymbol{x},\boldsymbol{W}\right)\right\} ,\quad\forall\boldsymbol{x}\in{\cal X}.
\end{equation}
This result directly yields the existence of the Jacobian of $\boldsymbol{h}^{\widetilde{F}}$,
given by
\begin{equation}
\nabla\boldsymbol{h}^{\widetilde{F}}\left(\boldsymbol{x}\right)\equiv\left[\hspace{-2pt}\hspace{-2pt}\begin{array}{c|c}
\\
\boldsymbol{I}_{N}\hspace{-2pt} & \mathbb{E}\left\{ \underline{\nabla}F\left(\boldsymbol{x},\boldsymbol{W}\right)\right\} \\
\\
\end{array}\hspace{-2pt}\hspace{-2pt}\right],\quad\forall\boldsymbol{x}\in{\cal X},
\end{equation}
which, of course, is the same as (\ref{eq:JAC_MUS1}), in the statement
of Lemma \ref{lem:Sub_Grad}.

Next, let us discuss differentiability of $g^{\widetilde{F}}$. We
know that, due to convexity of $F\left(\cdot,\boldsymbol{W}\right)$,
${\cal R}$ and $\left(\cdot\right)^{p}$, and because of monotonicity
of ${\cal R}$ and $\left(\cdot\right)^{p}$, the composite function
$\left({\cal R}\left(F\left(\cdot,\boldsymbol{W}\right)-\left(\bullet\right)\right)\right)^{p}$
is convex in both variables (at least separately). We would also like
to show that $\left({\cal R}\left(F\left(\cdot,\boldsymbol{W}\right)-\left(\bullet\right)\right)\right)^{p}$
is differentiable at each $\left(\boldsymbol{x},y\right)\in\mathrm{Graph}_{{\cal X}}\left(\mathbb{E}\left\{ F\left(\cdot,\boldsymbol{W}\right)\right\} \right)$,
almost everywhere relative to ${\cal P}$.

Indeed, fix an \textit{arbitrary} point $\left(\boldsymbol{x},y_{\boldsymbol{x}}\right)\equiv\left(\boldsymbol{x},\mathbb{E}\left\{ F\left(\boldsymbol{x},\boldsymbol{W}\right)\right\} \right)$
in $\mathrm{Graph}_{{\cal X}}\left(\mathbb{E}\left\{ F\left(\cdot,\boldsymbol{W}\right)\right\} \right)$.
By property $\mathbf{P1}$, we know that there is a certain event
$\mathsf{D}_{\boldsymbol{x}}\subseteq\Omega$, such that $F\left(\cdot,\boldsymbol{W}\left(\omega\right)\right)$
is differentiable at $\boldsymbol{x}\in{\cal X}$, for all $\omega\in\mathsf{D}_{\boldsymbol{x}}$.
Consequently, by the fact that the identity $\left(\bullet\right):\mathbb{R}\rightarrow\mathbb{R}$
is differentiable everywhere on $\mathbb{R}$, it follows that, for
every $\omega\in\mathsf{D}_{\boldsymbol{x}}$, the function $F\left(\cdot,\boldsymbol{W}\right)-\left(\bullet\right)$
is differentiable at $\left(\boldsymbol{x},y_{\boldsymbol{x}}\right)$.
Utilizing property $\mathbf{P2}$, for the same fixed point $\left(\boldsymbol{x},y_{\boldsymbol{x}}\right)$,
there exists another certain event $\mathsf{N}_{\boldsymbol{x}}\subseteq\Omega$,
such that, for every $\omega\in\mathsf{N}_{\boldsymbol{x}}$, $F\left(\boldsymbol{x},\boldsymbol{W}\left(\omega\right)\right)-y_{\boldsymbol{x}}\equiv F\left(\boldsymbol{x},\boldsymbol{W}\left(\omega\right)\right)-\mathbb{E}\left\{ F\left(\boldsymbol{x},\boldsymbol{W}\right)\right\} $
is \textit{not }in ${\cal A}$, the countable nullset containing the
nondifferentiability points of ${\cal R}$. Therefore, for every $\omega\in\mathsf{D}_{\boldsymbol{x}}\bigcap\mathsf{N}_{\boldsymbol{x}}$,
with ${\cal P}\left(\mathsf{D}_{\boldsymbol{x}}\bigcap\mathsf{N}_{\boldsymbol{x}}\right)\equiv1$,
$F\left(\cdot,\boldsymbol{W}\left(\omega\right)\right)-\left(\bullet\right)$
is differentiable at $\left(\boldsymbol{x},y_{\boldsymbol{x}}\right)$,
and ${\cal R}$ is differentiable at $F\left(\boldsymbol{x},\boldsymbol{W}\left(\omega\right)\right)-y_{\boldsymbol{x}}$,
implying that the composite function ${\cal R}\left(F\left(\cdot,\boldsymbol{W}\left(\omega\right)\right)-\left(\bullet\right)\right)$
is differentiable at $\left(\boldsymbol{x},y_{\boldsymbol{x}}\right)$,
as well. In other words, we have shown that the function ${\cal R}\left(F\left(\cdot,\boldsymbol{W}\left(\omega\right)\right)-\left(\bullet\right)\right)$
is differentiable at each arbitrary point $\left(\boldsymbol{x},y_{\boldsymbol{x}}\right)$
in the set $\mathrm{Graph}_{{\cal X}}\left(\mathbb{E}\left\{ F\left(\cdot,\boldsymbol{W}\right)\right\} \right)$,
for ${\cal P}-$almost every $\omega\in\Omega$. And since the function
$\left(\cdot\right)^{p}$ is differentiable everywhere on $\mathbb{R}$,
the preceding statement also holds for $\left({\cal R}\left(F\left(\cdot,\boldsymbol{W}\right)-\left(\bullet\right)\right)\right)^{p}$.

Further, let us determine the structure of the subdifferential of
$\left({\cal R}\left(F\left(\cdot,\boldsymbol{W}\right)-\left(\bullet\right)\right)\right)^{p}$.
Simply, because the functions ${\cal R}\left(F\left(\cdot,\boldsymbol{W}\right)-\left(\bullet\right)\right)$
and $\left(\cdot\right)^{p}$ are convex, with the latter being nondecreasing,
any subgradient of $\left({\cal R}\left(F\left(\cdot,\boldsymbol{W}\right)-\left(\bullet\right)\right)\right)^{p}$
may be expressed as
\begin{equation}
\underline{\nabla}\left({\cal R}\left(F\left(\boldsymbol{x},\boldsymbol{W}\right)-y\right)\right)^{p}=p\left({\cal R}\left(F\left(\boldsymbol{x},\boldsymbol{W}\right)-y\right)\right)^{p-1}\underline{\nabla}\left[{\cal R}\left(F\left(\boldsymbol{x},\boldsymbol{W}\right)-y\right)\right],
\end{equation}
for all $\left(\boldsymbol{x},y\right)\in\mathrm{Graph}_{{\cal X}}\left(\mathbb{E}\left\{ F\left(\cdot,\boldsymbol{W}\right)\right\} \right)$
(at least), where $\underline{\nabla}\left[{\cal R}\left(\cdot,\boldsymbol{W}-\left(\bullet\right)\right)\right]$
denotes \textit{any} subgradient of ${\cal R}\left(\cdot,\boldsymbol{W}-\left(\bullet\right)\right)$.
This is a direct application of the composition rule in subgradient
calculus. Likewise, another application of the composition rule to
the function ${\cal R}\left(F\left(\cdot,\boldsymbol{W}\right)-\left(\bullet\right)\right)$
yields
\begin{equation}
\underline{\nabla}\left[{\cal R}\left(F\left(\cdot,\boldsymbol{W}\right)-y\right)\right]=\underline{\nabla}{\cal R}\left(F\left(\boldsymbol{x},\boldsymbol{W}\right)-y\right)\left[\begin{array}{c}
\vphantom{{\displaystyle \int}}\underline{\nabla}F\left(\boldsymbol{x},\boldsymbol{W}\right)\\
\hline \vphantom{{\displaystyle \int}}-1
\end{array}\right]
\end{equation}
and, thus,
\begin{equation}
\underline{\nabla}\left({\cal R}\left(F\left(\boldsymbol{x},\boldsymbol{W}\right)-y\right)\right)^{p}\equiv p\left({\cal R}\left(F\left(\boldsymbol{x},\boldsymbol{W}\right)-y\right)\right)^{p-1}\underline{\nabla}{\cal R}\left(F\left(\boldsymbol{x},\boldsymbol{W}\right)-y\right)\left[\begin{array}{c}
\vphantom{{\displaystyle \int}}\underline{\nabla}F\left(\boldsymbol{x},\boldsymbol{W}\right)\\
\hline \vphantom{{\displaystyle \int}}-1
\end{array}\right],
\end{equation}
for all $\left(\boldsymbol{x},y\right)\in\mathrm{Graph}_{{\cal X}}\left(\mathbb{E}\left\{ F\left(\cdot,\boldsymbol{W}\right)\right\} \right)$.

We may now invoke (\citep{ShapiroLectures_2ND}, Theorem 7.51) for
each $\left(\boldsymbol{x},y\right)\in\mathrm{Graph}_{{\cal X}}\left(\mathbb{E}\left\{ F\left(\cdot,\boldsymbol{W}\right)\right\} \right)$,
from where we obtain that the function $g^{\widetilde{F}}$ is differentiable
\textit{everywhere} on $\mathrm{Graph}_{{\cal X}}\left(\mathbb{E}\left\{ F\left(\cdot,\boldsymbol{W}\right)\right\} \right)$
and that its gradient is given by
\begin{flalign}
\nabla g^{\widetilde{F}}\left(\boldsymbol{x},y\right) & \equiv\nabla\mathbb{E}\left\{ \left({\cal R}\left(F\left(\boldsymbol{x},\boldsymbol{W}\right)-y\right)\right)^{p}\right\} \nonumber \\
 & \equiv\mathbb{E}\left\{ \underline{\nabla}\left({\cal R}\left(F\left(\boldsymbol{x},\boldsymbol{W}\right)-y\right)\right)^{p}\right\} \nonumber \\
 & \equiv\mathbb{E}\left\{ p\left({\cal R}\left(F\left(\boldsymbol{x},\boldsymbol{W}\right)-y\right)\right)^{p-1}\underline{\nabla}{\cal R}\left(F\left(\boldsymbol{x},\boldsymbol{W}\right)-y\right)\left[\begin{array}{c}
\vphantom{{\displaystyle \int}}\underline{\nabla}F\left(\boldsymbol{x},\boldsymbol{W}\right)\\
\hline \vphantom{{\displaystyle \int}}-1
\end{array}\right]\right\} ,\label{eq:g_grad}
\end{flalign}
and we are done, since (\ref{eq:g_grad}) is the same as (\ref{eq:GRAD_MUS1}).

As far as the derivative of $\varrho$ is concerned, if $p\in\left(1,\infty\right)$
(if not, $\varrho$ is the identity), it exists everywhere on the
nonnegative semiaxis, except for the origin, and (\ref{eq:rho_MUS1})
is obviously true. Thus, from (\ref{eq:comp_rep_proof}), it is clear
that we should demand that
\begin{equation}
g^{\widetilde{F}}\left(\boldsymbol{h}^{\widetilde{F}}\left(\boldsymbol{x}\right)\right)\equiv\mathbb{E}\left\{ \left({\cal R}\left(F\left(\boldsymbol{x},\boldsymbol{W}\right)-\mathbb{E}\left\{ F\left(\boldsymbol{x},\boldsymbol{W}\right)\right\} \right)\right)^{p}\right\} >0,\quad\forall\boldsymbol{x}\in{\cal X}.
\end{equation}
Fix $\boldsymbol{x}\in{\cal X}$. Of course, because ${\cal R}$ is
nonnegative on $\mathbb{R}$, it is true that
\begin{equation}
\mathbb{E}\left\{ \left({\cal R}\left(F\left(\boldsymbol{x},\boldsymbol{W}\right)\hspace{-1pt}-\mathbb{E}\left\{ F\left(\boldsymbol{x},\boldsymbol{W}\right)\right\} \right)\right)^{p}\right\} \equiv0\iff{\cal R}\left(F\left(\boldsymbol{x},\boldsymbol{W}\right)\hspace{-1pt}-\mathbb{E}\left\{ F\left(\boldsymbol{x},\boldsymbol{W}\right)\right\} \right)\equiv0,\;{\cal P}-a.e.\label{eq:0_AE}
\end{equation}
Since, additionally, ${\cal R}$ is nondecreasing on $\mathbb{R}$,
the RHS statement of $\eqref{eq:0_AE}$ is in turn \textit{equivalent
to}
\begin{equation}
F\left(\boldsymbol{x},\boldsymbol{W}\right)-\mathbb{E}\left\{ F\left(\boldsymbol{x},\boldsymbol{W}\right)\right\} \le\sup\left\{ x\in\mathbb{R}\left|{\cal R}\left(x\right)\equiv0\right.\right\} \triangleq\kappa_{{\cal R}}\in\overline{\mathbb{R}},\quad{\cal P}-a.e.,
\end{equation}
where, in general, $\kappa_{{\cal R}}\equiv+\infty$ if and only if
${\cal R}\left(x\right)\equiv0$, for all $x\in\mathbb{R}$. However,
recall that, by assumption, ${\cal R}$ is not identically equal to
zero everywhere on $\mathbb{R}$; thus, $\kappa_{{\cal R}}\in\left[-\infty,\infty\right)$.
Finally, we have shown that
\begin{equation}
g^{\widetilde{F}}\left(\boldsymbol{h}^{\widetilde{F}}\left(\boldsymbol{x}\right)\right)\equiv0\iff{\cal P}\left(F\left(\boldsymbol{x},\boldsymbol{W}\right)-\mathbb{E}\left\{ F\left(\boldsymbol{x},\boldsymbol{W}\right)\right\} \le\kappa_{{\cal R}}\right)\equiv1,
\end{equation}
which means that, if ${\cal P}\left(F\left(\boldsymbol{x},\boldsymbol{W}\right)-\mathbb{E}\left\{ F\left(\boldsymbol{x},\boldsymbol{W}\right)\right\} \le\kappa_{{\cal R}}\right)<1$,
where $\kappa_{{\cal R}}$ is fixed in $\left[-\infty,\infty\right)$,
then
\begin{equation}
g^{\widetilde{F}}\left(\boldsymbol{h}^{\widetilde{F}}\left(\boldsymbol{x}\right)\right)\neq0\implies g^{\widetilde{F}}\left(\boldsymbol{h}^{\widetilde{F}}\left(\boldsymbol{x}\right)\right)>0.
\end{equation}
Enough said.\hfill{}\ensuremath{\blacksquare}

\subsection{\label{subsec:Some-Auxiliary-Results}Some Auxiliary Results}

In this subsection, let us state some basic, elementary results stemming
from Assumption \ref{assu:F_AS_Main}, which will be helpful in both
the further characterization of condition ${\bf C3}$ of Assumption
\ref{assu:F_AS_Main}, and the asymptotic analysis of the \textit{$\textit{MESSAGE}^{p}$}
algorithm. First, a direct, but very useful consequence of condition
$\mathbf{C1}$ is summarized in the next proposition.
\begin{lem}
\textbf{\textup{(}}$\mathbb{E}\left\{ F\left(\cdot,\boldsymbol{W}\right)\right\} $\textbf{\textup{
is Lipschitz on ${\cal X}$ \& More)\label{prop:EF_LIP}}} Let condition
$\mathbf{C1}$ of Assumption \ref{assu:F_AS_1} be in effect. Then,
the functions $F\left(\cdot,\boldsymbol{W}\right)$ and $\mathbb{E}\left\{ F\left(\cdot,\boldsymbol{W}\right)\right\} $
satisfy
\begin{flalign}
\left|\mathbb{E}\left\{ F\left(\boldsymbol{x}_{1},\hspace{-1pt}\boldsymbol{W}\right)\right\} -\mathbb{E}\left\{ F\left(\boldsymbol{x}_{2},\hspace{-1pt}\boldsymbol{W}\right)\right\} \right| & \le G\left\Vert \boldsymbol{x}_{1}-\boldsymbol{x}_{2}\right\Vert _{2}\quad\text{and}\label{eq:EF_LIP_1}\\
\mathbb{E}\left\{ \left|F\left(\boldsymbol{x}_{1},\hspace{-1pt}\boldsymbol{W}\right)-F\left(\boldsymbol{x}_{2},\hspace{-1pt}\boldsymbol{W}\right)\right|\right\}  & \le2G\left\Vert \boldsymbol{x}_{1}-\boldsymbol{x}_{2}\right\Vert _{2},\label{eq:EF_LIP_2}
\end{flalign}
for all $\left(\boldsymbol{x}_{1},\boldsymbol{x}_{2}\right)\in{\cal X}\times{\cal X}$.
\end{lem}
\begin{proof}[Proof of Lemma \ref{prop:EF_LIP}]
Proof of (\ref{eq:EF_LIP_1}) is straightforward, and omitted. To
prove (\ref{eq:EF_LIP_2}), we use the definition of a subgradient
and our assumption that the random cost function $F\left(\cdot,\hspace{-1pt}\boldsymbol{W}\right)$
is measurably subdifferentiable on ${\cal X}$. For every $\left(\boldsymbol{x}_{1},\boldsymbol{x}_{2}\right)\in{\cal X}\times{\cal X}$
and everywhere on $\Omega$, it is true that
\begin{equation}
F\left(\boldsymbol{x}_{1},\hspace{-1pt}\boldsymbol{W}\right)-F\left(\boldsymbol{x}_{2},\hspace{-1pt}\boldsymbol{W}\right)\ge\left(\boldsymbol{x}_{1}-\boldsymbol{x}_{2}\right)^{\boldsymbol{T}}\underline{\nabla}F\left(\boldsymbol{x}_{2},\hspace{-1pt}\boldsymbol{W}\right),
\end{equation}
implying that
\begin{flalign}
F\left(\boldsymbol{x}_{2},\hspace{-1pt}\boldsymbol{W}\right)-F\left(\boldsymbol{x}_{1},\hspace{-1pt}\boldsymbol{W}\right) & \le\left(\boldsymbol{x}_{2}-\boldsymbol{x}_{1}\right)^{\boldsymbol{T}}\underline{\nabla}F\left(\boldsymbol{x}_{2},\hspace{-1pt}\boldsymbol{W}\right)\nonumber \\
 & \le\left\Vert \boldsymbol{x}_{2}-\boldsymbol{x}_{1}\right\Vert _{2}\left\Vert \underline{\nabla}F\left(\boldsymbol{x}_{2},\hspace{-1pt}\boldsymbol{W}\right)\right\Vert _{2}.
\end{flalign}
Since $\boldsymbol{x}_{1}$ and $\boldsymbol{x}_{2}$ are arbitrary,
it follows by symmetry that
\begin{flalign}
-\left(F\left(\boldsymbol{x}_{2},\hspace{-1pt}\boldsymbol{W}\right)-F\left(\boldsymbol{x}_{1},\hspace{-1pt}\boldsymbol{W}\right)\right) & \le\left\Vert \boldsymbol{x}_{2}-\boldsymbol{x}_{1}\right\Vert _{2}\left\Vert \underline{\nabla}F\left(\boldsymbol{x}_{1},\hspace{-1pt}\boldsymbol{W}\right)\right\Vert _{2}.
\end{flalign}
Consequently, we may write
\begin{equation}
\left|F\left(\boldsymbol{x}_{2},\hspace{-1pt}\boldsymbol{W}\right)-F\left(\boldsymbol{x}_{1},\hspace{-1pt}\boldsymbol{W}\right)\right|\le\left\Vert \boldsymbol{x}_{2}-\boldsymbol{x}_{1}\right\Vert _{2}\left(\left\Vert \underline{\nabla}F\left(\boldsymbol{x}_{1},\hspace{-1pt}\boldsymbol{W}\right)\right\Vert _{2}+\left\Vert \underline{\nabla}F\left(\boldsymbol{x}_{2},\hspace{-1pt}\boldsymbol{W}\right)\right\Vert _{2}\right),\label{eq:ELA_1}
\end{equation}
and taking expectations on both sides of (\ref{eq:ELA_1}) yields
\begin{flalign}
\mathbb{E}\left\{ \left|F\left(\boldsymbol{x}_{2},\hspace{-1pt}\boldsymbol{W}\right)-F\left(\boldsymbol{x}_{1},\hspace{-1pt}\boldsymbol{W}\right)\right|\right\}  & \le\left\Vert \boldsymbol{x}_{2}-\boldsymbol{x}_{1}\right\Vert _{2}\mathbb{E}\left\{ \left\Vert \underline{\nabla}F\left(\boldsymbol{x}_{1},\hspace{-1pt}\boldsymbol{W}\right)\right\Vert _{2}+\left\Vert \underline{\nabla}F\left(\boldsymbol{x}_{2},\hspace{-1pt}\boldsymbol{W}\right)\right\Vert _{2}\right\} \nonumber \\
 & \equiv\left\Vert \boldsymbol{x}_{2}-\boldsymbol{x}_{1}\right\Vert _{2}\left(\mathbb{E}\left\{ \left\Vert \underline{\nabla}F\left(\boldsymbol{x}_{1},\hspace{-1pt}\boldsymbol{W}\right)\right\Vert _{2}\right\} +\mathbb{E}\left\{ \left\Vert \underline{\nabla}F\left(\boldsymbol{x}_{2},\hspace{-1pt}\boldsymbol{W}\right)\right\Vert _{2}\right\} \right)\nonumber \\
 & \le\left\Vert \boldsymbol{x}_{2}-\boldsymbol{x}_{1}\right\Vert _{2}\left(\left\Vert \vphantom{\varint}\hspace{-2pt}\left\Vert \underline{\nabla}F\left(\boldsymbol{x}_{1},\hspace{-1pt}\boldsymbol{W}\right)\right\Vert _{2}\right\Vert _{{\cal L}_{P}}+\left\Vert \vphantom{\varint}\hspace{-2pt}\left\Vert \underline{\nabla}F\left(\boldsymbol{x}_{2},\hspace{-1pt}\boldsymbol{W}\right)\right\Vert _{2}\right\Vert _{{\cal L}_{P}}\right)\nonumber \\
 & \le\left\Vert \boldsymbol{x}_{2}-\boldsymbol{x}_{1}\right\Vert _{2}2G,
\end{flalign}
where we have exploited condition $\mathbf{C1}$. The claim is proved.
\end{proof}
Second, the next result is on the boundedness of the processes generated
by the \textit{$\textit{MESSAGE}^{p}$} algorithm, when $p>1$. It
is based on condition $\mathbf{C4}$ of Assumption \ref{assu:F_AS_Main},
as well as Assumption \ref{assu:Initial-Values}.
\begin{lem}
\textbf{\textup{(Case $p>1$: Iterate Boundedness)}}\label{lem:Iterate_Boundedness}
Fix $p>1$, let condition $\mathbf{C4}$ of Assumption \ref{assu:F_AS_Main}
be in effect. Also, choose $y^{0}$, $\beta_{0}$ and $z^{0}$, $\gamma_{0}$
according to Assumption \ref{assu:Initial-Values} and suppose that
$\beta_{n}\in\left(0,1\right]$, $\gamma_{n}\in\left(0,1\right]$,
for all $n\in\mathbb{N}$. Then, the composite process $\left\{ \left(\boldsymbol{x}^{n},y^{n},z^{n}\right)\right\} _{n\in\mathbb{N}}$
generated by the $\textit{MESSAGE}^{p}$ algorithm satisfies the uniform
pointwise bounds
\begin{flalign}
y^{n+1} & \in\left[m_{l},m_{h}\right]\\
z^{n+1} & \in\left[\varepsilon^{p},{\cal E}^{p}\right]\\
{\cal R}\hspace{-1pt}\left(\hspace{-1pt}F\left(\boldsymbol{x}^{n},\hspace{-1pt}\boldsymbol{W}_{2}^{n+1}\right)-y^{n}\right) & \in\left[\varepsilon,{\cal E}\right]\quad\text{and}\\
{\cal R}\hspace{-1pt}\left(\hspace{-1pt}F\left(\boldsymbol{x}^{n},\hspace{-1pt}\boldsymbol{W}_{2}^{n+1}\right)-\mathbb{E}\left\{ F\left(\boldsymbol{x}^{n},\boldsymbol{W}'\right)\right\} \right) & \in\left[\varepsilon,{\cal E}\right],\quad\forall n\in\mathbb{N},
\end{flalign}
almost everywhere relative to ${\cal P}$, where $\boldsymbol{W}'\sim{\cal P}_{\boldsymbol{W}}$.
\end{lem}
\begin{proof}[Proof of Lemma \ref{lem:Iterate_Boundedness}]
Let us start with $\left\{ \left(y^{n+1}\right)\right\} _{n\in\mathbb{N}}$.
It is true that
\begin{flalign}
y^{1} & \equiv\left(1-\beta_{0}\right)y^{0}+\beta_{0}F\left(\boldsymbol{x}^{0},\hspace{-1pt}\boldsymbol{W}_{1}^{1}\right)\nonumber \\
 & \ge\begin{cases}
\left(1-\beta_{0}\right)m_{l}+\beta_{0}m_{l}\equiv m_{l}, & \text{if }y^{0}\in\left[m_{l},m_{h}\right]\\
\left(1-1\right)y^{0}+m_{l}\equiv m_{l}, & \text{if }\beta_{0}\equiv1
\end{cases},\quad{\cal P}-a.e.,
\end{flalign}
and the result follows trivially by induction for all $n\in\mathbb{N}^{+}$,
and the fact that $\mathbb{N}$ is countable. The procedure bounding
$y^{n+1},n\in\mathbb{N}$ from above is exactly the same. Since we
have shown that $y^{n+1}\in\left[m_{l},m_{h}\right]$, for all $n\in\mathbb{N}$,
it also readily follows that
\begin{equation}
\varepsilon\equiv{\cal R}\hspace{-1pt}\left(m_{l}-m_{h}\right)\le{\cal R}\hspace{-1pt}\left(\hspace{-1pt}F\left(\boldsymbol{x}^{n},\hspace{-1pt}\boldsymbol{W}_{2}^{n+1}\right)-y^{n}\right)\le{\cal R}\hspace{-1pt}\left(m_{h}-m_{l}\right)\equiv{\cal E},
\end{equation}
and
\begin{equation}
\varepsilon\le{\cal R}\hspace{-1pt}\left(\hspace{-1pt}F\left(\boldsymbol{x}^{n},\hspace{-1pt}\boldsymbol{W}_{2}^{n+1}\right)-\mathbb{E}\left\{ F\left(\boldsymbol{x}^{n},\boldsymbol{W}'\right)\right\} \right)\le{\cal E},
\end{equation}
as well, almost everywhere relative to ${\cal P}$. As far as $\left\{ \left(z^{n+1}\right)\right\} _{n\in\mathbb{N}}$
is concerned, we work as above, that is,
\begin{flalign}
z^{1} & \equiv\left(1-\gamma_{0}\right)z^{0}+\gamma_{0}\left({\cal R}\hspace{-1pt}\left(\hspace{-1pt}F\left(\boldsymbol{x}^{0},\hspace{-1pt}\boldsymbol{W}_{2}^{1}\right)-y^{0}\right)\right)^{p}\nonumber \\
 & \ge\begin{cases}
\left(1-\gamma_{0}\right)\varepsilon^{p}+\gamma_{0}\varepsilon^{p}\equiv\varepsilon^{p}, & \text{if }z^{0}\in\left[\varepsilon^{p},{\cal E}^{p}\right]\\
\left(1-1\right)z^{0}+\varepsilon^{p}\equiv\varepsilon^{p}, & \text{if }\gamma_{0}\equiv1
\end{cases},\quad{\cal P}-a.e.,
\end{flalign}
and then we use induction, and similarly for the case of the upper
bound.
\end{proof}
Third, another expected, but also useful consequence of condition
$\mathbf{C4}$ is on the expansiveness of the composite function $\left({\cal R}\hspace{-1pt}\left(\left(\cdot\right)-\bullet\right)\right)^{p}$,
as follows.
\begin{lem}
\textbf{\textup{(}}$\left({\cal R}\hspace{-1pt}\left(\left(\cdot\right)-\bullet\right)\right)^{p}$\textbf{\textup{
is Lipschitz)}}\label{lem:P_is_LIP} Fix $p>1$ and let condition
$\mathbf{C4}$ of Assumption \ref{assu:F_AS_Main} be in effect. Then,
it is true that
\begin{flalign}
 & \hspace{-2pt}\hspace{-2pt}\hspace{-2pt}\hspace{-2pt}\hspace{-2pt}\hspace{-2pt}\hspace{-2pt}\hspace{-2pt}\hspace{-2pt}\hspace{-2pt}\hspace{-2pt}\hspace{-2pt}\left|\left({\cal R}\hspace{-1pt}\left(\hspace{-1pt}F\left(\boldsymbol{x}_{1},\hspace{-1pt}\boldsymbol{W}\right)-y_{1}\right)\right)^{p}-\left({\cal R}\hspace{-1pt}\left(\hspace{-1pt}F\left(\boldsymbol{x}_{2},\hspace{-1pt}\boldsymbol{W}'\right)-y_{2}\right)\right)^{p}\right|\nonumber \\
 & \quad\quad\quad\quad\le{\cal E}^{p-1}p\left(\left|F\left(\boldsymbol{x}_{1},\hspace{-1pt}\boldsymbol{W}\right)-F\left(\boldsymbol{x}_{2},\hspace{-1pt}\boldsymbol{W}'\right)\right|+\left|y_{1}-y_{2}\right|\right),
\end{flalign}
almost everywhere relative to ${\cal P}$, for all $\left(\left[\boldsymbol{x}_{1}\,y_{1}\right],\left[\boldsymbol{x}_{2}\,y_{2}\right]\right)\in\left[{\cal X}\times\mathrm{cl}\left\{ \left(m_{l},m_{h}\right)\right\} \right]^{2}$,
where $\boldsymbol{W}':\Omega\rightarrow\mathbb{R}^{M}$ may be taken
as any copy of $\boldsymbol{W}$.
\end{lem}
\begin{proof}[Proof of Lemma \ref{lem:P_is_LIP}]
Simply, using a telescoping argument and due to the fact that ${\cal R}$
is nonexpansive, we proceed directly, also exploiting Lemma \ref{lem:Iterate_Boundedness}
(with generic $\boldsymbol{W}$ and $\boldsymbol{W}'$ instead of
$\boldsymbol{W}_{2}^{n+1}$), yielding the inequalities
\begin{flalign}
 & \hspace{-2pt}\hspace{-2pt}\hspace{-2pt}\hspace{-2pt}\hspace{-2pt}\hspace{-2pt}\left|\left({\cal R}\hspace{-1pt}\left(\hspace{-1pt}F\left(\boldsymbol{x}_{1},\hspace{-1pt}\boldsymbol{W}\right)-y_{1}\right)\right)^{p}-\left({\cal R}\hspace{-1pt}\left(\hspace{-1pt}F\left(\boldsymbol{x}_{2},\hspace{-1pt}\boldsymbol{W}'\right)-y_{2}\right)\right)^{p}\right|\nonumber \\
 & \le\left|F\left(\boldsymbol{x}_{1},\hspace{-1pt}\boldsymbol{W}\right)-y_{1}-F\left(\boldsymbol{x}_{2},\hspace{-1pt}\boldsymbol{W}'\right)+y_{2}\right|\sum_{j\in\mathbb{N}_{p-1}}\left({\cal R}\left(\hspace{-1pt}F\hspace{-2pt}\left(\boldsymbol{x}_{1},\boldsymbol{W}\right)-y_{1}\right)\right)^{j}\left({\cal R}\left(F\hspace{-2pt}\left(\boldsymbol{x}_{2},\boldsymbol{W}'\right)-y_{2}\right)\right)^{p-1-j}\nonumber \\
 & \le\left(\left|F\left(\boldsymbol{x}_{1},\hspace{-1pt}\boldsymbol{W}\right)-F\left(\boldsymbol{x}_{2},\hspace{-1pt}\boldsymbol{W}'\right)\right|+\left|y_{1}-y_{2}\right|\right)\sum_{j\in\mathbb{N}_{p-1}}{\cal E}^{j}{\cal E}^{p-1-j}\nonumber \\
 & \equiv\left(\left|F\left(\boldsymbol{x}_{1},\hspace{-1pt}\boldsymbol{W}\right)-F\left(\boldsymbol{x}_{2},\hspace{-1pt}\boldsymbol{W}'\right)\right|+\left|y_{1}-y_{2}\right|\right){\cal E}^{p-1}p,\quad{\cal P}-a.e.,
\end{flalign}
for all $\left(\left[\boldsymbol{x}_{1}\,y_{1}\right],\left[\boldsymbol{x}_{2}\,y_{2}\right]\right)\in\left[{\cal X}\times\mathrm{cl}\left\{ \left(m_{l},m_{h}\right)\right\} \right]^{2}$.
\end{proof}
\begin{rem}
Observe that, since $\boldsymbol{W}'$ may be taken as \textit{any}
copy of $\boldsymbol{W}$ in Lemma \ref{lem:P_is_LIP}, the choice
$\boldsymbol{W}'\equiv\boldsymbol{W}$ is also perfectly valid.\hfill{}\ensuremath{\blacksquare}
\end{rem}

\subsection{\label{subsec:Proof-of-Proposition_5}Proof of Proposition \ref{prop:C3_Valid}}

To show case ${\bf \left(1\right)}$ of the first part of the result,
simply observe that, by assumption, $\underline{\nabla}{\cal R}\equiv\nabla{\cal R}$.
Thus, for every qualifying choice of $Q$, for every $\boldsymbol{x}\in{\cal X}$
and for every $\left(y_{1},y_{2}\right)\in\left[\mathrm{cl}\left\{ \left(m_{l},m_{h}\right)\right\} \right]^{2}$,
we may write
\begin{flalign}
\left\Vert \vphantom{\varint}\hspace{-2pt}\left|\left.\underline{\nabla}\left({\cal R}\left(z\right)\right)^{p}\right|_{z\equiv F\left(\boldsymbol{x},\hspace{-1pt}\boldsymbol{W}\right)-y_{1}}-\left.\underline{\nabla}\left({\cal R}\left(z\right)\right)^{p}\right|_{z\equiv F\left(\boldsymbol{x},\hspace{-1pt}\boldsymbol{W}\right)-y_{2}}\right|\right\Vert _{{\cal L}_{Q}} & \le\left\Vert \vphantom{\varint}D_{{\cal R},p}\left|y_{1}-y_{2}\right|\right\Vert _{{\cal L}_{Q}}\nonumber \\
 & \equiv D_{{\cal R},p}\left|y_{1}-y_{2}\right|,
\end{flalign}
and we are done.

Cases ${\bf \left(2\right)}$ and ${\bf \left(3\right)}$ of the result
will be based on the cdf-based representation of risk regularizers
(Theorem \ref{thm:IFF_RR}). Without loss of generality, assume that
${\cal R}$ is nonconstant. If it is, the problem is trivial. For
nonconstant ${\cal R}$, Theorem \ref{thm:IFF_RR} implies the existence
of a random variable $Y:\Omega\rightarrow\mathbb{R}$, with $\mathbb{E}\left\{ \left(x-Y\right)_{+}\right\} <\infty$,
with cdf $F_{Y}:\mathbb{R}\rightarrow\left[0,1\right]$, and of a
constant $C_{S}\in\left(0,1\right]$, such that
\begin{equation}
{\cal R}'_{+}\left(x\right)\equiv C_{S}F_{Y}\left(x\right),\quad\forall x\in\mathbb{R}.
\end{equation}
Of course, the random variable $Y$ may be taken as independent of
$F\left(\boldsymbol{x},\hspace{-1pt}\boldsymbol{W}\right)$, for all
$\boldsymbol{x}\in{\cal X}$. First, \textit{whenever} $p>2$, we
have, for every $\boldsymbol{x}\in{\cal X}$ and for every $\left(y_{1},y_{2}\right)\in\left[\mathrm{cl}\left\{ \left(m_{l},m_{h}\right)\right\} \right]^{2}$,
\begin{flalign}
 & \hspace{-2pt}\hspace{-2pt}\hspace{-2pt}\left\Vert \vphantom{\varint}\hspace{-2pt}\left|\left.\underline{\nabla}\left({\cal R}\left(z\right)\right)^{p}\right|_{z\equiv F\left(\boldsymbol{x},\hspace{-1pt}\boldsymbol{W}\right)-y_{1}}-\left.\underline{\nabla}\left({\cal R}\left(z\right)\right)^{p}\right|_{z\equiv F\left(\boldsymbol{x},\hspace{-1pt}\boldsymbol{W}\right)-y_{2}}\right|\right\Vert _{{\cal L}_{1}}\nonumber \\
 & \le p\left\Vert \vphantom{\varint}\hspace{-2pt}\left|\left({\cal R}\left(F\left(\boldsymbol{x},\hspace{-1pt}\boldsymbol{W}\right)\hspace{-2pt}-\hspace{-2pt}y_{1}\right)\right)^{p-1}\underline{\nabla}{\cal R}\left(F\left(\boldsymbol{x},\hspace{-1pt}\boldsymbol{W}\right)\hspace{-2pt}-\hspace{-2pt}y_{1}\right)-\left({\cal R}\left(F\left(\boldsymbol{x},\hspace{-1pt}\boldsymbol{W}\right)\hspace{-2pt}-\hspace{-2pt}y_{2}\right)\right)^{p-1}\underline{\nabla}{\cal R}\left(F\left(\boldsymbol{x},\hspace{-1pt}\boldsymbol{W}\right)\hspace{-2pt}-\hspace{-2pt}y_{2}\right)\right|\right\Vert _{{\cal L}_{1}}\nonumber \\
 & \le p\left\Vert \vphantom{\varint}\hspace{-2pt}\left({\cal R}\left(F\left(\boldsymbol{x},\hspace{-1pt}\boldsymbol{W}\right)-y_{1}\right)\right)^{p-1}\left|\underline{\nabla}{\cal R}\left(F\left(\boldsymbol{x},\hspace{-1pt}\boldsymbol{W}\right)-y_{1}\right)-\underline{\nabla}{\cal R}\left(F\left(\boldsymbol{x},\hspace{-1pt}\boldsymbol{W}\right)-y_{2}\right)\right|\right\Vert _{{\cal L}_{1}}\nonumber \\
 & \quad\quad+p\left\Vert \vphantom{\varint}\underline{\nabla}{\cal R}\left(F\left(\boldsymbol{x},\hspace{-1pt}\boldsymbol{W}\right)-y_{2}\right)\left|\left({\cal R}\left(F\left(\boldsymbol{x},\hspace{-1pt}\boldsymbol{W}\right)-y_{1}\right)\right)^{p-1}-\left({\cal R}\left(F\left(\boldsymbol{x},\hspace{-1pt}\boldsymbol{W}\right)-y_{2}\right)\right)^{p-1}\right|\right\Vert _{{\cal L}_{1}}\label{eq:PropC3_1}\\
 & \le p{\cal E}^{p-1}\left\Vert \vphantom{\varint}\hspace{-2pt}\left|\underline{\nabla}{\cal R}\left(F\left(\boldsymbol{x},\hspace{-1pt}\boldsymbol{W}\right)-y_{1}\right)-\underline{\nabla}{\cal R}\left(F\left(\boldsymbol{x},\hspace{-1pt}\boldsymbol{W}\right)-y_{2}\right)\right|\right\Vert _{{\cal L}_{1}}+p{\cal E}^{p-2}\left(p-1\right)\left|y_{1}-y_{2}\right|,\label{eq:PropC3_2}
\end{flalign}
where (\ref{eq:PropC3_1}) follows by the triangle inequality and
(\ref{eq:PropC3_2}) follows from Lemma \ref{lem:P_is_LIP}. Similarly,
for $p\equiv2$, we get
\begin{flalign}
 & \hspace{-2pt}\hspace{-2pt}\hspace{-2pt}\left\Vert \vphantom{\varint}\hspace{-2pt}\left|\left.\underline{\nabla}\left({\cal R}\left(z\right)\right)^{2}\right|_{z\equiv F\left(\boldsymbol{x},\hspace{-1pt}\boldsymbol{W}\right)-y_{1}}-\left.\underline{\nabla}\left({\cal R}\left(z\right)\right)^{2}\right|_{z\equiv F\left(\boldsymbol{x},\hspace{-1pt}\boldsymbol{W}\right)-y_{2}}\right|\right\Vert _{{\cal L}_{1}}\nonumber \\
 & \le2{\cal E}\left\Vert \vphantom{\varint}\hspace{-2pt}\left|\underline{\nabla}{\cal R}\left(F\left(\boldsymbol{x},\hspace{-1pt}\boldsymbol{W}\right)-y_{1}\right)-\underline{\nabla}{\cal R}\left(F\left(\boldsymbol{x},\hspace{-1pt}\boldsymbol{W}\right)-y_{2}\right)\right|\right\Vert _{{\cal L}_{1}}+2\left|y_{1}-y_{2}\right|,\label{eq:PropC3_2-1}
\end{flalign}
whereas, for $p\equiv1$, no further derivation is needed. As far
the involved ${\cal L}_{1}$-norm is concerned, since $\underline{\nabla}{\cal R}\equiv{\cal R}'_{+}$
by assumption, we may write, for every $\boldsymbol{x}\in{\cal X}$
and for every $\left(y_{1},y_{2}\right)\in\left[\mathrm{cl}\left\{ \left(m_{l},m_{h}\right)\right\} \right]^{2}$,
\begin{flalign}
 & \hspace{-2pt}\hspace{-2pt}\hspace{-2pt}\hspace{-2pt}\hspace{-2pt}\hspace{-2pt}\hspace{-2pt}\hspace{-2pt}\hspace{-2pt}\left\Vert \vphantom{\varint}\hspace{-2pt}\left|\underline{\nabla}{\cal R}\left(F\left(\boldsymbol{x},\hspace{-1pt}\boldsymbol{W}\right)-y_{1}\right)-\underline{\nabla}{\cal R}\left(F\left(\boldsymbol{x},\hspace{-1pt}\boldsymbol{W}\right)-y_{2}\right)\right|\right\Vert _{{\cal L}_{1}}\nonumber \\
 & \equiv\mathbb{E}\left\{ \left|\underline{\nabla}{\cal R}\left(F\left(\boldsymbol{x},\hspace{-1pt}\boldsymbol{W}\right)-y_{1}\right)-\underline{\nabla}{\cal R}\left(F\left(\boldsymbol{x},\hspace{-1pt}\boldsymbol{W}\right)-y_{2}\right)\right|\right\} \nonumber \\
 & =C_{S}\mathbb{E}\left\{ \left|F_{Y}\left(F\left(\boldsymbol{x},\hspace{-1pt}\boldsymbol{W}\right)-y_{1}\right)-F_{Y}\left(F\left(\boldsymbol{x},\hspace{-1pt}\boldsymbol{W}\right)-y_{2}\right)\right|\right\} \\
 & =C_{S}\mathbb{E}\left\{ \left(F_{Y}\left(F\left(\boldsymbol{x},\hspace{-1pt}\boldsymbol{W}\right)-\min\left\{ y_{1},y_{2}\right\} \right)-F_{Y}\left(F\left(\boldsymbol{x},\hspace{-1pt}\boldsymbol{W}\right)-\max\left\{ y_{1},y_{2}\right\} \right)\right)\right\} \nonumber \\
 & \equiv C_{S}\mathbb{E}\left\{ {\cal P}\left(\left.F\left(\boldsymbol{x},\hspace{-1pt}\boldsymbol{W}\right)-\max\left\{ y_{1},y_{2}\right\} <Y\le F\left(\boldsymbol{x},\hspace{-1pt}\boldsymbol{W}\right)-\min\left\{ y_{1},y_{2}\right\} \right|\boldsymbol{W}\right)\right\} \nonumber \\
 & \equiv C_{S}\mathbb{E}\left\{ \int\mathds{1}_{\left(F\left(\boldsymbol{x},\hspace{-1pt}\boldsymbol{W}\right)-\max\left\{ y_{1},y_{2}\right\} ,F\left(\boldsymbol{x},\hspace{-1pt}\boldsymbol{W}\right)-\min\left\{ y_{1},y_{2}\right\} \right]}\left(y\right)\text{d}{\cal P}_{Y}\left(y\right)\right\} \nonumber \\
 & =C_{S}\int\mathbb{E}\left\{ \mathds{1}_{\left(F\left(\boldsymbol{x},\hspace{-1pt}\boldsymbol{W}\right)-\max\left\{ y_{1},y_{2}\right\} ,F\left(\boldsymbol{x},\hspace{-1pt}\boldsymbol{W}\right)-\min\left\{ y_{1},y_{2}\right\} \right]}\left(y\right)\right\} \text{d}{\cal P}_{Y}\left(y\right)\label{eq:FUB}\\
 & \equiv C_{S}\int\mathbb{E}\left\{ \mathds{1}_{\left[y+\min\left\{ y_{1},y_{2}\right\} ,y+\max\left\{ y_{1},y_{2}\right\} \right)}\left(F\left(\boldsymbol{x},\hspace{-1pt}\boldsymbol{W}\right)\right)\right\} \text{d}{\cal P}_{Y}\left(y\right)\nonumber \\
 & \equiv C_{S}\int{\cal P}\left(y+\min\left\{ y_{1},y_{2}\right\} \le F\left(\boldsymbol{x},\hspace{-1pt}\boldsymbol{W}\right)<y+\max\left\{ y_{1},y_{2}\right\} \right)\text{d}{\cal P}_{Y}\left(y\right),
\end{flalign}
where (\ref{eq:FUB}) follows from Fubini's Theorem (the involved
double integral is always finite) on the product measure space $\left(\mathbb{R}\times\mathbb{R},\mathscr{B}\left(\mathbb{R}\right)\otimes\mathscr{B}\left(\mathbb{R}\right),{\cal P}_{Y}\times{\cal P}_{\boldsymbol{W}}^{\boldsymbol{x}}\right)$,
with ${\cal P}_{\boldsymbol{W}}^{\boldsymbol{x}}$ denoting the Borel
measure inducing $F_{\boldsymbol{W}}^{\boldsymbol{x}}$. Exploiting
the assumed \textit{continuity} of $F_{\boldsymbol{W}}^{\boldsymbol{x}}$
(Lipschitz or not), we also have, for every $\boldsymbol{x}\in{\cal X}$,
\begin{flalign}
 & \hspace{-2pt}\hspace{-2pt}\hspace{-2pt}\hspace{-2pt}\hspace{-2pt}\hspace{-2pt}\hspace{-2pt}\hspace{-2pt}\hspace{-2pt}\left\Vert \vphantom{\varint}\hspace{-2pt}\left|\underline{\nabla}{\cal R}\left(F\left(\boldsymbol{x},\hspace{-1pt}\boldsymbol{W}\right)-y_{1}\right)-\underline{\nabla}{\cal R}\left(F\left(\boldsymbol{x},\hspace{-1pt}\boldsymbol{W}\right)-y_{2}\right)\right|\right\Vert _{{\cal L}_{1}}\nonumber \\
 & \equiv C_{S}\int{\cal P}\left(y+\min\left\{ y_{1},y_{2}\right\} <F\left(\boldsymbol{x},\hspace{-1pt}\boldsymbol{W}\right)\le y+\max\left\{ y_{1},y_{2}\right\} \right)\text{d}{\cal P}_{Y}\left(y\right)\nonumber \\
 & \equiv C_{S}\int F_{\boldsymbol{W}}^{\boldsymbol{x}}\left(y+\max\left\{ y_{1},y_{2}\right\} \right)-F_{\boldsymbol{W}}^{\boldsymbol{x}}\left(y+\min\left\{ y_{1},y_{2}\right\} \right)\text{d}{\cal P}_{Y}\left(y\right)\nonumber \\
 & =C_{S}\int\left|F_{\boldsymbol{W}}^{\boldsymbol{x}}\left(y+y_{1}\right)-F_{\boldsymbol{W}}^{\boldsymbol{x}}\left(y+y_{2}\right)\right|\text{d}{\cal P}_{Y}\left(y\right),\label{eq:pre_result}
\end{flalign}
for all $\left(y_{1},y_{2}\right)\in\left[\mathrm{cl}\left\{ \left(m_{l},m_{h}\right)\right\} \right]^{2}$.
If the Lipschitz condition of case ${\bf \left(2\right)}$ is true,
we further have
\begin{flalign}
\left\Vert \vphantom{\varint}\hspace{-2pt}\left|\underline{\nabla}{\cal R}\left(F\left(\boldsymbol{x},\hspace{-1pt}\boldsymbol{W}\right)-y_{1}\right)-\underline{\nabla}{\cal R}\left(F\left(\boldsymbol{x},\hspace{-1pt}\boldsymbol{W}\right)-y_{2}\right)\right|\right\Vert _{{\cal L}_{1}} & \le C_{S}D_{\widetilde{F}}\int\left|y_{1}-y_{2}\right|\text{d}{\cal P}_{Y}\left(y\right)\nonumber \\
 & \equiv C_{S}D_{\widetilde{F}}\left|y_{1}-y_{2}\right|,\label{eq:Pre_1}
\end{flalign}
for all $\left(y_{1},y_{2}\right)\in\left[\mathrm{cl}\left\{ \left(m_{l},m_{h}\right)\right\} \right]^{2}$,
showing that condition $\mathbf{C3}$ is satisfied with
\begin{equation}
D\triangleq\begin{cases}
p{\cal E}^{p-1}C_{S}D_{\widetilde{F}}+p{\cal E}^{p-2}\left(p-1\right), & \text{if }p>2\\
2{\cal E}C_{S}D_{\widetilde{F}}+2, & \text{if }p\equiv2\\
C_{S}D_{\widetilde{F}}, & \text{if }p\equiv1
\end{cases}.
\end{equation}
by taking the supremum of (\ref{eq:Pre_1}) relative to $\boldsymbol{x}$
over ${\cal X}$. If the Lipschitz-in-Expectation condition of case
${\bf \left(3\right)}$ is true, we obtain the desired result of Proposition
\ref{prop:C3_Valid} in exactly the same fashion. In particular, when
$p\equiv1$, the equivalence in case ${\bf \left(3\right)}$ of Proposition
\ref{prop:C3_Valid} follows directly by (\ref{eq:pre_result}), and
the fact that $C_{S}\neq0$. Enough said.\hfill{}\ensuremath{\blacksquare}

\subsection{\label{subsec:RM_1}Proof of Lemma \ref{lem:INTER_1}}

The proof is simple, though somewhat tedious; essentially, it is an
exercise on using the triangle and Cauchy-Schwarz inequalities. First,
observe that, under Assumption \ref{assu:F_AS_Main}, it is true that
\begin{equation}
\sup_{\boldsymbol{x}\in{\cal X}}\left\Vert \hspace{-2pt}\vphantom{\varint}\left\Vert \underline{\nabla}F\left(\boldsymbol{x},\hspace{-1pt}\boldsymbol{W}\right)\right\Vert _{2}\right\Vert _{{\cal L}_{2}}\le\sup_{\boldsymbol{x}\in{\cal X}}\left\Vert \hspace{-2pt}\vphantom{\varint}\left\Vert \underline{\nabla}F\left(\boldsymbol{x},\hspace{-1pt}\boldsymbol{W}\right)\right\Vert _{2}\right\Vert _{{\cal L}_{P}}\le G<\infty,
\end{equation}
for $P\in\left[2,\infty\right]$, due to condition ${\bf C1}$.

\textit{Fix} $n\in\mathbb{N}$ and let $p>1$. Under Assumption \ref{assu:F_AS_Main},
by nonexpansiveness of the projection operator onto the closed and
convex set ${\cal X}$, and by the triangle inequality, we have
\begin{flalign}
\left\Vert \boldsymbol{x}^{n+1}\hspace{-1pt}-\hspace{-1pt}\boldsymbol{x}^{n}\right\Vert _{2} & \equiv\left\Vert \Pi_{{\cal X}}\hspace{-2pt}\left\{ \boldsymbol{x}^{n}\hspace{-2pt}-\hspace{-2pt}\alpha_{n}\hspace{-2pt}\widehat{\nabla}^{n+1}\phi^{\widetilde{F}}\left(\boldsymbol{x}^{n},y^{n},z^{n}\right)\hspace{-2pt}\right\} \hspace{-1pt}\hspace{-1pt}-\hspace{-1pt}\Pi_{{\cal X}}\hspace{-2pt}\left\{ \boldsymbol{x}^{n}\right\} \right\Vert _{2}\nonumber \\
 & \le\left\Vert \alpha_{n}\hspace{-2pt}\widehat{\nabla}^{n+1}\phi^{\widetilde{F}}\left(\boldsymbol{x}^{n},y^{n},z^{n}\right)\right\Vert _{2}\nonumber \\
 & \equiv\alpha_{n}\left\Vert \underline{\nabla}F\hspace{-2pt}\left(\boldsymbol{x}^{n},\hspace{-1pt}\boldsymbol{W}_{2}^{n+1}\right)\hspace{-1pt}\hspace{-2pt}+\hspace{-2pt}c\Delta^{n+1}\left(\boldsymbol{x}^{n},y^{n},z^{n}\right)\right\Vert _{2}\nonumber \\
 & \le\alpha_{n}\left\Vert \underline{\nabla}F\hspace{-2pt}\left(\boldsymbol{x}^{n},\hspace{-1pt}\boldsymbol{W}_{2}^{n+1}\right)\hspace{-1pt}\right\Vert _{2}+\alpha_{n}c\left\Vert \Delta^{n+1}\left(\boldsymbol{x}^{n},y^{n},z^{n}\right)\right\Vert _{2},
\end{flalign}
where, by Lemmata \ref{lem:Iterate_Boundedness} and \ref{lem:P_is_LIP},
\begin{flalign}
\left\Vert \Delta^{n+1}\left(\boldsymbol{x}^{n},y^{n},z^{n}\right)\right\Vert _{2} & \equiv\left\Vert \left(z^{n}\right)^{\left(1-p\right)/p}\left(\underline{\nabla}F\hspace{-2pt}\left(\boldsymbol{x}^{n},\hspace{-1pt}\boldsymbol{W}_{2}^{n+1}\right)-\underline{\nabla}F\hspace{-2pt}\left(\boldsymbol{x}^{n},\hspace{-1pt}\boldsymbol{W}_{1}^{n+1}\right)\right)\right.\nonumber \\
 & \quad\quad\times\left.\underline{\nabla}{\cal R}\hspace{-2pt}\left(F\hspace{-2pt}\left(\boldsymbol{x}^{n},\hspace{-1pt}\boldsymbol{W}_{2}^{n+1}\right)\hspace{-2pt}-y^{n}\right)\hspace{-2pt}\hspace{-2pt}\left({\cal R}\hspace{-2pt}\left(F\hspace{-2pt}\left(\boldsymbol{x}^{n},\hspace{-1pt}\boldsymbol{W}_{2}^{n+1}\right)\hspace{-2pt}-y^{n}\right)\right)^{p-1}\right\Vert _{2}\nonumber \\
 & \equiv\left(\dfrac{{\cal E}}{\varepsilon}\right)^{\hspace{-1pt}p-1}\left\Vert \left(\underline{\nabla}F\hspace{-2pt}\left(\boldsymbol{x}^{n},\hspace{-1pt}\boldsymbol{W}_{2}^{n+1}\right)-\underline{\nabla}F\hspace{-2pt}\left(\boldsymbol{x}^{n},\hspace{-1pt}\boldsymbol{W}_{1}^{n+1}\right)\right)\right\Vert _{2}\nonumber \\
 & \le\mathsf{R}_{p}\left(\left\Vert \underline{\nabla}F\hspace{-2pt}\left(\boldsymbol{x}^{n},\hspace{-1pt}\boldsymbol{W}_{2}^{n+1}\right)\right\Vert _{2}+\left\Vert \underline{\nabla}F\hspace{-2pt}\left(\boldsymbol{x}^{n},\hspace{-1pt}\boldsymbol{W}_{1}^{n+1}\right)\right\Vert _{2}\right),
\end{flalign}
almost everywhere relative to ${\cal P}$, where we have defined $\mathsf{R}_{p}\triangleq\left({\cal E}\varepsilon^{-1}\right)^{\hspace{-1pt}p-1}$.
Consequently, we may bound the $\ell_{2}$-norm of $\boldsymbol{x}^{n+1}\hspace{-1pt}-\hspace{-1pt}\boldsymbol{x}^{n}$
from above as
\begin{equation}
\left\Vert \boldsymbol{x}^{n+1}\hspace{-1pt}-\hspace{-1pt}\boldsymbol{x}^{n}\right\Vert _{2}\le\alpha_{n}\left(1+c\mathsf{R}_{p}\right)\left\Vert \underline{\nabla}F\hspace{-2pt}\left(\boldsymbol{x}^{n},\hspace{-1pt}\boldsymbol{W}_{2}^{n+1}\right)\hspace{-1pt}\right\Vert _{2}+\alpha_{n}c\left(\dfrac{{\cal E}}{\varepsilon}\right)^{\hspace{-1pt}p-1}\left\Vert \underline{\nabla}F\hspace{-2pt}\left(\boldsymbol{x}^{n},\hspace{-1pt}\boldsymbol{W}_{1}^{n+1}\right)\right\Vert _{2},
\end{equation}
almost everywhere relative to ${\cal P}$. This, of course, implies
that
\begin{flalign}
 & \hspace{-2pt}\hspace{-2pt}\hspace{-2pt}\mathbb{E}_{\mathscr{D}^{n}}\left\{ \left\Vert \boldsymbol{x}^{n+1}\hspace{-1pt}-\hspace{-1pt}\boldsymbol{x}^{n}\right\Vert _{2}^{2}\hspace{-1pt}\right\} \nonumber \\
 & \le\alpha_{n}^{2}\mathbb{E}_{\mathscr{D}^{n}}\left\{ \hspace{-1pt}\hspace{-1pt}\left(\hspace{-2pt}\left(1\hspace{-2pt}+c\mathsf{R}_{p}\right)\hspace{-1pt}\hspace{-2pt}\left\Vert \underline{\nabla}F\hspace{-2pt}\left(\boldsymbol{x}^{n},\hspace{-1pt}\boldsymbol{W}_{2}^{n+1}\right)\hspace{-1pt}\right\Vert _{2}+c\mathsf{R}_{p}\left\Vert \underline{\nabla}F\hspace{-2pt}\left(\boldsymbol{x}^{n},\hspace{-1pt}\boldsymbol{W}_{1}^{n+1}\right)\right\Vert _{2}\right)^{2}\hspace{-1pt}\right\} ,\label{eq:INTER_1_1}
\end{flalign}
almost everywhere relative to ${\cal P}$, as well. Let us focus more
closely on the conditional expectation on the RHS of (\ref{eq:INTER_1_1}).
First, by the substitution rule for conditional expectations, which
is guaranteed to be valid in all our discussions in this paper, due
to the existence of regular conditional distributions on Borel spaces
(see, for instance, \citep{Durrett2010probability}), it is true that
\begin{flalign}
 & \hspace{-2pt}\hspace{-2pt}\hspace{-2pt}\mathbb{E}_{\mathscr{D}^{n}}\left\{ \hspace{-2pt}\hspace{-1pt}\left(\hspace{-2pt}\left(1+c\mathsf{R}_{p}\right)\hspace{-1pt}\hspace{-2pt}\left\Vert \underline{\nabla}F\hspace{-2pt}\left(\boldsymbol{x}^{n},\hspace{-1pt}\boldsymbol{W}_{2}^{n+1}\right)\hspace{-1pt}\right\Vert _{2}+c\mathsf{R}_{p}\left\Vert \underline{\nabla}F\hspace{-2pt}\left(\boldsymbol{x}^{n},\hspace{-1pt}\boldsymbol{W}_{1}^{n+1}\right)\right\Vert _{2}\right)^{2}\hspace{-1pt}\right\} \nonumber \\
 & \equiv\hspace{-2pt}\left.\mathbb{E}\left\{ \hspace{-2pt}\left(\hspace{-2pt}\left(1+c\mathsf{R}_{p}\right)\hspace{-1pt}\hspace{-2pt}\left\Vert \underline{\nabla}F\hspace{-2pt}\left(\boldsymbol{x},\hspace{-1pt}\boldsymbol{W}_{2}^{n+1}\right)\hspace{-1pt}\right\Vert _{2}+c\mathsf{R}_{p}\left\Vert \underline{\nabla}F\hspace{-2pt}\left(\boldsymbol{x},\hspace{-1pt}\boldsymbol{W}_{1}^{n+1}\right)\right\Vert _{2}\right)^{2}\right\} \right|_{\boldsymbol{x}\equiv\boldsymbol{x}^{n}},\label{eq:INTER_1_2}
\end{flalign}
almost everywhere relative to ${\cal P}$, where, in the RHS of (\ref{eq:INTER_1_2}),
expectation is with respect to the product measure ${\cal P}_{\boldsymbol{W}}\times{\cal P}_{\boldsymbol{W}}$
on the Borel measurable space $\left(\mathbb{R}^{M}\times\mathbb{R}^{M},\mathscr{B}\left(\mathbb{R}^{M}\right)\otimes\mathscr{B}\left(\mathbb{R}^{M}\right)\right)$.
This due to mutual independence of $\boldsymbol{W}_{1}^{n+1}$ and
$\boldsymbol{W}_{2}^{n+1}$, and also their independence relative
to $\mathscr{D}^{n}$. Then, by the triangle inequality of the ${\cal L}_{2}$-norm
on the aforementioned product probability space, we may write, for
$\boldsymbol{x}\in{\cal X}$,
\begin{flalign}
 & \hspace{-2pt}\hspace{-2pt}\hspace{-2pt}\hspace{-2pt}\hspace{-2pt}\hspace{-2pt}\hspace{-2pt}\hspace{-2pt}\hspace{-2pt}\sqrt{\mathbb{E}\left\{ \hspace{-2pt}\left(\hspace{-2pt}\left(1+c\mathsf{R}_{p}\right)\hspace{-1pt}\hspace{-2pt}\left\Vert \underline{\nabla}F\hspace{-2pt}\left(\boldsymbol{x},\hspace{-1pt}\boldsymbol{W}_{2}^{n+1}\right)\hspace{-1pt}\right\Vert _{2}+c\mathsf{R}_{p}\left\Vert \underline{\nabla}F\hspace{-2pt}\left(\boldsymbol{x},\hspace{-1pt}\boldsymbol{W}_{1}^{n+1}\right)\right\Vert _{2}\right)^{2}\right\} }\nonumber \\
 & \quad\quad\equiv\hspace{-2pt}\left\Vert \hspace{-2pt}\vphantom{\int}\left(1+c\mathsf{R}_{p}\right)\hspace{-1pt}\hspace{-2pt}\left\Vert \underline{\nabla}F\hspace{-2pt}\left(\boldsymbol{x},\hspace{-1pt}\boldsymbol{W}_{2}^{n+1}\right)\hspace{-1pt}\right\Vert _{2}+c\mathsf{R}_{p}\left\Vert \underline{\nabla}F\hspace{-2pt}\left(\boldsymbol{x},\hspace{-1pt}\boldsymbol{W}_{1}^{n+1}\right)\right\Vert _{2}\right\Vert _{{\cal L}_{2}}\nonumber \\
 & \quad\quad\le\hspace{-2pt}\left(1+c\mathsf{R}_{p}\right)\left\Vert \hspace{-2pt}\vphantom{\int}\left\Vert \underline{\nabla}F\hspace{-2pt}\left(\boldsymbol{x},\hspace{-1pt}\boldsymbol{W}_{2}^{n+1}\right)\right\Vert _{2}\right\Vert _{{\cal L}_{2}}\hspace{-2pt}+c\mathsf{R}_{p}\left\Vert \hspace{-2pt}\vphantom{\int}\left\Vert \underline{\nabla}F\hspace{-2pt}\left(\boldsymbol{x},\hspace{-1pt}\boldsymbol{W}_{1}^{n+1}\right)\right\Vert _{2}\right\Vert _{{\cal L}_{2}}\nonumber \\
 & \quad\quad\le\hspace{-2pt}\left(2c\mathsf{R}_{p}+1\right)G,
\end{flalign}
or, by taking squares on both sides,
\begin{equation}
\mathbb{E}\left\{ \hspace{-2pt}\left(\hspace{-2pt}\left(1+c\mathsf{R}_{p}\right)\hspace{-1pt}\hspace{-2pt}\left\Vert \underline{\nabla}F\hspace{-2pt}\left(\boldsymbol{x},\hspace{-1pt}\boldsymbol{W}_{2}^{n+1}\right)\hspace{-1pt}\right\Vert _{2}+c\mathsf{R}_{p}\left\Vert \underline{\nabla}F\hspace{-2pt}\left(\boldsymbol{x},\hspace{-1pt}\boldsymbol{W}_{1}^{n+1}\right)\hspace{-1pt}\right\Vert _{2}\right)^{2}\right\} \hspace{-1pt}\le\hspace{-1pt}\left(2c\hspace{-1pt}+\hspace{-1pt}1\right)^{2}G^{2},
\end{equation}
almost everywhere relative to ${\cal P}$. Thus, it follows that
\begin{equation}
\mathbb{E}_{\mathscr{D}^{n}}\left\{ \left\Vert \boldsymbol{x}^{n+1}\hspace{-1pt}-\hspace{-1pt}\boldsymbol{x}^{n}\right\Vert _{2}^{2}\hspace{-1pt}\right\} \le\alpha_{n}^{2}\left(2c\mathsf{R}_{p}+1\right)^{2}G^{2},\label{eq:INTER_1_3}
\end{equation}
almost everywhere relative to ${\cal P}$. In case $p\equiv1$, it
may be easily shown that the respective bound may be recovered by
setting $p\equiv1$ in (\ref{eq:INTER_1_3}) (pretending that ${\cal E}$
and $\varepsilon$ are finite). Finally, note that, for every value
of $p$, (\ref{eq:INTER_1_3}) holds for each $n\in\mathbb{N}$, and
$\mathbb{N}$ is, of course, countable. Enough said.\hfill{}\ensuremath{\blacksquare}

\subsection{\label{subsec:RM_0}Proof of Lemma \ref{lem:INTER_1-1}}

\textit{Fix} $n\in\mathbb{N}$, and let $y^{n}\hspace{-1pt}\hspace{-1pt}-\hspace{-1pt}{\cal S}^{\widetilde{F}}\hspace{-1pt}\hspace{-1pt}\left(\boldsymbol{x}^{n}\right)\triangleq E_{{\cal S}}^{n}$,
for brevity. Then, we may write
\begin{align}
\left|E_{{\cal S}}^{n+1}\right|^{2} & \hspace{-2pt}\equiv\hspace{-2pt}\left|\left(1-\beta_{n}\right)y^{n}+\beta_{n}F\hspace{-2pt}\left(\boldsymbol{x}^{n},\boldsymbol{W}_{1}^{n+1}\right)\hspace{-1pt}\hspace{-1pt}-\hspace{-1pt}{\cal S}^{\widetilde{F}}\hspace{-1pt}\hspace{-1pt}\left(\boldsymbol{x}^{n+1}\right)\right|^{2}\nonumber \\
 & \hspace{-2pt}\equiv\hspace{-2pt}\left|\left(1-\beta_{n}\right)\left(y^{n}-{\cal S}^{\widetilde{F}}\hspace{-1pt}\hspace{-1pt}\left(\boldsymbol{x}^{n}\right)\right)+\beta_{n}\left(F\hspace{-2pt}\left(\boldsymbol{x}^{n},\boldsymbol{W}_{1}^{n+1}\right)\hspace{-1pt}\hspace{-1pt}-{\cal S}^{\widetilde{F}}\hspace{-1pt}\hspace{-1pt}\left(\boldsymbol{x}^{n}\right)\right)+{\cal S}^{\widetilde{F}}\hspace{-1pt}\hspace{-1pt}\left(\boldsymbol{x}^{n}\right)-\hspace{-1pt}{\cal S}^{\widetilde{F}}\hspace{-1pt}\hspace{-1pt}\left(\boldsymbol{x}^{n+1}\right)\right|^{2}\nonumber \\
 & \hspace{-2pt}\equiv\hspace{-2pt}\left|\left(1-\beta_{n}\right)E_{{\cal S}}^{n}+\beta_{n}\left(F\hspace{-2pt}\left(\boldsymbol{x}^{n},\boldsymbol{W}_{1}^{n+1}\right)\hspace{-1pt}\hspace{-1pt}-{\cal S}^{\widetilde{F}}\hspace{-1pt}\hspace{-1pt}\left(\boldsymbol{x}^{n}\right)\right)+{\cal S}^{\widetilde{F}}\hspace{-1pt}\hspace{-1pt}\left(\boldsymbol{x}^{n}\right)-\hspace{-1pt}{\cal S}^{\widetilde{F}}\hspace{-1pt}\hspace{-1pt}\left(\boldsymbol{x}^{n+1}\right)\right|^{2}\nonumber \\
 & \hspace{-2pt}\le\hspace{-2pt}\left(1+\beta_{n}\right)\hspace{-2pt}\left|\left(1-\beta_{n}\right)E_{{\cal S}}^{n}+\beta_{n}\hspace{-2pt}\left(F\hspace{-2pt}\left(\boldsymbol{x}^{n},\boldsymbol{W}_{1}^{n+1}\right)\hspace{-1pt}\hspace{-1pt}-{\cal S}^{\widetilde{F}}\hspace{-1pt}\hspace{-1pt}\left(\boldsymbol{x}^{n}\right)\right)\right|^{2}\nonumber \\
 & \quad+\hspace{-2pt}\left(1+\beta_{n}^{-1}\right)\hspace{-2pt}\hspace{-1pt}\left|{\cal S}^{\widetilde{F}}\hspace{-1pt}\hspace{-1pt}\left(\boldsymbol{x}^{n}\right)\hspace{-1pt}\hspace{-1pt}-\hspace{-1pt}{\cal S}^{\widetilde{F}}\hspace{-1pt}\hspace{-1pt}\left(\boldsymbol{x}^{n+1}\right)\right|^{2}\nonumber \\
 & \hspace{-2pt}\equiv\hspace{-2pt}\left(1+\beta_{n}\right)\left(1-\beta_{n}\right)^{2}\left|E_{{\cal S}}^{n}\right|^{2}+\left(1+\beta_{n}\right)\beta_{n}^{2}\left|F\hspace{-2pt}\left(\boldsymbol{x}^{n},\boldsymbol{W}_{1}^{n+1}\right)\hspace{-1pt}\hspace{-1pt}-{\cal S}^{\widetilde{F}}\hspace{-1pt}\hspace{-1pt}\left(\boldsymbol{x}^{n}\right)\right|^{2}\nonumber \\
 & \quad+2\left(1-\beta_{n}^{2}\right)\beta_{n}E_{{\cal S}}^{n}\left(F\hspace{-2pt}\left(\boldsymbol{x}^{n},\boldsymbol{W}_{1}^{n+1}\right)\hspace{-1pt}\hspace{-1pt}-{\cal S}^{\widetilde{F}}\hspace{-1pt}\hspace{-1pt}\left(\boldsymbol{x}^{n}\right)\right)+\left(1+\beta_{n}^{-1}\right)\left|{\cal S}^{\widetilde{F}}\hspace{-1pt}\hspace{-1pt}\left(\boldsymbol{x}^{n}\right)-\hspace{-1pt}{\cal S}^{\widetilde{F}}\hspace{-1pt}\hspace{-1pt}\left(\boldsymbol{x}^{n+1}\right)\right|^{2}\nonumber \\
 & \hspace{-2pt}\le\hspace{-2pt}\left(1-\beta_{n}\right)\left|E_{{\cal S}}^{n}\right|^{2}+2\beta_{n}^{2}\left|F\hspace{-2pt}\left(\boldsymbol{x}^{n},\boldsymbol{W}_{1}^{n+1}\right)\hspace{-1pt}\hspace{-1pt}-{\cal S}^{\widetilde{F}}\hspace{-1pt}\hspace{-1pt}\left(\boldsymbol{x}^{n}\right)\right|^{2}\nonumber \\
 & \quad+2\left(1+\beta_{n}\right)\left(1-\beta_{n}\right)\beta_{n}E_{{\cal S}}^{n}\left(F\hspace{-2pt}\left(\boldsymbol{x}^{n},\boldsymbol{W}_{1}^{n+1}\right)\hspace{-1pt}\hspace{-1pt}-{\cal S}^{\widetilde{F}}\hspace{-1pt}\hspace{-1pt}\left(\boldsymbol{x}^{n}\right)\right)+2\beta_{n}^{-1}G^{2}\left\Vert \boldsymbol{x}^{n+1}\hspace{-1pt}\hspace{-1pt}-\hspace{-1pt}\boldsymbol{x}^{n}\right\Vert _{2}^{2},
\end{align}
where we have used our assumption that $\beta_{n}\le1$. Taking expectations
relative to $\mathscr{D}^{n}$ on both sides, we have
\begin{equation}
\mathbb{E}_{\mathscr{D}^{n}}\hspace{-2pt}\left\{ \left|E_{{\cal S}}^{n+1}\right|^{2}\right\} \le\left(1-\beta_{n}\right)\left|E_{{\cal S}}^{n}\right|^{2}+\beta_{n}^{2}2V+0+\beta_{n}^{-1}2G^{2}\mathbb{E}_{\mathscr{D}^{n}}\hspace{-2pt}\left\{ \left\Vert \boldsymbol{x}^{n+1}\hspace{-1pt}\hspace{-1pt}-\hspace{-1pt}\boldsymbol{x}^{n}\right\Vert _{2}^{2}\right\} ,
\end{equation}
almost everywhere relative to ${\cal P}$. The fact that $\mathbb{N}$
is countable completes the proof.\hfill{}\ensuremath{\blacksquare}

\subsection{\label{subsec:RM_2}Proof of Lemma \ref{lem:INTER_Middle}}

\textit{Fix} $n\in\mathbb{N}$. By adding and subtracting appropriate
terms as in the proof of Lemma \ref{lem:INTER_1-1} above, it is then
easy to show that the difference $z^{n}-{\cal D}^{\widetilde{F}}\left(\boldsymbol{x}^{n},y^{n}\right)\triangleq E_{{\cal D}}^{n}$
may be expressed as
\begin{flalign}
E_{{\cal D}}^{n+1} & \equiv\left(1-\gamma_{n}\right)E_{{\cal D}}^{n}+\left({\cal D}^{\widetilde{F}}\left(\boldsymbol{x}^{n},y^{n}\right)-{\cal D}^{\widetilde{F}}\left(\boldsymbol{x}^{n+1},y^{n+1}\right)\right)\nonumber \\
 & \quad\quad\quad\quad+\gamma_{n}\left(\left({\cal R}\left(F\hspace{-2pt}\left(\boldsymbol{x}^{n},\hspace{-1pt}\boldsymbol{W}_{2}^{n+1}\right)-y^{n}\right)\right)^{p}-{\cal D}^{\widetilde{F}}\left(\boldsymbol{x}^{n},y^{n}\right)\right).
\end{flalign}
Let us consider the quantity $\left|E_{{\cal D}}^{n}\right|^{2}$.
We may expand the square one time, yielding
\begin{flalign}
\left|E_{{\cal D}}^{n+1}\right|^{2} & \le\left(1+\gamma_{n}\right)\left|\left(1-\gamma_{n}\right)E_{{\cal D}}^{n}+\gamma_{n}\left(\left({\cal R}\left(F\hspace{-2pt}\left(\boldsymbol{x}^{n},\hspace{-1pt}\boldsymbol{W}_{2}^{n+1}\right)-y^{n}\right)\right)^{p}-{\cal D}^{\widetilde{F}}\left(\boldsymbol{x}^{n},y^{n}\right)\right)\right|^{2}\nonumber \\
 & \quad\quad+\left(1+\gamma_{n}^{-1}\right)\left|{\cal D}^{\widetilde{F}}\left(\boldsymbol{x}^{n},y^{n}\right)-{\cal D}^{\widetilde{F}}\left(\boldsymbol{x}^{n+1},y^{n+1}\right)\right|^{2}\nonumber \\
 & \equiv\left(1+\gamma_{n}\right)\left(1-\gamma_{n}\right)^{2}\left|E_{{\cal D}}^{n}\right|^{2}+\left(1+\gamma_{n}\right)\gamma_{n}^{2}\left|\left({\cal R}\left(F\hspace{-2pt}\left(\boldsymbol{x}^{n},\hspace{-1pt}\boldsymbol{W}_{2}^{n+1}\right)-y^{n}\right)\right)^{p}-{\cal D}^{\widetilde{F}}\left(\boldsymbol{x}^{n},y^{n}\right)\right|^{2}\nonumber \\
 & \quad\quad+2\left(1-\gamma_{n}^{2}\right)\gamma_{n}E_{{\cal D}}^{n}\left(\left({\cal R}\left(F\hspace{-2pt}\left(\boldsymbol{x}^{n},\hspace{-1pt}\boldsymbol{W}_{2}^{n+1}\right)-y^{n}\right)\right)^{p}-{\cal D}^{\widetilde{F}}\left(\boldsymbol{x}^{n},y^{n}\right)\right)\nonumber \\
 & \quad\quad\quad\quad+\left(1+\gamma_{n}^{-1}\right)\left|{\cal D}^{\widetilde{F}}\left(\boldsymbol{x}^{n},y^{n}\right)-{\cal D}^{\widetilde{F}}\left(\boldsymbol{x}^{n+1},y^{n+1}\right)\right|^{2}
\end{flalign}
As in the proof of Lemma \ref{lem:INTER_1-1}, taking conditional
expectations relative to $\mathscr{D}^{n}$ on both sides and since
$\gamma_{n}\le1$, we get
\begin{align}
\mathbb{E}_{\mathscr{D}^{n}}\left\{ \left|E_{{\cal D}}^{n+1}\right|^{2}\right\}  & \le\left(1-\gamma_{n}\right)\left|E_{{\cal D}}^{n}\right|^{2}+2\gamma_{n}^{2}\mathbb{E}_{\mathscr{D}^{n}}\left\{ \left|\left({\cal R}\left(F\hspace{-2pt}\left(\boldsymbol{x}^{n},\hspace{-1pt}\boldsymbol{W}_{2}^{n+1}\right)-y^{n}\right)\right)^{p}-{\cal D}^{\widetilde{F}}\left(\boldsymbol{x}^{n},y^{n}\right)\right|^{2}\right\} \nonumber \\
 & \quad\quad+2\gamma_{n}^{-1}\mathbb{E}_{\mathscr{D}^{n}}\left\{ \left|{\cal D}^{\widetilde{F}}\left(\boldsymbol{x}^{n},y^{n}\right)-{\cal D}^{\widetilde{F}}\left(\boldsymbol{x}^{n+1},y^{n+1}\right)\right|^{2}\right\} +0.\label{eq:blah_blah}
\end{align}
Next, we consider the last two nonzero terms of the RHS of (\ref{eq:blah_blah})
separately. First, we may write
\begin{flalign}
 & \hspace{-2pt}\hspace{-2pt}\hspace{-2pt}\hspace{-2pt}\hspace{-2pt}\hspace{-2pt}\hspace{-2pt}\hspace{-2pt}\hspace{-2pt}\hspace{-2pt}\hspace{-2pt}\hspace{-2pt}\hspace{-2pt}\hspace{-2pt}\hspace{-2pt}\mathbb{E}_{\mathscr{D}^{n}}\left\{ \left|\left({\cal R}\left(F\hspace{-2pt}\left(\boldsymbol{x}^{n},\hspace{-1pt}\boldsymbol{W}_{2}^{n+1}\right)-y^{n}\right)\right)^{p}-{\cal D}^{\widetilde{F}}\left(\boldsymbol{x}^{n},y^{n}\right)\right|^{2}\right\} \nonumber \\
 & \le\mathbb{E}_{\mathscr{D}^{n}}\left\{ \left({\cal R}\left(F\hspace{-2pt}\left(\boldsymbol{x}^{n},\hspace{-1pt}\boldsymbol{W}_{2}^{n+1}\right)-y^{n}\right)\right)^{2p}\right\} +\mathbb{E}_{\mathscr{D}^{n}}\left\{ \left({\cal D}^{\widetilde{F}}\left(\boldsymbol{x}^{n},y^{n}\right)\right)^{2}\right\} \nonumber \\
 & \le2{\cal E}^{2p},\quad{\cal P}-a.e..\label{eq:MyGod_0-1}
\end{flalign}
Second, observe that, by Lemmata \ref{lem:P_is_LIP} and \ref{prop:EF_LIP},
we get
\begin{flalign}
 & \hspace{-2pt}\hspace{-2pt}\hspace{-2pt}\hspace{-2pt}\hspace{-2pt}\hspace{-2pt}\hspace{-2pt}\hspace{-2pt}\hspace{-2pt}\left|{\cal D}^{\widetilde{F}}\hspace{-2pt}\left(\boldsymbol{x}^{n+1},y^{n+1}\right)\hspace{-2pt}-{\cal D}^{\widetilde{F}}\hspace{-2pt}\left(\boldsymbol{x}^{n},y^{n}\right)\right|\nonumber \\
 & \equiv\left|\mathbb{E}_{\mathscr{D}^{n+1}}\hspace{-2pt}\left\{ \left({\cal R}\left(F\hspace{-2pt}\left(\boldsymbol{x}^{n+1},\hspace{-1pt}\boldsymbol{W}'\right)\hspace{-2pt}-\hspace{-2pt}y^{n+1}\right)\right)^{p}-\left({\cal R}\left(F\hspace{-2pt}\left(\boldsymbol{x}^{n},\hspace{-1pt}\boldsymbol{W}'\right)\hspace{-2pt}-\hspace{-2pt}y^{n}\right)\right)^{p}\right\} \right|\nonumber \\
 & \le\mathbb{E}_{\mathscr{D}^{n+1}}\hspace{-2pt}\left\{ \left|\left({\cal R}\left(F\hspace{-2pt}\left(\boldsymbol{x}^{n+1},\hspace{-1pt}\boldsymbol{W}'\right)\hspace{-2pt}-\hspace{-2pt}y^{n+1}\right)\right)^{p}-\left({\cal R}\left(F\hspace{-2pt}\left(\boldsymbol{x}^{n},\hspace{-1pt}\boldsymbol{W}'\right)\hspace{-2pt}-\hspace{-2pt}y^{n}\right)\right)^{p}\right|\right\} \nonumber \\
 & \le{\cal E}^{p-1}p\left(\mathbb{E}_{\mathscr{D}^{n+1}}\hspace{-2pt}\left\{ \left|F\hspace{-2pt}\left(\boldsymbol{x}^{n+1},\hspace{-1pt}\boldsymbol{W}'\right)\hspace{-2pt}-F\hspace{-2pt}\left(\boldsymbol{x}^{n},\hspace{-1pt}\boldsymbol{W}'\right)\right|\right\} +\left|y^{n+1}-y^{n}\right|\right)\nonumber \\
 & \le{\cal E}^{p-1}p\left(2G\left\Vert \boldsymbol{x}^{n+1}-\boldsymbol{x}^{n}\right\Vert _{2}+\beta_{n}\left|F\hspace{-2pt}\left(\boldsymbol{x}^{n},\hspace{-1pt}\boldsymbol{W}_{1}^{n+1}\right)-y^{n}\right|\right)\nonumber \\
 & \le{\cal E}^{p-1}p\left(2G\left\Vert \boldsymbol{x}^{n+1}-\boldsymbol{x}^{n}\right\Vert _{2}+\beta_{n}\left(m_{h}-m_{l}\right)\right)\nonumber \\
 & \equiv2G{\cal E}^{p-1}p\left\Vert \boldsymbol{x}^{n+1}-\boldsymbol{x}^{n}\right\Vert _{2}+\beta_{n}{\cal E}^{p-1}p\left(m_{h}-m_{l}\right),\quad{\cal P}-a.e.
\end{flalign}
Additionally, it is true that
\begin{equation}
\left|{\cal D}^{\widetilde{F}}\hspace{-2pt}\left(\boldsymbol{x}^{n+1},y^{n+1}\right)\hspace{-2pt}-{\cal D}^{\widetilde{F}}\hspace{-2pt}\left(\boldsymbol{x}^{n},y^{n}\right)\right|^{2}\le8G^{2}{\cal E}^{2p-2}p^{2}\left\Vert \boldsymbol{x}^{n+1}-\boldsymbol{x}^{n}\right\Vert _{2}^{2}+\beta_{n}^{2}2{\cal E}^{2p-2}p^{2}\left(m_{h}-m_{l}\right)^{2}.\label{eq:MyGod_2-1}
\end{equation}
Combining (\ref{eq:MyGod_2-1}), (\ref{eq:MyGod_0-1}) and (\ref{eq:blah_blah}),
we end up with the inequality
\begin{flalign}
\mathbb{E}_{\mathscr{D}^{n}}\left\{ \left|E_{{\cal D}}^{n+1}\right|^{2}\right\}  & =\left(1-\gamma_{n}\right)\left|E_{{\cal D}}^{n}\right|^{2}\nonumber \\
 & \hspace{-2pt}\hspace{-2pt}\hspace{-2pt}\hspace{-2pt}\hspace{-2pt}\hspace{-2pt}\hspace{-2pt}\hspace{-2pt}\hspace{-2pt}\hspace{-2pt}\hspace{-2pt}\hspace{-2pt}\hspace{-2pt}\hspace{-2pt}\hspace{-2pt}\hspace{-2pt}\hspace{-2pt}\hspace{-2pt}\hspace{-2pt}\hspace{-2pt}\hspace{-2pt}+\gamma_{n}^{-1}16G^{2}{\cal E}^{2p-2}p^{2}\mathbb{E}_{\mathscr{D}^{n}}\left\{ \left\Vert \boldsymbol{x}^{n+1}-\boldsymbol{x}^{n}\right\Vert _{2}^{2}\right\} +\beta_{n}^{2}\gamma_{n}^{-1}4{\cal E}^{2p-2}p^{2}\left(m_{h}-m_{l}\right)^{2}+\gamma_{n}^{2}4{\cal E}^{2p},
\end{flalign}
being valid almost everywhere relative to ${\cal P}$. But $\mathbb{N}$
is countable.\hfill{}\ensuremath{\blacksquare}

\subsection{\label{subsec:RM_1-1-1}Proof of Lemma \ref{lem:INTER_1-1-1}}

\textit{As usual, fix} $n\in\mathbb{N}^{+}$, and let $p>1$. Nonexpansiveness
of the projection operator onto ${\cal X}$ yields
\begin{flalign}
\left\Vert \boldsymbol{x}^{n+1}\hspace{-1pt}-\hspace{-1pt}\boldsymbol{x}^{*}\right\Vert _{2}^{2} & \equiv\left\Vert \Pi_{{\cal X}}\hspace{-2pt}\left\{ \boldsymbol{x}^{n}\hspace{-2pt}-\hspace{-2pt}\alpha_{n}\hspace{-2pt}\widehat{\nabla}^{n+1}\phi^{\widetilde{F}}\left(\boldsymbol{x}^{n},y^{n},z^{n}\right)\hspace{-2pt}\right\} \hspace{-1pt}\hspace{-1pt}-\hspace{-1pt}\Pi_{{\cal X}}\hspace{-2pt}\left\{ \boldsymbol{x}^{*}\right\} \right\Vert _{2}^{2}\nonumber \\
 & \le\left\Vert \boldsymbol{x}^{n}-\boldsymbol{x}^{*}-\alpha_{n}\hspace{-2pt}\widehat{\nabla}^{n+1}\phi^{\widetilde{F}}\left(\boldsymbol{x}^{n},y^{n},z^{n}\right)\right\Vert _{2}^{2}\nonumber \\
 & =\left\Vert \boldsymbol{x}^{n}-\boldsymbol{x}^{*}\right\Vert _{2}^{2}+\alpha_{n}^{2}\left\Vert \underline{\nabla}F\hspace{-2pt}\left(\boldsymbol{x}^{n},\hspace{-1pt}\boldsymbol{W}_{2}^{n+1}\right)\hspace{-2pt}\hspace{-1pt}+\hspace{-2pt}c\Delta^{n+1}\left(\boldsymbol{x}^{n},y^{n},z^{n}\right)\right\Vert _{2}^{2}\nonumber \\
 & \quad\quad-2\alpha_{n}\left(\boldsymbol{x}^{n}-\boldsymbol{x}^{*}\right)^{\boldsymbol{T}}\left(\underline{\nabla}F\hspace{-2pt}\left(\boldsymbol{x}^{n},\hspace{-1pt}\boldsymbol{W}_{2}^{n+1}\right)\hspace{-2pt}\hspace{-1pt}+\hspace{-2pt}c\Delta^{n+1}\hspace{-1pt}\hspace{-1pt}\left(\boldsymbol{x}^{n},{\cal S}^{\widetilde{F}}\hspace{-1pt}\hspace{-1pt}\left(\boldsymbol{x}^{n}\right),{\cal D}^{\widetilde{F}}\hspace{-1pt}\hspace{-1pt}\left(\boldsymbol{x}^{n}\right)\right)\hspace{-1pt}\right)+\boldsymbol{U}^{n+1}\nonumber \\
 & \equiv\left\Vert \boldsymbol{x}^{n}-\boldsymbol{x}^{*}\right\Vert _{2}^{2}+\alpha_{n}^{2}\left\Vert \widehat{\nabla}^{n+1}\phi^{\widetilde{F}}\left(\boldsymbol{x}^{n},y^{n},z^{n}\right)\right\Vert _{2}^{2}\nonumber \\
 & \quad\quad-2\alpha_{n}\left(\boldsymbol{x}^{n}-\boldsymbol{x}^{*}\right)^{\boldsymbol{T}}\widehat{\nabla}^{n+1}\phi^{\widetilde{F}}\left(\boldsymbol{x}^{n},{\cal S}^{\widetilde{F}}\hspace{-1pt}\hspace{-1pt}\left(\boldsymbol{x}^{n}\right),{\cal D}^{\widetilde{F}}\hspace{-1pt}\hspace{-1pt}\left(\boldsymbol{x}^{n}\right)\right)+\boldsymbol{U}^{n+1},\label{eq:BIG_P1_1}
\end{flalign}
everywhere on $\Omega$, where the function $\boldsymbol{U}^{n+1}:\Omega\rightarrow\mathbb{R}$
is defined as
\begin{flalign}
 & \hspace{-2pt}\hspace{-2pt}\hspace{-2pt}\hspace{-2pt}\hspace{-2pt}\boldsymbol{U}^{n+1}\\
 & \hspace{-2pt}\hspace{-2pt}\hspace{-2pt}\hspace{-2pt}\hspace{-2pt}\triangleq2c\alpha_{n}\left(\boldsymbol{x}^{n}-\boldsymbol{x}^{*}\right)^{\boldsymbol{T}}\left(\Delta^{n+1}\hspace{-1pt}\hspace{-1pt}\left(\boldsymbol{x}^{n},{\cal S}^{\widetilde{F}}\hspace{-1pt}\hspace{-1pt}\left(\boldsymbol{x}^{n}\right),{\cal D}^{\widetilde{F}}\hspace{-1pt}\hspace{-1pt}\left(\boldsymbol{x}^{n}\right)\right)-\Delta^{n+1}\hspace{-1pt}\hspace{-1pt}\left(\boldsymbol{x}^{n},y^{n},z^{n}\right)\hspace{-1pt}\right)\nonumber \\
 & \hspace{-2pt}\hspace{-2pt}\hspace{-2pt}\hspace{-2pt}\hspace{-2pt}\equiv2c\alpha_{n}\left(\boldsymbol{x}^{n}-\boldsymbol{x}^{*}\right)^{\boldsymbol{T}}\left(\underline{\nabla}F\hspace{-2pt}\left(\boldsymbol{x}^{n},\hspace{-1pt}\boldsymbol{W}_{2}^{n+1}\right)-\underline{\nabla}F\hspace{-2pt}\left(\boldsymbol{x}^{n},\hspace{-1pt}\boldsymbol{W}_{1}^{n+1}\right)\right)\nonumber \\
 & \hspace{-2pt}\hspace{-2pt}\hspace{-2pt}\hspace{-2pt}\times\hspace{-2pt}\hspace{-2pt}\left(\hspace{-2pt}\underset{\triangleq\mathsf{A}_{1}^{n}\left(\boldsymbol{x}^{n}\right)}{\underbrace{\left.\underline{\nabla}\left({\cal R}\hspace{-1pt}\left(r\right)\right)^{p}\right|_{r\equiv F\hspace{-2pt}\left(\boldsymbol{x}^{n},\boldsymbol{W}_{2}^{n+1}\right)-{\cal S}^{\widetilde{F}}\left(\boldsymbol{x}^{n}\right)}}}\underset{\triangleq\mathsf{C}_{1}^{n}\left(\boldsymbol{x}^{n}\right)}{\underbrace{\left({\cal D}^{\widetilde{F}}\hspace{-2pt}\left(\boldsymbol{x}^{n}\right)\hspace{-2pt}\right)^{\hspace{-2pt}\left(1-p\right)/p}}}\hspace{-2pt}-\underset{\triangleq\mathsf{A}_{2}^{n}\left(\boldsymbol{x}^{n},y^{n}\right)}{\underbrace{\left.\underline{\nabla}\left({\cal R}\hspace{-1pt}\left(r\right)\right)^{p}\right|_{r\equiv F\left(\boldsymbol{x}^{n},\boldsymbol{W}_{2}^{n+1}\right)-y^{n}}}}\underset{\triangleq\mathsf{C}_{2}^{n}\left(z^{n}\right)}{\underbrace{\left(z^{n}\right)^{\hspace{-2pt}\left(1-p\right)/p}}}\hspace{-2pt}\right)\hspace{-1pt}\hspace{-2pt}.\hspace{-2pt}\hspace{-2pt}\hspace{-2pt}\nonumber 
\end{flalign}
From the proof of Lemma \ref{lem:INTER_1} (Section \ref{subsec:RM_1}),
it readily follows that
\begin{equation}
\mathbb{E}_{\mathscr{D}^{n}}\left\{ \left\Vert \widehat{\nabla}^{n+1}\phi^{\widetilde{F}}\left(\boldsymbol{x}^{n},y^{n},z^{n}\right)\right\Vert _{2}^{2}\right\} \le\left(2c\mathsf{R}_{p}+1\right)^{2}G^{2},\quad{\cal P}-a.e.
\end{equation}
Hence, taking conditional expectations on both sides of (\ref{eq:BIG_P1_1})
relative to $\mathscr{D}^{n}$, we have
\begin{flalign}
 & \hspace{-2pt}\hspace{-2pt}\hspace{-2pt}\hspace{-2pt}\hspace{-2pt}\hspace{-2pt}\hspace{-2pt}\hspace{-2pt}\hspace{-2pt}\mathbb{E}_{\mathscr{D}^{n}}\hspace{-2pt}\left\{ \left\Vert \boldsymbol{x}^{n+1}\hspace{-1pt}-\hspace{-1pt}\boldsymbol{x}^{*}\right\Vert _{2}^{2}\right\} \nonumber \\
 & \le\hspace{-2pt}\left\Vert \boldsymbol{x}^{n}-\boldsymbol{x}^{*}\right\Vert _{2}^{2}+\alpha_{n}^{2}\left(2c\mathsf{R}_{p}+1\right)^{2}G^{2}\nonumber \\
 & \quad\quad-2\alpha_{n}\left(\boldsymbol{x}^{n}-\boldsymbol{x}^{*}\right)^{\boldsymbol{T}}\mathbb{E}_{\mathscr{D}^{n}}\left\{ \widehat{\nabla}^{n+1}\phi^{\widetilde{F}}\left(\boldsymbol{x}^{n},{\cal S}^{\widetilde{F}}\hspace{-1pt}\hspace{-1pt}\left(\boldsymbol{x}^{n}\right),{\cal D}^{\widetilde{F}}\hspace{-1pt}\hspace{-1pt}\left(\boldsymbol{x}^{n}\right)\right)\right\} \hspace{-2pt}+\mathbb{E}_{\mathscr{D}^{n}}\hspace{-2pt}\left\{ \boldsymbol{U}^{n+1}\right\} \nonumber \\
 & \equiv\hspace{-2pt}\left\Vert \boldsymbol{x}^{n}-\boldsymbol{x}^{*}\right\Vert _{2}^{2}+\alpha_{n}^{2}\left(2c\mathsf{R}_{p}+1\right)^{2}G^{2}\hspace{-2pt}-2\alpha_{n}\left(\boldsymbol{x}^{n}-\boldsymbol{x}^{*}\right)^{\boldsymbol{T}}\nabla\phi^{\widetilde{F}}\left(\boldsymbol{x}^{n}\right)\hspace{-2pt}+\mathbb{E}_{\mathscr{D}^{n}}\hspace{-2pt}\left\{ \boldsymbol{U}^{n+1}\right\} \nonumber \\
 & \le\hspace{-2pt}\left\Vert \boldsymbol{x}^{n}-\boldsymbol{x}^{*}\right\Vert _{2}^{2}+\alpha_{n}^{2}\left(2c\mathsf{R}_{p}+1\right)^{2}G^{2}\hspace{-2pt}-2\alpha_{n}\left(\phi^{\widetilde{F}}\left(\boldsymbol{x}^{n}\right)-\phi_{*}^{\widetilde{F}}\right)\hspace{-2pt}+\mathbb{E}_{\mathscr{D}^{n}}\hspace{-2pt}\left\{ \boldsymbol{U}^{n+1}\right\} ,\label{eq:BIG_P1_2}
\end{flalign}
almost everywhere relative to ${\cal P}$, where in the last inequality,
we have exploited our assumption that the objective function $\phi^{\widetilde{F}}$
is convex. Therefore, our main concern now is properly bounding $\mathbb{E}_{\mathscr{D}^{n}}\left\{ \boldsymbol{U}^{n+1}\right\} $.
By Cauchy-Schwarz, $\boldsymbol{U}^{n+1}$ may be bounded from above
as
\begin{flalign}
\boldsymbol{U}^{n+1} & \le2c\alpha_{n}\left\Vert \boldsymbol{x}^{n}-\boldsymbol{x}^{*}\right\Vert _{2}\left(\left\Vert \underline{\nabla}F\hspace{-2pt}\left(\boldsymbol{x}^{n},\hspace{-1pt}\boldsymbol{W}_{1}^{n+1}\right)\right\Vert _{2}+\left\Vert \underline{\nabla}F\hspace{-2pt}\left(\boldsymbol{x}^{n},\hspace{-1pt}\boldsymbol{W}_{2}^{n+1}\right)\right\Vert _{2}\right)\nonumber \\
 & \quad\times\left|\mathsf{A}_{1}^{n}\left(\boldsymbol{x}^{n}\right)\mathsf{C}_{1}^{n}\left(\boldsymbol{x}^{n}\right)-\mathsf{A}_{2}^{n}\left(\boldsymbol{x}^{n},y^{n}\right)\mathsf{C}_{2}^{n}\left(z^{n}\right)\right|,\label{eq:Long_0}
\end{flalign}
everywhere on $\Omega$, as well. 

Now, \textit{let Assumption \ref{assu:F_AS_Main} be in effect}. Making
use of the fact that $\underline{\nabla}{\cal R}$ is uniformly upper
bounded by unity and of Lemmata \ref{lem:Iterate_Boundedness} and
\ref{lem:P_is_LIP}, and the resulting inequality
\begin{flalign}
\mathsf{A}_{1}^{n}\left(\boldsymbol{x}^{n}\right) & \equiv\left.\underline{\nabla}\left({\cal R}\hspace{-1pt}\left(r\right)\right)^{p}\right|_{r\equiv F\hspace{-2pt}\left(\boldsymbol{x}^{n},\boldsymbol{W}_{2}^{n+1}\right)-{\cal S}^{\widetilde{F}}\left(\boldsymbol{x}^{n}\right)}\nonumber \\
 & \equiv p\left({\cal R}\hspace{-1pt}\left(F\hspace{-2pt}\left(\boldsymbol{x}^{n},\boldsymbol{W}_{2}^{n+1}\right)-{\cal S}^{\widetilde{F}}\left(\boldsymbol{x}^{n}\right)\right)\right)^{p-1}\underline{\nabla}{\cal R}\hspace{-1pt}\left(F\hspace{-2pt}\left(\boldsymbol{x}^{n},\boldsymbol{W}_{2}^{n+1}\right)-{\cal S}^{\widetilde{F}}\left(\boldsymbol{x}^{n}\right)\right)\nonumber \\
 & \le p{\cal E}^{p-1},
\end{flalign}
we may further bound the absolute difference on the RHS of (\ref{eq:Long_0})
from above as
\begin{flalign}
 & \hspace{-2pt}\hspace{-2pt}\hspace{-2pt}\hspace{-2pt}\hspace{-2pt}\hspace{-2pt}\hspace{-2pt}\hspace{-2pt}\hspace{-2pt}\hspace{-2pt}\hspace{-2pt}\hspace{-2pt}\left|\mathsf{A}_{1}^{n}\left(\boldsymbol{x}^{n}\right)\mathsf{C}_{1}^{n}\left(\boldsymbol{x}^{n}\right)-\mathsf{A}_{2}^{n}\left(\boldsymbol{x}^{n},y^{n}\right)\mathsf{C}_{2}^{n}\left(z^{n}\right)\right|\nonumber \\
 & \le\mathsf{C}_{2}^{n}\left(z^{n}\right)\hspace{-2pt}\left|\mathsf{A}_{1}^{n}\hspace{-2pt}\left(\boldsymbol{x}^{n}\right)\hspace{-2pt}-\hspace{-1pt}\mathsf{A}_{2}^{n}\hspace{-2pt}\left(\boldsymbol{x}^{n},y^{n}\right)\hspace{-1pt}\right|\hspace{-1pt}+\mathsf{A}_{1}^{n}\left(\boldsymbol{x}^{n}\right)\hspace{-1pt}\left|\mathsf{C}_{1}^{n}\hspace{-2pt}\left(\boldsymbol{x}^{n}\right)\hspace{-2pt}-\hspace{-1pt}\mathsf{C}_{2}^{n}\hspace{-2pt}\left(z^{n}\right)\hspace{-1pt}\right|\nonumber \\
 & \le\left(\dfrac{1}{\varepsilon}\right)^{\hspace{-1pt}p-1}\hspace{-2pt}\left|\mathsf{A}_{1}^{n}\hspace{-2pt}\left(\boldsymbol{x}^{n}\right)\hspace{-2pt}-\hspace{-1pt}\mathsf{A}_{2}^{n}\hspace{-2pt}\left(\boldsymbol{x}^{n},y^{n}\right)\hspace{-1pt}\right|\hspace{-1pt}+p{\cal E}^{p-1}\hspace{-1pt}\left|\mathsf{C}_{1}^{n}\hspace{-2pt}\left(\boldsymbol{x}^{n}\right)\hspace{-2pt}-\hspace{-1pt}\mathsf{C}_{2}^{n}\hspace{-2pt}\left(z^{n}\right)\hspace{-1pt}\right|\nonumber \\
 & \le\left(\dfrac{1}{\varepsilon}\right)^{\hspace{-1pt}p-1}\hspace{-2pt}\left|\mathsf{A}_{1}^{n}\hspace{-2pt}\left(\boldsymbol{x}^{n}\right)\hspace{-2pt}-\hspace{-1pt}\mathsf{A}_{2}^{n}\hspace{-2pt}\left(\boldsymbol{x}^{n},y^{n}\right)\hspace{-1pt}\right|\hspace{-1pt}+\left(p-1\right)\dfrac{{\cal E}^{p-1}}{\varepsilon^{2p-1}}\hspace{-1pt}\left|z^{n}\hspace{-2pt}-\hspace{-2pt}{\cal D}^{\widetilde{F}}\hspace{-2pt}\left(\boldsymbol{x}^{n}\right)\hspace{-1pt}\right|,\hspace{-2pt}\hspace{-2pt}\label{eq:Long_1}
\end{flalign}
almost everywhere relative to ${\cal P}$. Utilizing (\ref{eq:Long_1})
and taking conditional expectations relative to $\mathscr{D}^{n}$
on both sides of (\ref{eq:Long_0}), we have
\begin{flalign}
\hspace{-2pt}\hspace{-2pt}\hspace{-2pt}\hspace{-2pt}\hspace{-2pt}\hspace{-2pt}\hspace{-2pt}\mathbb{E}_{\mathscr{D}^{n}}\left\{ \boldsymbol{U}^{n+1}\right\}  & \le2c\alpha_{n}\left\Vert \boldsymbol{x}^{n}-\boldsymbol{x}^{*}\right\Vert _{2}\nonumber \\
 & \hspace{-2pt}\hspace{-2pt}\hspace{-2pt}\hspace{-2pt}\hspace{-2pt}\hspace{-2pt}\hspace{-2pt}\hspace{-2pt}\hspace{-2pt}\hspace{-2pt}\hspace{-2pt}\hspace{-2pt}\hspace{-2pt}\hspace{-2pt}\hspace{-2pt}\hspace{-2pt}\hspace{-2pt}\hspace{-2pt}\hspace{-2pt}\hspace{-2pt}\hspace{-2pt}\hspace{-2pt}\hspace{-2pt}\hspace{-2pt}\hspace{-2pt}\hspace{-2pt}\hspace{-2pt}\times\hspace{-2pt}\left(\hspace{-2pt}\hspace{-2pt}\left(\dfrac{1}{\varepsilon}\right)^{\hspace{-1pt}p-1}\mathbb{E}_{\mathscr{D}^{n}}\left\{ \hspace{-2pt}\left(\left\Vert \underline{\nabla}F\hspace{-2pt}\left(\boldsymbol{x}^{n},\hspace{-1pt}\boldsymbol{W}_{1}^{n+1}\right)\right\Vert _{2}+\left\Vert \underline{\nabla}F\hspace{-2pt}\left(\boldsymbol{x}^{n},\hspace{-1pt}\boldsymbol{W}_{2}^{n+1}\right)\right\Vert _{2}\right)\hspace{-2pt}\left|\mathsf{A}_{1}^{n}\left(\boldsymbol{x}^{n}\right)-\mathsf{A}_{2}^{n}\left(\boldsymbol{x}^{n},y^{n}\right)\right|\right\} \right.\nonumber \\
 & \hspace{-2pt}\hspace{-2pt}\hspace{-2pt}\hspace{-2pt}\hspace{-2pt}\hspace{-2pt}\hspace{-2pt}\hspace{-2pt}\hspace{-2pt}\hspace{-2pt}\hspace{-2pt}\hspace{-2pt}\hspace{-2pt}\hspace{-2pt}\hspace{-2pt}\hspace{-2pt}\hspace{-2pt}\hspace{-2pt}\hspace{-2pt}\hspace{-2pt}\hspace{-2pt}\hspace{-2pt}+\left(p-1\right)\dfrac{{\cal E}^{p-1}}{\varepsilon^{2p-1}}\left.\mathbb{E}_{\mathscr{D}^{n}}\left\{ \hspace{-2pt}\left(\left\Vert \underline{\nabla}F\hspace{-2pt}\left(\boldsymbol{x}^{n},\hspace{-1pt}\boldsymbol{W}_{1}^{n+1}\right)\right\Vert _{2}+\left\Vert \underline{\nabla}F\hspace{-2pt}\left(\boldsymbol{x}^{n},\hspace{-1pt}\boldsymbol{W}_{2}^{n+1}\right)\right\Vert _{2}\right)\hspace{-2pt}\left|z^{n}-{\cal D}^{\widetilde{F}}\hspace{-2pt}\left(\boldsymbol{x}^{n}\right)\right|\right\} \vphantom{\left(\dfrac{{\cal E}}{\varepsilon}\right)^{\hspace{-1pt}p-1}}\hspace{-2pt}\hspace{-2pt}\right)\hspace{-2pt},\hspace{-2pt}\label{eq:Long_2}
\end{flalign}
almost everywhere relative to ${\cal P}$. Let us consider each of
the three terms on the RHS of (\ref{eq:Long_2}) separately. Regarding
the term related to $z^{n}$, exploiting condition $\mathbf{C1}$,
we may write
\begin{flalign}
 & \hspace{-2pt}\hspace{-2pt}\hspace{-2pt}\hspace{-2pt}\mathbb{E}_{\mathscr{D}^{n}}\left\{ \left(\left\Vert \underline{\nabla}F\hspace{-2pt}\left(\boldsymbol{x}^{n},\hspace{-1pt}\boldsymbol{W}_{1}^{n+1}\right)\right\Vert _{2}+\left\Vert \underline{\nabla}F\hspace{-2pt}\left(\boldsymbol{x}^{n},\hspace{-1pt}\boldsymbol{W}_{2}^{n+1}\right)\right\Vert _{2}\right)\right\} \nonumber \\
 & \equiv\left.\left\Vert \vphantom{\int}\hspace{-2pt}\left\Vert \underline{\nabla}F\hspace{-2pt}\left(\boldsymbol{x},\hspace{-1pt}\boldsymbol{W}_{1}^{n+1}\right)\right\Vert _{2}\hspace{-2pt}+\hspace{-1pt}\left\Vert \underline{\nabla}F\hspace{-2pt}\left(\boldsymbol{x},\hspace{-1pt}\boldsymbol{W}_{2}^{n+1}\right)\right\Vert _{2}\right\Vert _{{\cal L}_{P}}\right|_{\boldsymbol{x}\equiv\boldsymbol{x}^{n}}\hspace{-2pt}\left|z^{n}-{\cal D}^{\widetilde{F}}\hspace{-2pt}\left(\boldsymbol{x}^{n}\right)\right|\nonumber \\
 & \le2G\hspace{-2pt}\hspace{-1pt}\left|z^{n}-{\cal D}^{\widetilde{F}}\hspace{-2pt}\left(\boldsymbol{x}^{n}\right)\right|\hspace{-2pt},\quad{\cal P}-a.e.\label{eq:B_0}
\end{flalign}
As far as the remaining term of (\ref{eq:Long_2}) containing $\mathsf{A}_{1}^{n}$
and $\mathsf{A}_{2}^{n}$ is concerned, the situation is somewhat
more complicated. Recall that, by assumption, we have $P\in\left[2,\infty\right]$,
$P/\left(P-1\right)\le Q\le\infty$, and $P^{-1}+Q^{-1}\le1$. Then,
we may  invoke the \textit{generalized} H\"older's Inequality for
finite measure spaces (on the \textit{appropriate} Borel space), along
with condition ${\bf C3}$, obtaining
\begin{flalign}
 & \hspace{-2pt}\hspace{-2pt}\hspace{-2pt}\mathbb{E}_{\mathscr{D}^{n}}\left\{ \hspace{-2pt}\left(\left\Vert \underline{\nabla}F\hspace{-2pt}\left(\boldsymbol{x}^{n},\hspace{-1pt}\boldsymbol{W}_{1}^{n+1}\right)\right\Vert _{2}+\left\Vert \underline{\nabla}F\hspace{-2pt}\left(\boldsymbol{x}^{n},\hspace{-1pt}\boldsymbol{W}_{2}^{n+1}\right)\right\Vert _{2}\right)\hspace{-2pt}\left|\mathsf{A}_{1}^{n}\left(\boldsymbol{x}^{n}\right)-\mathsf{A}_{2}^{n}\left(\boldsymbol{x}^{n},y^{n}\right)\right|\right\} \nonumber \\
 & \le\hspace{-1pt}\left.\left\Vert \vphantom{\int}\hspace{-2pt}\left\Vert \underline{\nabla}F\hspace{-2pt}\left(\boldsymbol{x},\hspace{-1pt}\boldsymbol{W}_{1}^{n+1}\right)\right\Vert _{2}\hspace{-2pt}+\hspace{-1pt}\left\Vert \underline{\nabla}F\hspace{-2pt}\left(\boldsymbol{x},\hspace{-1pt}\boldsymbol{W}_{2}^{n+1}\right)\right\Vert _{2}\right\Vert _{{\cal L}_{P}}\right|_{\boldsymbol{x}\equiv\boldsymbol{x}^{n}}\left.\left\Vert \vphantom{\int}\hspace{-2pt}\left|\mathsf{A}_{1}^{n}\left(\boldsymbol{x}\right)-\mathsf{A}_{2}^{n}\left(\boldsymbol{x},y\right)\right|\right\Vert _{{\cal L}_{Q}}\right|_{\boldsymbol{x}\equiv\boldsymbol{x}^{n},y\equiv y^{n}}\nonumber \\
 & \le\hspace{-1pt}2G\left.\left\Vert \vphantom{\int}\hspace{-2pt}\left|\left.\underline{\nabla}\left({\cal R}\hspace{-1pt}\left(r\right)\right)^{p}\right|_{r\equiv F\hspace{-2pt}\left(\boldsymbol{x},\boldsymbol{W}_{2}^{n+1}\right)-{\cal S}^{\widetilde{F}}\left(\boldsymbol{x}\right)}\hspace{-2pt}-\hspace{-2pt}\left.\underline{\nabla}\left({\cal R}\hspace{-1pt}\left(r\right)\right)^{p}\right|_{r\equiv F\left(\boldsymbol{x},\boldsymbol{W}_{2}^{n+1}\right)-y}\right|\right\Vert _{{\cal L}_{Q}}\right|_{\boldsymbol{x}\equiv\boldsymbol{x}^{n},y\equiv y^{n}}\nonumber \\
 & \le\hspace{-1pt}2GD\left.\vphantom{\int}\hspace{-2pt}\left|y-{\cal S}^{\widetilde{F}}\hspace{-1pt}\hspace{-1pt}\left(\boldsymbol{x}\right)\right|\right|_{\boldsymbol{x}\equiv\boldsymbol{x}^{n},y\equiv y^{n}}\equiv\hspace{-1pt}2GD\left|y^{n}-{\cal S}^{\widetilde{F}}\hspace{-1pt}\hspace{-1pt}\left(\boldsymbol{x}^{n}\right)\right|,\label{eq:finally?}
\end{flalign}
almost everywhere relative to ${\cal P}$. Combining (\ref{eq:Long_2})
with (\ref{eq:B_0}) and (\ref{eq:finally?}), we readily obtain the
upper bound
\begin{flalign}
\mathbb{E}_{\mathscr{D}^{n}}\left\{ \boldsymbol{U}^{n+1}\right\}  & \le4\widetilde{\mathsf{B}}_{p}Gc\alpha_{n}\left\Vert \boldsymbol{x}^{n}\hspace{-2pt}-\hspace{-1pt}\boldsymbol{x}^{*}\right\Vert _{2}\hspace{-2pt}\left(\left|y^{n}\hspace{-2pt}-\hspace{-1pt}{\cal S}^{\widetilde{F}}\hspace{-1pt}\hspace{-1pt}\left(\boldsymbol{x}^{n}\right)\right|\hspace{-2pt}+\hspace{-2pt}\left|z^{n}\hspace{-2pt}-\hspace{-1pt}{\cal D}^{\widetilde{F}}\hspace{-2pt}\left(\boldsymbol{x}^{n}\right)\right|\right)\hspace{-2pt},\quad{\cal P}-a.e.,
\end{flalign}
where the constant $\widetilde{\mathsf{B}}_{p}<\infty$ is defined
as
\begin{equation}
\widetilde{\mathsf{B}}_{p}\triangleq\max\left\{ \hspace{-2pt}\left(\dfrac{1}{\varepsilon}\right)^{\hspace{-1pt}p-1}\hspace{-2pt}D,\left(p-1\right)\dfrac{{\cal E}^{p-1}}{\varepsilon^{2p-1}}\right\} \hspace{-2pt}.
\end{equation}
As a next step, recalling Lemma \ref{lem:P_is_LIP}, we observe that
\begin{flalign}
 & \hspace{-2pt}\hspace{-2pt}\hspace{-2pt}\hspace{-2pt}\hspace{-2pt}\hspace{-2pt}\left|z^{n}\hspace{-2pt}-\hspace{-1pt}{\cal D}^{\widetilde{F}}\hspace{-2pt}\left(\boldsymbol{x}^{n}\right)\right|\nonumber \\
 & \equiv\left|z^{n}-{\cal D}^{\widetilde{F}}\hspace{-2pt}\left(\boldsymbol{x}^{n},y^{n}\right)+{\cal D}^{\widetilde{F}}\hspace{-2pt}\left(\boldsymbol{x}^{n},y^{n}\right)-\hspace{-1pt}{\cal D}^{\widetilde{F}}\hspace{-2pt}\left(\boldsymbol{x}^{n}\right)\right|\nonumber \\
 & \le\left|z^{n}-{\cal D}^{\widetilde{F}}\hspace{-2pt}\left(\boldsymbol{x}^{n},y^{n}\right)\right|+\left|{\cal D}^{\widetilde{F}}\hspace{-2pt}\left(\boldsymbol{x}^{n},y^{n}\right)-\hspace{-1pt}{\cal D}^{\widetilde{F}}\hspace{-2pt}\left(\boldsymbol{x}^{n}\right)\right|\nonumber \\
 & \equiv\left|z^{n}-{\cal D}^{\widetilde{F}}\hspace{-2pt}\left(\boldsymbol{x}^{n},y^{n}\right)\right|\nonumber \\
 & \quad\quad+\left|\mathbb{E}_{\mathscr{D}^{n}}\hspace{-2pt}\left\{ \left({\cal R}\left(F\hspace{-2pt}\left(\boldsymbol{x}^{n},\hspace{-1pt}\boldsymbol{W}'\right)\hspace{-2pt}-\hspace{-2pt}{\cal S}^{\widetilde{F}}\left(\boldsymbol{x}^{n}\right)\right)\right)^{p}\right\} -\mathbb{E}_{\mathscr{D}^{n}}\hspace{-2pt}\left\{ \left({\cal R}\left(F\hspace{-2pt}\left(\boldsymbol{x}^{n},\hspace{-1pt}\boldsymbol{W}'\right)\hspace{-2pt}-\hspace{-2pt}y^{n}\right)\right)^{p}\right\} \right|\nonumber \\
 & \equiv\left|z^{n}-{\cal D}^{\widetilde{F}}\hspace{-2pt}\left(\boldsymbol{x}^{n},y^{n}\right)\right|+\left|\mathbb{E}_{\mathscr{D}^{n}}\hspace{-2pt}\left\{ \left({\cal R}\left(F\hspace{-2pt}\left(\boldsymbol{x}^{n},\hspace{-1pt}\boldsymbol{W}'\right)\hspace{-2pt}-\hspace{-2pt}{\cal S}^{\widetilde{F}}\left(\boldsymbol{x}^{n}\right)\right)\right)^{p}\hspace{-2pt}-\hspace{-1pt}\left({\cal R}\left(F\hspace{-2pt}\left(\boldsymbol{x}^{n},\hspace{-1pt}\boldsymbol{W}'\right)\hspace{-2pt}-\hspace{-2pt}y^{n}\right)\right)^{p}\right\} \hspace{-1pt}\right|\nonumber \\
 & \le\left|z^{n}-{\cal D}^{\widetilde{F}}\hspace{-2pt}\left(\boldsymbol{x}^{n},y^{n}\right)\right|+\mathbb{E}_{\mathscr{D}^{n}}\hspace{-2pt}\left\{ \left|\left({\cal R}\left(F\hspace{-2pt}\left(\boldsymbol{x}^{n},\hspace{-1pt}\boldsymbol{W}'\right)\hspace{-2pt}-\hspace{-2pt}{\cal S}^{\widetilde{F}}\left(\boldsymbol{x}^{n}\right)\right)\right)^{p}\hspace{-2pt}-\hspace{-1pt}\left({\cal R}\left(F\hspace{-2pt}\left(\boldsymbol{x}^{n},\hspace{-1pt}\boldsymbol{W}'\right)\hspace{-2pt}-\hspace{-2pt}y^{n}\right)\right)^{p}\right|\right\} \nonumber \\
 & \le\left|z^{n}-{\cal D}^{\widetilde{F}}\hspace{-2pt}\left(\boldsymbol{x}^{n},y^{n}\right)\right|+{\cal E}^{p-1}p\left|y^{n}-{\cal S}^{\widetilde{F}}\hspace{-1pt}\hspace{-1pt}\left(\boldsymbol{x}^{n}\right)\right|,\quad{\cal P}-a.e..
\end{flalign}
This yields, in turn,
\begin{flalign}
\mathbb{E}_{\mathscr{D}^{n}}\left\{ \boldsymbol{U}^{n+1}\right\}  & \le4\mathsf{B}_{p}Gc\alpha_{n}\left\Vert \boldsymbol{x}^{n}\hspace{-2pt}-\hspace{-1pt}\boldsymbol{x}^{*}\right\Vert _{2}\hspace{-2pt}\left|y^{n}\hspace{-2pt}-\hspace{-1pt}{\cal S}^{\widetilde{F}}\hspace{-1pt}\hspace{-1pt}\left(\boldsymbol{x}^{n}\right)\right|\nonumber \\
 & \quad\quad+4\mathsf{B}_{p}Gc\alpha_{n}\left\Vert \boldsymbol{x}^{n}\hspace{-2pt}-\hspace{-1pt}\boldsymbol{x}^{*}\right\Vert _{2}\left|z^{n}-{\cal D}^{\widetilde{F}}\hspace{-2pt}\left(\boldsymbol{x}^{n},y^{n}\right)\right|\hspace{-2pt},
\end{flalign}
where $\mathsf{B}_{p}\triangleq\left(1+{\cal E}^{p-1}p\right)\widetilde{\mathsf{B}}{}_{p}$.
Finally, invoking Lemmata \ref{lem:INTER_1-1} and \ref{lem:INTER_Middle},
$\mathbb{E}_{\mathscr{D}^{n}}\left\{ \boldsymbol{U}^{n+1}\right\} $
may be further bounded as
\begin{flalign}
\mathbb{E}_{\mathscr{D}^{n}}\left\{ \boldsymbol{U}^{n+1}\right\}  & \le4\mathsf{B}_{p}^{2}G^{2}c^{2}\dfrac{\alpha_{n}^{2}}{\beta_{n}}\left\Vert \boldsymbol{x}^{n}\hspace{-2pt}-\hspace{-1pt}\boldsymbol{x}^{*}\right\Vert _{2}^{2}+\beta_{n}\left|y^{n}\hspace{-2pt}-\hspace{-1pt}{\cal S}^{\widetilde{F}}\hspace{-1pt}\hspace{-1pt}\left(\boldsymbol{x}^{n}\right)\right|^{2}\nonumber \\
 & \quad\quad+4\mathsf{B}_{p}^{2}G^{2}c^{2}\dfrac{\alpha_{n}^{2}}{\gamma_{n}}\left\Vert \boldsymbol{x}^{n}\hspace{-2pt}-\hspace{-1pt}\boldsymbol{x}^{*}\right\Vert _{2}^{2}+\gamma_{n}\left|z^{n}-{\cal D}^{\widetilde{F}}\hspace{-2pt}\left(\boldsymbol{x}^{n},y^{n}\right)\right|^{2},
\end{flalign}
almost everywhere relative to ${\cal P}$.

As a result, invoking Lemma \ref{lem:Iterate_Boundedness}, (\ref{eq:BIG_P1_2})
may be further bounded from above as
\begin{flalign}
 & \hspace{-2pt}\hspace{-2pt}\hspace{-2pt}\mathbb{E}_{\mathscr{D}^{n}}\left\{ \left\Vert \boldsymbol{x}^{n+1}\hspace{-1pt}-\hspace{-1pt}\boldsymbol{x}^{*}\right\Vert _{2}^{2}\right\} \nonumber \\
 & \le\left(1+4\mathsf{B}_{p}^{2}G^{2}c^{2}\left(\dfrac{\alpha_{n}^{2}}{\beta_{n}}+\dfrac{\alpha_{n}^{2}}{\gamma_{n}}\right)\hspace{-2pt}\right)\left\Vert \boldsymbol{x}^{n}-\boldsymbol{x}^{*}\right\Vert _{2}^{2}+\alpha_{n}^{2}\left(2c\mathsf{R}_{p}+1\right)^{2}G^{2}-2\alpha_{n}\left(\phi^{\widetilde{F}}\left(\boldsymbol{x}^{n}\right)-\phi_{*}^{\widetilde{F}}\right)\nonumber \\
 & \quad\quad+\beta_{n}\hspace{-1pt}\left|y^{n}\hspace{-1pt}\hspace{-1pt}-\hspace{-1pt}{\cal S}^{\widetilde{F}}\hspace{-1pt}\hspace{-1pt}\left(\boldsymbol{x}^{n}\right)\right|^{2}\hspace{-2pt}+\hspace{-1pt}\gamma_{n}\left|z^{n}-{\cal D}^{\widetilde{F}}\left(\boldsymbol{x}^{n},y^{n}\right)\right|^{2}\label{eq:Final?}
\end{flalign}
almost everywhere relative to ${\cal P}$, yielding (\ref{eq:BIG_P1_3-1})
in the statement of Lemma \ref{lem:INTER_Middle}. The fact that $\mathbb{N}^{+}$
is countable completes the proof when $p>1$. For the case where $p\equiv1$,
exactly the same procedure yields the constant $\mathsf{B}_{p}\equiv D$,
whereas the ratio $\alpha_{n}^{2}/\gamma_{n}$ and last term on the
RHS of (\ref{eq:Final?}) (left to right) \textit{disappears}.\hfill{}\ensuremath{\blacksquare}

\subsection{\label{subsec:Chung}A Generalization of Chung's Lemma}

The following result is a generalization of Chung's Lemma \citep{Chung1954},
which is an old and well-known result for analyzing convergence rates
of stochastic approximation algorithms.
\begin{lem}
\textbf{\textup{(Generalized Chung's Lemma)}}\label{lem:Chung} Consider
any nonnegative sequence $\left\{ S^{n}\right\} _{n\in\mathbb{N}}$,
such that
\begin{equation}
S^{n+1}\le\left(1-\alpha_{n}\right)S^{n}+C\beta_{n},\quad\forall n\in\mathbb{N},
\end{equation}
where $\left\{ \alpha_{n}\right\} _{n\in\mathbb{N}}$, $\left\{ \beta_{n}\right\} _{n\in\mathbb{N}}$
are also nonnegative sequences, and $C\ge0$. Suppose that
\begin{equation}
n^{+}\triangleq\min\left\{ \left.n\in\mathbb{N}\right|\alpha_{n'}\le1,\;\forall n'\in\mathbb{N}^{n}\right\} \in\left[0,\infty\right),
\end{equation}
and choose $n_{o}\in\mathbb{N}^{n^{+}}$. Then, for every $n\in\mathbb{N}^{n_{o}}$,
it is true that
\begin{equation}
S^{n+1}\le S^{n_{o}}\prod_{i\in\mathbb{N}_{n}^{n_{o}}}\left(1-\alpha_{i}\right)+C\sum_{i\in\mathbb{N}_{n}^{n_{o}}}\beta_{i}\prod_{j\in\mathbb{N}_{n}^{i+1}}\left(1-\alpha_{j}\right),
\end{equation}
where, by convention, $\prod_{j\in\mathbb{N}_{n}^{n+1}}\left(\cdot\right)\equiv\prod_{j=n+1}^{n}\left(\cdot\right)\equiv1$.
\end{lem}
\begin{proof}[Proof of Lemma \ref{lem:Chung}]
Use simple induction; enough said.
\end{proof}

\subsection{\label{subsec:Rate_Gen}Proof of Lemma \ref{lem:Rate_Generator}}

In the following, we consider the case where $p>1$. If $p\equiv1$,
the proof is similar, albeit simpler. From the proof of Lemma \ref{lem:INTER_1-1-1},
we have already shown that
\begin{equation}
\hspace{-2pt}\mathbb{E}_{\mathscr{D}^{n}}\hspace{-2pt}\left\{ \left\Vert \boldsymbol{x}^{n+1}\hspace{-1pt}-\hspace{-1pt}\boldsymbol{x}^{*}\right\Vert _{2}^{2}\right\} \hspace{-2pt}\le\hspace{-2pt}\left\Vert \boldsymbol{x}^{n}-\boldsymbol{x}^{*}\right\Vert _{2}^{2}+\alpha_{n}^{2}\left(2c\mathsf{R}_{p}+1\right)^{2}G^{2}\hspace{-2pt}-2\alpha_{n}\left(\phi^{\widetilde{F}}\left(\boldsymbol{x}^{n}\right)\hspace{-2pt}-\phi_{*}^{\widetilde{F}}\right)\hspace{-2pt}+\mathbb{E}_{\mathscr{D}^{n}}\left\{ \boldsymbol{U}^{n+1}\right\} ,\label{eq:BIG_P1_2-1}
\end{equation}
and that
\begin{equation}
\mathbb{E}_{\mathscr{D}^{n}}\left\{ \boldsymbol{U}^{n+1}\right\} \hspace{-1pt}\hspace{-1pt}\le\hspace{-1pt}4\mathsf{B}_{p}Gc\alpha_{n}\left\Vert \boldsymbol{x}^{n}\hspace{-2pt}-\hspace{-1pt}\boldsymbol{x}^{*}\right\Vert _{2}\left|y^{n}\hspace{-2pt}-\hspace{-1pt}{\cal S}^{\widetilde{F}}\hspace{-1pt}\hspace{-1pt}\left(\boldsymbol{x}^{n}\right)\right|\hspace{-1pt}+\hspace{-1pt}4\mathsf{B}_{p}Gc\alpha_{n}\left\Vert \boldsymbol{x}^{n}\hspace{-2pt}-\hspace{-1pt}\boldsymbol{x}^{*}\right\Vert _{2}\left|z^{n}\hspace{-1pt}\hspace{-1pt}-\hspace{-1pt}{\cal D}^{\widetilde{F}}\hspace{-2pt}\left(\boldsymbol{x}^{n},y^{n}\right)\right|,
\end{equation}
almost everywhere relative to ${\cal P}$, for each $n\in\mathbb{N}$.

First, we exploit our assumption in regard to strong convexity of
the objective $\phi^{\widetilde{F}}$, and taking expectations on
both sides of (\ref{eq:BIG_P1_2-1}), we get
\begin{equation}
\mathbb{E}\left\{ \left\Vert \boldsymbol{x}^{n+1}\hspace{-1pt}-\hspace{-1pt}\boldsymbol{x}^{*}\right\Vert _{2}^{2}\right\} \le\left(1-2\sigma\alpha_{n}\right)\mathbb{E}\left\{ \left\Vert \boldsymbol{x}^{n}-\boldsymbol{x}^{*}\right\Vert _{2}^{2}\right\} +\alpha_{n}^{2}\left(2c\mathsf{R}_{p}+1\right)^{2}G^{2}+\mathbb{E}\left\{ \boldsymbol{U}^{n+1}\right\} ,
\end{equation}
being true for all $n\in\mathbb{N}$. Let us focus on appropriately
bounding the term $\mathbb{E}\left\{ \boldsymbol{U}^{n+1}\right\} $,
for $n\in\mathbb{N}$. It is true that
\begin{flalign}
\mathbb{E}_{\mathscr{D}^{n}}\left\{ \boldsymbol{U}^{n+1}\right\}  & \le\dfrac{\sigma}{2}\alpha_{n}\left\Vert \boldsymbol{x}^{n}\hspace{-2pt}-\hspace{-1pt}\boldsymbol{x}^{*}\right\Vert _{2}^{2}+\alpha_{n}\dfrac{8\mathsf{B}_{p}^{2}G^{2}c^{2}}{\sigma}\left|y^{n}\hspace{-2pt}-\hspace{-1pt}{\cal S}^{\widetilde{F}}\hspace{-1pt}\hspace{-1pt}\left(\boldsymbol{x}^{n}\right)\right|^{2}\nonumber \\
 & \quad\quad+\dfrac{\sigma}{2}\alpha_{n}\left\Vert \boldsymbol{x}^{n}\hspace{-2pt}-\hspace{-1pt}\boldsymbol{x}^{*}\right\Vert _{2}^{2}+\alpha_{n}\dfrac{8\mathsf{B}_{p}^{2}G^{2}c^{2}}{\sigma}\left|z^{n}-{\cal D}^{\widetilde{F}}\hspace{-2pt}\left(\boldsymbol{x}^{n},y^{n}\right)\right|^{2}\nonumber \\
 & \equiv\sigma\alpha_{n}\left\Vert \boldsymbol{x}^{n}\hspace{-2pt}-\hspace{-1pt}\boldsymbol{x}^{*}\right\Vert _{2}^{2}+\alpha_{n}\dfrac{8\mathsf{B}_{p}^{2}G^{2}c^{2}}{\sigma}\left(\left|y^{n}\hspace{-2pt}-\hspace{-1pt}{\cal S}^{\widetilde{F}}\hspace{-1pt}\hspace{-1pt}\left(\boldsymbol{x}^{n}\right)\right|^{2}+\left|z^{n}-{\cal D}^{\widetilde{F}}\hspace{-2pt}\left(\boldsymbol{x}^{n},y^{n}\right)\right|^{2}\right),
\end{flalign}
almost everywhere relative to ${\cal P}$. Again, taking expectations
on both sides, we have
\begin{flalign}
\hspace{-1pt}\hspace{-1pt}\hspace{-1pt}\hspace{-1pt}\hspace{-1pt}\mathbb{E}\hspace{-1pt}\hspace{-1pt}\left\{ \boldsymbol{U}^{n+1}\hspace{-1pt}\right\}  & \hspace{-1pt}\hspace{-2pt}\le\hspace{-2pt}\sigma\alpha_{n}\mathbb{E}\hspace{-1pt}\hspace{-1pt}\left\{ \left\Vert \boldsymbol{x}^{n}\hspace{-1pt}\hspace{-1pt}-\hspace{-1pt}\boldsymbol{x}^{*}\right\Vert _{2}^{2}\right\} \hspace{-2pt}+\hspace{-1pt}\hspace{-1pt}\alpha_{n}\hspace{-1pt}\dfrac{8\mathsf{B}_{p}^{2}G^{2}c^{2}}{\sigma}\hspace{-1pt}\hspace{-1pt}\hspace{-1pt}\left(\hspace{-1pt}\hspace{-1pt}\mathbb{E}\hspace{-1pt}\hspace{-1pt}\left\{ \left|y^{n}\hspace{-2pt}-\hspace{-1pt}{\cal S}^{\widetilde{F}}\hspace{-1pt}\hspace{-1pt}\left(\boldsymbol{x}^{n}\right)\right|^{2}\right\} \hspace{-2pt}+\hspace{-1pt}\mathbb{E}\hspace{-1pt}\hspace{-1pt}\left\{ \left|z^{n}\hspace{-1pt}\hspace{-1pt}-\hspace{-1pt}{\cal D}^{\widetilde{F}}\hspace{-2pt}\left(\boldsymbol{x}^{n},y^{n}\right)\right|^{2}\right\} \hspace{-1pt}\hspace{-1pt}\right)\hspace{-1pt}\hspace{-1pt},\hspace{-1pt}\hspace{-1pt}\hspace{-1pt}
\end{flalign}
 for all $n\in\mathbb{N}$. For brevity, we hereafter make the identifications
\begin{align}
A^{n} & \triangleq\mathbb{E}\left\{ \left\Vert \boldsymbol{x}^{n}-\boldsymbol{x}^{*}\right\Vert _{2}^{2}\right\} ,\\
B^{n} & \triangleq\mathbb{E}\left\{ \left|y^{n}\hspace{-2pt}-\hspace{-1pt}{\cal S}^{\widetilde{F}}\hspace{-1pt}\hspace{-1pt}\left(\boldsymbol{x}^{n}\right)\right|^{2}\right\} \quad\text{and}\\
C^{n} & \triangleq\mathbb{E}\left\{ \left|z^{n}-{\cal D}^{\widetilde{F}}\hspace{-2pt}\left(\boldsymbol{x}^{n},y^{n}\right)\right|^{2}\right\} ,\quad\forall n\in\mathbb{N}.
\end{align}
As a result, we arrive at the inequalities
\begin{flalign}
A^{n+1} & \le\left(1-2\sigma\alpha_{n}\right)A^{n}+\alpha_{n}^{2}\left(2c\mathsf{R}_{p}+1\right)^{2}G^{2}+\mathbb{E}\left\{ \boldsymbol{U}^{n+1}\right\} \quad\text{and}\\
\mathbb{E}\left\{ \boldsymbol{U}^{n+1}\right\}  & \le\sigma\alpha_{n}A^{n}+\alpha_{n}\dfrac{8\mathsf{B}_{p}^{2}G^{2}c^{2}}{\sigma}\left(B^{n}+C^{n}\right),\quad\forall n\in\mathbb{N},
\end{flalign}
which imply that
\begin{equation}
A^{n+1}\le\left(1-\sigma\alpha_{n}\right)A^{n}+\alpha_{n}^{2}\left(2c\mathsf{R}_{p}+1\right)^{2}G^{2}+\alpha_{n}\dfrac{8\mathsf{B}_{p}^{2}G^{2}c^{2}}{\sigma}\left(B^{n}+C^{n}\right),\quad\forall n\in\mathbb{N}.
\end{equation}
By Lemmata \ref{lem:INTER_1}, \ref{lem:INTER_1-1} and \ref{lem:INTER_Middle},
we know that (by taking expectations on both sides), for every $n\in\mathbb{N}^{+}$,
$B^{n}$ and $C^{n}$ satisfy the recursive inequalities
\begin{align}
\hspace{-2pt}\hspace{-2pt}\hspace{-2pt}\hspace{-2pt}\hspace{-2pt}\hspace{-2pt}B^{n} & \hspace{-1pt}\hspace{-1pt}\le\hspace{-1pt}\hspace{-1pt}\left(1\hspace{-1pt}-\hspace{-1pt}\beta_{n-1}\right)\hspace{-1pt}B^{n-1}\hspace{-1pt}+\hspace{-1pt}\dfrac{\alpha_{n-1}^{2}}{\beta_{n-1}}2\left(2c\mathsf{R}_{p}\hspace{-1pt}+\hspace{-1pt}1\right)^{2}G^{4}\hspace{-1pt}+\hspace{-1pt}\beta_{n-1}^{2}2V\quad\text{and}\\
\hspace{-2pt}\hspace{-2pt}\hspace{-2pt}\hspace{-2pt}\hspace{-2pt}\hspace{-2pt}C^{n} & \hspace{-1pt}\hspace{-1pt}\le\hspace{-1pt}\hspace{-1pt}\left(1\hspace{-1pt}-\hspace{-1pt}\gamma_{n-1}\right)\hspace{-1pt}C^{n-1}\hspace{-1pt}+\hspace{-1pt}\dfrac{\alpha_{n-1}^{2}}{\gamma_{n-1}}16G^{4}{\cal E}^{2p-2}p^{2}\hspace{-2pt}\left(2c\mathsf{R}_{p}\hspace{-1pt}+\hspace{-1pt}1\right)^{2}\hspace{-1pt}+\hspace{-1pt}\dfrac{\beta_{n-1}^{2}}{\gamma_{n-1}}4{\cal E}^{2p-2}p^{2}\hspace{-1pt}\left(m_{h}\hspace{-1pt}-\hspace{-1pt}m_{l}\right)\hspace{-1pt}+\hspace{-1pt}\gamma_{n-1}^{2}2{\cal E}^{2p}.\hspace{-2pt}\hspace{-2pt}\hspace{-2pt}
\end{align}
Now, for simplicity and clarity, let us define a strictly \textit{problem
dependent} constant
\begin{multline}
\quad\quad\quad\Sigma\hspace{-1pt}\triangleq\max\left\{ \vphantom{\left(2c\mathsf{R}_{p}+1\right)^{2}}\left(2c\mathsf{R}_{p}+1\right)^{2}G^{2},8\mathsf{B}_{p}^{2}G^{2}c^{2},2\left(2c\mathsf{R}_{p}+1\right)^{2}G^{4},2V,\right.\\
\left.\vphantom{\left(2c\mathsf{R}_{p}+1\right)^{2}}16G^{4}{\cal E}^{2p-2}p^{2}\left(2c\mathsf{R}_{p}+1\right)^{2},4{\cal E}^{2p-2}p^{2}\left(m_{h}-m_{l}\right),2{\cal E}^{2p}\right\} .\quad\quad\quad
\end{multline}
Then, it is true that
\begin{flalign}
A^{n+1} & \le\left(1-\sigma\alpha_{n}\right)A^{n}+\alpha_{n}^{2}\Sigma+\alpha_{n}\dfrac{\Sigma}{\sigma}B^{n}+\alpha_{n}\dfrac{\Sigma}{\sigma}C^{n},\quad\text{with}\label{eq:A}\\
B^{n} & \le\left(1-\beta_{n-1}\right)B^{n-1}+\dfrac{\alpha_{n-1}^{2}}{\beta_{n-1}}\Sigma+\beta_{n-1}^{2}\Sigma\quad\text{and}\label{eq:B}\\
C^{n} & \le\left(1-\gamma_{n-1}\right)C^{n-1}+\dfrac{\alpha_{n-1}^{2}}{\gamma_{n-1}}\Sigma+\dfrac{\beta_{n-1}^{2}}{\gamma_{n-1}}\Sigma+\gamma_{n-1}^{2}\Sigma,\quad\forall n\in\mathbb{N}^{+}.\label{eq:C}
\end{flalign}
Next, let $\left\{ \boldsymbol{\Delta}_{B}^{n}\ge0\right\} _{n\in\mathbb{N}^{+}}$
and $\left\{ \boldsymbol{\Delta}_{C}^{n}\ge0\right\} _{n\in\mathbb{N}^{+}}$
be two \textit{auxiliary correction sequences}, to be determined shortly.
It then follows that
\begin{equation}
\boldsymbol{\Delta}_{B}^{n}B^{n}\le\left(1-\beta_{n-1}\right)\boldsymbol{\Delta}_{B}^{n}B^{n}+\dfrac{\alpha_{n-1}^{2}}{\beta_{n-1}}\boldsymbol{\Delta}_{B}^{n}\Sigma+\beta_{n-1}^{2}\boldsymbol{\Delta}_{B}^{n}\Sigma,
\end{equation}
implying that, for every $n\in\mathbb{N}^{+}$,
\begin{align}
\left(\boldsymbol{\Delta}_{B}^{n}+\alpha_{n}\dfrac{\Sigma}{\sigma}\right)B^{n} & \le\left(1-\beta_{n-1}\right)\left(\boldsymbol{\Delta}_{B}^{n}+\alpha_{n}\dfrac{\Sigma}{\sigma}\right)B^{n-1}\nonumber \\
 & +\dfrac{\alpha_{n-1}^{2}}{\beta_{n-1}}\boldsymbol{\Delta}_{B}^{n}\Sigma+\beta_{n-1}^{2}\boldsymbol{\Delta}_{B}^{n}\Sigma+\dfrac{\alpha_{n}\alpha_{n-1}^{2}}{\beta_{n-1}}\dfrac{\Sigma^{2}}{\sigma}+\alpha_{n}\beta_{n-1}^{2}\dfrac{\Sigma^{2}}{\sigma}.\label{eq:Delta_B}
\end{align}
Exactly the same procedure for (\ref{eq:C}) yields
\begin{align}
\left(\boldsymbol{\Delta}_{C}^{n}+\alpha_{n}\dfrac{\Sigma}{\sigma}\right)C^{n} & \le\left(1-\gamma_{n-1}\right)\left(\boldsymbol{\Delta}_{C}^{n}+\alpha_{n}\dfrac{\Sigma}{\sigma}\right)C^{n-1}\nonumber \\
 & \hspace{-2pt}\hspace{-2pt}\hspace{-2pt}\hspace{-2pt}\hspace{-2pt}\hspace{-2pt}\hspace{-2pt}\hspace{-2pt}\hspace{-2pt}\hspace{-2pt}\hspace{-2pt}\hspace{-2pt}\hspace{-2pt}\hspace{-2pt}\hspace{-2pt}\hspace{-2pt}\hspace{-2pt}\hspace{-2pt}\hspace{-2pt}\hspace{-2pt}\hspace{-2pt}\hspace{-2pt}\hspace{-2pt}\hspace{-2pt}\hspace{-2pt}\hspace{-2pt}+\dfrac{\alpha_{n-1}^{2}}{\gamma_{n-1}}\boldsymbol{\Delta}_{C}^{n}\Sigma+\dfrac{\beta_{n-1}^{2}}{\gamma_{n-1}}\boldsymbol{\Delta}_{C}^{n}\Sigma+\gamma_{n-1}^{2}\boldsymbol{\Delta}_{C}^{n}\Sigma+\dfrac{\alpha_{n}\alpha_{n-1}^{2}}{\gamma_{n-1}}\dfrac{\Sigma^{2}}{\sigma}+\dfrac{\alpha_{n}\beta_{n-1}^{2}}{\gamma_{n-1}}\dfrac{\Sigma^{2}}{\sigma}+\alpha_{n}\gamma_{n-1}^{2}\dfrac{\Sigma^{2}}{\sigma},\label{eq:Delta_C}
\end{align}
for all $n\in\mathbb{N}^{+}$. Simply combining (\ref{eq:A}), (\ref{eq:Delta_B})
and (\ref{eq:Delta_C}), we obtain
\begin{flalign}
 & \hspace{-2pt}\hspace{-2pt}\hspace{-2pt}\hspace{-2pt}A^{n+1}+\left(\boldsymbol{\Delta}_{B}^{n}+\alpha_{n}\dfrac{\Sigma}{\sigma}\right)B^{n}+\left(\boldsymbol{\Delta}_{C}^{n}+\alpha_{n}\dfrac{\Sigma}{\sigma}\right)C^{n}\nonumber \\
 & \hspace{-2pt}\hspace{-2pt}\le\left(1-\sigma\alpha_{n}\right)A^{n}+\left(1-\beta_{n-1}\right)\left(\boldsymbol{\Delta}_{B}^{n}+\alpha_{n}\dfrac{\Sigma}{\sigma}\right)B^{n-1}+\left(1-\gamma_{n-1}\right)\left(\boldsymbol{\Delta}_{C}^{n}+\alpha_{n}\dfrac{\Sigma}{\sigma}\right)C^{n-1}\nonumber \\
 & \hspace{-2pt}\hspace{-2pt}+\alpha_{n}^{2}\Sigma+\alpha_{n}\dfrac{\Sigma}{\sigma}B^{n}+\alpha_{n}\dfrac{\Sigma}{\sigma}C^{n}\nonumber \\
 & \hspace{-2pt}\hspace{-2pt}+\dfrac{\alpha_{n-1}^{2}}{\beta_{n-1}}\boldsymbol{\Delta}_{B}^{n}\Sigma\hspace{-1pt}+\hspace{-1pt}\beta_{n-1}^{2}\boldsymbol{\Delta}_{B}^{n}\Sigma\hspace{-1pt}+\hspace{-1pt}\dfrac{\alpha_{n}\alpha_{n-1}^{2}}{\beta_{n-1}}\dfrac{\Sigma^{2}}{\sigma}\hspace{-1pt}+\hspace{-1pt}\alpha_{n}\beta_{n-1}^{2}\dfrac{\Sigma^{2}}{\sigma}\nonumber \\
 & \hspace{-2pt}\hspace{-2pt}+\dfrac{\alpha_{n-1}^{2}}{\gamma_{n-1}}\boldsymbol{\Delta}_{C}^{n}\Sigma\hspace{-1pt}+\hspace{-1pt}\dfrac{\beta_{n-1}^{2}}{\gamma_{n-1}}\boldsymbol{\Delta}_{C}^{n}\Sigma\hspace{-1pt}+\hspace{-1pt}\gamma_{n-1}^{2}\boldsymbol{\Delta}_{C}^{n}\Sigma\hspace{-1pt}+\hspace{-1pt}\dfrac{\alpha_{n}\alpha_{n-1}^{2}}{\gamma_{n-1}}\dfrac{\Sigma^{2}}{\sigma}\hspace{-1pt}+\hspace{-1pt}\dfrac{\alpha_{n}\beta_{n-1}^{2}}{\gamma_{n-1}}\dfrac{\Sigma^{2}}{\sigma}\hspace{-1pt}+\hspace{-1pt}\alpha_{n}\gamma_{n-1}^{2}\dfrac{\Sigma^{2}}{\sigma},\;\forall n\in\mathbb{N}^{+},
\end{flalign}
or, equivalently,
\begin{flalign}
 & \hspace{-2pt}\hspace{-2pt}\hspace{-2pt}\hspace{-2pt}A^{n+1}+\boldsymbol{\Delta}_{B}^{n}B^{n}+\boldsymbol{\Delta}_{C}^{n}C^{n}\nonumber \\
 & \hspace{-2pt}\hspace{-2pt}\le\left(1-\sigma\alpha_{n}\right)A^{n}+\left(1-\beta_{n-1}\right)\left(\boldsymbol{\Delta}_{B}^{n}+\alpha_{n}\dfrac{\Sigma}{\sigma}\right)B^{n-1}+\left(1-\gamma_{n-1}\right)\left(\boldsymbol{\Delta}_{C}^{n}+\alpha_{n}\dfrac{\Sigma}{\sigma}\right)C^{n-1}\nonumber \\
 & \hspace{-2pt}\hspace{-2pt}+\alpha_{n}^{2}\Sigma\nonumber \\
 & \hspace{-2pt}\hspace{-2pt}+\dfrac{\alpha_{n-1}^{2}}{\beta_{n-1}}\boldsymbol{\Delta}_{B}^{n}\Sigma\hspace{-1pt}+\hspace{-1pt}\beta_{n-1}^{2}\boldsymbol{\Delta}_{B}^{n}\Sigma\hspace{-1pt}+\hspace{-1pt}\dfrac{\alpha_{n}\alpha_{n-1}^{2}}{\beta_{n-1}}\dfrac{\Sigma^{2}}{\sigma}\hspace{-1pt}+\hspace{-1pt}\alpha_{n}\beta_{n-1}^{2}\dfrac{\Sigma^{2}}{\sigma}\nonumber \\
 & \hspace{-2pt}\hspace{-2pt}+\dfrac{\alpha_{n-1}^{2}}{\gamma_{n-1}}\boldsymbol{\Delta}_{C}^{n}\Sigma\hspace{-1pt}+\hspace{-1pt}\dfrac{\beta_{n-1}^{2}}{\gamma_{n-1}}\boldsymbol{\Delta}_{C}^{n}\Sigma\hspace{-1pt}+\hspace{-1pt}\gamma_{n-1}^{2}\boldsymbol{\Delta}_{C}^{n}\Sigma\hspace{-1pt}+\hspace{-1pt}\dfrac{\alpha_{n}\alpha_{n-1}^{2}}{\gamma_{n-1}}\dfrac{\Sigma^{2}}{\sigma}\hspace{-1pt}+\hspace{-1pt}\dfrac{\alpha_{n}\beta_{n-1}^{2}}{\gamma_{n-1}}\dfrac{\Sigma^{2}}{\sigma}\hspace{-1pt}+\hspace{-1pt}\alpha_{n}\gamma_{n-1}^{2}\dfrac{\Sigma^{2}}{\sigma},\;\forall n\in\mathbb{N}^{+}.\label{eq:REC_1}
\end{flalign}
Driven by the form of (\ref{eq:REC_1}), we would first like to choose
$\boldsymbol{\Delta}_{B}^{n}$ and $\boldsymbol{\Delta}_{C}^{n}$
such that, \textit{at least for $n$ sufficiently large},
\begin{flalign}
\left(1-\beta_{n-1}\right)\left(\boldsymbol{\Delta}_{B}^{n}+\alpha_{n}\dfrac{\Sigma}{\sigma}\right) & \le\left(1-\sigma\alpha_{n}\right)\boldsymbol{\Delta}_{B}^{n},\quad\text{and}\label{eq:Conditions_1}\\
\left(1-\gamma_{n-1}\right)\left(\boldsymbol{\Delta}_{C}^{n}+\alpha_{n}\dfrac{\Sigma}{\sigma}\right) & \le\left(1-\sigma\alpha_{n}\right)\boldsymbol{\Delta}_{C}^{n}.
\end{flalign}
The following procedure is exactly the same for both terms, so let
us take, say, $\boldsymbol{\Delta}_{B}^{n}$. We may write (\ref{eq:Conditions_1})
equivalently as
\begin{flalign}
\left(1-\beta_{n-1}\right)\alpha_{n}\dfrac{\Sigma}{\sigma} & \le\left(\beta_{n-1}-\sigma\alpha_{n}\right)\boldsymbol{\Delta}_{B}^{n}.
\end{flalign}
Since, by condition ${\bf G1}$,
\begin{equation}
\sigma\alpha_{n}\le\dfrac{K-1}{K}\min\left\{ \beta_{n-1},\gamma_{n-1}\right\} <\beta_{n-1}\le1,
\end{equation}
for every $n\in\mathbb{N}^{n_{o}}$ and some globally fixed $n_{o}\in\mathbb{N}^{+}$,
(\ref{eq:Conditions_1}) is also equivalent to
\begin{equation}
\dfrac{\left(1-\beta_{n-1}\right)\alpha_{n}}{\beta_{n-1}-\sigma\alpha_{n}}\dfrac{\Sigma}{\sigma}\le\boldsymbol{\Delta}_{B}^{n},\quad\forall n\in\mathbb{N}^{n_{o}}.
\end{equation}
Therefore, it suffices to choose
\begin{equation}
\boldsymbol{\Delta}_{B}^{n}\triangleq\dfrac{\alpha_{n}}{\beta_{n-1}-\sigma\alpha_{n}}\dfrac{\Sigma}{\sigma},\quad\forall n\in\mathbb{N}^{n_{o}}.
\end{equation}
Additionally, by defining $\boldsymbol{\Delta}_{B}^{n_{o}-1}\triangleq\boldsymbol{\Delta}_{B}^{n_{o}}$,
it is easy to see that conditions ${\bf G1}$ and ${\bf G2}$ imply
that
\begin{flalign}
\dfrac{\alpha_{n}}{\beta_{n-1}}\dfrac{\Sigma}{\sigma}\le\boldsymbol{\Delta}_{B}^{n} & \le K\dfrac{\alpha_{n}}{\beta_{n-1}}\dfrac{\Sigma}{\sigma}\quad\text{and}\\
\boldsymbol{\Delta}_{B}^{n} & \le\boldsymbol{\Delta}_{B}^{n-1},\quad\forall n\in\mathbb{N}^{n_{o}},
\end{flalign}
respectively. Of course, similar results hold in a completely analogous
fashion for $\boldsymbol{\Delta}_{C}^{n}$, for every $n\in\mathbb{N}^{n_{o}}$.
Consequently, letting
\begin{equation}
J^{n}\triangleq A^{n}+\boldsymbol{\Delta}_{B}^{n-1}B^{n-1}+\boldsymbol{\Delta}_{C}^{n-1}C^{n-1},\quad\forall n\in\mathbb{N}^{+},
\end{equation}
and defining the constant
\begin{equation}
\widetilde{\Sigma}\triangleq\dfrac{\Sigma}{\sigma^{2}}\max\left\{ \left(K+1\right)\dfrac{\Sigma}{\sigma^{2}},\left(K+1\right)\Sigma,1\right\} ,
\end{equation}
standard manipulations show that the RHS of expression (\ref{eq:REC_1})
may be further bounded as
\begin{equation}
\hspace{-1pt}\hspace{-1pt}\hspace{-1pt}\hspace{-1pt}J^{n+1}\hspace{-1pt}\hspace{-1pt}\le\hspace{-1pt}\left(1\hspace{-1pt}-\hspace{-1pt}\sigma\alpha_{n}\right)J^{n}\hspace{-1pt}+\hspace{-1pt}\widetilde{\Sigma}\left(\hspace{-1pt}\sigma^{2}\alpha_{n}^{2}\hspace{-1pt}+\hspace{-1pt}\dfrac{\sigma^{3}\alpha_{n}\alpha_{n-1}^{2}}{\beta_{n-1}^{2}}\hspace{-1pt}+\hspace{-1pt}\sigma\alpha_{n}\beta_{n-1}\hspace{-1pt}+\hspace{-1pt}\dfrac{\sigma^{3}\alpha_{n}\alpha_{n-1}^{2}}{\gamma_{n-1}^{2}}\hspace{-1pt}+\hspace{-1pt}\dfrac{\sigma\alpha_{n}\beta_{n-1}^{2}}{\gamma_{n-1}^{2}}\hspace{-1pt}+\hspace{-1pt}\sigma\alpha_{n}\gamma_{n-1}\hspace{-1pt}\right)\hspace{-1pt}\hspace{-1pt},
\end{equation}
for all $n\in\mathbb{N}^{n_{o}}$.

Lastly, let us now show the last part of Lemma \ref{lem:Rate_Generator}.
If, additionally, the assumptions of Lemma \ref{lem:Iter_Error_Bound}
are in effect, it follows that
\begin{equation}
\sup_{n\in\mathbb{N}}B^{n}<\infty\quad\text{and}\quad\sup_{n\in\mathbb{N}}C^{n}<\infty.
\end{equation}
Thus, we may write
\begin{equation}
A^{n+1}\le\left(1-\sigma\alpha_{n}\right)A^{n}+\alpha_{n}^{2}\sigma^{2}\Lambda+\alpha_{n}\sigma\Lambda,\quad\forall n\in\mathbb{N}.
\end{equation}
where
\begin{equation}
\Lambda\triangleq\dfrac{G^{2}}{\sigma^{2}}\max\left\{ \hspace{-2pt}\left(2c\mathsf{R}_{p}+1\right)^{2},8\mathsf{B}_{p}^{2}c^{2}\left(\sup_{n\in\mathbb{N}}B^{n}+\sup_{n\in\mathbb{N}}C^{n}\right)\hspace{-2pt}\right\} 
\end{equation}
We now make use of conditions ${\bf G1}$ and ${\bf G2}$. It is true
that
\begin{flalign}
A^{1} & \le\left(1-\sigma\alpha_{0}\right)A^{0}+\alpha_{0}^{2}\sigma^{2}\Lambda+\alpha_{0}\sigma\Lambda\nonumber \\
 & \le A^{0}+\max\left\{ \sup_{n\in\mathbb{N}_{n_{o}-1}}\alpha_{n}^{2}\sigma^{2},1\right\} \Lambda+\sup_{n\in\mathbb{N}}\alpha_{n}\sigma\Lambda\triangleq D<\infty,
\end{flalign}
and, of course, $A^{0}\le D$. We use simple induction. Suppose that,
for some $n\in\mathbb{N}$, $A^{n}\le D$. Then, there are two possibilities
for $\sigma\alpha_{n}\ge0$. Either $\sigma\alpha_{n}>1$, which is
of course only possible if $n\in\mathbb{N}_{n_{o}-1}$, in which case
we have
\begin{flalign}
A^{n+1} & \le\left(1-\sigma\alpha_{n}\right)A^{n}+\alpha_{n}^{2}\sigma^{2}\Lambda+\alpha_{n}\sigma\Lambda\nonumber \\
 & \le\alpha_{n}^{2}\sigma^{2}\Lambda+\alpha_{n}\sigma\Lambda\nonumber \\
 & \le\max\left\{ \sup_{n\in\mathbb{N}_{n_{o}-1}}\alpha_{n}^{2}\sigma^{2},1\right\} \Lambda+\sup_{n\in\mathbb{N}}\alpha_{n}\sigma\Lambda\nonumber \\
 & \equiv D,
\end{flalign}
or $\sigma\alpha_{n}\le1$, which might happen for any $n\in\mathbb{N}$,
yielding
\begin{flalign}
A^{n+1} & \le\left(1-\sigma\alpha_{n}\right)A^{n}+\alpha_{n}^{2}\sigma^{2}\Lambda+\alpha_{n}\sigma\Lambda\nonumber \\
 & \le\left(1-\sigma\alpha_{n}\right)D+\alpha_{n}\sigma\left(\alpha_{n}\sigma\Lambda+\Lambda\right)\nonumber \\
 & \le\left(1-\sigma\alpha_{n}\right)D+\alpha_{n}\sigma\left(\sup_{n\in\mathbb{N}}\alpha_{n}\sigma\Lambda+\max\left\{ \sup_{n\in\mathbb{N}_{n_{o}-1}}\alpha_{n}^{2}\sigma^{2},1\right\} \Lambda\right)\nonumber \\
 & \le D-\sigma\alpha_{n}D+\alpha_{n}\sigma D\nonumber \\
 & \equiv D.
\end{flalign}
As a result, we have shown that $A^{n+1}\le D$, as well, implying
that
\begin{equation}
\sup_{n\in\mathbb{N}}A^{n}\le D<\infty.
\end{equation}
By definition of $J^{n}$, for $n\in\mathbb{N}^{+}$, we may write
(utilizing condition ${\bf G2}$)
\begin{flalign}
\sup_{n\in\mathbb{N}^{+}}J^{n} & \equiv\sup_{n\in\mathbb{N}^{+}}A^{n}+\boldsymbol{\Delta}_{B}^{n-1}B^{n-1}+\boldsymbol{\Delta}_{C}^{n-1}C^{n-1}\nonumber \\
 & \le\sup_{n\in\mathbb{N}^{+}}A^{n}+\sup_{n\in\mathbb{N}^{+}}\boldsymbol{\Delta}_{B}^{n-1}\sup_{n\in\mathbb{N}^{+}}B^{n-1}+\sup_{n\in\mathbb{N}^{+}}\boldsymbol{\Delta}_{C}^{n-1}\sup_{n\in\mathbb{N}^{+}}C^{n-1}\nonumber \\
 & \equiv\sup_{n\in\mathbb{N}^{+}}A^{n}+\max\hspace{-1pt}\hspace{-1pt}\left\{ \sup_{n\in\mathbb{N}^{n_{o}-2}}\hspace{-1pt}\boldsymbol{\Delta}_{B}^{n},\boldsymbol{\Delta}_{B}^{n_{o}-1}\hspace{-1pt}\right\} \hspace{-1pt}\sup_{n\in\mathbb{N}^{+}}B^{n-1}+\max\hspace{-1pt}\hspace{-1pt}\left\{ \sup_{n\in\mathbb{N}^{n_{o}-2}}\hspace{-1pt}\boldsymbol{\Delta}_{C}^{n},\boldsymbol{\Delta}_{C}^{n_{o}-1}\hspace{-1pt}\right\} \hspace{-1pt}\sup_{n\in\mathbb{N}^{+}}C^{n-1}\nonumber \\
 & <\infty,
\end{flalign}
and the proof is now complete.\hfill{}\ensuremath{\blacksquare}

\subsection{\label{subsec:Rate_Sconvex}Proof of Lemma \ref{lem:Rate_SConvex}}

First, let us verify conditions $\mathbf{G1}$ and $\mathbf{G2}$
of Lemma \ref{lem:Rate_Generator}. For $\mathbf{G1}$, we perform,
for every $n\in\mathbb{N}^{2}$, the equivalence test
\begin{equation}
\sigma\alpha_{n}\equiv\dfrac{1}{n}\le\dfrac{K-1}{K}\dfrac{1}{\left(n-1\right)^{\tau_{2}}}\equiv\dfrac{K-1}{K}\min\left\{ \beta_{n-1},\gamma_{n-1}\right\} \iff K\ge\dfrac{1}{1-\dfrac{\left(n-1\right)^{\tau_{2}}}{n}},
\end{equation}
which implies that any $K\ge2$ works. Therefore, $\mathbf{G1}$ is
satisfied for all $n\in\mathbb{N}^{2}$ by choosing, say, $K\equiv2$.
To verify $\mathbf{G2}$ for the sequence $\left\{ \beta_{n}\right\} _{n\in\mathbb{N}}$,
we would also like to show that
\begin{equation}
\alpha_{n+1}\beta_{n-1}\equiv\dfrac{1}{\sigma\left(n+1\right)}\dfrac{1}{\left(n-1\right)^{\tau_{2}}}\le\dfrac{1}{\sigma n}\dfrac{1}{n^{\tau_{2}}}\equiv\alpha_{n}\beta_{n},
\end{equation}
for all \textit{sufficiently large} $n\in\mathbb{N}^{2}$. Indeed,
it is a standard calculus exercise to show that
\begin{equation}
\dfrac{n^{\tau_{2}}}{\left(n-1\right)^{\tau_{2}}}\le\dfrac{n+1}{n},\quad\forall n\in\left[\dfrac{1}{1-\tau_{2}^{1/\left(\tau_{2}+1\right)}},\infty\right)\bigcap\mathbb{N}\subseteq\mathbb{N}^{3}.
\end{equation}
To verify $\mathbf{G2}$ for the sequence $\left\{ \gamma_{n}\right\} _{n\in\mathbb{N}}$,
it suffices to observe that
\begin{equation}
\dfrac{n^{\tau_{2}}}{\left(n-1\right)^{\tau_{2}}}>\dfrac{n^{\tau_{3}}}{\left(n-1\right)^{\tau_{3}}},\quad\forall n\in\mathbb{N}^{2},
\end{equation}
implying that
\begin{equation}
\dfrac{n^{\tau_{3}}}{\left(n-1\right)^{\tau_{3}}}\le\dfrac{n+1}{n}\:\iff\:\alpha_{n+1}\gamma_{n-1}\le\alpha_{n}\gamma_{n},\ \forall n\in\left[\dfrac{1}{1-\tau_{2}^{1/\left(\tau_{2}+1\right)}},\infty\right)\bigcap\mathbb{N}.
\end{equation}
Therefore, we may apply Lemma \ref{lem:Rate_Generator} by choosing
\begin{equation}
n_{o}\equiv\text{\ensuremath{n_{o}}}\left(\tau_{2}\right)\equiv\left\lceil \dfrac{1}{1-\tau_{2}^{1/\left(\tau_{2}+1\right)}}\right\rceil ,\label{eq:n_depend}
\end{equation}
in which case it must be true that, for every $n\in\mathbb{N}^{n_{o}}$,
\begin{flalign}
J^{n+1} & \le\left(1-\sigma\alpha_{n}\right)J^{n}+\widetilde{\Sigma}\left(\sigma^{2}\alpha_{n}^{2}+\dfrac{\sigma^{3}\alpha_{n}\alpha_{n-1}^{2}}{\beta_{n-1}^{2}}+\sigma\alpha_{n}\beta_{n-1}+\dfrac{\sigma^{3}\alpha_{n}\alpha_{n-1}^{2}}{\gamma_{n-1}^{2}}+\dfrac{\sigma\alpha_{n}\beta_{n-1}^{2}}{\gamma_{n-1}^{2}}+\sigma\alpha_{n}\gamma_{n-1}\right)\nonumber \\
 & \le\left(1-\sigma\alpha_{n}\right)J^{n}+\widetilde{\Sigma}\left(\sigma^{2}\alpha_{n}^{2}+\dfrac{\sigma^{3}\alpha_{n-1}^{3}}{\beta_{n-1}^{2}}+\sigma\alpha_{n-1}\beta_{n-1}+\dfrac{\sigma^{3}\alpha_{n-1}^{3}}{\gamma_{n-1}^{2}}+\dfrac{\sigma\alpha_{n-1}\beta_{n-1}^{2}}{\gamma_{n-1}^{2}}+\sigma\alpha_{n-1}\gamma_{n-1}\right)\nonumber \\
 & \le\left(1-\sigma\alpha_{n}\right)J^{n}+\widetilde{\Sigma}\left(\sigma^{2}\alpha_{n}^{2}+2\dfrac{\sigma^{3}\alpha_{n-1}^{3}}{\beta_{n-1}^{2}}+\dfrac{\sigma\alpha_{n-1}\beta_{n-1}^{2}}{\gamma_{n-1}^{2}}+2\sigma\alpha_{n-1}\gamma_{n-1}\right)\nonumber \\
 & \equiv\left(1-\dfrac{1}{n}\right)J^{n}+\widetilde{\Sigma}\left(\dfrac{1}{n^{2}}+2\dfrac{1}{\left(n-1\right)^{3-2\tau_{2}}}+\dfrac{1}{\left(n-1\right)^{1+2\tau_{2}-2\tau_{3}}}+2\dfrac{1}{\left(n-1\right)^{1+\tau_{3}}}\right)\nonumber \\
 & \le\left(1-\dfrac{1}{n}\right)J^{n}+\widetilde{\Sigma}\left(\dfrac{1}{n^{2}}+2\Lambda_{2}\left(\tau_{2}\right)\dfrac{1}{n^{3-2\tau_{2}}}+\Lambda_{3}\left(\tau_{2}\right)\dfrac{1}{n^{1+2\tau_{2}-2\tau_{3}}}+2\Lambda_{23}\left(\tau_{2}\right)\dfrac{1}{n^{1+\tau_{3}}}\right),
\end{flalign}
where we have used the fact that, for our choice of the sequence $\left\{ \alpha_{n}\right\} _{n\in\mathbb{N}}$,
it is true that $\alpha_{n}\le\alpha_{n-1}$, for all $n\in\mathbb{N}^{+}$,
and the constants $\Lambda_{2}$,$\Lambda_{3}$ and $\Lambda_{23}$
are defined as
\begin{flalign}
\Lambda_{2}\left(\tau_{2}\right) & \triangleq\left(\dfrac{n_{o}\left(\tau_{2}\right)}{n_{o}\left(\tau_{2}\right)-1}\right)^{3-2\tau_{2}}<8,\\
\Lambda_{3}\left(\tau_{2}\right) & \triangleq\left(\dfrac{n_{o}\left(\tau_{2}\right)}{n_{o}\left(\tau_{2}\right)-1}\right)^{1+2\tau_{2}-2\tau_{3}}<8\quad\text{and}\\
\Lambda_{23}\left(\tau_{2}\right) & \triangleq\left(\dfrac{n_{o}\left(\tau_{2}\right)}{n_{o}\left(\tau_{2}\right)-1}\right)^{1+\tau_{3}}<4<8.
\end{flalign}
Consequently, we obtain the bound
\begin{equation}
J^{n+1}\le\left(1-\dfrac{1}{n}\right)J^{n}+16\widetilde{\Sigma}\left(\dfrac{1}{n^{2}}+\dfrac{1}{n^{3-2\tau_{2}}}+\dfrac{1}{n^{1+2\tau_{2}-2\tau_{3}}}+\dfrac{1}{n^{1+\tau_{3}}}\right),\label{eq:REC_SUB}
\end{equation}
being true for all $n\in\mathbb{N}^{n_{o}}$, where $\text{\ensuremath{n_{o}}}$
depends on $\tau_{2}$ according to (\ref{eq:n_depend}).

Let us now apply the Generalized Chung's Lemma (Lemma \ref{lem:Chung})
to the recursion (\ref{eq:REC_SUB}). For every $n\in\mathbb{N}^{n_{o}}$,
we have
\begin{flalign}
J^{n+1} & \le J^{n_{o}}\prod_{i\in\mathbb{N}_{n}^{n_{o}}}\left(1-\dfrac{1}{i}\right)+16\widetilde{\Sigma}\sum_{i\in\mathbb{N}_{n}^{n_{o}}}\left(\dfrac{1}{i^{2}}+\dfrac{1}{i^{3-2\tau_{2}}}+\dfrac{1}{i^{1+2\tau_{2}-2\tau_{3}}}+\dfrac{1}{i^{1+\tau_{3}}}\right)\prod_{j\in\mathbb{N}_{n}^{i+1}}\left(1-\dfrac{1}{j}\right)\nonumber \\
 & \equiv J^{n_{o}}\prod_{i\in\mathbb{N}_{n}^{n_{o}}}\left(\dfrac{i-1}{i}\right)+16\widetilde{\Sigma}\sum_{i\in\mathbb{N}_{n-1}^{n_{o}}}\left(\dfrac{1}{i^{2}}+\dfrac{1}{i^{3-2\tau_{2}}}+\dfrac{1}{i^{1+2\tau_{2}-2\tau_{3}}}+\dfrac{1}{i^{1+\tau_{3}}}\right)\prod_{j\in\mathbb{N}_{n}^{i+1}}\left(\dfrac{j-1}{j}\right)\nonumber \\
 & \quad\quad\quad\quad+16\widetilde{\Sigma}\left(\dfrac{1}{n^{2}}+\dfrac{1}{n^{3-2\tau_{2}}}+\dfrac{1}{n^{1+2\tau_{2}-2\tau_{3}}}+\dfrac{1}{n^{1+\tau_{3}}}\right)\nonumber \\
 & \equiv J^{n_{o}}\dfrac{n_{o}-1}{n}+16\widetilde{\Sigma}\sum_{i\in\mathbb{N}_{n-1}^{n_{o}}}\left(\dfrac{1}{i^{2}}+\dfrac{1}{i^{3-2\tau_{2}}}+\dfrac{1}{i^{1+2\tau_{2}-2\tau_{3}}}+\dfrac{1}{i^{1+\tau_{3}}}\right)\dfrac{i}{n}\nonumber \\
 & \quad\quad\quad\quad+16\widetilde{\Sigma}\left(\dfrac{1}{n^{2}}+\dfrac{1}{n^{3-2\tau_{2}}}+\dfrac{1}{n^{1+2\tau_{2}-2\tau_{3}}}+\dfrac{1}{n^{1+\tau_{3}}}\right)\nonumber \\
 & \equiv J^{n_{o}}\dfrac{n_{o}-1}{n}+\dfrac{16\widetilde{\Sigma}}{n}\sum_{i\in\mathbb{N}_{n-1}^{n_{o}}}\left(\dfrac{1}{i}+\dfrac{1}{i^{2-2\tau_{2}}}+\dfrac{1}{i^{2\tau_{2}-2\tau_{3}}}+\dfrac{1}{i^{\tau_{3}}}\right)\nonumber \\
 & \quad\quad\quad\quad+16\widetilde{\Sigma}\left(\dfrac{1}{n^{2}}+\dfrac{1}{n^{3-2\tau_{2}}}+\dfrac{1}{n^{1+2\tau_{2}-2\tau_{3}}}+\dfrac{1}{n^{1+\tau_{3}}}\right).\label{eq:mid_point}
\end{flalign}
Due to our assumption that $1/2\le\tau_{3}<\tau_{2}<1$, it holds
that $2-2\tau_{2}\neq1$ and $2\tau_{2}-2\tau_{3}\neq1$. Consequently,
it is true that
\begin{flalign}
\sum_{i\in\mathbb{N}_{n-1}^{n_{o}}}\dfrac{1}{i^{2-2\tau_{2}}} & <\dfrac{1}{n_{o}^{2-2\tau_{2}}}+\dfrac{n^{1-\left(2-2\tau_{2}\right)}}{1-\left(2-2\tau_{2}\right)},\\
\sum_{i\in\mathbb{N}_{n-1}^{n_{o}}}\dfrac{1}{i^{2\tau_{2}-2\tau_{3}}} & <\dfrac{1}{n_{o}^{2\tau_{2}-2\tau_{3}}}+\dfrac{n^{1-\left(2\tau_{2}-2\tau_{3}\right)}}{1-\left(2\tau_{2}-2\tau_{3}\right)}\quad\text{and}\\
\sum_{i\in\mathbb{N}_{n-1}^{n_{o}}}\dfrac{1}{i^{\tau_{3}}} & <\dfrac{1}{n_{o}^{\tau_{3}}}+\dfrac{n^{1-\tau_{3}}}{1-\tau_{3}},
\end{flalign}
whereas
\begin{equation}
\sum_{i\in\mathbb{N}_{n-1}^{n_{o}}}\dfrac{1}{i}<\dfrac{1}{n_{o}}+\log\left(n\right).
\end{equation}
By defining the quantity
\begin{equation}
\mathsf{R}\equiv\mathsf{R}\left(\tau_{2},\tau_{3}\right)\triangleq\dfrac{1}{1-\max\left\{ 2-2\tau_{2},2\tau_{2}-2\tau_{3},\tau_{3}\right\} }>1,
\end{equation}
we may further bound (\ref{eq:mid_point}) from above as
\begin{flalign}
J^{n+1} & \le J^{n_{o}}\dfrac{n_{o}-1}{n}+16\widetilde{\Sigma}\mathsf{R}\left(\dfrac{\log\left(n\right)}{n}+\dfrac{1}{n^{2-2\tau_{2}}}+\dfrac{1}{n^{2\tau_{2}-2\tau_{3}}}+\dfrac{1}{n^{\tau_{3}}}\right)\nonumber \\
 & \quad\quad\quad\quad+\dfrac{64\widetilde{\Sigma}}{n}+16\widetilde{\Sigma}\left(\dfrac{1}{n^{2}}+\dfrac{1}{n^{3-2\tau_{2}}}+\dfrac{1}{n^{1+2\tau_{2}-2\tau_{3}}}+\dfrac{1}{n^{1+\tau_{3}}}\right)\nonumber \\
 & \le\dfrac{J^{n_{o}}n_{o}+64\widetilde{\Sigma}}{n}+32\widetilde{\Sigma}\mathsf{R}\left(\dfrac{\log\left(n\right)}{n}+\dfrac{1}{n^{2-2\tau_{2}}}+\dfrac{1}{n^{2\tau_{2}-2\tau_{3}}}+\dfrac{1}{n^{\tau_{3}}}\right)\nonumber \\
 & \le\dfrac{n_{o}\left(J^{n_{o}}+64\widetilde{\Sigma}\right)}{n}+32\widetilde{\Sigma}\mathsf{R}\left(\dfrac{1}{n^{1/2}}+\dfrac{1}{n^{2-2\tau_{2}}}+\dfrac{1}{n^{2\tau_{2}-2\tau_{3}}}+\dfrac{1}{n^{\tau_{3}}}\right)\nonumber \\
 & \le\dfrac{n_{o}\left(J^{n_{o}}+64\widetilde{\Sigma}\right)}{n}+128\widetilde{\Sigma}\mathsf{R}\dfrac{1}{n^{\min\left\{ 1/2,2-2\tau_{2},2\tau_{2}-2\tau_{3},\tau_{3}\right\} }}\nonumber \\
 & \equiv\dfrac{n_{o}\left(J^{n_{o}}+64\widetilde{\Sigma}\right)}{n}+128\widetilde{\Sigma}\mathsf{R}\dfrac{1}{n^{\min\left\{ 2-2\tau_{2},2\tau_{2}-2\tau_{3}\right\} }},
\end{flalign}
For our stepsize choices, it is trivial to see that the remaining
condition $\mathbf{G3}$ of Lemma \ref{lem:Rate_Generator} is also
satisfied. Therefore, it must be true that
\begin{equation}
J^{n+1}\le\dfrac{n_{o}\left(\sup_{n\in\mathbb{N}^{+}}J^{n}+64\widetilde{\Sigma}\right)}{n}+\dfrac{128\widetilde{\Sigma}\mathsf{R}}{n^{\min\left\{ 2-2\tau_{2},2\tau_{2}-2\tau_{3}\right\} }},\quad\forall n\in\mathbb{N}^{n_{o}},
\end{equation}
where $\sup_{n\in\mathbb{N}^{+}}J^{n}<\infty$, and defining another
constant $\widehat{\Sigma}\triangleq\max\left\{ \left(\sup_{n\in\mathbb{N}^{+}}J^{n}+64\widetilde{\Sigma}\right),128\widetilde{\Sigma}\right\} $,
we end up with the inequality
\begin{equation}
\mathbb{E}\left\{ \left\Vert \boldsymbol{x}^{n+1}-\boldsymbol{x}^{*}\right\Vert _{2}^{2}\right\} \le J^{n+1}\le\dfrac{\widehat{\Sigma}n_{o}}{n}+\dfrac{\widehat{\Sigma}\mathsf{R}}{n^{\min\left\{ 2-2\tau_{2},2\tau_{2}-2\tau_{3}\right\} }},\quad\forall n\in\mathbb{N}^{n_{o}},\label{eq:RATE_1}
\end{equation}
completing the proof for first part Theorem \ref{lem:Rate_SConvex}.

To prove the second part, let
\begin{equation}
\tau_{2}\equiv\dfrac{3+\epsilon}{4}\quad\text{and}\quad\tau_{3}\equiv\dfrac{1+\delta\epsilon}{2},
\end{equation}
for some $\epsilon\in\left[0,1\right)$ and $\delta\in\left(0,1\right)$.
Then, for the exponents of the corresponding terms in (\ref{eq:RATE_1}),
we have the identities
\begin{flalign}
2-2\tau_{2} & \equiv\dfrac{1-\epsilon}{2}\quad\text{and}\\
2\tau_{2}-2\tau_{3} & \equiv\dfrac{1-\epsilon\left(2\delta-1\right)}{2},
\end{flalign}
out of which the first is the smallest. Additionally, it also true
that
\begin{flalign}
\mathsf{R}\left(\tau_{2},\tau_{3}\right)\equiv\mathsf{R}\left(\delta,\epsilon\right) & \equiv\dfrac{1}{1-\max\left\{ \dfrac{1-\epsilon}{2},\dfrac{1-\epsilon\left(2\delta-1\right)}{2},\dfrac{1+\delta\epsilon}{2}\right\} }\nonumber \\
 & \equiv\dfrac{1}{1-\dfrac{1}{2}\max\left\{ 1-\epsilon,1-\epsilon\left(2\delta-1\right),1+\delta\epsilon\right\} }\nonumber \\
 & \equiv\dfrac{2}{1-\epsilon\max\left\{ 1-2\delta,\delta\right\} }\nonumber \\
 & <\dfrac{2}{1-\epsilon}.
\end{flalign}
As a result, we may further bound (\ref{eq:RATE_1}) as
\begin{flalign}
\mathbb{E}\left\{ \left\Vert \boldsymbol{x}^{n+1}-\boldsymbol{x}^{*}\right\Vert _{2}^{2}\right\} \le J^{n+1} & \le\dfrac{\widehat{\Sigma}n_{o}\left(\epsilon\right)}{n}+\dfrac{\widehat{\Sigma}\mathsf{R}\left(\delta,\epsilon\right)}{n^{\min\left\{ \left(1-\epsilon\right)/2,\left(1-\epsilon\left(2\delta-1\right)\right)/2\right\} }}\nonumber \\
 & \equiv\dfrac{\widehat{\Sigma}n_{o}\left(\epsilon\right)}{n}+\dfrac{\widehat{\Sigma}\mathsf{R}\left(\delta,\epsilon\right)}{n^{\left(1-\epsilon\right)/2}}\nonumber \\
 & \le\dfrac{\widehat{\Sigma}n_{o}\left(\epsilon\right)}{n^{\left(1-4\epsilon\right)/2}}+\dfrac{\widehat{\Sigma}\dfrac{2}{1-\epsilon}}{n^{\left(1-\epsilon\right)/2}}\nonumber \\
 & \equiv\dfrac{\widehat{\Sigma}\left(n_{o}\left(\epsilon\right)+\dfrac{2}{1-\epsilon}\right)}{n^{\left(1-\epsilon\right)/2}},
\end{flalign}
for every $n\in\mathbb{N}^{n_{o}\text{\ensuremath{\left(\epsilon\right)}}}$.\hfill{}\ensuremath{\blacksquare}\vspace{25bp}

\[
{\fontsize{80}{80}\selectfont\ensuremath{\smiley}}
\]

\clearpage{}

\bibliographystyle{plainnat}
\phantomsection\addcontentsline{toc}{section}{\refname}\bibliography{library_fixed}

\begin{center}
\vspace{70bp}
{\fontsize{80}{80}\selectfont \S}
\par\end{center}
\end{document}